\documentclass[11pt,a4paper]{report}

\usepackage{titlesec}
\usepackage{fancyhdr}
\usepackage{a4wide}
\usepackage{graphicx}
\usepackage{float}
\usepackage{amssymb}
\usepackage{amsmath}
\usepackage{amsfonts}
\usepackage{dsfont}
\usepackage{amsthm}
\usepackage{color}
\usepackage{xcolor,colortbl}
\usepackage{mathrsfs}
\usepackage{array}
\usepackage{eucal}
\usepackage{tikz}
\usepackage[T1]{fontenc}
\usepackage{inputenc}
\usepackage[english]{babel}
\usepackage{lmodern}
\usepackage{hyperref}
\usepackage{cleveref}
\usepackage{mathtools}
\usepackage{geometry}
\usepackage{changepage}
\usepackage{zref}

\makeatletter
\zref@newprop{chapsec}{\thesection}
\zref@addprop{main}{chapsec}
\makeatother
\newcommand{\chapsecname}[1]{
  \hyperref[#1]{\S\zref[chapsec]{#1}}
}

\usepackage{bm}
\changepage{0pt}{}{}{}{}{0pt}{}{0pt}{6pt}
\usepackage[numbers]{natbib}
\usepackage{bbm}
\usepackage{pgfplots}
\newcommand{\Indicator}{\rho}
\definecolor{RoyalBlue}{cmyk}{1, 0.50, 0, 0}
\definecolor{Gray}{gray}{0.90}

\counterwithout{section}{chapter}
\renewcommand{\thechapter}{\Roman{chapter}}
\renewcommand{\thesection}{\thechapter.\arabic{section}}

\usepackage{hyperref}
\hypersetup{
pdfpagemode=none,
pdftoolbar=true,        
pdfmenubar=true,        
pdffitwindow=false,     
pdfstartview={Fit},    
pdftitle={A few techniques to achieve invisibility in waveguides},    
pdfauthor={L. Chesnel},     
pdfsubject={},  
pdfcreator={L. Chesnel},   
pdfproducer={L. Chesnel}, 
pdfkeywords={Waveguides, scattering, invisibility, asymptotic analysis, spectral theory, complex resonances, shape derivative}, 
pdfnewwindow=true,      
colorlinks=true,       
linkcolor=magenta,          
citecolor=red,        
filecolor=cyan,      
urlcolor=blue           
}

\newcommand{\dsp}{\displaystyle}
\newcommand{\eps}{\varepsilon}
\newcommand{\om}{\omega}
\newcommand{\Om}{\Omega}
\newcommand{\mrm}[1]{\mathrm{#1}}

\newcommand{\Cplx}{\mathbb{C}}
\newcommand{\N}{\mathbb{N}}
\newcommand{\R}{\mathbb{R}}
\newcommand{\Z}{\mathbb{Z}}

\newcommand{\mL}{\mrm{L}}
\newcommand{\mH}{\mrm{H}}
\newcommand{\mV}{\mrm{V}}

\newcommand{\mX}{\mrm{X}}
\newcommand{\mY}{\mrm{Y}}
\newcommand{\mZ}{\mrm{Z}}
\newcommand{\loc}{\mbox{\scriptsize loc}}
\renewcommand{\ker}{\mrm{ker}}

\renewcommand{\dim}{\mrm{dim}}

\newtheorem{theorem}{Theorem}[chapter]
\newtheorem{lemma}[theorem]{Lemma}
\newtheorem{remark}[theorem]{Remark}
\newtheorem{definition}[theorem]{Definition}
\newtheorem{corollary}[theorem]{Corollary}
\newtheorem{proposition}[theorem]{Proposition}

\begin{document}

\thispagestyle{empty}

\noindent\rule{\linewidth}{.8pt}
\texttt{Summer school EUR MINT 2025 -- Control, Inverse Problems and Spectral Theory}\\[-8pt]
\rule{\linewidth}{.8pt}

\vspace{0.5cm}
\begin{center}
\includegraphics[angle=0,width=14cm]{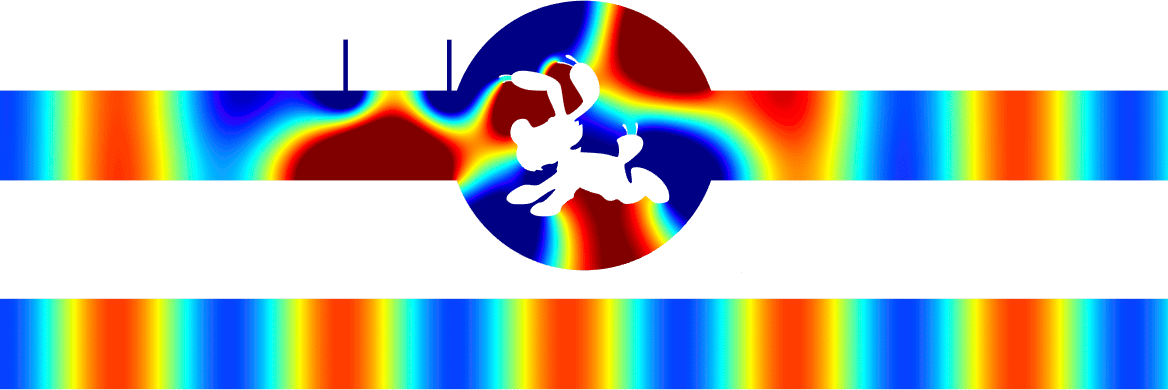}
\end{center}

\vspace{0.5cm}

\begin{center}
\tikz \node[rectangle,inner sep=12pt,fill=none,draw=black,line width=1mm,rounded corners=1mm]{
\begin{minipage}{0.8\textwidth} 
\begin{center}
{\sc \bf\fontsize{22}{22}\selectfont
A few techniques to achieve  \\[6pt]
 invisibility in waveguides }
\end{center}
\end{minipage}
};\\[16pt]

\textsc{\Large Lucas Chesnel}\\[10pt]
Version \today\\[10pt]
\begin{minipage}{0.95\textwidth}
{\small
\begin{tabular}{|l}
Inria, Ensta Paris, Institut Polytechnique de Paris \\[2pt]
E-mail: \texttt{lucas.chesnel@inria.fr} 
\end{tabular}\\[3pt]

}
\end{minipage}

\vspace{1cm}

\noindent \includegraphics[width=14cm]{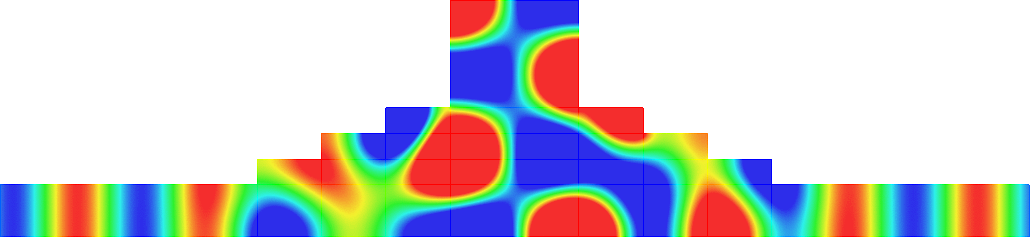}\\[10pt]
\includegraphics[width=14cm]{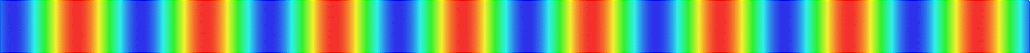}

\vspace{1.4cm}

\includegraphics[width=7.7cm]{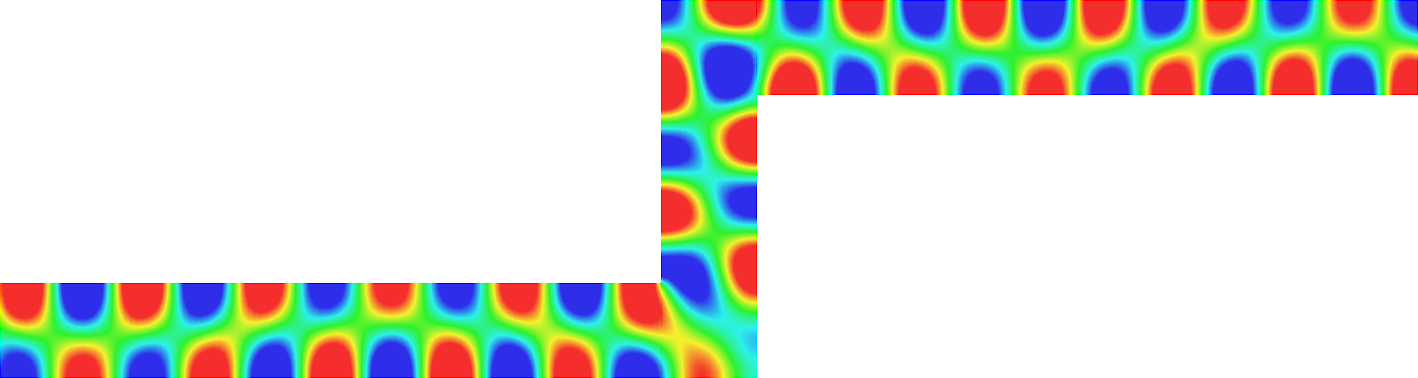}\qquad\includegraphics[trim={0cm 0.6cm 3cm 0.2cm},clip,width=7.7cm]{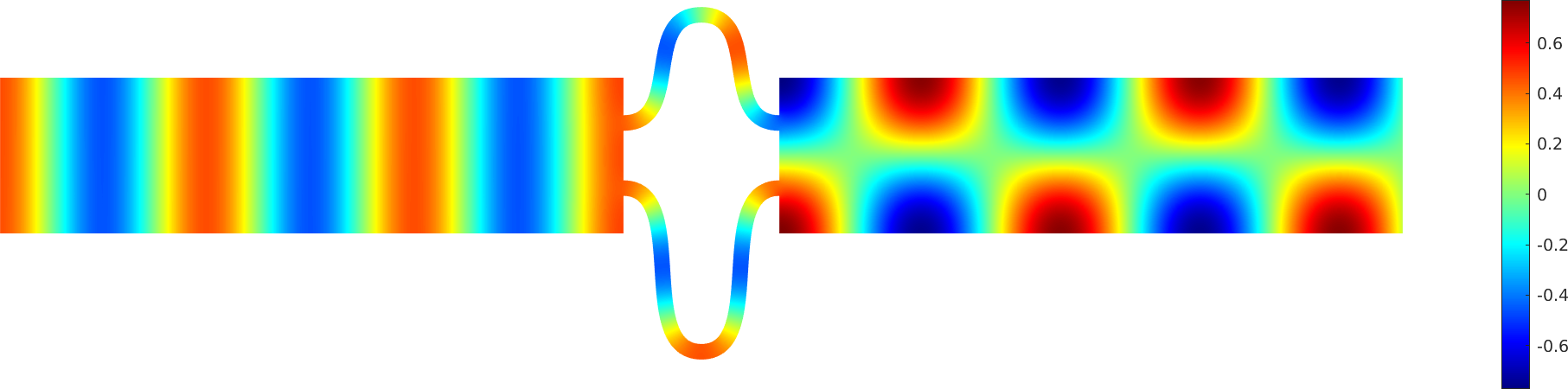}\\[20pt]\includegraphics[width=7.7cm]{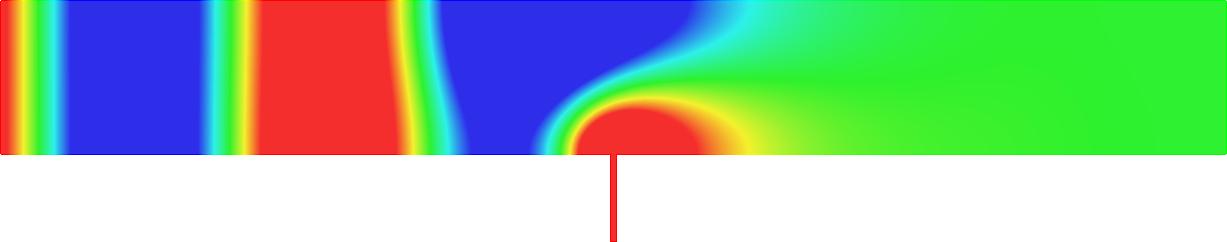}\qquad\includegraphics[trim={0cm 0.6cm 3cm 0.2cm},clip,width=7.7cm]{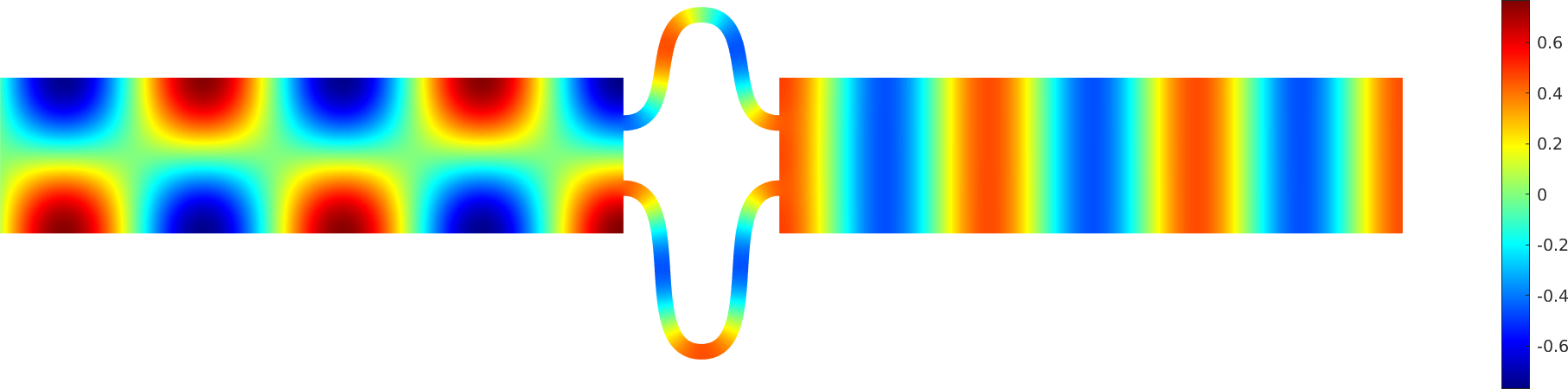}

\vspace{0.2cm}

\end{center}

\vspace{-0.4cm}

\newpage

\section*{Introduction}

\vspace{-0.7cm}

\begin{figure}[!ht]
\centering
\includegraphics[width=4.5cm]{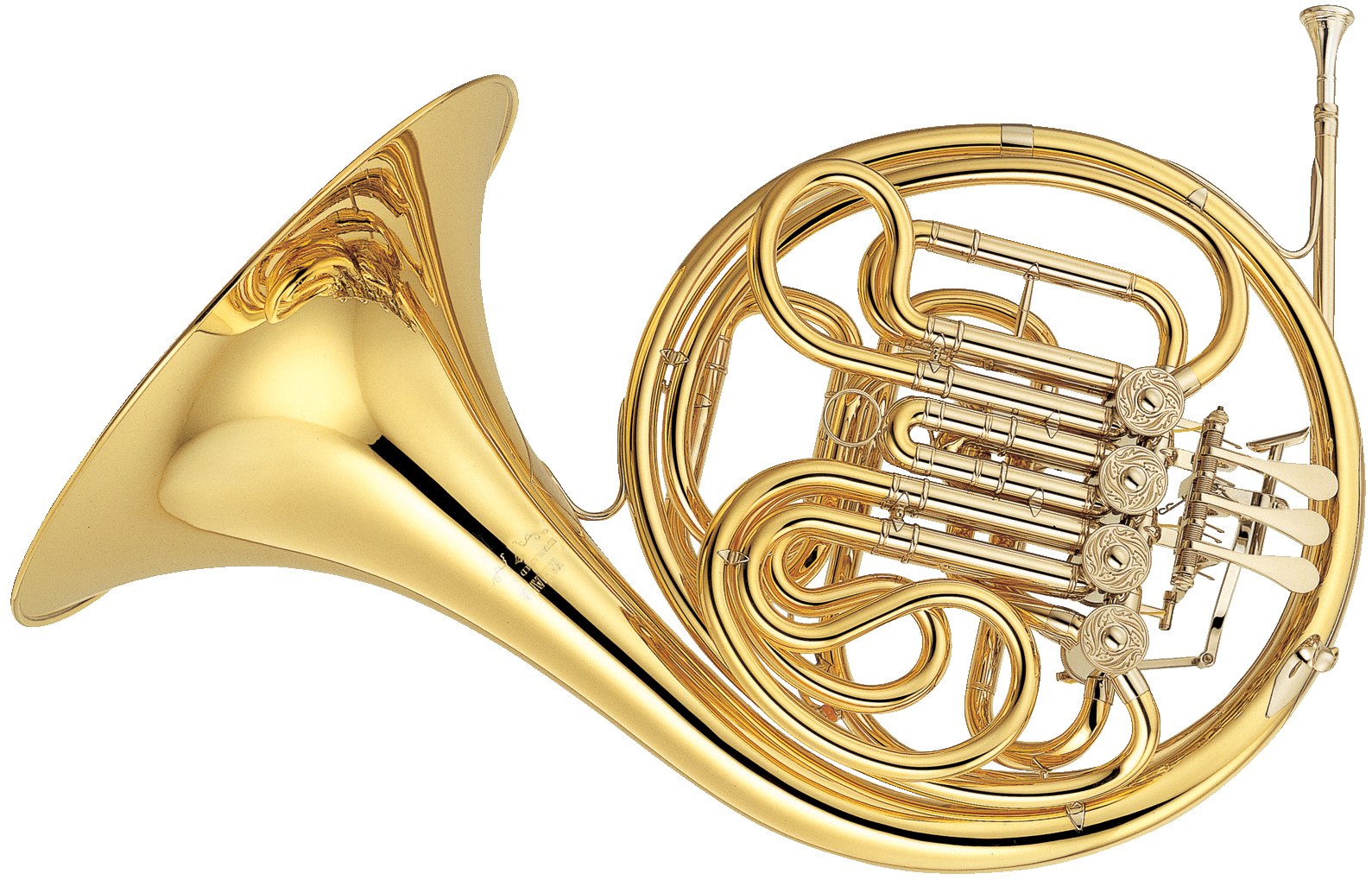}\qquad
\includegraphics[width=4.5cm]{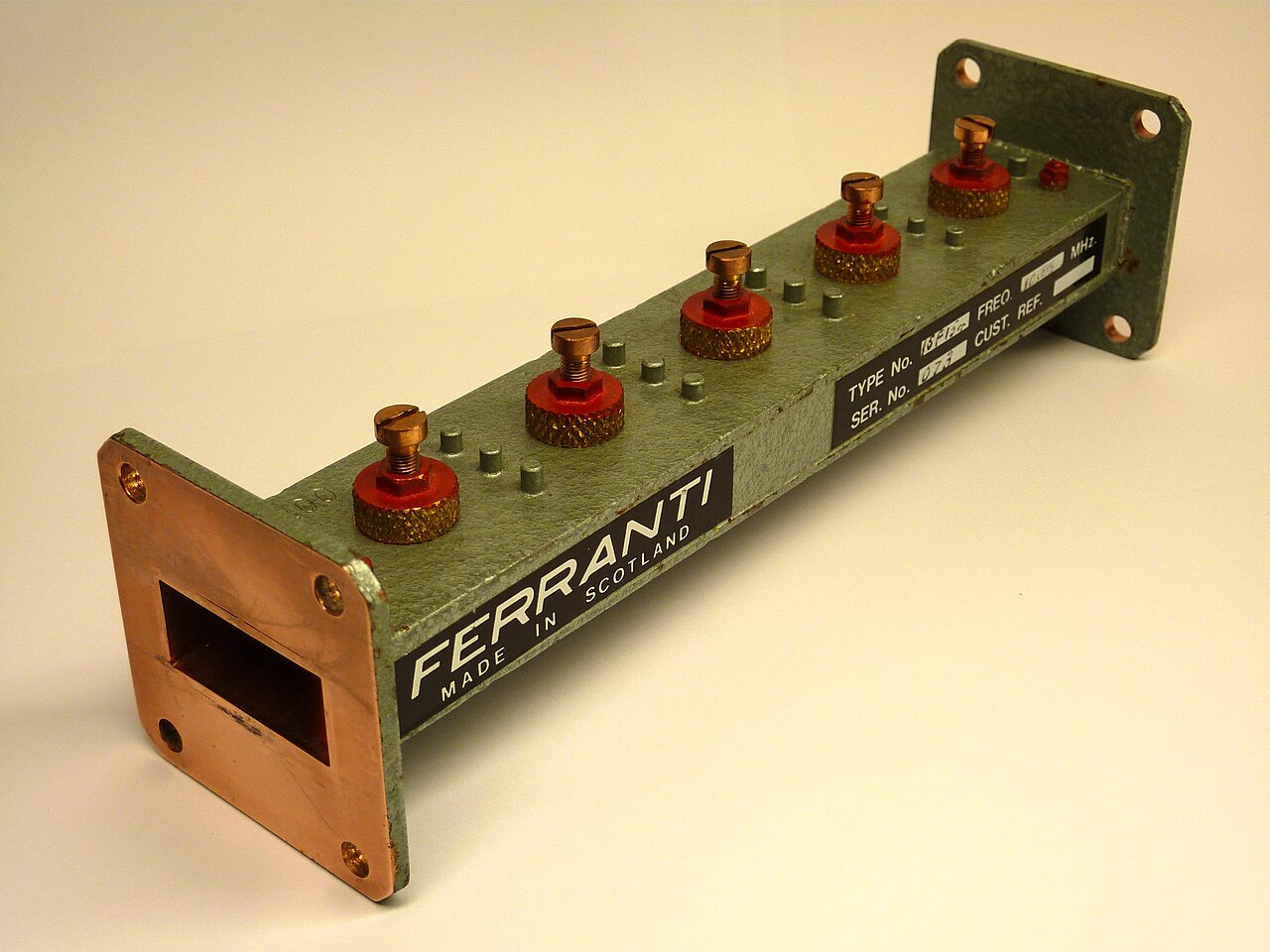}\quad\,
\includegraphics[width=4.5cm]{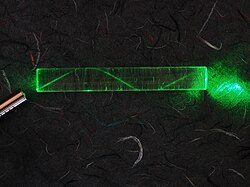}\\[-4pt]
\caption{Examples of waveguides (source Wikipedia). \label{WaveguideExamples}}
\end{figure}

\noindent The aim of these lecture notes is to consider a concrete problem, namely the identification of situations of invisibility in waveguides, to present techniques and tools of applied mathematics that can be useful in other contexts. We will be interested in the propagation of scalar waves in guides which are unbounded in one direction. Such problems arise in many fields of physics. For example air ducts and horns in acoustics carry sound waves in musical instruments as well as in loudspeakers.  Conductive metal pipes are exploited to propagate
high frequency radio waves while optical fibers serve as waveguides for light in electromagnetism. Waveguides problems also appear in water waves theory, in classical mechanics or in quantum mechanics. In general, the diffraction of an incident wave in such  structures in presence of an obstacle generates a reflection and a transmission characterized by some scattering coefficients. Broadly speaking, our goal is to play with the geometry, the frequency and/or the index material to control these scattering coefficients.\\
\newline
This document is divided in four chapters. In the first one, we present classical results concerning waveguide theory. This is a rather long story and the aim here is not to be exhaustive but instead to present the main ideas and ingredients that will be useful to address the invisibility problem. In Chapter \ref{chapterContinuation}, we develop perturbative techniques based in particular on the use of shape derivatives to design invisible defects of the reference geometry. With these approaches, in principle we construct small amplitude invisible obstacles. In Chapter \ref{SectionResonance}, we exploit resonant phenomena to provide examples of larger invisible obstacles. There, we also propose a method to hide given objects by perturbing (in a singular way) the boundaries of the waveguide. Finally, in Chapter \ref{SectionSpectral} we change the point of view, assume that the obstacle is given, and construct a non-selfadjoint operator whose eigenvalues coincide with frequencies such that there are incident fields which produce zero reflection.\\
\newline 
Our approaches mainly rely on techniques of asymptotic analysis as well as spectral theory for selfadjoint and non-selfadjoint operators. Wherever possible, we will illustrate the results by numerical experiments.\\
\newline
The first chapter contains classical material. In the next three,  more recent results are presented. They have been obtained with different collaborators, among them Antoine Bera, Anne-Sophie Bonnet-Ben Dhia, J\'er\'emy Heleine, Sergei Nazarov, Vincent Pagneux. I thank them warmly.\\
\newline
These lecture notes have been written as support material for a one-week course (5.5 hours in total) that I delivered at the Institut de Math\'ematiques de Toulouse in the period 23-27 June 2025 as part of the summer school Control, Inverse Problems and Spectral Theory. They have been proofread several times. However, it is always difficult to eliminate all typos. I would be grateful to anyone who finds any to send them to \url{lucas.chesnel@inria.fr}. Remarks, suggestions are also welcome. \\
\newline
\noindent\textbf{Key words.} Waveguides, scattering, invisibility, asymptotic analysis, spectral theory, complex resonances, shape derivative.

\newpage
\tableofcontents

\chapter{Waveguide problems}\label{ChapWaveguides}

\section{Setting}

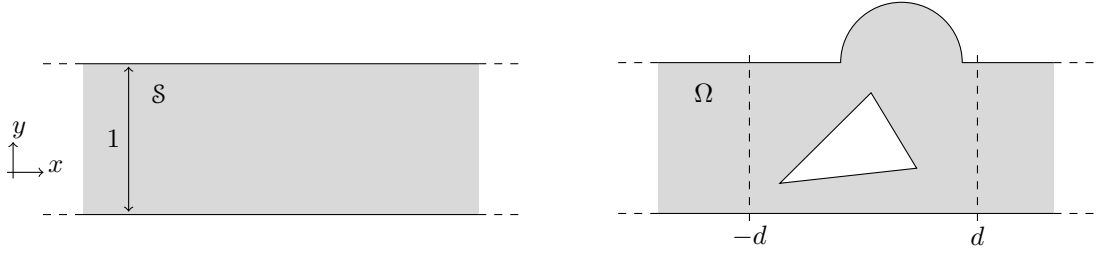
\begin{figure}[!ht]
\centering
\begin{tikzpicture}[scale=2]
\draw[fill=gray!30,draw=none](-1.3,0) rectangle (1.3,1);
\draw (-1.3,0)--(1.3,0);
\draw (-1.3,1)--(1.3,1);
\draw[dashed] (-1.3,0)--(-1.6,0);
\draw[dashed] (-1.3,1)--(-1.6,1);
\draw[dashed] (1.3,0)--(1.6,0);
\draw[dashed] (1.3,1)--(1.6,1);
\draw[<->] (-1,0.02)--(-1,0.98);
\node at (-1.1,0.5){\small $1$};
\begin{scope}[shift={(-1.8,0.2)},scale=0.4]
\draw[->] (0,0.2)--(0.6,0.2);
\draw[->] (0.1,0.1)--(0.1,0.7);
\node at (0.8,0.3){\small $x$};
\node at (0.2,0.9){\small $y$};
\end{scope}
\node at (-0.8,0.8){\small $\mathcal{S}$};
\phantom{\node at (0.8,-0.15){\small $d$};}
\end{tikzpicture}\qquad\quad
\begin{tikzpicture}[scale=2]
\draw[fill=gray!30] (0.3,1) circle (0.4) ;
\draw[fill=gray!30,draw=none](-1.3,0) rectangle (1.3,1);
\draw (-1.3,0)--(1.3,0);
\draw (-1.3,1)--(-0.1,1);
\draw (0.7,1)--(1.3,1);
\draw[fill=white] (-0.5,0.2)--(0.1,0.8)--(0.4,0.3)--cycle;
\draw[dashed] (-1.3,0)--(-1.6,0);
\draw[dashed] (-1.3,1)--(-1.6,1);
\draw[dashed] (1.3,0)--(1.6,0);
\draw[dashed] (1.3,1)--(1.6,1);
\node at (-1,0.8){\small $\Om$};
\draw[-,dashed] (-0.7,-0.05)--(-0.7,1.05);
\draw[-,dashed] (0.8,-0.05)--(0.8,1.05);
\node at (-0.7,-0.15){\small $-d$};
\node at (0.8,-0.15){\small $d$};
\end{tikzpicture}\vspace{-0.3cm}
\caption{Left: reference strip $\mathcal{S}$. Right: perturbed waveguide $\Om$.\label{PictureWaveguide}}
\end{figure}

\noindent In this chapter, we present general results concerning waveguide theory. To make it simple, we stick to a 2D scalar problem. The analysis below can be quite directly adapted to deal with 3D geometries. However it is important that the cross section of the waveguide be bounded. The study differs when this assumption does not hold, for example, for an infinite layer in $\R^3$.\\
\newline
Set $I\coloneqq(0;1)$ and consider $\Om\subset\R^2$ a waveguide which coincides with the reference strip $\mathcal{S}\coloneqq\{(x,y)\in\R\times I\}$ outside of a compact region located in the zone $\{(x,y)\in\R^2\,|\,|x|<d\}$ for some $d>0$ (see Figure \ref{PictureWaveguide}). We assume that the domain $\Om$ is connected with Lipschitz boundary. Let us study the wave equation, for $t\ge0$,
\begin{equation}\label{WaveEquation}
\begin{array}{|rcll}
\cfrac{1}{c^2}\cfrac{\partial^2 \mathcal{U}}{\partial t^2}-\Delta \mathcal{U}&=&\mathcal{F}&\mbox{ in }\Om \\[2pt]
\mathcal{U}&=&0&\mbox{ on }\partial\Om,
\end{array}
\end{equation}
with some initial conditions. Here $c>0$, the celerity of waves in the homogeneous medium filling $\Om$, is assumed to be constant. 
The Dirichlet Boundary Conditions (BCs) are relevant in certain circumstances in electromagnetism when the Maxwell's problem has some invariance with respect to one spatial variable. Assume that the excitation $\mathcal{F}$ is time harmonic, \textit{i.e.} of the form
\[
\mathcal{F}(x,y,t)=f(x,y)e^{-i\om t},
\]
for some pulsation $\om>0$ corresponding to a temporal period $T\coloneqq2\pi/\om$. Then it is natural to look for solutions of (\ref{WaveEquation}) which are also harmonic for long times\footnote{This is the limiting amplitude principle, which holds in general, but can be violated in rare circumstances due to trapped modes that we will meet later. For specific studies on this aspect, which requires a rather long analysis, see \cite{Eidu65,Eidu69,Wede84}.}. More precisely, we are led to search for $\mathcal{U}$ solving (\ref{WaveEquation}) of the form
\begin{equation}\label{HarmoU}
\mathcal{U}(x,y,t)=u(x,y)e^{-i\om t}.
\end{equation}
Inserting (\ref{HarmoU}) in (\ref{WaveEquation}), we find that $u$ satisfies the problem 
\begin{equation}\label{WaveguidePb}
\begin{array}{|rcll}
-\Delta u-k^2u&=&f&\mbox{ in }\Om \\[2pt]
u&=&0&\mbox{ on }\partial\Om
\end{array}
\end{equation}
where $k\coloneqq\om/c>0$ denotes the wavenumber.\\
\newline 
We wish to endow (\ref{WaveguidePb}) with a well-suited functional framework. This is not straightforward for two reasons. First, the form associated with (\ref{WaveguidePb}) is not coercive except when $k$ is small. Second, the domain $\Om$ is unbounded so that the term involving $k$ cannot be seen as a compact perturbation of the principal part.\\
\newline
Before proceeding further, we introduce a few spaces that will be useful in the analysis. Denote by $\mL^2(\Om)$ the usual Lebesgue space of complex valued square-integrable functions. It is a Hilbert space for the inner product 
\[
(u,v)_{\mL^2(\Om)}=\int_{\Om} u \overline{v}\,dxdy,\qquad\forall u,v\in \mL^2(\Om).
\]
The reason for considering complex valued functions comes from the possible existence of complex modes, see \S\ref{ParagraphModesD} below. We will also work with the Sobolev spaces
\[
\begin{array}{l}
\mH^1(\Om)\coloneqq\{u\in\mL^2(\Om)\,|\,\nabla u\in (\mL^2(\Om))^2\}\\[2pt]
\mH^1_0(\Om)\coloneqq\{u\in\mH^1(\Om)\,|\,u=0\mbox{ on  }\partial\Om\}
\end{array}
\]
that we endow with the inner product
\[
(u,v)_{\mH^1(\Om)}=\int_{\Om}\nabla u\cdot\nabla \overline{v}+u\overline{v}\,dxdy.
\]
They also are Hilbert spaces. Let us mention that in the definition of $\mH^1_0(\Om)$, classically, the value of $u$ on $\partial\Om$ should be understood in the sense of traces (see \textit{e.g.} \cite[Thm.\,3.37]{McLe00}). We define the norms 
\[
\|\cdot\|_{\mL^2(\Om)}=(\cdot,\cdot)_{\mL^2(\Om)}^{1/2},\qquad \|\cdot\|_{\mH^1(\Om)}=(\cdot,\cdot)_{\mH^1(\Om)}^{1/2}.
\]
For a non-empty  set $\mathcal{O}$, $\mathscr{C}^{\infty}_0(\mathcal{O})$ refer to the space of infinitely differentiable functions whose support is bounded and in $\mathcal{O}$. Finally, for a continuous linear map $T: \mrm{X}\to \mrm{Y}$ between two Banach spaces, we use the norm
\[
\|T\|=\underset{x\in\mX\setminus\{0\}}{\sup}\cfrac{\|Tx\|_{\mrm{Y}}}{\|x\|_{\mrm{X}}}\,.
\]

\section{Dirichlet problem for $0<k<\pi$}

Assume that $f$ in (\ref{WaveguidePb}) belongs to $\mL^2(\Om)$.  The natural variational formulation of that problem writes
\begin{equation}\label{PbVariaLow}
\begin{array}{|l}
\mbox{Find }u\in\mrm{H}^1_0(\Om) \mbox{ such that }\\[3pt]
a(u,v)=\ell(v),\qquad \forall v\in\mH^1_0(\Om),
\end{array}
\end{equation}
with 
\[
a(u,v)=\int_{\Om}\nabla u\cdot\nabla \overline{v}-k^2 u\overline{v}\,dxdy,\qquad\qquad \ell(v)=\int_{\Om}f\overline{v}\,dxdy.
\]
The sesquilinear form $a(\cdot,\cdot)$ is continuous in $\mH^1_0(\Om)$. Therefore, with the Riesz representation theorem, we can introduce the linear bounded operator $A(k):\mH^1_0(\Om)\to\mH^1_0(\Om)$ such that
\begin{equation}\label{DefOpAk}
(A(k)u,v)_{\mH^1(\Om)} =a(u,v),\qquad\forall u,v\in \mH^1_0(\Om).
\end{equation}
In this section, we establish the following statement.
\begin{theorem}\label{ThmDLowFreq}
Pick $k\in(0;\pi)$. The operator $A(k)$ decomposes as 
\[
A(k)=B+K
\]
where $B:\mH^1_0(\Om)\to\mH^1_0(\Om)$ is an isomorphism and $K:\mH^1_0(\Om)\to\mH^1_0(\Om)$ is compact ($B$ and $K$ are allowed to depend on $k$).
\end{theorem}
\begin{proof}
Define the sesquilinear form $b(\cdot,\cdot)$ such that 
\[
b(u,v)=\int_{\Om}\nabla u\cdot\nabla \overline{v}+((1+k^2)\mathbbm{1}_{\Om_d}-k^2) u\overline{v}\,dxdy,\qquad\forall u,v\in \mH^1_0(\Om),
\]
where $\mathbbm{1}_{\Om_d}$ stands for the indicator function of the set $\Om_d\coloneqq\{(x,y)\in\Om\,|\,|x|<d\}$. Since $b(\cdot,\cdot)$ is continuous in $\mH^1_0(\Om)$, we can define the bounded operator $B:\mH^1_0(\Om)\to\mH^1_0(\Om)$ such that 
\[
(Bu,v)_{\mH^1(\Om)}=b(u,v),\qquad\forall u,v\in \mH^1_0(\Om).
\]
From the Lax-Milgram theorem, to show that $B$ is an isomorphism, it suffices to prove that $b(\cdot,\cdot)$ is coercive in $\mH^1_0(\Om)$. In Lemma \ref{LemmaPoincare} below, we establish the 1D Poincar\'e inequality 
\begin{equation}\label{Poincare1D}
\pi^2\int_I |\varphi|^2\,dt \le \int_I |\partial_t\varphi|^2\,dt,\qquad\forall \varphi\in\mH^1_0(I)\coloneqq\{\psi\in\mH^1(I)\,|\,\psi(0)=\psi(1)=0\},
\end{equation}
where we recall that $I=(0;1)$. Integrating this estimate with respect to $x\in(-\infty;-d)\cup(d;+\infty)$ for $u\in\mathscr{C}^\infty_0(\Om)$ and using the density of $\mathscr{C}^\infty_0(\Om)$ in $\mH^1_0(\Om)$, we obtain
\[
\pi^2\int_{\Om\setminus\overline{\Om_d}}|u|^2\,dxdy \le \int_{\Om\setminus\overline{\Om_d}}|\nabla u|^2\,dxdy,\qquad\forall u\in\mH^1_0(\Om).
\]
Therefore we can write, for all $u\in\mH^1_0(\Om)$,
\[
\begin{array}{rcl}
b(u,u)& = & \dsp\int_{\Om\setminus\overline{\Om_d}}|\nabla u|^2-k^2 |u|^2\,dxdy+\|u\|^2_{\mH^1(\Om_d)} \\[10pt]
& \ge & \dsp\bigg(1-\cfrac{k^2}{\pi^2}\bigg)\int_{\Om\setminus\overline{\Om_d}}|\nabla u|^2\,dxdy+\|u\|^2_{\mH^1(\Om_d)} \\[10pt]
& \ge & (1+\pi^{-2})^{-1}\bigg(1-\cfrac{k^2}{\pi^2}\bigg)\|u\|^2_{\mH^1(\Om\setminus\overline{\Om_d})}+\|u\|^2_{\mH^1(\Om_d)} \ge \alpha\,\|u\|^2_{\mH^1(\Om)}
\end{array}
\]
with $\alpha=(1+\pi^{-2})^{-1}(1-k^2/\pi^2)>0$. This shows that $b(\cdot,\cdot)$ is coercive in $\mH^1_0(\Om)$.\\
\newline
Now set $K=A(k)-B$. We have 
\begin{equation}\label{DefOpK}
(Ku,v)_{\mH^1(\Om)}=-(1+k^2)\int_{\Om_d}u\overline{v}\,dxdy,\qquad\forall u,v\in \mH^1_0(\Om).
\end{equation}
To establish that $K:\mH^1_0(\Om)\to\mH^1_0(\Om)$ is compact, we have to prove that from any bounded sequence $(u_n)$ of functions of $\mH^1_0(\Om)$, we can extract a subsequence such that $(Ku_n)$ converges in $\mH^1_0(\Om)$. By taking $v=Ku$ in (\ref{DefOpK}), we obtain, for all $u\in\mH^1_0(\Om)$, 
\[
\|Ku\|_{\mH^1(\Om)} \le (1+k^2)\,\|u\|_{\mL^2(\Om_d)}. 
\]
In particular, for $u_{mn}\coloneqq u_m-u_n$, we obtain
\begin{equation}\label{EstimCauchy}
\|Ku_{mn}\|_{\mH^1(\Om)} \le (1+k^2)\,\|u_{mn}\|_{\mL^2(\Om_d)}. 
\end{equation}
Since $\Om_d$ is bounded and has Lipschitz boundary, the Rellich theorem ensures that the embedding of $\mH^1(\Om_d)$ in $\mL^2(\Om_d)$ is compact. We deduce that we can extract from  $(u_n)$, which is bounded in $\mH^1(\Om)$ and so in $\mH^1(\Om_d)$, a subsequence, still denoted by $(u_n)$, such that $(u_n)$ converges in $\mL^2(\Om_d)$. Thus $(u_n)$ is a Cauchy sequence in $\mL^2(\Om_d)$. From (\ref{EstimCauchy}), we infer that $(Ku_n)$ is a Cauchy sequence in $\mH^1_0(\Om)$. Since this space is complete, we infer that $(Ku_n)$ indeed converges in $\mH^1_0(\Om)$.
\end{proof}
\noindent This shows that $A(k)$ satisfies the Fredholm alternative (see \textit{e.g.} \cite[Chap.\,12]{Wlok87} for a short presentation of the theory of Fredholm operators). Either $A(k)$ is injective and in this case it is an isomorphism of $\mH^1_0(\Om)$. Or $A(k)$ has a kernel of finite dimension $\mrm{span}(u_1,\dots,u_P)$. In that case, since $A(k)$ is selfadjoint because it is symmetric and bounded, the equation 
\begin{equation}\label{PbCompa}
A(k)u=F\quad\mbox{ in }\mH^1_0(\Om)
\end{equation}
has a solution, defined up to $\mrm{span}(u_1,\dots,u_P)$, if and only if $F$ satisfies the compatibility conditions 
\begin{equation}\label{PbCompa1}
(F,u_p)_{\mH^1(\Om)}=0,\qquad p=1,\dots,P.
\end{equation}
Note that showing that the conditions (\ref{PbCompa1}) are necessary for the existence of a solution is straightforward: simply multiply (\ref{PbCompa}) by $\overline{u_p}$ and integrate by parts.\\
\newline
We prove now the Poincar\'e inequality needed in (\ref{Poincare1D}). 
\begin{lemma}\label{LemmaPoincare}
We have 
\begin{equation}\label{EstimatePoinc1D0}
\inf_{\varphi\in\mH^1_0(I)\setminus\{0\}} \cfrac{\|\partial_t\varphi\|^2_{\mL^2(I)}}{\|\varphi\|^2_{\mL^2(I)}}=\pi^2
\end{equation}
so that there holds
\begin{equation}\label{EstimatePoinc1D}
\pi^2\,\|\varphi\|^2_{\mL^2(I)}\le \|\partial_t\varphi\|^2_{\mL^2(I)},\qquad\forall \varphi\in\mH^1_0(I).
\end{equation}
\end{lemma}
\begin{proof}
To obtain 1D Poincar\'e inequalities as (\ref{EstimatePoinc1D}), a classical approach consists in working with explicit representations. More precisely, for $\varphi\in\mathscr{C}^\infty_0(I)$, we can write, for $s\in(0;1)$,
\[
\varphi(s)=\int_{0}^s \partial_t\varphi(t)\,dt.
\]
According to the Cauchy-Schwarz inequality in $\mL^2$, this implies, for all $s\in(0;1/2)$,
\[
|\varphi(s)|^2 \le \int_{0}^s1\,dt\int_{0}^s|\partial_t\varphi(t)|^2\,dt \le s\,\|\partial_t\varphi\|^2_{\mL^2(0;1/2)}.
\]
Integrating this identity between $0$ and $1/2$, we obtain
\[
\|\varphi\|^2_{\mL^2(0;1/2)}\le \cfrac{1}{8}\,\|\partial_t\varphi\|^2_{\mL^2(0;1/2)}.
\]
By establishing a similar estimate on $(1/2;1)$ (note that $\varphi(1)=0$) and using the density of $\mathscr{C}^\infty_0(I)$ in $\mH^1_0(I)$, we find 
\[
\|\varphi\|^2_{\mL^2(I)}\le \cfrac{1}{8}\,\|\partial_t\varphi\|^2_{\mL^2(I)}\qquad\Leftrightarrow\qquad 8\,\|\varphi\|^2_{\mL^2(I)}\le \|\partial_t\varphi\|^2_{\mL^2(I)},\qquad\forall\varphi\in\mH^1_0(I).
\]
This is a nice Poincar\'e inequality but it is not optimal (observe that (\ref{EstimatePoinc1D}) is better because $8<\pi^2$). Looking for the best Poincar\'e inequality leads us to consider the minimization problem 
\begin{equation}\label{PbInf}
\inf_{\varphi\in\mH^1_0(I)\setminus\{0\}} \cfrac{\|\partial_t\varphi\|^2_{\mL^2(I)}}{\|\varphi\|^2_{\mL^2(I)}}\,.
\end{equation}
Below we prove that this infimum, equal to some $\lambda>0$, is actually a minimum because it is attained at some functions $u\in\mH^1_0(I)\setminus\{0\}$. Moreover we establish that these quantities  satisfy
\begin{equation}\label{PbSpectral}
-\partial_{tt}^2u=\lambda u\qquad\mbox{ in }I.
\end{equation}
In other words, $\lambda$ is an eigenvalue (the smallest) of the Dirichlet Laplacian and $u$ is a corresponding eigenfunction. Since $I=(0;1)$, a direct computation gives $\lambda=\pi^2$ with $u(t)=\sin(\pi t)$ (up to a multiplicative constant which does not change the ratio in (\ref{PbInf})). Thus we obtain (\ref{EstimatePoinc1D0}) and so (\ref{EstimatePoinc1D}).\\
\newline
To prove that the infimum in (\ref{PbInf}) is reached, let us first remark that solving (\ref{PbInf}) is equivalent to solving the constrained minimization problem 
\[
\inf_{\varphi\in \mathscr{B}}\left\{ J(\varphi)=\int_{I}|\partial_t\varphi|^2\,dt\right\}
\]
with $\mathscr{B}\coloneqq\{\varphi\in\mH^1_0(I)\,|\,\int_I |\varphi|^2\,dt=1\}$. The functional $J$ is positive in $\mathscr{B}$, therefore we have $\inf_\mathscr{B} J \ge0$. Consider $(u_n)\in \mathscr{B}^{\N}$ a minimizing sequence for $J$, \textit{i.e.} a sequence of functions of $\mathscr{B}$ such that 
\[
\lim_{n\to+\infty}J(u_n)=\inf_\mathscr{B} J.
\]
The sequence $(J(u_n))$ is bounded and so $(u_n)$ is bounded in $\mH^1_0(I)$. Then we know that we can extract a subsequence, still denoted by $(u_n)$, such that there is some $u\in \mH^1_0(I)$ such that
\[
u_n \rightharpoonup u \text{ weakly in } \mH^1_0(I),\qquad u_n \rightarrow u \text{ strongly in } \mL^2(I).
\]
Let us detail the proof of this classical result. Since $(u_n)$ is bounded in $\mH^1_0(I)$, from \cite[Thm.\,3.17]{Brez11}, we know that we can extract a subsequence, still denoted by $(u_n)$, such that there is some $u\in \mH^1_0(I)$ such that $
u_n \rightharpoonup u \text{ weakly in } \mH^1_0(I)$. 
Moreover, the Rellich-Kondrachov theorem, see \textit{e.g.} \cite[Thm.\,9.16+Rmk.\,9.20]{Brez11}, ensures that we can extract a subsequence, still denoted by $(u_n)$, such that $(u_n)$ converges strongly in $\mL^2(I)$ to some $\tilde{u}$. It remains to show that $\tilde{u}$ coincides with $u$. Using that the map $v\mapsto (u-\tilde{u},v)_{\mL^2(I)}$ is continuous in $\mH^1_0(I)$, according to the Riesz representation theorem, we can introduce $E\in\mH^1_0(I)$ such that
\[
(E,v)_{\mH^1(I)}=(u-\tilde{u},v)_{\mL^2(I)},\qquad\forall v\in \mH^1_0(I).
\]
This implies, for all $n\in\N$, $
(E,u_n)_{\mH^1(I)}=(u-\tilde{u},u_n)_{\mL^2(I)}$. 
Taking the limit $n\to+\infty$, we deduce 
\[
(u-\tilde{u},u)_{\mL^2(I)}=(E,u)_{\mH^1(I)}=\lim_{n\to+\infty}(E,u_n)_{\mH^1(I)}=\lim_{n\to+\infty}(u-\tilde{u},u_n)_{\mL^2(I)}=(u-\tilde{u},\tilde{u})_{\mL^2(I)}.
\]
This gives $\|u-\tilde{u}\|^2_{\mL^2(I)}=0$ and so $u=\tilde{u}$.\\
\newline 
Now, by writing 
\[
0\le (\partial_t(u-u_n),\partial_t(u-u_n))_{\mL^2(\Om)}=\|\partial_tu\|^2_{\mL^2(I)}+\|\partial_tu_n\|^2_{\mL^2(I)}-2\Re e\,(\partial_tu,\partial_tu_n)_{\mL^2(I)},
\]
we obtain $2\Re e\,(\partial_tu,\partial_tu_n)_{\mL^2(I)}-\|\partial_tu\|^2_{\mL^2(I)}\le \|\partial_t u_n\|^2_{\mL^2(I)}$. By taking the limit inferior, we deduce
\[
\|\partial_tu\|^2_{\mL^2(I)} \leq \liminf\limits_{n\to+\infty}\|\partial_tu_n\|^2_{\mL^2(I)}
\]
and so
\[
J(u)\leq \liminf\limits_{n\to+\infty} J(u_n) = \inf_\mathscr{B} J.
\]
But, for all $n\in\N$, we have $ \|u_n\|_{\mL^2(I)}=1$. Since $(u_n)$ converges strongly to $u$ in $\mL^2(I)$, we deduce $\| u\|_{\mL^2(I)}=1$, which shows that $u$ belongs to $\mathscr{B}$. Thus $J$ attains its infimum in $\mathscr{B}$ at $u$.\\
\newline
Set $\lambda=J(u)$. For all $v\in \mH^1_0(I)\setminus\{0\}$, we have $v/\|v\|_{\mL^2(I)}\in \mathscr{B}$ and so 
\[
J(v/\|v\|_{\mL^2(I)}) \ge \lambda.
\]
This gives, for all $v\in \mH^1_0(I)\setminus\{0\}$,
\[
\int_I |\partial_t v|^2 - \lambda  |v|^2\,dt \ge0.
\]
Then we are led to consider the following minimization problem without constraint 
\[
\min_{v\in \mH^1_0(I)}\left\{\tilde J(v)=\int_I |\partial_t v|^2 - \lambda  |v|^2\,dt\right\}.
\]
According to what precedes, the functional $\tilde J$ is non negative in $\mH^1_0(I)$. Moreover, we have $\tilde J(u)=0$ and $u\in \mH^1_0(I)$. Therefore $u$  is a minimizer of $\tilde J$ in $\mH^1_0(I)$. Thus we must have, for all $v\in \mH^1_0(I)$, $s\in\Cplx$,
\[
0=\tilde J(u) \le \tilde J(u+sv).
\]
This is equivalent to
\[
0 \le \Re e\,\left(\overline{s}\int_I \partial_t u\,\partial_t \overline{v}-\lambda   u\overline{v}\,dt\right)+|s|^2\tilde J(v),\qquad\forall v\in  \mH^1_0(I),\,s\in\Cplx,
\]
which holds if and only if
\[
\int_I \partial_t u\,\partial_t \overline{v}\,dt=\lambda\int_I   u\overline{v}\,dt,\qquad\forall v\in  \mH^1_0(I).
\]
We deduce that, in the sense of distributions, we must have the equation (\ref{PbSpectral}). 
\end{proof}

\noindent In Theorem \ref{ThmDLowFreq}, we proved that $A(k)$ satisfies the Fredholm alternative for $k\in(0;\pi)$. We give now a result of injectivity in certain geometries.

\begin{proposition}\label{PropInjD}
Assume that $k\in(0;\pi)$ and $\Om\subset \mathcal{S}=\R\times I$. Then the operator $A(k)$ defined in (\ref{DefOpAk}) is injective, and so is an isomorphism of $\mH^1_0(\Om)$. In that case, Problem (\ref{PbVariaLow}) admits a unique solution for all $f\in\mL^2(\Om)$.
\end{proposition}

\begin{figure}[!ht]
\centering
\begin{tikzpicture}[scale=2]
\draw[fill=gray!30,draw=none](-1.3,0) rectangle (1.3,1);
\draw (-1.3,0)--(1.3,0);
\draw (-1.3,1)--(-0.1,1);
\draw (0.7,1)--(1.3,1);
\draw[fill=white] (-0.8,0.2)--(-0.2,0.8)--(0.1,0.3)--cycle;
\draw[dashed] (-1.3,0)--(-1.6,0);
\draw[dashed] (-1.3,1)--(-1.6,1);
\draw[dashed] (1.3,0)--(1.6,0);
\draw[dashed] (1.3,1)--(1.6,1);
\clip (-0.1,1) rectangle  + (0.8,-0.5);
\draw[fill=white] (0.3,1) circle (0.4) ;
\end{tikzpicture}\qquad\quad
\begin{tikzpicture}[scale=2]
\draw[fill=gray!30] (0.3,1) circle (0.4) ;
\draw[fill=gray!30,draw=none](-1.3,0) rectangle (1.3,1);
\draw (-1.3,0)--(1.3,0);
\draw (-1.3,1)--(-0.1,1);
\draw (0.7,1)--(1.3,1);
\draw[dashed] (-1.3,0)--(-1.6,0);
\draw[dashed] (-1.3,1)--(-1.6,1);
\draw[dashed] (1.3,0)--(1.6,0);
\draw[dashed] (1.3,1)--(1.6,1);
\end{tikzpicture}
\caption{Left: example of domain where $A(k)$ is an isomorphism for all $k\in(0;\pi)$. Right: example of domain where one can show that $A(k)$ is not an isomorphism for all $k\in(0;\pi)$.\label{PictureWaveguideInj}}
\end{figure}
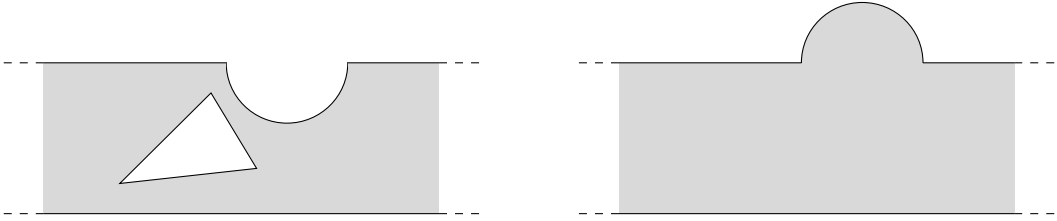

\begin{remark}
In Figure \ref{PictureWaveguideInj} left, we give an example of domain $\Om$ satisfying the assumption $\Om\subset \mathcal{S}$. We emphasize that the injectivity of $A(k)$ for all $k\in(0;\pi)$ does not always hold and in waveguides with exterior bumps (with respect to the reference strip $\mathcal{S}$) as in Figure \ref{PictureWaveguideInj} right, $A(k)$ can have a non zero kernel for certain $k\in(0;\pi)$. 
\end{remark}
\begin{proof}
When $\Om\subset\mathcal{S}$, if a function belongs to $\mH^1_0(\Om)$, then its extension by zero to $\mathcal{S}$ is an element of $\mH^1_0(\mathcal{S})$. Therefore we have
\begin{equation}\label{Compa}
\inf_{u \in\mH^1_0(\mathcal{S})\setminus\{0\}} \cfrac{\dsp\int_\Om |\nabla u|^2\,dxdy}{\dsp\int_{\Om} |u|^2\,dxdy}\le \inf_{u \in\mH^1_0(\Om)\setminus\{0\}} \cfrac{\dsp\int_\Om |\nabla u|^2\,dxdy}{\dsp\int_{\Om} |u|^2\,dxdy}\,.
\end{equation}
Now integrating the 1D Poincar\'e inequality (\ref{EstimatePoinc1D}) with respect to $x\in\R$ for $u\in\mathscr{C}^\infty_0(\mathcal{S})$ and using the density of $\mathscr{C}^\infty_0(\mathcal{S})$ in $\mH^1_0(\mathcal{S})$, we obtain
\[
\pi^2\int_{\mathcal{S}}|u|^2\,dxdy \le \int_{\mathcal{S}}|\nabla u|^2\,dxdy,\qquad\forall u\in\mH^1_0(\mathcal{S}).
\]
From (\ref{Compa}), this gives 
\[
\pi^2\int_{\Om}|u|^2\,dxdy \le \int_{\Om}|\nabla u|^2\,dxdy,\qquad\forall u\in\mH^1_0(\Om).
\]
Therefore, if $u\in\mH^1_0(\Om)$ is such that $A(k)u=0$, then we have
\[
\begin{array}{rcl}
0=a(u,u)&=&\dsp\int_{\Om}|\nabla u|^2-k^2|u|^2\,dxdy \\[6pt]
&\ge & (\pi^2-k^2)\dsp\int_{\Om}|u|^2\,dxdy,
\end{array}
\]
which ensures that $u\equiv0$ in $\Om$. This shows that $A(k)$ is injective and Theorem \ref{ThmDLowFreq} together with the Fredholm alternative guarantee that $A(k)$ is an isomorphism of $\mH^1_0(\Om)$.
\end{proof}

\noindent In the cases where $A(k)$ is not injective, one can show as in Proposition \ref{PropoTrappedModes} below that the elements of its kernel are exponentially decaying as $x\to\pm\infty$ and so are localized in a neighborhood of the geometric perturbation (see an example in Figure \ref{TrappedD}). For this reason, the functions of $\ker\,A(k)$ are usually called trapped modes.

\begin{figure}[!ht]
\centering
\includegraphics[height=2cm]{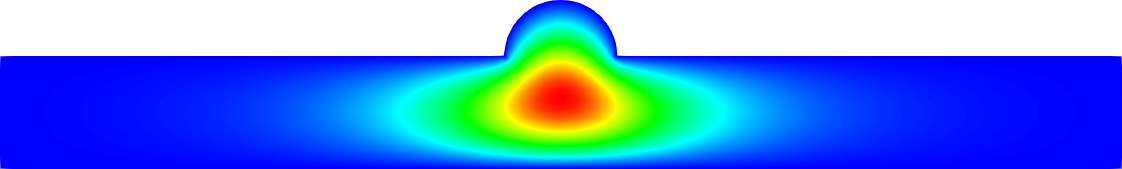}
\caption{Example of trapped mode for Problem (\ref{WaveguidePb}). 
The above function belongs to the kernel of $A(k)$, \textit{i.e.} it is an element of $\mH^1_0(\Om)\setminus\{0\}$ which satisfies $\Delta u+k^2u=0$ in $\Om$,  for a certain $k>0$.  \label{TrappedD}}
\end{figure}

\section{Dirichlet problem for $k>\pi$}

In the previous paragraph, we studied Problem (\ref{WaveguidePb}) for $k<\pi$. Now we wish to understand what happens for $k\ge\pi$. To proceed, we first compute what one usually calls the ``modes'' of (\ref{WaveguidePb}). They play a key role in the physical phenomena and so in the mathematical properties of (\ref{WaveguidePb}).

\subsection{Computation of modes}\label{ParagraphModesD}

The modes are defined as the solutions with separate variables 
\begin{equation}\label{SepVaria}
u(x,y)=\alpha(x)\varphi(y)\not\equiv0,
\end{equation}
say in $\mathscr{C}^2(\mathcal{S})$, which solve Problem (\ref{WaveguidePb}) in the reference strip $\mathcal{S}$ for $f\equiv0$. Let us emphasize that they are \textit{a priori} complex valued. Inserting (\ref{SepVaria}) in the equation $\Delta u+k^2u=0$ in $\mathcal{S}$, we obtain
\[
\alpha''(x)\varphi(y)+\alpha(x)\varphi''(y)+k^2\alpha(x)\varphi(y)=0.
\]
Dividing this identity by $\alpha(x)\varphi(y)$ (note that $u$ cannot vanish on a non-empty open set otherwise it would be zero everywhere due to the unique continuation principle), we find that one must have
\begin{equation}\label{PbFourierD}
\begin{array}{|l}
-\varphi''(y)=\lambda\,\varphi(y)\quad\mbox{ in }I\\[3pt]
\varphi(0)=\varphi(1)=0
\end{array}
\end{equation}
and 
\begin{equation}\label{ModeAlpha}
-\alpha''(x)=(k^2-\lambda)\,\alpha(x)\qquad \mbox{ in }\R
\end{equation}
for some constant $\lambda$ to be determined. Problem (\ref{PbFourierD}) is a spectral problem: we wish to find the values of $\lambda$ such that (\ref{PbFourierD}) admits a non zero solution $\varphi$. A direct calculation shows that the eigenpairs of (\ref{PbFourierD}) are given by
\begin{equation}\label{BaseFourierDirichlet}
\lambda_n=n^2\pi^2,\qquad\qquad \varphi_n(y)=\sqrt{2}\sin(n\pi y),\qquad n\in\N^\ast\coloneqq\{1,2,\dots\}.
\end{equation}
Note that the $\varphi_n$ have been chosen such that they satisfy the orthonormality conditions 
\[
(\varphi_m,\varphi_n)_{\mL^2(I)}=\delta_{m,n}
\]
where $\delta_{m,n}$ stands for the Kronecker symbol. Then solving the second order ODE (\ref{ModeAlpha}), finally we find that when $k\ne n\pi$ for all $n\in\N^\ast$, the modes 
coincide with the family $\{w^\pm_n\}_{n\in\N^\ast}$ where 
\begin{equation}\label{DefModes}
w^\pm_n(x,y)=e^{\pm i\beta_n x}\varphi_n(y),\qquad\quad\beta_n\coloneqq\sqrt{k^2-n^2\pi^2}.
\end{equation}
Here and below, the complex square root is chosen (this is a convention) such that if $z=r e^{i\theta}$ with $r\ge0$ and $\theta\in[0;2\pi)$, then $\sqrt{z}=\sqrt{r} e^{i \theta/2}$. As a consequence, for any $z\in\Cplx$, there holds $\Im m\,\sqrt{z}\ge0$.\\
\newline
Let us make a few observations concerning these modes. To set ideas, introduce $N\in\N$ such that $k\in(N\pi;(N+1)\pi)$.\\[4pt]
$\star$ For $n=1,...,N$ (ignore this case if $N=0$), we have
\[
w_n^\pm(x,y)=e^{\pm i\sqrt{k^2-n^2\pi^2}x}\varphi_n(y).
\] 
Since $\sqrt{k^2-n^2\pi^2}>0$ for $n=1,...,N$, these modes do not decay at infinity. They are called propagating modes. For a fixed $k>0$, there is always a finite number of propagating modes. Moreover they do not exist when $k\in(0;\pi)$ (the situation studied in the previous paragraph). On the contrary, for all $k>\pi$, the modes $w_1^\pm$ are propagating. Going back to time-domain, we observe that these modes lead to consider solutions of (\ref{WaveEquation}) of the form
\[
\mathcal{W}_n^\pm(x,y,t)=e^{i(\pm \sqrt{k^2-n^2\pi^2}x-\om t)}\varphi_n(y).
\]
The waves $\mathcal{W}_n^+$ propagate to the right while the $\mathcal{W}_n^-$ propagate to the left. For this reason, we will say that the $w_n^+$ are rightgoing modes while the $w_n^-$ are leftgoing.\\[4pt]
$\star$ For $n=N+1,N+2,...$, we have
\[
w_n^\pm(x,y)=e^{\mp \sqrt{n^2\pi^2-k^2}x}\varphi_n(y).
\] 
Since $\sqrt{n^2\pi^2-k^2}>0$ for $n=N+1,N+2,\dots$, these modes  are exponentially decaying as $x\to\pm\infty$ and exponentially growing as $x\to\mp\infty$. There are an infinite number of them.\\
\newline
Though these modes have been computed for the problem in the reference strip $\mathcal{S}$, we will also use them in the analysis of Problem (\ref{WaveguidePb}) in the perturbed domain $\Om$.

\subsection{Ill-posedness in $\mH^1_0(\Om)$}

In this paragraph, our goal is to show that the existence of propagating modes for $k>\pi$ is responsible for the ill-posedness in the Fredholm sense of the operator 
$A(k):\mH^1_0(\Om)\to\mH^1_0(\Om)$ defined in (\ref{DefOpAk}).

\begin{definition}\label{definition fredholm op}
Let $\mrm{X}$ and $\mrm{Y}$ be two Banach spaces, and let
$T: \mrm{X}\to \mrm{Y}$ be a continuous linear map.
The operator $T$ is said to be a Fredholm operator if and 
only if the following two conditions are fulfilled\\[6pt]
$i)$  $\mrm{dim}(\mrm{ker}\,T ) <+\infty$ and $\mrm{range}\, T$ is closed;\\[6pt]
$ii)$ $\mrm{dim}(\mrm{coker}\,T ) <+\infty$ where $\mrm{coker}\,T  
\;\dsp{ \coloneqq\; \big(\mrm{Y}/\mrm{range}\, T\big) }$.\\[6pt]
Besides, the index of a Fredholm operator $T$ 
is defined by $\mrm{ind}\,T = \mrm{dim}(\mrm{ker}\,T ) - 
\mrm{dim}(\mrm{coker}\,T)$.
\end{definition}

\noindent The result of Theorem \ref{ThmDLowFreq} guarantees that $A(k)$ is Fredholm of index zero when $k\in(0;\pi)$. To prove that $A(k)$ is not Fredholm when $k>\pi$, we start by recalling a lemma due to  J. Peetre \cite{Peet61} (see also Theorem 12.12 in \cite{Wlok87}).
\begin{lemma}\label{lemme Peetre}
Let $\mX$, $\mY$, $\mZ$ be three reflexive Banach spaces, such that $\mX$ is compactly embedded into $\mZ$. Let
$T:\mX\to \mY$ be a continuous linear map. Then the assertions below are equivalent:\\[6pt]
\phantom{i}\quad ~ \hspace{1.2cm}~\phantom{i}i) $\dim(\ker\,T)<+\infty$ and $\mrm{range}\,T$ is closed in $\mY$;\\[5pt]
\phantom{i}\quad ~ \hspace{1.2cm}~ii) there exists $C>0$ such that $\|u\|_{\mX}\le C\,(\|Tu\|_{\mY}+\|u\|_{\mZ})$, $\forall u\in \mX$.
\end{lemma}

\begin{proposition}
For $k>\pi$, the operator $A(k):\mH^1_0(\Om)\to\mH^1_0(\Om)$  defined in (\ref{DefOpAk}) is not Fredholm.
\end{proposition}
\begin{remark}
We exclude the case $k=\pi$ because the computations are a bit different. However in that situation too one can prove that $A(k):\mH^1_0(\Om)\to\mH^1_0(\Om)$ is not Fredholm.
\end{remark}
\begin{proof}
Set again $\Om_d\coloneqq\{(x,y)\in\Om\,|\,|x|<d\}$ where $d$ appears before (\ref{WaveguidePb}). Our goal is to show that we cannot have the existence of $C>0$ such that there holds
\begin{equation}\label{EstimAP}
\|u\|_{\mH^1(\Om)}\le C\,(\|A(k)u\|_{\mH^1(\Om)}+\|u\|_{\mL^2(\Om_d)}),\qquad \forall u\in\mH^1_0(\Om).
\end{equation}
To proceed, consider some functions $\psi_+,\psi_-\in\mathscr{C}^\infty(\R)$ such that 
\[
\psi_+(x)=\begin{array}{|ll}
 1 & \mbox{ for }x> d+1 \\[2pt]
0 & \mbox{ for }x< d
\end{array}\qquad \quad\psi_-(x)=\begin{array}{|ll}
1 & \mbox{ for }x< 0\\[2pt]
0 & \mbox{ for }x>1.
\end{array}
\]
Then for $m\in\N$, set $\psi_m(x)=\psi_+(x)\psi_-(x-m)$ and 
\[
u_m(x,y)=\psi_m(x) w_1^+(x,y)=\psi_m(x) e^{i\beta_1 x}\varphi_1(y),
\]
where $w_1^+$ is the mode appearing in (\ref{DefModes}) which is propagating for $k>\pi$ (because then $\beta_1=\sqrt{k^2-\pi^2}\in\R$).

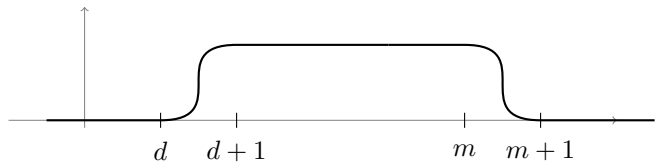
\begin{figure}[!ht]
\centering
\begin{tikzpicture}[scale=1]
\draw[gray,very thin,->] (-4,0)--(4,0);
\draw[gray,very thin,->] (-3,-0.1)--(-3,1.5);
\draw[thick] (-3.5,0)--(-2,0) .. controls (-1,0) and (-2,1) .. (-1,1) -- (1,1);
\draw[thick] (4.5,0)--(3,0) .. controls (2,0) and (3,1) .. (2,1) -- (1,1);
\draw (-2,-0.1)--(-2,0.1);
\draw (-1,-0.1)--(-1,0.1);
\draw (3,-0.1)--(3,0.1);
\draw (2,-0.1)--(2,0.1);
\node at (-2,-0.4){\small $d$};
\node at (-1,-0.4){\small $d+1$};
\node at (3,-0.4){\small $m+1$};
\node at (2,-0.4){\small $m$};
\end{tikzpicture}
\caption{Graph of the cut-off function $\psi_m$.\label{CutOffFunction}} 
\end{figure}

\noindent Exploiting that the support of $u_m$ becomes unbounded as $m\to+\infty$, it is straightforward to show that 
\[
\lim_{m\to+\infty}\|u_m\|_{\mH^1(\Om)}=+\infty.
\]
On the other hand, clearly  $(\|u_m\|_{\mL^2(\Om_d)})$ remains bounded as $m\to+\infty$. Now, for $v\in \mH^1_0(\Om)$, we have
\begin{equation}\label{ExpressionAkm}
(A(k)u_m,v)_{\mH^1(\Om)} = \int_{\Om}\nabla u_m\cdot\nabla v-k^2 u_mv\,dxdy=-\int_{\Om}(\Delta u_m +k^2 u_m)v\,dxdy.
\end{equation}
But there holds
\[
\begin{array}{rcl}
\Delta u_m +k^2 u_m&=&\psi_m(\Delta w_1^+ +k^2 w_1^+)+2\nabla \psi_m\cdot\nabla w_1^++w_1^+\Delta\psi_m \\[3pt]
&=&2\nabla \psi_m\cdot\nabla w_1^++w_1^+\Delta\psi_m.
\end{array}
\]
By observing that $\psi_m$, $\Delta\psi_m$ are non zero only in $(d;d+1)\times(0;1)\cup  (m;m+1)\times(0;1)$ and that their norms in $\mL^\infty(\Om)$ remain bounded independently of $m\in\N$, we find that there exists $C>0$ independent of $m$ such that we have
\[
\|\Delta u_m +k^2 u_m\|_{\mL^2(\Om)} \le C.
\]
By taking $v=A(k)u_m$ in (\ref{ExpressionAkm}), we conclude that $(A(k)u_m)$ remains bounded in  $\mH^1_0(\Om)$ as $m\to+\infty$. This shows that Estimate (\ref{EstimAP}) does not hold.\\
Finally, since $\Om_d$ is bounded and has Lipschitz boundary, the embedding of $\mH^1_0(\Om)$ in  $\mL^2(\Om_d)$ is compact. From Lemma \ref{lemme Peetre}, we deduce that $A(k):\mH^1_0(\Om)\to\mH^1_0(\Om)$ is not Fredholm when $k>\pi$.
\end{proof}
\noindent By working in weighted Sobolev spaces, one can show that $A(k)$ has a kernel of finite dimension for all $k>\pi$. Therefore, the loss of Fredholmness is due to the fact that the range of $A(k)$ is not closed when $k>\pi$. Thus, even by removing the kernel if it is not trivial, we cannot create an operator in $\mH^1_0(\Om)$ which admits a continuous inverse when propagating modes exist. This leads us to think that we have to take them into account in the functional framework.\\
\newline
To proceed, we will apply some strategy which is classical in applied mathematics related to physics: we will add a bit of dissipation in the medium characterized by some parameter $\eta>0$ and then take the limit as $\eta\to0$. More precisely, with dissipation, the definition of the physical solution, the one in $\mH^1_0(\Om)$, becomes obvious. Then we will define the solution without dissipation as the limit as $\eta\to0$ of the solution with dissipation. This is called the limiting absorption  principle in scattering theory (mind the difference with the limiting \textit{amplitude} principle mentioned before (\ref{HarmoU})). In fluid mechanics, dissipation is more often referred to as viscosity but the idea is the same.

\subsection{Problem with dissipation}\label{ParaPbDissip}

To model dissipation, let us work on the problem 
\begin{equation}\label{WaveguidePbDissipation}
\begin{array}{|rcll}
-\Delta u_\eta-(k^2+ik\eta)u_\eta&=&f&\mbox{ in }\Om \\[2pt]
u_\eta&=&0&\mbox{ on }\partial\Om
\end{array}
\end{equation}
with $\eta>0$. To get an idea of why this is a relevant way to model dissipation, let us come back to time domain. With the time harmonic convention $\mathcal{U}_\eta(x,y,t)=u_\eta(x,y)e^{-i\om t}$, Problem (\ref{WaveguidePbDissipation}) originates from the study of the wave equation, for $t\ge0$,
\begin{equation}\label{WaveguidePbDissipationTime}
\begin{array}{|rcll}
\cfrac{1}{c^2}\,\cfrac{\partial^2 \mathcal{U}_\eta}{\partial t^2}+\cfrac{\eta}{c}\,\cfrac{\partial \mathcal{U}_\eta}{\partial t}-\Delta \mathcal{U}_\eta&=&\mathcal{F}&\mbox{ in }\Om \\[2pt]
\mathcal{U}_\eta&=&0&\mbox{ on }\partial\Om,
\end{array}
\end{equation}
with some initial conditions. Assume that the forcing term $\mathcal{F}$ is null. Then multiplying (\ref{WaveguidePbDissipationTime}) by $\partial_t\overline{\mathcal{U}_\eta}$ and integrating in $\Om$, we obtain
\[
\cfrac{\partial }{\partial t}\,\cfrac{1}{2}\,\int_{\Om}\cfrac{1}{c^2}\,\bigg|\cfrac{\partial \mathcal{U}_\eta}{\partial t}\bigg|^2+|\nabla \mathcal{U}_\eta|^2\,dxdy =-\cfrac{\eta}{c} \int_{\Om}\bigg|\cfrac{\partial \mathcal{U}_\eta}{\partial t}\bigg|^2\,dxdy.
\]
Therefore, the energy
\[
E(t)=\cfrac{1}{2}\,\int_{\Om}\cfrac{1}{c^2}\,\bigg|\cfrac{\partial \mathcal{U}_\eta}{\partial t}\bigg|^2+|\nabla \mathcal{U}_\eta|^2\,dxdy
\]
is indeed decreasing, due to the term $\eta\partial_t\mathcal{U}_\eta$, when $\eta>0$.\\
\newline
The variational formulation associated with (\ref{WaveguidePbDissipation}) writes 
\begin{equation}\label{WaveguidePbDissipationVaria}
\begin{array}{|l}
\mbox{Find }u_\eta\in\mrm{H}^1_0(\Om) \mbox{ such that }\\[3pt]
a_\eta(u_\eta,v)=\ell(v),\qquad \forall v\in\mH^1_0(\Om),
\end{array}
\end{equation}
where the sesquilinear (resp. antilinear) forms $a_\eta(\cdot,\cdot)$ (resp. $\ell(\cdot)$) are such that  
\[
a_\eta(w,v)=\int_{\Om}\nabla w\cdot\nabla \overline{v}-(k^2+ik\eta )w \overline{v}\,dxdy,\qquad\qquad \ell(v)=\int_{\Om}f\overline{v}\,dxdy.
\]

\begin{theorem}\label{ThmDissip}
For all $k>0$, for all $\eta>0$, Problem (\ref{WaveguidePbDissipationVaria}) admits a unique solution $u_\eta\in\mrm{H}^1_0(\Om)$.
\end{theorem}
\begin{proof}
For $v\in\mH^1_0(\Om)$, we have 
\[
\Re e\,a_\eta(v,v)=\int_{\Om}|\nabla v|^2-k^2|v|^2 \,dxdy,\qquad \quad\Im  m\,a_\eta(v,v)=-k\eta\int_{\Om}|v|^2 \,dxdy,
\]
Therefore, we obtain, for $\gamma>0$, 
\[
\Re e\,\big( (1+i\gamma)a_\eta(v,v)\big)=\Re e\,a_\eta(v,v)-\gamma\,\Im  m\,a_\eta(v,v)=\int_{\Om}|\nabla v|^2+(\gamma k\eta-k^2)\,|v|^2 \,dxdy.
\]
Thus for any $\eta>0$, for $\gamma>0$ large enough, the form $(1+i\gamma)a_\eta(\cdot,\cdot)$ is coercive in $\mH^1_0(\Om)$. With the complex version of the Lax-Milgram theorem, this is enough to conclude that Problem (\ref{WaveguidePbDissipationVaria}) admits a unique solution in that case.
\end{proof}

\begin{figure}[!ht]
\centering
\begin{tikzpicture}[scale=2.2]
\draw[fill=gray!30] (0.3,1) circle (0.4) ;
\draw[fill=gray!30,draw=none](-1,0) rectangle (1.1,1);
\draw (-1.3,0)--(1.3,0);
\draw (-1.3,1)--(-0.1,1);
\draw (0.7,1)--(1.3,1);
\draw[fill=white] (-0.5,0.2)--(0.1,0.8)--(0.4,0.3)--cycle;
\draw[dashed] (-1.3,0)--(-1.6,0);
\draw[dashed] (-1.3,1)--(-1.6,1);
\draw[dashed] (1.3,0)--(1.6,0);
\draw[dashed] (1.3,1)--(1.6,1);
\node at (-0.4,0.8){ $\Om_L$};
\draw[-,dashed] (-0.7,-0.05)--(-0.7,1);
\draw[-,dashed] (0.8,-0.05)--(0.8,1);
\node at (-0.7,-0.15){\small $-d$};
\node at (0.8,-0.15){\small $d$};
\draw[-] (-1,-0.05)--(-1,1);
\draw[-] (1.1,-0.05)--(1.1,1);
\node at (-1,-0.15){\small $-L$};
\node at (1.1,-0.15){\small $L$};
\node at (1.3,0.5){ $\Sigma_L$};
\node at (-1.2,0.5){ $\Sigma_{-L}$};
\node at (0.35,0.6){ $\Gamma$};
\node at (-0.2,1.2){ $\Gamma$};
\node at (0.05,-0.15){ $\Gamma$};
\end{tikzpicture}\vspace{-0.3cm}
\caption{Domain $\Om_L$.\label{PictureDtN}}
\end{figure}

\noindent Now we wish to take the limit $\eta$ tends to zero in (\ref{WaveguidePbDissipationVaria}). It will be more convenient to work in a bounded domain. Introduce some $L>d$ and define 
\[
\Om_L\coloneqq\{(x,y)\in\Om\,|\,|x|<L\}
\]
(see Figure \ref{PictureDtN}). Assume that $f$ in (\ref{WaveguidePbDissipation}) is given in $\mL^2(\Om)$ such that $f(x,y)=0$ for $\pm x >L$. We will first derive a problem set in $\Om_L$ whose solution coincides with $u_\eta|_{\Om_L}$. To proceed, we must impose \textit{ad hoc} transparent conditions on the artificial boundaries 
\[
\Sigma_{L}\coloneqq\{L\}\times I,\qquad  \Sigma_{-L}\coloneqq\{-L\}\times I
\]
that do not create spurious reflections.\\
\newline
Let us work first on $\Sigma_L$. Denote by $\mH^{1/2}_{00}(\Sigma_L)$ the space of traces on $\Sigma_L$ of elements of $\mH^1_0(\Om_L)$. It coincides with the functions which belong to $\mH^{1/2}(\{L\}\times \R)$ when extended by zero. Classically, see \textit{e.g.} \cite[\S1.4]{CaCo06}, we endow $\mH^{1/2}_{00}(\Sigma_L)$ with the norm
\[
\|\varphi\|_{\mH^{1/2}(\Sigma_L)}\coloneqq \sum_{n=1}^\infty \left(1+(n\pi)^2\right)^{1/2}\,|a_n|^2\qquad\mbox{ where }\qquad a_n\coloneqq \int_0^1\varphi(L,y)\,\varphi_n(y)\,dy,
\] 
the $\varphi_n$ being the ones introduced in (\ref{BaseFourierDirichlet}). Let $\mH^{-1/2}(\Sigma_L)$ stand for the dual space of $\mH^{1/2}_{00}(\Sigma_L)$ endowed with the norm
\[
\|\Phi\|_{\mH^{-1/2}(\Sigma_L)}\coloneqq\underset{\varphi\in\mH^{1/2}_{00}(\Sigma_L)\setminus\{0\}}{\sup}\cfrac{|\langle \Phi,\varphi\rangle_{\Sigma_L}|}{\|\varphi\|_{\mH^{1/2}(\Sigma_L)}}\,,
\]
where $\langle \cdot,\cdot\rangle_{\Sigma_L}$ denotes the linear duality pairing between $\mH^{-1/2}(\Sigma_{L})$ and $\mH^{1/2}_{00}(\Sigma_{L})$. Set $\mathcal{S}_+\coloneqq(L;+\infty)\times I$ and define the Dirichlet-to-Neumann operator such that
\[
\begin{array}{lccc}
\Lambda_+^\eta:& \mH^{1/2}_{00}(\Sigma_L) & \to& \mH^{-1/2}(\Sigma_L) \\[6pt]
 & \varphi & \mapsto  & \cfrac{\partial v_{\varphi}}{\partial \nu}\,,
\end{array}
\]
where $\partial_\nu=\partial_x$ on $\Sigma_L$ and $v_{\varphi}\in \mH^1(\mathcal{S}_+)$ is the function such that
\begin{equation}\label{PbSP}
\begin{array}{|rcll}
\Delta v_{\varphi}+(k^2+ik\eta)v_{\varphi}&=&0&\mbox{ in }\mathcal{S}_+ \\[2pt]
v_{\varphi}&=&0&\mbox{ on }\partial\Om\cap\partial \mathcal{S}_+\\[2pt]
v_{\varphi}&=&\varphi&\mbox{ on }\Sigma_L.
\end{array}
\end{equation}
One establishes that (\ref{PbSP}) admits a unique solution by adapting the proof of Theorem \ref{ThmDissip} by working with a lifting function on $\Sigma_L$. Let us show that
 $\Lambda_+^\eta$ is indeed valued in $\mH^{-1/2}(\Sigma_L)$ and more precisely, that it is a continuous operator from $\mH^{1/2}_{00}(\Sigma_L)$ to $\mH^{-1/2}(\Sigma_L)$. To proceed, and actually this will be useful in the sequel, the simplest way is to write an explicit representation of the action of $\Lambda_+^\eta$. We work again with the modes of (\ref{PbSP}). Assume that $k$ is not equal to one of the $n\pi$, $n\in\N^\ast$. In every transverse section $\{x\}\times I$ of $\mathcal{S}_+$, we have the decomposition in Fourier series
\begin{equation}\label{DecompoDissip}
v_{\varphi}(x,y)=\sum_{n=1}^{+\infty} \alpha_n(x)\,\varphi_n(y)
\end{equation}
where the $\alpha_n$ are to be determined. Inserting (\ref{DecompoDissip}) into (\ref{WaveguidePbDissipation}), similarly to (\ref{ModeAlpha}), we find that the $\alpha_n$ must be of the form
\[
\alpha_n(x)=A_n\,e^{i\beta_n^\eta x}+B_n\,e^{-i\beta_n^\eta  x}
\]
for some constants $A_n$, $B_n\in\Cplx$ with 
\[
\beta_n^\eta=\sqrt{k^2+i\eta-n^2\pi^2}.
\]
But according to our convention for the complex square root after (\ref{DefModes}), for $\eta>0$, the imaginary part of $\beta_n^\eta$ is positive for all $n\in\N^\ast$. As a consequence, $x\mapsto e^{i\beta_n^\eta x}$ is exponentially decaying at $+\infty$ while $x\mapsto e^{-i\beta_n^\eta x}$ is exponentially growing. Since $v_{\varphi}$ belongs to $\mH^1(\mathcal{S}_+)$, we must impose $B_n=0$ for all $n\in\N^\ast$. Thus for $x>L$, we have the expansion
\[
v_{\varphi}(x,y)=\sum_{n=1}^{+\infty} (\varphi,\varphi_n)_{\mL^2(\Sigma_L)}\,e^{i\beta_n^\eta (x-L)}\,\varphi_n(y)
\]
so that there holds
\[
\Lambda_+^\eta(\varphi)=\dsp\sum_{n=1}^{+\infty} i\beta_n^\eta\,(\varphi,\varphi_n)_{\mL^2(\Sigma_L)}\,\varphi_n(y).
\]
For $\varphi'\in \mH^{1/2}_{00}(\Sigma_L)$, we obtain
\begin{equation}\label{EstimaDtNCont0}
|\langle \Lambda_+^\eta(\varphi),\overline{\varphi'}\rangle_{\Sigma_L}| \le \dsp\sum_{n=1}^{+\infty} |\beta_n^\eta|\,|(\varphi,\varphi_n)_{\mL^2(\Sigma_L)}|\,|(\varphi',\varphi_n)_{\mL^2(\Sigma_L)}|. 
\end{equation}
But one can write, for all $n\in\N^\ast$,
\[
|\beta_n^\eta|^4=(k^2-n^2\pi^2)^2+\eta^2 \le 2k^4+2n^4\pi^4+\eta^2\quad\mbox{ and so }\quad\cfrac{|\beta_n^\eta|^4}{n^4\pi^4}\le \cfrac{2k^4}{\pi^4} 
+2+\cfrac{\eta^2}{\pi^4}\,.
\]
Thus for any given $\eta_0>0$, there is a constant $C>0$ independent of $\eta\in(0;\eta_0]$ such that
\[
\underset{n\in\N^\ast}{\sup} \cfrac{|\beta_n^\eta|}{n\pi} \le C.
\]
Exploiting this in (\ref{EstimaDtNCont0}) gives
\begin{equation}\label{EstimaDtNCont}
\begin{array}{rcl}
|\langle \Lambda_+^\eta(\varphi),\overline{\varphi'}\rangle_{\Sigma_L}|
&\le & C\,\left(\dsp\sum_{n=1}^{+\infty}n\pi|(\varphi,\varphi_n)_{\mL^2(\Sigma_L)}|^2\right)^{1/2} \left(\dsp\sum_{n=1}^{+\infty}n\pi|(\varphi',\varphi_n)_{\mL^2(\Sigma_L)}|^2\right)^{1/2} \\[14pt]
&\le & C\,\|\varphi\|_{\mH^{1/2}(\Sigma_L)}\,\|\varphi'\|_{\mH^{1/2}(\Sigma_L)}.
\end{array}
\end{equation}
Estimate (\ref{EstimaDtNCont}) guarantees that $\Lambda_+^\eta(\varphi)$ indeed belongs to $\mH^{-1/2}(\Sigma_L)$ and that
\begin{equation}\label{ContOp}
\|\Lambda_+^\eta(\varphi)\|_{\mH^{-1/2}(\Sigma_L)} \le C\,\|\varphi\|_{\mH^{1/2}(\Sigma_L)}.
\end{equation}
Since the operator $\Lambda_+^\eta$ is clearly linear, (\ref{ContOp}) ensures that $\Lambda_+^\eta: \mH^{1/2}_{00}(\Sigma_L)  \to \mH^{-1/2}(\Sigma_L)$ is continuous. Moreover, one has 
\begin{equation}\label{EndContDtN}
\|\Lambda_+^\eta\| = \underset{\varphi\in \mH^{1/2}_{00}(\Sigma_L)\setminus\{0\}}{\sup} \cfrac{\|\Lambda_+^\eta(\varphi)\|_{\mH^{-1/2}(\Sigma_L)}}{\|\varphi\|_{\mH^{1/2}(\Sigma_L)}} \le C.
\end{equation} 
Importantly, note that if $u_\eta$ solves (\ref{WaveguidePbDissipationVaria}), then it satisfies
\begin{equation}\label{CondTransP}
\cfrac{\partial u_\eta}{\partial \nu} =\Lambda_+^\eta( u_\eta|_{\Sigma_L})\quad\mbox{ on }\Sigma_L,
\end{equation}
where again $\partial_\nu=\partial_x$ on $\Sigma_L$.\\
\newline
We work completely similarly in $\mathcal{S}_-\coloneqq(-\infty;-L)\times I$ and define the Dirichlet-to-Neumann operator such that
\[
\begin{array}{lccc}
\Lambda_-^\eta:& \mH^{1/2}_{00}(\Sigma_{-L}) & \to& \mH^{-1/2}(\Sigma_{-L}) \\[6pt]
 & \varphi & \mapsto  & \cfrac{\partial v_{\varphi}}{\partial \nu}\,,
\end{array}
\]
where this time $\partial_\nu=-\partial_x$ on $\Sigma_{-L}$ and $v_{\varphi}\in \mH^1(\mathcal{S}_-)$ is the function such that
\[
\begin{array}{|rcll}
\Delta v_{\varphi}+(k^2+ik\eta)v_{\varphi}&=&0&\mbox{ in }\mathcal{S}_- \\[2pt]
v_{\varphi}&=&0&\mbox{ on }\partial\Om\cap\partial \mathcal{S}_-\\[2pt]
v_{\varphi}&=&\varphi&\mbox{ on }\Sigma_{-L}.
\end{array}
\]
One has the representation 
\[
\Lambda_-^\eta(\varphi)=\dsp\sum_{n=1}^{+\infty} i\beta_n^\eta\,(\varphi,\varphi_n)_{\mL^2(\Sigma_{-L})}\,\varphi_n(y),
\]
which allows one to prove as above that the operator $\Lambda_-^\eta:\mH^{1/2}_{00}(\Sigma_{-L})  \to \mH^{-1/2}(\Sigma_{-L})$  is continuous. Moreover, if $u_\eta$ solves (\ref{WaveguidePbDissipationVaria}), then it satisfies
\begin{equation}\label{CondTransM}
\cfrac{\partial u_\eta}{\partial \nu} =\Lambda_-^\eta( u_\eta|_{\Sigma_{-L}})\quad\mbox{ on }\Sigma_{-L},
\end{equation}
with $\partial_\nu=-\partial_x$ on $\Sigma_{-L}$.\\
\newline
Now we have everything to write our problem in $\Om_L$. Set $\Gamma\coloneqq\partial\Om\cap\partial\Om_L$ and $\mH^1_0(\Om_L;\Gamma)\coloneqq\{v\in\mH^1(\Om_L)\,|\,v=0\mbox{ on }\Gamma\}$. 

\begin{proposition}
If $u_\eta\in\mH^1_0(\Om)$ solves (\ref{WaveguidePbDissipation}), then its restriction to $\Om_L$ satisfies
\begin{equation}\label{WaveguidePbDissipationBounded}
\begin{array}{|rcll}
\multicolumn{4}{|l}{\mbox{Find }u_\eta\in\mH^1_0(\Om_L;\Gamma)\mbox{ such that }}\\[2pt]
-\Delta u_\eta-(k^2+ik\eta)u_\eta&=&f&\mbox{ in }\Om_L \\[2pt]
\cfrac{\partial u_\eta}{\partial \nu} & = & \Lambda_\pm^\eta(u_\eta|_{\Sigma_{\pm L}}) &\mbox{ on }\Sigma_{\pm L}
\end{array}
\end{equation} 
where $\partial_{\nu}=\pm\partial_x$ at $x=\pm L$. Conversely, if  $u_\eta\in\mH^1_0(\Om)$ satisfies (\ref{WaveguidePbDissipationBounded}), it can be extended as a solution in $\mH^1_0(\Om)$ of (\ref{WaveguidePbDissipation}).
\end{proposition}
\begin{proof}
The first part of the statement comes from (\ref{CondTransP}), (\ref{CondTransM}). Now if $u_\eta\in\mH^1_0(\Om)$ solves (\ref{WaveguidePbDissipationBounded}), define $\hat{u}_\eta$ such that
\[
\hat{u}_\eta(x,y)=\begin{array}{|ll}
u_\eta(x,y) & \mbox{ in }\Om_L\\[2pt]
\dsp\sum_{n=1}^{+\infty} (u_\eta,\varphi_n)_{\mL^2(\Sigma_{\pm L})}\,e^{\pm i\beta_n^\eta (x\mp L)}\,\varphi_n(y)
 & \mbox{ in }\mathcal{S}_{\pm}.
\end{array}
\]
The function $\hat{u}_\eta$ satisfies $\Delta \hat{u}_\eta+(k^2+ik\eta)\hat{u}_\eta=f$ in $\Om_L\cup \mathcal{S}_+\cup \mathcal{S}_-$ (remember that $f$ is assumed to be zero in $\mathcal{S}_+\cup \mathcal{S}_-$). Moreover we have $[\hat{u}_\eta]|_{\Sigma_{\pm L}}=0$ as well as $[\partial_x \hat{u}_\eta]|_{\Sigma_{\pm L}}=0$ where $[\cdot]|_{\Sigma_{\pm L}}$ stands for the jump at $x=\pm L$. This is enough to conclude that $\hat{u}_\eta$ is the solution of (\ref{WaveguidePbDissipation}) in $\mH^1_0(\Om)$.
\end{proof}

\subsection{Problem without dissipation}\label{ParaLimiting}

Consider again some $f\in\mL^2(\Om)$ such that $f(x,y)=0$ for $|x|>L$. Taking the limit $\eta$ tends to zero in (\ref{WaveguidePbDissipationBounded}), we are led to consider the problem 
\begin{equation}\label{WaveguidePbBounded}
\begin{array}{|rcll}
\multicolumn{4}{|l}{\mbox{Find }u\in\mH^1_0(\Om_L;\Gamma)\mbox{ such that }}\\[2pt]
-\Delta u-k^2u&=&f&\mbox{ in }\Om_L \\[2pt]
\cfrac{\partial u}{\partial \nu} & = & \Lambda_\pm(u|_{\Sigma_{\pm L}}) &\mbox{ on }\Sigma_{\pm L}
\end{array}
\end{equation} 
where $\Lambda_\pm$ are the operators such that
\begin{equation}\label{DefDTN}
\begin{array}{lrcl}
\Lambda_\pm:& \mH^{1/2}_{00}(\Sigma_{\pm L}) & \to& \mH^{-1/2}(\Sigma_{\pm L}) \\[6pt]
 & \varphi & \mapsto  & \dsp\sum_{n=1}^{+\infty} i\beta_n\, (\varphi,\varphi_n)_{\mL^2(\Sigma_{\pm L})}\,\varphi_n(y).
\end{array}
\end{equation}
Here the $\beta_n$ are the ones introduced in (\ref{DefModes}). As in (\ref{EstimaDtNCont})--(\ref{EndContDtN}), one shows that the operators $\Lambda_\pm: \mH^{1/2}_{00}(\Sigma_{\pm L})  \to \mH^{-1/2}(\Sigma_{\pm L})$ are continuous.\\
\newline
The study of Problem (\ref{WaveguidePbBounded}) leads to consider the following variational problem 
\begin{equation}\label{PbNumBounded}
\begin{array}{|l}
\mbox{Find }u\in\mH^1_0(\Om_L;\Gamma) \mbox{ such that }\\[3pt]
a^{\mrm{out}}(u,v)=\ell(v),\qquad \forall v\in\mH^1_0(\Om_L;\Gamma),
\end{array}
\end{equation}
with 
\begin{equation}\label{DefForms}
\begin{array}{rcl}
a^{\mrm{out}}(u,v)&=&\dsp\int_{\Om_L}\hspace{-0.2cm}\nabla u\cdot\nabla\overline{v}-k^2 u\overline{v}\,dxdy-\langle\Lambda_+(u|_{\Sigma_{L}}),\overline{v} \rangle_{\Sigma_L}-\langle\Lambda_-(u|_{\Sigma_{-L}}),\overline{v} \rangle_{\Sigma_{-L}}\\[12pt]
\ell(v)&=&\dsp\int_{\Om_L}f\overline{v}\,dxdy.
\end{array}
\end{equation}
It is important to understand that the Dirichlet-to-Neumann operators $\Lambda_\pm$ in (\ref{WaveguidePbBounded}) on $\Sigma_{\pm L}$ have a double action. First, they allow to look for a solution which decomposes on propagating modes which are not in $\mH^1_0(\Om)$. However we cannot allow for all propagating modes in the functional space because otherwise by combining them, we could create non zero functions satisfying the homogeneous problem, \textit{i.e} we would obtain for all frequencies a non zero kernel. By working with the $\Lambda_\pm$, we also impose   radiation conditions and select outgoing behaviors (this explains the choice of the index ${}^{\mrm{out}}$ for ``outgoing''). To sum up, for $\pm x>L$ the solution, which is the physical one, is searched as a superposition of outgoing propagating modes and evanescent modes.\\
\newline
By exploiting the continuity of $\Lambda_\pm: \mH^{1/2}_{00}(\Sigma_{\pm L})  \to \mH^{-1/2}(\Sigma_{\pm L})$ as well as the fact that the trace mappings from $\mH^1_0(\Om_L)$ to $\mH^{1/2}_{00}(\Sigma_{\pm L})$ are also continuous, we deduce that the sesquilinear form $a^{\mrm{out}}(\cdot,\cdot)$ is continuous in $\mH^1_0(\Om_L)$. Therefore, with the Riesz representation theorem, we can introduce the linear operator $A^{\mrm{out}}(k):\mH^1_0(\Om_L;\Gamma)\to\mH^1_0(\Om_L;\Gamma)$ such that
\begin{equation}\label{DefOpAkOut}
(A^{\mrm{out}}(k)u,v)_{\mH^1(\Om_L)} =a^{\mrm{out}}(u,v),\qquad\forall u,v\in \mH^1_0(\Om_L;\Gamma).
\end{equation}
One has the following important statement.
\begin{theorem}\label{ThmFredholm}
For $k\in(\pi;+\infty)\setminus\{\N\pi\}$, the operator $A^{\mrm{out}}(k)$ decomposes as 
\[
A^{\mrm{out}}(k)=B^{\mrm{out}}+K
\]
where $B^{\mrm{out}}:\mH^1_0(\Om_L;\Gamma)\to\mH^1_0(\Om_L;\Gamma)$ is an isomorphism and $K:\mH^1_0(\Om_L;\Gamma)\to\mH^1_0(\Om_L;\Gamma)$ is compact ($B^{\mrm{out}}$ and $K$ are allowed to depend on $k$).
\end{theorem}
\begin{proof}
Define the continuous operator  $B^{\mrm{out}}:\mH^1_0(\Om_L;\Gamma)\to\mH^1_0(\Om_L;\Gamma)$ such that 
\[
(B^{\mrm{out}}u,v)_{\mH^1(\Om_L)}=b^{\mrm{out}}(u,v),\qquad\forall u,v\in \mH^1_0(\Om_L;\Gamma),
\]
with
\[
b^{\mrm{out}}(u,v)=\int_{\Om_L}\nabla u\cdot\nabla\overline{v}+u\overline{v}\,dxdy-\langle\Lambda_+(u|_{\Sigma_{L}}),\overline{v} \rangle_{\Sigma_L}-\langle\Lambda_-(u|_{\Sigma_{-L}}),\overline{v}\rangle_{\Sigma_{-L}}.
\]
For $u\in\mH^1_0(\Om_L;\Gamma)$, we have 
\begin{equation}\label{EstimCoer1}
\Re e\,b^{\mrm{out}}(u,u)=\|u\|^2_{\mH^1(\Om_L)}-\Re e\,\langle\Lambda_+(u|_{\Sigma_{L}}),\overline{u}\rangle_{\Sigma_L}-\Re e\,\langle\Lambda_-(u|_{\Sigma_{-L}}),\overline{u} \rangle_{\Sigma_{-L}}.
\end{equation}
But (\ref{DefDTN}) provides
\[
\langle\Lambda_\pm(u|_{\Sigma_{\pm L}}),\overline{u}\rangle_{\Sigma_{\pm L}} = \dsp\sum_{n=1}^{+\infty} i\beta_n\, |(u,\varphi_n)_{\mL^2(\Sigma_{\pm L})}|^2.
\]
Introduce $N\in\N^\ast$ such that $k\in(N\pi;(N+1)\pi)$. The $i\beta_n$ for $n=1,\dots,N$ are purely imaginary. On the other hand, the $i\beta_n$ for $n>N$ are real negative. Thus we have
\[
\Re e\,\langle\Lambda_\pm(u|_{\Sigma_{\pm L}}),\overline{u}\rangle_{\Sigma_{\pm L}}=\sum_{n=N+1}^{+\infty} i\beta_n\, |(u,\varphi_n)_{\mL^2(\Sigma_{\pm L})}|^2=\sum_{n=N+1}^{+\infty} -\sqrt{n^2\pi^2-k^2}\, |(u,\varphi_n)_{\mL^2(\Sigma_{\pm L})}|^2<0.
\]
With (\ref{EstimCoer1}), this gives
\[
\Re e\,b^{\mrm{out}}(u,u) \ge \|u\|^2_{\mH^1(\Om_L)}.
\]
From the complex version of the Lax-Milgram theorem, we deduce that $B^{\mrm{out}}:\mH^1_0(\Om_L;\Gamma)\to\mH^1_0(\Om_L;\Gamma)$ is an isomorphism.\\
\newline
Now set $K=A^{\mrm{out}}(k)-B^{\mrm{out}}$. We have 
\[
(Ku,v)_{\mH^1(\Om_L)}=-(1+k^2)\int_{\Om_L}u\overline{v}\,dxdy,\qquad\forall u,v\in \mH^1_0(\Om_L;\Gamma).
\] 
Since $\Om_L$ is bounded and has Lipschitz boundary, the embedding of $\mH^1_0(\Om_L;\Gamma)$ in $\mL^2(\Om_L)$ is compact and we can show that $K:\mH^1_0(\Om_L;\Gamma)\to\mH^1_0(\Om_L;\Gamma)$ is compact by working as for the operator $K$ appearing in the proof of Theorem \ref{ThmDLowFreq}.
\end{proof}
\noindent This shows that $A^{\mrm{out}}(k)$ satisfies the Fredholm alternative. Either $A^{\mrm{out}}(k)$ is injective and in this case it is an isomorphism of $\mH^1_0(\Om_L;\Gamma)$. Or $A^{\mrm{out}}(k)$ has a kernel of finite dimension $\mrm{span}(u_1,\dots,u_P)$ and in that case the equation 
\[
A^{\mrm{out}}(k)u=F\quad\mbox{ in }\mH^1_0(\Om_L;\Gamma)
\]
has a solution, defined up to $\mrm{span}(u_1,\dots,u_P)$, if and only if $F$ satisfies the compatibility conditions 
\begin{equation}\label{PbCompa1Bis}
(F,u_p)_{\mH^1(\Om_L)}=0,\qquad p=1,\dots,P.
\end{equation}
If $u$ solves (\ref{WaveguidePbBounded}), then by defining $\hat{u}$ such that 
\begin{equation}\label{defExten}
\hat{u}(x,y)=\begin{array}{|ll}
u(x,y) & \mbox{ in }\Om_L \\[3pt]
\dsp\sum_{n=1}^{+\infty} (u,\varphi_n)_{\mL^2(\Sigma_{\pm L})}\,e^{\pm i\beta_n (x\mp L)}\,\varphi_n(y) & \mbox{ for }\pm x>L,
\end{array}
\end{equation}
we obtain a solution of (\ref{WaveguidePb}). In general, this $\hat{u}$ does not belong to $\mH^1_0(\Om)$ because it involves propagating modes.\\
\newline
In the cases where $A^{\mrm{out}}(k)$ is not injective, we show now that the elements $u$ of its kernel do not decompose on the propagating modes. As a consequence, the corresponding extensions $\hat{u}$ are exponentially decaying as $x\to\pm\infty$. Again, we call them trapped modes. 

\begin{proposition}\label{PropoTrappedModes}
Assume that $k\in(N\pi;(N+1)\pi)$ for some $N\in\N^\ast$. Consider some  $u$ in $\ker\,A^{\mrm{out}}(k)$. Then its corresponding extension $\hat{u}$ defined via (\ref{defExten}) decays as $O(e^{-\sqrt{(N+1)^2\pi^2-k^2}|x|})$ as $x\to\pm\infty$ and so belongs to $\mH^1_0(\Om)$. 
\end{proposition}
\begin{proof}
If $u$ is an element of $\ker\,A^{\mrm{out}}(k)$, we have $a^{\mrm{out}}(u,u)=0$ and so 
\[
0=\Im m\,a^{\mrm{out}}(u,u)=-\sum_{n=1}^{N} \beta_n\, (|(u,\varphi_n)_{\mL^2(\Sigma_{+L})}|^2+|(u,\varphi_n)_{\mL^2(\Sigma_{-L})}|^2).
\]
Since the $\beta_1,\dots,\beta_N$ are all positive, this proves that the corresponding $\hat{u}$ in (\ref{defExten}) decomposes only on the evanescent modes.
\end{proof}
\noindent Finally, we show that the occurrences where $A^{\mrm{out}}(k)$ is not injective are rare.
\begin{proposition}\label{PropoDiscreteTrappedModes}
The set of $k\in(\pi;+\infty)\setminus\{\N\pi\}$ such that $A^{\mrm{out}}(k)$ is not injective is discrete and accumulates only at $+\infty$.
\end{proposition}
\begin{proof}
We reproduce the proof presented in \cite[Thm.\,3.3]{BoSt94} (see also \cite[Thm.\,2.2]{BoLu12}). 
Assume that $k\in(N\pi;(N+1)\pi)$ for some $N\in\N^\ast$. If $u$ belongs to $\ker\,A^{\mrm{out}}(k)$, then according to Proposition \ref{PropoTrappedModes}, it does not decompose on the propagating modes. As a consequence, from (\ref{PbNumBounded}), we find that it satisfies
\begin{equation}\label{SpectralPb}
\mathfrak{a}(u,v)=
\lambda \dsp\int_{\Om_L}u\overline{v}\,dxdy,\qquad\forall v\in \mH^1_0(\Om_L;\Gamma),
\end{equation}
with $\lambda=k^2$ and 
\[
\mathfrak{a}(u,v)=\dsp\int_{\Om_L}\nabla u\cdot\nabla\overline{v}\,dxdy+\sum_\pm\sum_{n=N+1}^{+\infty} \sqrt{n^2\pi^2-k^2}\, (u,\varphi_n)_{\mL^2(\Sigma_{\pm L})}\overline{(v,\varphi_n)_{\mL^2(\Sigma_{\pm L})}}.
\]
Note that contrary to the form $a^{\mrm{out}}(\cdot,\cdot)$ defined in (\ref{DefForms}) which involves the propagating modes, $\mathfrak{a}(\cdot,\cdot)$ is hermitian. Using also that $\mH^1_0(\Om_L;\Gamma)$ is compactly embedded in $\mL^2(\Om_L)$, one shows classically that for a given $k$, the spectrum of Problem (\ref{SpectralPb}) is made of an increasing sequence of 
non-negative eigenvalues $(\lambda_m(k))_{m\in\N^\ast}$ which tends to $+\infty$. Additionally, the theory of selfadjoint operators (see \textit{e.g} \cite[Chap.\,3]{Wein74}) ensures that one has the min-max formula, for all $m\in\N^\ast$,
\[
\lambda_m(k) =\underset{ V_m\subset\mathcal{V}_m}{\min}\,\underset{v\in V_m\setminus\{0\}}{\max}\cfrac{\mathfrak{a}(v,v)}{\|v\|_{\mL^2(\Om_L)}}\,,
\]
where $\mathcal{V}_m$ denotes the set of all m-dimensional subspaces of $\mH^1_0(\Om_L;\Gamma)$. Using this characterization, one proves that for any $m\in\N^\ast$, $k\mapsto \lambda_m(k)$ is non-increasing and continuous on $(N\pi;(N+1)\pi)$. As a consequence, for each $m\in\N^\ast$, the fixed point equation  $\lambda_m(k)=k^2$ has at most one solution in $(N\pi;(N+1)\pi)$. Working a bit more, one establishes that the equation  $\lambda_m(k)=k^2$ can have a solution $k\in (N\pi;(N+1)\pi)$ only for a finite number of $m\in\N^\ast$. This is enough to conclude.\\ 
\newline
Let us emphasize that the non zero $u$ satisfying (\ref{SpectralPb}) are not necessarily trapped modes because test functions are missing in (\ref{SpectralPb}) compared to (\ref{PbNumBounded}). Therefore, though this approach suffices to guarantees the discreteness of eigenvalues for $A(k)$, it cannot be used directly to prove the existence of trapped modes.
\end{proof}

\subsection{Limiting absorption principle}\label{ParaLimitingAbs}

In this paragraph, we prove that the dissipative solution converges to the solution of Problem (\ref{WaveguidePbBounded}) without dissipation as $\eta$ tends to zero.

\begin{theorem}
Fix $k\in(\pi;+\infty)\setminus\{\N\pi\}$ and assume that the operator $A^{\mrm{out}}(k)$ defined in (\ref{DefOpAkOut}) is injective. Then there is $\eta_0>0$ such that we have
\[
\|u-u_\eta\|_{\mH^1(\Om_L)} \le C\eta\,\|f\|_{\mL^2(\Om_L)},\qquad \forall \eta\in(0;\eta_0],
\]
with a constant $C>0$ which is independent of $\eta$. Here $u$, $u_\eta$ are the functions solving respectively (\ref{WaveguidePbBounded}), (\ref{WaveguidePbDissipationBounded}).
\end{theorem}
\begin{proof}
Let $F$ be the element of $\mH^1_0(\Om_L;\Gamma)$ such that 
\[
(F,v)_{\mH^1(\Om_L)}=\ell(v),\qquad\forall v\in\mH^1_0(\Om_L;\Gamma).
\]
Denote also by $A^\eta:\mH^1_0(\Om_L;\Gamma)\to\mH^1_0(\Om_L;\Gamma)$ the operator such that
\[
(A^\eta v,w)_{\mH^1(\Om_L)}=\dsp\int_{\Om_L}\hspace{-0.2cm}\nabla v\cdot\nabla\overline{w}-(k^2+ik\eta) v\overline{w}\,dxdy-\langle\Lambda^\eta_+(v|_{\Sigma_{L}}),\overline{w} \rangle_{\Sigma_L}-\langle\Lambda^\eta_-(v|_{\Sigma_{-L}}),\overline{w} \rangle_{\Sigma_{-L}}
\]  
for all $v,w\in\mH^1_0(\Om_L;\Gamma)$. We have, for all $\eta>0$,
\[
A^\eta u_\eta=F=A^{\mrm{out}} u. 
\]
This gives $A^{\mrm{out}} (u-u_\eta)=(A^\eta-A^{\mrm{out}})u_\eta$. Since $A^{\mrm{out}}$ is assumed to be injective, Theorem \ref{ThmFredholm} together with the Fredholm alternative guarantee that $A^{\mrm{out}}$ is an isomorphism of $\mH^1_0(\Om_L;\Gamma)$. Thus we can write
\begin{equation}\label{Relation1}
u-u_\eta=(A^{\mrm{out}})^{-1} (A^\eta-A^{\mrm{out}})u_\eta.
\end{equation}
Let us estimate
\begin{equation}\label{DefNormDif}
\|A^\eta-A^{\mrm{out}}\| =\underset{v\in\mH^1_0(\Om_L;\Gamma)\setminus\{0\}}{\sup}\,\cfrac{\|(A^\eta -A^{\mrm{out}}) v\|_{\mH^1(\Om_L)}}{\|v\|_{\mH^1(\Om_L)}}\,.
\end{equation}
From the definition of $A^{\mrm{out}}$, $A^\eta$, one finds, for $v,w\in\mH^1_0(\Om_L;\Gamma)$, 
\[
\begin{array}{ll}
&((A^\eta -A^{\mrm{out}}) v,w)_{\mH^1(\Om_L)} \\[8pt]
=&-ik\eta\dsp\int_{\Om_L} v\,\overline{w}\,dxdy -\langle\Lambda^\eta_+(v|_{\Sigma_{L}})-\Lambda_+(v|_{\Sigma_{L}}),\overline{w} \rangle_{\Sigma_L}-\langle\Lambda^\eta_-(v|_{\Sigma_{-L}})-\Lambda_-(v|_{\Sigma_{-L}}),\overline{w }\rangle_{\Sigma_{-L}}. 
\end{array}
\]
Taking $w=(A^\eta -A^{\mrm{out}}) v$ in this equality and using the continuity of the trace mappings from $\mH^1(\Om_L)$ to $\mH^{1/2}(\Sigma_{\pm L})$, we deduce the existence of a constant
$C>0$, which may change from one line to another below but which remains independent of $\eta\in(0;\eta_0]$, such that
\begin{equation}\label{DefNormDif2}
\begin{array}{l}
\|(A^\eta -A^{\mrm{out}}) v\|_{\mH^1(\Om_L)} \le k\eta\,\|v\|_{\mH^1(\Om_L)}\\[9pt]
\hspace{1.6cm}+C\,(\|\Lambda^\eta_+(v|_{\Sigma_{L}})-\Lambda_+(v|_{\Sigma_{L}})\|_{\mH^{-1/2}(\Sigma_{L})}+\|\Lambda^\eta_-(v|_{\Sigma_{-L}})-\Lambda_-(v|_{\Sigma_{-L}})\|_{\mH^{-1/2}(\Sigma_{-L})}).
\end{array}
\end{equation} 
Let us estimate the second term in the right-hand side above. For $\varphi,\varphi'\in \mH^{1/2}_{00}(\Sigma_{\pm L})$, we can write
\begin{equation}\label{EstimaDtNCont0Bis}
|\langle \Lambda_\pm^\eta(\varphi)-\Lambda_\pm(\varphi),\overline{\varphi'}\rangle_{\Sigma_{\pm L}}| \le \dsp\sum_{n=1}^{+\infty} |\beta_n^\eta-\beta_n|\,|(\varphi,\varphi_n)_{\mL^2(\Sigma_{\pm L})}|\,|(\varphi',\varphi_n)_{\mL^2(\Sigma_{\pm L})}|. 
\end{equation}
Working from the expressions of $\beta_n^\eta$, $\beta_n$, one shows that for any given $\eta_0>0$, there is a constant $C>0$ independent of $\eta\in(0;\eta_0]$ such that 
\[
\underset{n\in\N^\ast}{\sup} \cfrac{|\beta_n^\eta-\beta_n|}{n\pi} \le C\eta.
\]
Exploiting this in (\ref{EstimaDtNCont0Bis}) gives
\[
\begin{array}{rcl}
|\langle \Lambda_\pm^\eta(\varphi)-\Lambda_\pm(\varphi),\overline{\varphi'}\rangle_{\Sigma_{\pm L}}|
&\le &  \dsp C\eta\,\left(\dsp\sum_{n=1}^{+\infty}n\pi|(\varphi,\varphi_n)_{\mL^2(\Sigma_{\pm L})}|^2\right)^{1/2} \left(\dsp\sum_{n=1}^{+\infty}n\pi|(\varphi',\varphi_n)_{\mL^2(\Sigma_{\pm L})}|^2\right)^{1/2} \\[12pt]
&\le &  \dsp C\eta\,\|\varphi\|_{\mH^{1/2}(\Sigma_{ L})}\|\varphi'\|_{\mH^{1/2}(\Sigma_{ L})}.
\end{array}
\]
From this, we deduce the estimates
\begin{equation}\label{DefNormDif3}
\|\Lambda^\eta_\pm(v|_{\Sigma_{\pm L}})-\Lambda_\pm(v|_{\Sigma_{\pm L}})\|_{\mH^{-1/2}(\Sigma_{\pm L})}\le C\,\eta\,\|v\|_{\mH^{1/2}(\Sigma_{\pm L})} \le C\,\eta\,\|v\|_{\mH^{1}(\Om_L)}.
\end{equation}
By combining (\ref{DefNormDif2}), (\ref{DefNormDif3}) and  (\ref{DefNormDif}), one gets, for $\eta\in(0;\eta_0]$,
\begin{equation}\label{Relation2}
\|A^\eta-A^{\mrm{out}}\| \le C \eta,
\end{equation}
Then, gathering (\ref{Relation2}) and (\ref{Relation1}), we find
\begin{equation}\label{FirstEstimate}
\|u-u_\eta\|_{\mH^1(\Om_L)} \le C \eta\, \|u_\eta\|_{\mH^1(\Om_L)}.
\end{equation}
Finally, by using the inequality $\|u_\eta\|_{\mH^1(\Om_L)}\le \|u-u_\eta\|_{\mH^1(\Om_L)}+\|u\|_{\mH^1(\Om_L)}  $ in (\ref{FirstEstimate}), we obtain, for $\eta$ small,
\[
\|u-u_\eta\|_{\mH^1(\Om_L)} \le C\eta\,\|u\|_{\mH^1(\Om_L)}\le C\eta\,\|(A^{\mrm{out}})^{-1}f\|_{\mH^1(\Om_L)} \le C\eta\,\|f\|_{\mL^2(\Om_L)}.
\]
This completes the proof.
\end{proof}

\subsection{Approximated problem with truncated Dirichlet-to-Neumann operators}\label{ParaTroncDTN}
In practice, to solve (\ref{PbNumBounded}) numerically as in \S\ref{ParaNumApprox} below, we will need to truncate the series appearing in the Dirichlet-to-Neumann operators. In this section, we prove error estimates between the exact solution and the one obtained via this approximation.\\
\newline
Recall that we have introduced $L>d>0$ such that $\Om$ coincides with the strip $\R\times I$ for $|x|>d$ (see Figure \ref{PictureDtN}). Again we set 
\[
\Om_d\coloneqq\{(x,y)\in\Om\,|\,|x|<d\},\qquad \Om_L\coloneqq\{(x,y)\in\Om\,|\,|x|<L\}
\]
so that one has $\Om_d\subset\Om_L$. In this section, we make the additional assumption that $f\in\mL^2(\Om)$ is such that $f(x,y)=0$ for $|x|>d$. For the technique we present below, it is important to exploit that $f$ vanishes in a neighborhood of $\Sigma_L\cup \Sigma_{-L}$. Assume that $k\in(N\pi;(N+1)\pi)$ for some $N\in\N^\ast$. For $M\ge N$, define the truncated Dirichlet-to-Neumann operators such that
\[
\begin{array}{lrcl}
\Lambda^M_\pm:& \mH^{1/2}_{00}(\Sigma_{\pm L}) & \to& \mH^{-1/2}(\Sigma_{\pm L}) \\[6pt]
 & \varphi & \mapsto  & \dsp\sum_{n=1}^{M} i\beta_n\, (\varphi,\varphi_n)_{\mL^2(\Sigma_{\pm L})}\,\varphi_n(y).
\end{array}
\]
Note that it is crucial to take into account at least all propagating modes. Then, consider the approximated problem
\begin{equation}\label{PbNumBoundedTrunc}
\begin{array}{|l}
\mbox{Find }u_M\in\mH^1_0(\Om_L;\Gamma) \mbox{ such that }\\[3pt]
a^{\mrm{out}}_M(u_M,v)=\ell(v),\qquad \forall v\in\mH^1_0(\Om_L;\Gamma)
\end{array}
\end{equation}
where $\ell(\cdot)$ is the form defined in (\ref{DefForms}) and 
\begin{equation}\label{DefFormsTrunc}
\begin{array}{rcl}
a^{\mrm{out}}_M(u,v)&=&\dsp\int_{\Om_L}\hspace{-0.2cm}\nabla u\cdot\nabla\overline{v}-k^2 u\overline{v}\,dxdy-\langle\Lambda^M_+(u|_{\Sigma_{L}}),\overline{v} \rangle_{\Sigma_L}-\langle\Lambda^M_-(u|_{\Sigma_{-L}}),\overline{v} \rangle_{\Sigma_{-L}}.
\end{array}
\end{equation}
\noindent Our main result in the section is the following:
\begin{theorem}\label{MainThm}
Assume that $k\in(N\pi;(N+1)\pi)$, $N\in\N^\ast$, and that $f\in\mL^2(\Om)$ satisfies $f(x,y)=0$ for $|x|>d$. Assume also that the operator $A^{\mrm{out}}(k)$ defined in (\ref{DefOpAkOut}) is injective. Then\\[3pt]
-  (\ref{PbNumBounded}) admits a unique solution $u$;\\[2pt]
- (\ref{PbNumBoundedTrunc}) admits a unique solution $u_M$ for $M$ large enough.\\[4pt]
Moreover, there exists $M_0\ge N$ such that one has the estimate
\begin{equation}\label{Error_estimate}
\|u-u_M\|_{\mH^1(\Om_d)} \le C\,e^{-2|\beta_{M+1}|(L-d)}\|f\|_{\mrm{L}^2(\Om_d)},\qquad\forall M\ge M_0.
\end{equation}
Here $C>0$ is a constant independent of $d$, $L$ and $M$.
\end{theorem}
\begin{corollary}\label{CoroMainThm}
With the same assumptions and notation as Theorem \ref{MainThm}, one has the error estimate
\begin{equation}\label{Error_estimate2}
\|u-u_M\|_{\mH^1(\Om_L)} \le C\,e^{-|\beta_{M+1}|(L-d)}\|f\|_{\mrm{L}^2(\Om_d)},\qquad\forall M\ge M_0.
\end{equation}
\end{corollary}
\begin{remark}
Note the difference of exponent in the right hand side of the error estimates (\ref{Error_estimate}), (\ref{Error_estimate2}) set respectively in $\Om_d$ and $\Om_L$. 
\end{remark}
\begin{proof}[Proof of Theorem \ref{MainThm}]
We have already mentioned that Theorem \ref{ThmFredholm} and the Fredholm alternative ensure that (\ref{PbNumBounded}) has a unique solution $u$ when $A^{\mrm{out}}(k)$ is injective.\\
\newline
Now introduce the following problem set in the small domain $\Om_d$
\begin{equation}\label{PbNumBounded_d}
\begin{array}{|l}
\mbox{Find }v\in\mH^1_0(\Om_d;\Gamma) \mbox{ such that }\\[3pt]
a^{\infty}_d(v,w)=\dsp\int_{\Om_d}f\overline{w}\,dxdy,\qquad \forall w\in\mH^1_0(\Om_d;\Gamma),
\end{array}
\end{equation}
with 
\[
\begin{array}{rcl}
a^{\infty}_d(v,w)&=&\dsp\int_{\Om_d}\hspace{-0.2cm}\nabla v\cdot\nabla\overline{w}-k^2 v\overline{w}\,dxdy-\langle\Lambda^{\infty}_+(v|_{\Sigma_{d}}),\overline{w} \rangle_{\Sigma_d}-\langle\Lambda^{\infty}_-(v|_{\Sigma_{-d}}),\overline{w} \rangle_{\Sigma_{-d}}.
\end{array}
\]
Here the quantities $\mH^1_0(\Om_d;\Gamma)$, $\Lambda^{\infty}_\pm$ are defined as $\mH^1_0(\Om_L;\Gamma)$, $\Lambda_\pm$ by replacing $L$ by $d$. With the Riesz representation theorem, introduce the operator $A^\infty_d:\mH^1_0(\Om_d;\Gamma)\to\mH^1_0(\Om_d;\Gamma)$ such that
\[
(A^\infty_dv,w)_{\mH^1(\Om_d)} =a^{\infty}_d(v,w),\qquad\forall v,w\in \mH^1_0(\Om_d;\Gamma).
\]
Reproducing what has been done in Theorem \ref{ThmFredholm}, one establishes that $A^\infty_d$ is Fredholm of index zero. Moreover, working as in (\ref{defExten}), one shows that if $u$ satisfies (\ref{PbNumBounded}), then $u|_{\Om_d}$ solves (\ref{PbNumBounded_d}). Conversely, if $v$ satisfies (\ref{PbNumBounded_d}), $v$ can be extended to a solution of (\ref{PbNumBounded}). Since $A^{\mrm{out}}(k):\mH^1_0(\Om_L;\Gamma)\to\mH^1_0(\Om_L;\Gamma)$ is an isomorphism, we deduce that $A^\infty_d:\mH^1_0(\Om_d;\Gamma)\to\mH^1_0(\Om_d;\Gamma)$ is also an isomorphism.\\
\newline
Below, we shall prove the following lemma:
\begin{lemma}\label{lemmaEstim}
If $u_M$ satisfies (\ref{PbNumBoundedTrunc}), then $u_M|_{\Om_d}$ solves the problem
\begin{equation}\label{PbNumBoundedTruncPert}
\begin{array}{|l}
\mbox{Find }v\in\mH^1_0(\Om_d;\Gamma) \mbox{ such that }\\[3pt]
a^{\infty}_d(v,w)+r^M_d(v,w)=\dsp\int_{\Om_d}f\overline{w}\,dxdy,\qquad \forall w\in\mH^1_0(\Om_d;\Gamma),
\end{array}
\end{equation}
with 
\[
r^M_d(v,w)=\sum_{n=M+1}^{+\infty}\cfrac{2i\beta_n}{1+e^{2|\beta_n|(L-d)}}\,\left(
(v,\varphi_n)_{\Sigma_{d}}\overline{(w,\varphi_n)_{\Sigma_{ d}}}+(v,\varphi_n)_{\Sigma_{- d}}\overline{(w,\varphi_n)_{\Sigma_{- d}}}\,\right).
\]
\end{lemma}
\noindent Introduce the operator $\tilde{A}^M_d:\mH^1_0(\Om_d;\Gamma)\to \mH^1_0(\Om_d;\Gamma)$ such that
\[
(\tilde{A}^M_d v,w)_{\mH^1(\Om_d)} =a^{\infty}_d(v,w)+r^M_d(v,w),\qquad\forall v,w\in \mH^1_0(\Om_d;\Gamma).
\]
We emphasize that $\tilde{A}^M_d$ is different from what would be the operator corresponding to a problem with truncated Dirichlet-to-Neumann maps at $\Sigma_{\pm d}$. From the expression of $r^M_d(\cdot,\cdot)$, we get
\[
\|A^\infty_d-\tilde{A}^M_d\| = \underset{ v\in\mH^1_0(\Om_d;\Gamma)\setminus\{0\}}{\sup} \cfrac{\|(A^\infty_d-\tilde{A}^M_d)v\|_{\mH^1(\Om_d)}}{\|v\|_{\mH^1(\Om_d)}} \le C\,e^{-2|\beta_{M+1}|(L-d)}
\]
where, here and below, $C>0$ is a constant which may change from one line to another but which remains independent of $d$, $L$ and $M$. Since $A^\infty_d$ is an isomorphism, this guarantees that $\tilde{A}^M_d$ is an isomorphism for $M$ large enough. Working by contradiction, with Lemma \ref{lemmaEstim} we infer that the operator associated with (\ref{DefFormsTrunc}) is injective for $M$ large enough. Since one can show that this operator is Fredholm of index zero, this proves that (\ref{PbNumBoundedTrunc}) admits a unique solution $u_M$ for $M$ large enough. Additionally, we have 
\[
A^\infty_d u=\tilde{A}^M_d u_M\quad\Leftrightarrow\quad A^\infty_d(u-u_M)=(\tilde{A}^M_d-A^\infty_d) u_M \quad\Leftrightarrow\quad u-u_M=(A^\infty_d)^{-1}(\tilde{A}^M_d-A^\infty_d) u_M.
\]
This gives
\[
\|u-u_M\|_{\mH^1(\Om_d)} \le C\,e^{-2|\beta_{M+1}|(L-d)}\|u_M\|_{\mH^1(\Om_d)}.
\]
Using the inequality $\|u_M\|_{\mH^1(\Om_d)}-\|u\|_{\mH^1(\Om_d)} \le \|u-u_M\|_{\mH^1(\Om_d)} $, we obtain for $M$ large enough,
\begin{equation}\label{EstimInterErrEstim}
(1-C\,e^{-2|\beta_{M+1}|(L-d)})\|u_M\|_{\mH^1(\Om_d)} \le \|u\|_{\mH^1(\Om_d)}.
\end{equation}
Thus, we deduce
\[
\|u-u_M\|_{\mH^1(\Om_d)} \le C\,e^{-2|\beta_{M+1}|(L-d)}\|u\|_{\mH^1(\Om_d)} \le  C\,e^{-2|\beta_{M+1}|(L-d)}\|f\|_{\mrm{L}^2(\Om_d)},
\] 
which ends the proof of the theorem.
\end{proof}
\begin{proof}[Proof of Corollary \ref{CoroMainThm}]
Let us now establish an error estimate not only in $\mH^1(\Om_d)$ but in $\mH^1(\Om_L)$. Let us start from $a^{\mrm{out}}(u,v)=a^{\mrm{out}}_M(u_M,v)$ for all $v\in\mH^1_0(\Om_L;\Gamma)$. This implies 
\begin{equation}\label{ErrorEstima1}
a^{\mrm{out}}(u-u_M,v)=a^{\mrm{out}}_M(u_M,v)-a^{\mrm{out}}(u_M,v),\qquad\forall v\in \mH^1_0(\Om_L;\Gamma).
\end{equation}
For the right-hand side, using in particular the equality (\ref{sys4}) below, we find 
\begin{equation}\label{ErrorEstima2}
\begin{array}{ll}
&a^{\mrm{out}}_M(u_M,v)-a^{\mrm{out}}(u_M,v)\\[10pt]
=&\dsp\sum_{n=M+1}^{+\infty}i\beta_n\,\left((u_M,\varphi_n)_{\Sigma_{L}}\overline{(v,\varphi_n)_{\Sigma_{L}}}+(u_M,\varphi_n)_{\Sigma_{-L}}\overline{(v,\varphi_n)_{\Sigma_{-L}}}\,\right)\\[10pt]
=&\dsp\sum_{n=M+1}^{+\infty}\cfrac{2i\beta_n}{e^{|\beta_n|(L-d)}+e^{-|\beta_n|(L-d)}}\,\left((u_M,\varphi_n)_{\Sigma_{d}}\overline{(v,\varphi_n)_{\Sigma_{L}}}+(u_M,\varphi_n)_{\Sigma_{-d}}\overline{(v,\varphi_n)_{\Sigma_{-L}}}\,\right).
\end{array}
\end{equation}
Combining (\ref{ErrorEstima1}) and (\ref{ErrorEstima2}), we obtain
\[
\|A^{\mrm{out}}(k)(u-u_M)\|_{\mH^1(\Om_L)} \le C\,e^{-|\beta_{M+1}|(L-d)}\,\|u_M\|_{\mH^1(\Om_d)}\le C\,e^{-|\beta_{M+1}|(L-d)}\|f\|_{\mrm{L}^2(\Om_d)}. 
\]
The last inequality above is a consequence of (\ref{EstimInterErrEstim}). Finally since $A^{\mrm{out}}(k):\mH^1_0(\Om_L;\Gamma)\to\mH^1_0(\Om_L;\Gamma)$ is an isomorphism, this yields (\ref{Error_estimate2}).
\end{proof}
\noindent Finally, we give the proof of the intermediate lemma above.
\begin{proof}[Proof of Lemma \ref{lemmaEstim}]
If $u_M$ solves (\ref{PbNumBoundedTrunc}), using decomposition in Fourier series, we obtain 
\begin{equation}\label{decompo_Modes}
u_M(x,y)=\sum_{n=1}^{+\infty}(a^\pm_n e^{\pm i\beta_n(x\mp d)}+b_n^\pm e^{\mp i\beta_n(x\mp d)})\,\varphi_n(y)\quad\mbox{ for }d<\pm x <L,
\end{equation}
for certain complex constants $a^\pm_n$, $b^\pm_n$. By evaluating these expressions respectively on $\Sigma_{\pm d}$ and $\Sigma_{\pm L}$, we obtain
\begin{eqnarray}
a^\pm_n+b^\pm_n&=&(u_M,\varphi_n)_{\Sigma_{\pm d}} \label{sys1}\\
 a^\pm_ne^{i\beta_n(L-d)}+b^\pm_ne^{- i\beta_n(L-d)}&=&(u_M,\varphi_n)_{\Sigma_{\pm L}}.\label{sys1bis}
\end{eqnarray}
On the other hand, when $u_M$ solves (\ref{PbNumBoundedTrunc}), it satisfies on $\Sigma_{\pm L}$
\[
\cfrac{\partial u_M}{\partial \nu} = \Lambda^M_\pm(u|_{\Sigma_{\pm L}})\qquad\Leftrightarrow\qquad\pm\cfrac{\partial u_M}{\partial x}=\sum_{n=1}^M i\beta_n (u_M,\varphi_n)_{\Sigma_{\pm L}}\varphi_n(y).
\]
From (\ref{decompo_Modes}), this implies
\begin{equation}\label{sys2}
\begin{array}{|ll}
a^\pm_ne^{i\beta_n(L-d)}-b^\pm_ne^{-i\beta_n(L-d)}=(u_M,\varphi_n)_{\Sigma_{\pm L}}&\mbox{ for }n\le M\\[3pt]
a^\pm_ne^{i\beta_n(L-d)}-b^\pm_ne^{-i\beta_n(L-d)}=0&\mbox{ for }n > M.
\end{array}
\end{equation}
Gathering (\ref{sys1bis}) and (\ref{sys2}) gives
\begin{equation}\label{sys3}
\begin{array}{|ll}
a^\pm_n=e^{-i\beta_n (L-d)}(u_M,\varphi_n)_{\Sigma_{\pm L}},\quad b^\pm_n=0, &\quad\mbox{ for }n\le M\\[3pt]
a^\pm_n=e^{-i\beta_n (L-d)}(u_M,\varphi_n)_{\Sigma_{\pm L}}/2,\quad b^\pm_n=e^{i\beta_n (L-d)}(u_M,\varphi_n)_{\Sigma_{\pm L}}/2, &\quad\mbox{ for }n > M.
\end{array}
\end{equation}
In particular, with (\ref{sys1}) and (\ref{sys3}), we obtain
\begin{equation}\label{sys4}
(u_M,\varphi_n)_{\Sigma_{\pm L}}=\cfrac{2}{e^{-i\beta_n(L-d)}+e^{i\beta_n(L-d)}}\,(u_M,\varphi_n)_{\Sigma_{\pm d}}\quad\mbox{ for }n > M.
\end{equation}
Using successively (\ref{decompo_Modes}), (\ref{sys1}), (\ref{sys3}), (\ref{sys4}), we obtain, on $\Sigma_{\pm d}$,
\[
\begin{array}{rcl}
\pm\cfrac{\partial u_M}{\partial x}&=&\dsp\sum_{n=1}^{+\infty} i\beta_n (a^\pm_n-b^\pm_n)\varphi_n(y)\\[12pt]
&=&\dsp\sum_{n=1}^{+\infty} i\beta_n (a^\pm_n+b^\pm_n)\varphi_n(y)-\sum_{n=1}^{+\infty} 2i\beta_n b^\pm_n\varphi_n(y)\\[12pt]
&=&\dsp\sum_{n=1}^{+\infty} i\beta_n (u_M,\varphi_n)_{\Sigma_{\pm d}}\varphi_n(y)-\sum_{n=M+1}^{+\infty} 2i\beta_n b^\pm_n\varphi_n(y)\\[12pt]
&=&\dsp\sum_{n=1}^{+\infty} i\beta_n (u_M,\varphi_n)_{\Sigma_{\pm d}}\varphi_n(y)-\sum_{n=M+1}^{+\infty} 2i\beta_n \cfrac{(u_M,\varphi_n)_{\Sigma_{\pm d}}}{1+e^{2|\beta_n|(L-d)}}\,\varphi_n(y).
\end{array}
\]
This is enough to obtain (\ref{PbNumBoundedTruncPert}).
\end{proof}

\subsection{Scattering problem}\label{ParaScaPb}

In (\ref{WaveguidePbBounded}), we considered a problem with a source term $f$. In the following, we will be mostly interested in scattering problems for incident waves. To keep things as simple as possible, we assume all through this paragraph that $k\in(\pi;2\pi)$ so that only the modes $w^\pm_1$ can propagate. To make short, we denote them by $w_\pm$ so that
\begin{equation}\label{Defwpm}
w_\pm(x,y)=e^{\pm i\beta_1 x}\varphi_1(y)=e^{\pm i\sqrt{k^2-\pi^2} x}\varphi_1(y).
\end{equation}
The scattering of the rightgoing wave $w_+$ coming from the left branch of the waveguide leads us to consider the problem
\begin{equation}\label{WaveguidePbSca}
\begin{array}{|rcll}
\multicolumn{4}{|l}{\mbox{Find }u_+\in\mH^1_{0,\loc}(\Om)\mbox{ such that }u_+-w_+\mbox{ is outgoing and } }\\[3pt]
\Delta u_++k^2u_+&=&0&\mbox{ in }\Om \\[2pt]
u_+&=&0&\mbox{ on }\partial\Om.
\end{array}
\end{equation}
Here $\mH^1_{0,\loc}(\Om)$ denotes the set of measurable functions $v$ such that $\zeta v$ belongs to $\mH^1_0(\Om)$ for all $\zeta\in\mathscr{C}^{\infty}_0(\R^2)$. Moreover, the sentence $u_+-w_+$ is outgoing means that we impose the decomposition
\[
u_+-w_+=\left\{\begin{array}{ll}
\dsp\sum_{n=1}^{+\infty} \alpha^+_ne^{ i\beta_n x}\,\varphi_n(y) & \mbox{ for }x>L\\[10pt]
\dsp\sum_{n=1}^{+\infty} \alpha^-_ne^{-i\beta_n x}\,\varphi_n(y) & \mbox{ for } x<-L,
\end{array}\right.
\]
for some $\alpha^\pm_n\in\Cplx$. In this context, $u_+$ is usually called the total field associated with the incident field $w_+$ while the quantity $u_+-w_+$ is the scattered field. 
\begin{proposition}\label{PropoWPDirichletSca}
For all $k\in(\pi;2\pi)$, (\ref{WaveguidePbSca}) admits a solution. It is unique if trapped modes do not exist.
\end{proposition}
\begin{proof}
Introduce some cut-off function $\zeta\in\mathscr{C}^{\infty}(\R)$ such that  
\[
\zeta(x)=\begin{array}{|ll}
1 & \mbox{ for }x\le-L\\[2pt]
0 & \mbox{ for }x\ge-d.
\end{array}
\] 
If $u_+$ satisfies (\ref{WaveguidePbSca}), then $u^s_+=u_+-\zeta w_+$, which is outgoing, solves Problem (\ref{WaveguidePbBounded}) with 
\begin{equation}\label{DefFSca}
\begin{array}{rcl}
f&\coloneqq&\Delta(\zeta w_+)+k^2(\zeta w_+)\\[3pt]
&=&(\Delta w_++k^2 w_+)\zeta+2\nabla w_+\cdot \nabla\zeta+w_+\Delta\zeta \ = \ 2\nabla w_+\cdot \nabla\zeta+w_+\Delta\zeta.
\end{array}
\end{equation}
Note in particular that $f$ is indeed an element of $\mL^2(\Om)$ such that $f(x,y)=0$ for $\pm x\ge L$. Conversely, if $u^s_+$ satisfies (\ref{WaveguidePbBounded}) with the above $f$, then $u_+\coloneqq u^s_++\zeta w_+$ is a solution of (\ref{WaveguidePbSca}). Therefore it is sufficient to focus our attention on the study of (\ref{WaveguidePbBounded}) with $f$ defined in (\ref{DefFSca}).\\ 
\newline
If trapped modes do not exist, Theorem \ref{ThmFredholm} together with the Fredholm alternative ensure that (\ref{WaveguidePbBounded}) admits a unique solution.\\
Now if trapped modes $\mrm{span}(u_1,\dots,u_P)$, $P\ge1$, exist, let us prove that for the particular $f$ considered in (\ref{DefFSca}) related to the incident mode, Problem (\ref{WaveguidePbBounded}) still has a solution. To proceed, we have to show that $f$ satisfies the compatibility conditions appearing in (\ref{PbCompa1Bis}). For $p=1,\dots,P$, integrating twice by parts in $\Om_L$ and using that $\Delta u_p+k^2 u_p=0$ in $\Om$, we obtain
\[
(f,u_p)_{\mL^2(\Om_L)} =\dsp -\int_{\Om_L}\left(\Delta(\zeta w_+)+k^2(\zeta w_+)\right)\overline{u_p}\,dxdy =-\int_{\Sigma_{-L}}\cfrac{\partial w_+}{\partial \nu}\,\overline{u_p}-w_+\cfrac{\partial \overline{u_p}}{\partial \nu}\,dy.
\]
Above we used that $\zeta =1$ on $\Sigma_{-L}$ and $\zeta =0$ on $\Sigma_{L}$. Now we observe that $w_+$ is propagating while Proposition \ref{PropoTrappedModes} ensures that $u_p$ decomposes only on the evanescent modes. From the orthogonality of the family $\{\varphi_n\}_{n\in\N^\ast}$ in $\mL^2(I)$, we conclude that $(f,u_p)_{\mL^2(\Om_L)} =0$.
\end{proof}
\noindent Set $R_+\coloneqq\alpha_1^-$, $T_+\coloneqq1+\alpha_1^+$ so that for $u_+$ we have the representation
\begin{equation}\label{RepresentationupD}
u_+=\begin{array}{|ll}
w_++R_+w_-+\tilde{u}_+ & \mbox{ for }x < - L \\[3pt]
\phantom{w_+mi}T_+w_++\tilde{u}_+ & \mbox{ for }x>L,
\end{array}
\end{equation}
with $\tilde{u}_+\in\mH^1_0(\Om)$. The quantities $R_+$, $T_+$ are usually called the reflection and transmission coefficients. Let us emphasize that they are uniquely defined, even if trapped modes exist for the chosen $k$. Indeed since trapped modes decay as $O(e^{-\sqrt{4\pi^2-k^2}|x|})$ as $x\to\pm\infty$ according to Proposition \ref{PropoTrappedModes}, the scattering coefficients $R_+$, $T_+$ are insensitive to their existence.\\
\newline
In the same way, one shows that the problem 
\[
\begin{array}{|rcll}
\multicolumn{4}{|l}{\mbox{Find }u_-\in\mH^1_{0,\loc}(\Om)\mbox{ such that }u_--w_-\mbox{ is outgoing and } }\\[3pt]
\Delta u_-+k^2u_-&=&0&\mbox{ in }\Om \\[2pt]
u_-&=&0&\mbox{ on }\partial\Om.
\end{array}
\]
admits a solution for all $k\in(\pi;2\pi)$. Then for $u_-$, one has the decomposition 
\begin{equation}\label{RepresentationumD}
u_-=\begin{array}{|ll}
\phantom{w_+mi}T_-w_-+\tilde{u}_- & \mbox{ for }x < - L \\[3pt]
w_-+R_-w_++\tilde{u}_- & \mbox{ for }x>L,
\end{array}
\end{equation}
where $R_-$, $T_-\in\Cplx$, $\tilde{u}_-\in\mH^1_0(\Om)$. It corresponds to the scattering of the leftgoing wave $w_-$ coming from the right branch of the waveguide.\\
\newline
With the coefficients $R_\pm$, $T_\pm$ appearing in the decompositions (\ref{RepresentationupD}), (\ref{RepresentationumD}), we form the scattering matrix
\[
\mathbb{S}\coloneqq\left( 
\begin{array}{cc}
R_+ & T_+ \\[2pt]
T_- & R_- 
\end{array}
\right)\in\Cplx^{2\times2}.
\]
As known in physics, the matrix $\mathbb{S}$ has a very rigid structure. More precisely, we have the following statement:
\begin{proposition}\label{PropoStructureS}
The scattering matrix $\mathbb{S}$ is symmetric ($T_+=T_-$) and unitary ($\mathbb{S}\,\overline{\mathbb{S}}^{\top}=\mrm{Id}_{2\times2}$).
\end{proposition}
\begin{proof}
For $l>0$, set $\Om_{l}\coloneqq\{(x,y)\in\Om\,|\,|x|<l\}$, $\Sigma_{\pm l}\coloneqq\{\pm l\}\times I$. We have
\[
0=\int_{\Om_{l}} (\Delta u_++k^2u_+)u_--u_+(\Delta u_-+k^2u_-)\,dxdy=\int_{\Sigma_{l}\cup\Sigma_{-l}}\cfrac{\partial u_+}{\partial \nu}\,u_--u_+\cfrac{\partial u_-}{\partial \nu}\,dy.
\]
For $l\ge L$, using decompositions (\ref{RepresentationupD}), (\ref{RepresentationumD}), we obtain on $\Sigma_{-l}$ 
\[
\cfrac{\partial u_+}{\partial \nu}=-\cfrac{\partial u_+}{\partial x}= -i\beta_1 (w_+-R_+\,w_-)-\partial_x\tilde{u}_+,\qquad\quad \cfrac{\partial u_-}{\partial \nu}=-\cfrac{\partial u_-}{\partial x}= i\beta_1 T_- \,w_--\partial_x\tilde{u}_-,
\]
and on $\Sigma_{l}$ 
\[
\cfrac{\partial u_+}{\partial \nu}=\cfrac{\partial u_+}{\partial x}= i\beta_1 T_+ \,w_++\partial_x\tilde{u}_+,\qquad\quad \cfrac{\partial u_-}{\partial \nu}=\cfrac{\partial u_-}{\partial x}= -i\beta_1 (w_--R_-\,w_+)+\partial_x\tilde{u}_-.
\]
By exploiting that the family $(\varphi_n)_{n\in\N^\ast}$ is orthonormal in $\mL^2(I)$, we can write 
\[
\begin{array}{rcl}
0&=&\dsp\int_{\Sigma_{l}\cup\Sigma_{-l}}\cfrac{\partial u_+}{\partial \nu}\,u_--u_+\cfrac{\partial u_-}{\partial \nu}\,dy \\[10pt]
&=& -i\beta_1 (w_+-R_+\,w_-)\,T_-\,w_--i\beta_1(w_++R_+\,w_-)\,T_-\,w_- \\
 & &+i\beta_1\,T_+\,w_+\, (w_-+R_-\,w_+)+i\beta_1\,T_+\,w_+\, (w_--R_-\,w_+) +\dsp\int_{\Sigma_{l}\cup\Sigma_{-l}}\cfrac{\partial \tilde{u}_+}{\partial \nu}\,\tilde{u}_--\tilde{u}_+\cfrac{\partial \tilde{u}_-}{\partial \nu}\,dy\\[10pt]
& = & 2i\beta_1(T_+-T_-)+\dsp\int_{\Sigma_{l}\cup\Sigma_{-l}}\cfrac{\partial \tilde{u}_+}{\partial \nu}\,\tilde{u}_--\tilde{u}_+\cfrac{\partial \tilde{u}_-}{\partial \nu}\,dy.
\end{array}
\]
Taking the limit $l\to+\infty$, we deduce $0=2i\beta_1(T_+-T_-)$, which gives $T_+=T_-$. Then working similarly from the identities
\[
0=\int_{\Om_{l}} (\Delta u_\pm+k^2u_\pm)\overline{u_\pm}-u_\pm(\Delta \overline{u_\pm}+k^2\overline{u_\pm})\,dxdy=\int_{\Sigma_{l}\cup\Sigma_{-l}}\cfrac{\partial u_\pm}{\partial \nu}\,\overline{u_\pm}-u_\pm\cfrac{\partial \overline{u_\pm}}{\partial \nu}\,dy,
\]
one establishes the relations of conservation of energy
\begin{equation}\label{ConservationNRJ_Dirichlet}
|R_\pm|^2+|T_\pm|^2=1.
\end{equation}
Finally, by using that 
\[
0=\int_{\Om_{l}} (\Delta u_\pm+k^2u_\pm)\overline{u_\mp}-u_\pm(\Delta \overline{u_\mp}+k^2\overline{u_\mp})\,dxdy=\int_{\Sigma_{l}\cup\Sigma_{-l}}\cfrac{\partial u_\pm}{\partial \nu}\,\overline{u_\mp}-u_\pm\cfrac{\partial \overline{u_\mp}}{\partial \nu}\,dy,
\]
one gets 
\begin{equation}\label{CheckUnitarity}
R_{\pm}\overline{T_{\mp}}+T_{\pm}\overline{R_{\mp}}=0.
\end{equation}
Identities (\ref{ConservationNRJ_Dirichlet}) and (\ref{CheckUnitarity}) ensure that $\mathbb{S}\,\overline{\mathbb{S}}^{\top}=\mrm{Id}_{2\times2}$.
\end{proof}
\noindent In the following, we simply set $T\coloneqq T_+=T_-$ so that we have
\[
\mathbb{S}=\left( 
\begin{array}{cc}
R_+ & T \\[2pt]
T & R_- 
\end{array}
\right).
\]
In our study, we will need some formulas expressing the values of the scattering coefficients $R_\pm$, $T$ with respect to $u_\pm$.
\begin{proposition}\label{PropoRepresentationD}
For $k\in(\pi;2\pi)$, the coefficients $R_{\pm}$, $T$ appearing in (\ref{RepresentationupD}), (\ref{RepresentationumD}) in the decompositions of $u_+$, $u_-$ satisfy
\[
R_{\pm}=\cfrac{1}{2i\beta_1}\dsp\int_{\Sigma_{L}\cup\Sigma_{-L}}\cfrac{\partial u_\pm}{\partial \nu}\,\overline{w_\mp}-u_\pm\cfrac{\partial \overline{w_\mp}}{\partial \nu}\,dy\quad \mbox{ and }\quad T=1+\cfrac{1}{2i\beta_1}\dsp\int_{\Sigma_{L}\cup\Sigma_{-L}}\cfrac{\partial u_\pm}{\partial \nu}\,\overline{w_\pm}-u_\pm\cfrac{\partial \overline{w_\pm}}{\partial \nu}\,dy.
\]
\end{proposition}
\begin{proof}
From (\ref{RepresentationupD}), we find on $\Sigma_{-L}$
\[
\cfrac{\partial u_+}{\partial \nu}=-\cfrac{\partial u_+}{\partial x}= -i\beta_1 (w_+-R_+\,w_-)-\partial_x\tilde{u}_+,\qquad\quad \cfrac{\partial \overline{w_-}}{\partial \nu}=-\cfrac{\partial w_+}{\partial x}= -i\beta_1 w_+,
\]
and on $\Sigma_{L}$ 
\[
\cfrac{\partial u_+}{\partial \nu}=\cfrac{\partial u_+}{\partial x}= i\beta_1 T \,w_++\partial_x\tilde{u}_+,\qquad\quad \cfrac{\partial \overline{w_-}}{\partial \nu}=\cfrac{\partial w_+}{\partial x}= i\beta_1 w_+.
\]
By exploiting that the exponentially decaying modes appearing in the decomposition of $\tilde{u}$ are orthogonal to $w_\pm$ in $\mL^2(\Sigma_{\pm L})$, we obtain
\[
\begin{array}{ll}
&\dsp\cfrac{1}{2i\beta_1}\dsp\dsp\int_{\Sigma_{L}\cup\Sigma_{-L}}\cfrac{\partial u_+}{\partial \nu}\,\overline{w_-}-u_+\cfrac{\partial \overline{w_-}}{\partial \nu}\,dy \\[10pt]
=&\dsp\cfrac{1}{2}\dsp\dsp\int_{\Sigma_{-L}}- (w_+-R_+\,w_-)w_++(w_++R_+\,w_-) w_+\,dy+\cfrac{1}{2}\dsp\dsp\int_{\Sigma_{L}} T\,(w_+^2)- T\,(w_+^2)\,dy=R_+.
\end{array}
\]
Formulas for $R_-$ and $T$ are obtained in a similar manner. 
\end{proof}

\subsection{Numerical approximation}\label{ParaNumApprox}

Below we will have to compute numerical approximations of the quantities $u_\pm$, in particular to obtain $\mathbb{S}$. There are ``three infinities''  which are unpleasant to solve (\ref{WaveguidePbSca}) with a computer. First, the domain $\Om$ is unbounded. To face this difficulty, we will exploit the above analysis and consider a formulation set in $\Om_L$ involving the Dirichlet-to-Neumann operators similar to (\ref{PbNumBounded}). Second, the radiation conditions involve an infinite number of terms. As in \S\ref{ParaTroncDTN}, we will truncate the series appearing in the Dirichlet-to-Neumann operators at rank $M\ge N$ where $N$ is the number of propagating modes. Third, $\mH^1_0(\Om_L;\Gamma)$ is a space of infinite dimension. 
We will work with a finite element method and solve a variational formulation in a space of finite dimension.\\
\newline
To set ideas, assume that $k\in(\pi;2\pi)$ and consider the approximation of $u_+$. Since $u_+-w_+$ is outgoing, we have the conditions
\[
\cfrac{\partial (u_+-w_+)}{\partial \nu} =  \Lambda_\pm((u_+-w_+)|_{\Sigma_{\pm L}}) \quad \mbox{ on }\Sigma_{\pm L}.
\]
This gives 
\[
\cfrac{\partial u_+}{\partial \nu} =  \Lambda_\pm(u_+|_{\Sigma_{\pm L}})+\cfrac{\partial w_+}{\partial \nu}-\Lambda_\pm(w_+|_{\Sigma_{\pm L}}) \quad \mbox{ on }\Sigma_{\pm L}.
\]
Since $w_+$ is rightgoing, using the definitions (\ref{DefDTN}) of $\Lambda_\pm$, we obtain 
\[
\partial_\nu w_+-\Lambda_+(w_+|_{\Sigma_{L}})=0\quad\mbox{ on }\Sigma_{L},\qquad\qquad \partial_\nu w_+-\Lambda_-(w_+|_{\Sigma_{-L}})=-2i\beta_1w_+\quad\mbox{ on }\Sigma_{-L}.
\]
Thus $u_+$ solves the problem
\begin{equation}\label{DefVariaD}
\begin{array}{|l}
\mbox{Find }u_+\in\mH^1_0(\Om_L;\Gamma) \mbox{ such that for all }v\in\mH^1_0(\Om_L;\Gamma),\\[3pt]
\dsp\int_{\Om_L}\hspace{-0.2cm}\nabla u_+\cdot\nabla\overline{v}-k^2 u\overline{v}\,dxdy-\langle\Lambda_+(u_+|_{\Sigma_{L}}),\overline{v} \rangle_{\Sigma_L}-\langle\Lambda_-(u_+|_{\Sigma_{-L}}),\overline{v} \rangle_{\Sigma_{-L}}=-2i\beta_1\int_{\Sigma_{-L}}\hspace{-0.2cm}w_+\overline{v}\,dy.
\end{array}
\end{equation}
Note that (\ref{DefDTN}) yields
\[
\langle\Lambda_\pm(u|_{\Sigma_{\pm L}}),\overline{v}\rangle_{\Sigma_{\pm L}} = \dsp\sum_{n=1}^{+\infty} i\beta_n\, (u,\varphi_n)_{\mL^2(\Sigma_{\pm L})}\,\overline{(v,\varphi_n)_{\mL^2(\Sigma_{\pm L})}}.
\]
Now introduce $(\mathcal{T}_h)_h$ a shape regular family of triangulations of $\overline{\Omega_L}$ (in other words, we mesh the domain $\overline{\Omega_L}$ with triangles). Define the family of Lagrange finite element spaces
\begin{equation}\label{DefVh}
\mV_h \coloneqq \left\{v\in \mH^1_0(\Om_L;\Gamma)\mbox{ such that }v|_{\tau}\in \mathbb{P}_q(\tau)\mbox{ for all }\tau\in\mathcal{T}_h\right\},
\end{equation}
where $\mathbb{P}_q(\tau)$ is the space of polynomials of degree at most $q$ on the triangle $\tau$. Finally, the problem we solve writes
\begin{equation}\label{FormuNumD}
\begin{array}{|l}
\mbox{Find }u^h_+\in\mV_h \mbox{ such that for all }v^h\in\mV_h,\\[3pt]
\dsp\int_{\Om_L}\hspace{-0.2cm}\nabla u^h_+\cdot\nabla\overline{v^h}-k^2 u^h_+\overline{v^h}\,dxdy-
 \dsp\sum_{n=1}^{M} i\beta_n\, (u^h_+,\varphi_n)_{\mL^2(\Sigma_{ L})}\,\overline{(v^h,\varphi_n)_{\mL^2(\Sigma_{ L})}}\\[10pt]
\hspace{4.95cm}-\dsp\sum_{n=1}^{M} i\beta_n\, (u^h_+,\varphi_n)_{\mL^2(\Sigma_{ -L})}\,\overline{(v^h,\varphi_n)_{\mL^2(\Sigma_{-L})}}
=-2i\beta_1\int_{\Sigma_{-L}}\hspace{-0.3cm}w_+\overline{v^h}\,dy.
\end{array}\hspace{-0.8cm}
\end{equation} 
Using in particular Corollary \ref{CoroMainThm}, one can show (see \cite{Gold82} and \cite[Chap.\,5]{BoLu21} for more details) that for $h$ small enough, $L$ large enough  $u^h_+$ yields a good approximation of $u_+$. Then replacing $u_+$ by $u^h_+$ in the exact formulas 
\begin{equation}\label{ScndDefRT}
R_+ = \int_{\Sigma_{-L}} (u_+-w_+)\,w_+\,dy,\qquad T = \int_{\Sigma_{L}} u_+w_-\,dy,
\end{equation}
we obtain good approximations of the scattering coefficients. Let us emphasize that numerically, it is more interesting to use expressions (\ref{ScndDefRT}) than the ones provided by Proposition \ref{PropoRepresentationD} for $R_+$, $T$ because they offer better precision. Identities of Proposition \ref{PropoRepresentationD} will be useful for theoretical purposes.

\section{Neumann problem}

In acoustics, we are led to study the same Helmholtz equation as in (\ref{WaveguidePb}) but with homogeneous Neumann BCs which model sound hard walls. This yields the problem 
\begin{equation}\label{WaveguidePbNeumann}
\begin{array}{|rcll}
\Delta u+k^2u&=&0&\mbox{ in }\Om \\[2pt]
\partial_{\nu}u&=&0&\mbox{ on }\partial\Om,
\end{array}
\end{equation}
where $u$ stands for the pressure of the fluid in $\Om$ and $\partial_{\nu}$ refers to the outward normal derivative on $\partial\Om$. As seen in the study of the Dirichlet case, the modes play a key role in the analysis. This time, we need to compute the solutions of (\ref{WaveguidePbNeumann}) with separate variables in the reference strip $\mathcal{S}$. Reproducing what has been done in \S\ref{ParagraphModesD}, we find that they coincide with the family $\{w^\pm_n\}_{n\in\N}$ where 
\begin{equation}\label{ModesNeumann}
w^\pm_n(x,y)=e^{\pm i\beta_n x}\varphi_n(y),\qquad\beta_n\coloneqq\sqrt{k^2-n^2\pi^2},\quad\varphi_n(y)=\begin{array}{|ll}
1 & \mbox{ if }n=0\\
\sqrt{2}\cos(n\pi y) & \mbox{ for }n>0.
\end{array}
\end{equation}
In particular for all $k>0$, we have $w^\pm_0(x,y)=e^{\pm ikx}$, which shows that unlike the Dirichlet case, propagating modes always exist (this is also related to the fact that we have no Poincar\'e inequality as (\ref{EstimatePoinc1D}) in $\mH^1(I)$). As a consequence, for all $k>0$, $\mH^1(\Om)$ is not an adapted functional framework to study (\ref{ModesNeumann}). The solution must decompose on the propagating modes and radiation conditions must be imposed to select the outgoing behavior. The method, based in particular on the limiting absorption principle, is completely similar to what has been done in \S\ref{ParaPbDissip}, \ref{ParaLimiting}, \ref{ParaLimitingAbs}. We do not detail it and instead simply present the main results concerning the corresponding scattering problem. \\
\newline 
To stick to the simplest setting, we assume that $k$ belongs to $(0;\pi)$ so that only the modes $w^\pm_0$ can propagate. We denote them by $w_\pm$, so that
\begin{equation}\label{Defwpm}
w_\pm(x,y)=e^{\pm i k x}.
\end{equation}
Note that $w_\pm$ are plane waves, they do not depend on the variable $y$ (which was not the case for the Dirichlet problem). The scattering of the rightgoing wave $w_+$ in $\Om$ leads to consider the solution $u_+\in\mH^1_{\loc}(\Om)$ of (\ref{WaveguidePbNeumann}) admitting the expansion
\begin{equation}\label{RepresentationNeumannP}
u_+=\begin{array}{|ll}
w_++R_+w_-+\tilde{u}_+ & \mbox{ for }x < - L \\[3pt]
\phantom{w_+mi}T_+w_++\tilde{u}_+ & \mbox{ for }x>L,
\end{array}
\end{equation}
with $R_+$, $T_+\in\Cplx$ and $\tilde{u}_+\in\mH^1(\Om)$. More precisely, by adapting what has been done in \S\ref{ParaScaPb}, one can show that (\ref{WaveguidePbNeumann}) always admits a solution with the expansion (\ref{RepresentationNeumannP}). This solution is uniquely defined if trapped modes (solutions of (\ref{WaveguidePbNeumann}) in $\mH^1(\Om)$) do not exist. Besides, by working as in Proposition \ref{PropoTrappedModes}, one shows that trapped modes, if they exist, decay as $O(e^{-\sqrt{\pi^2-k^2}|x|})$ as $x\to\pm\infty$ so that the scattering coefficients $R_+$, $T_+$ in (\ref{RepresentationNeumannP}) are always uniquely defined.\\
\newline
Similarly, the scattering of the leftgoing plane wave $w_-$ in $\Om$ leads to consider the solution $u_-\in\mH^1_{\loc}(\Om)$ of (\ref{WaveguidePbNeumann}) admitting the expansion
\begin{equation}\label{RepresentationNeumannM}
u_-=\begin{array}{|ll}
\phantom{w_+mi}T_-w_-+\tilde{u}_- & \mbox{ for }x < - L \\[3pt]
w_-+R_-w_++\tilde{u}_- & \mbox{ for }x>L,
\end{array}
\end{equation}
with $R_-$, $T_-\in\Cplx$ and $\tilde{u}_-\in\mH^1(\Om)$.\\
\newline
As in Proposition \ref{PropoStructureS}, one proves that $T_+=T_-$ and we set $T\coloneqq T_+=T_-$. The scattering matrix 
\begin{equation}\label{DefScaMatN}
\mathbb{S}=\left( 
\begin{array}{cc}
R_+ & T \\[2pt]
T & R_- 
\end{array}
\right)\in\Cplx^{2\times2}
\end{equation}
is symmetric and unitary ($\mathbb{S}\,\overline{\mathbb{S}}^{\top}=\mrm{Id}_{2\times2}$). In particular, we have the relations of conservation of energy
\begin{equation}\label{ConservationNRJ_Neumann}
|R_\pm|^2+|T|^2=1.
\end{equation}
As in the proof of Proposition \ref{PropoRepresentationD}, one establishes the following statement:
\begin{proposition}\label{PropoRepresentationN}
For $k\in(0;\pi)$, the coefficients $R_{\pm}$, $T$ appearing in (\ref{RepresentationNeumannP}), (\ref{RepresentationNeumannM}) in the decompositions of $u_+$, $u_-$ satisfy
\[
R_{\pm}=\cfrac{1}{2ik}\dsp\int_{\Sigma_{L}\cup\Sigma_{-L}}\cfrac{\partial u_\pm}{\partial \nu}\,\overline{w_\mp}-u_\pm\cfrac{\partial \overline{w_\mp}}{\partial \nu}\,dy\quad \mbox{ and }\quad T=1+\cfrac{1}{2ik}\dsp\int_{\Sigma_{L}\cup\Sigma_{-L}}\cfrac{\partial u_\pm}{\partial \nu}\,\overline{w_\pm}-u_\pm\cfrac{\partial \overline{w_\pm}}{\partial \nu}\,dy.
\]
\end{proposition}
\noindent Similarly to (\ref{DefVariaD}), one finds that $u_+$ solves the problem
\[
\begin{array}{|l}
\mbox{Find }u_+\in\mH^1(\Om_L) \mbox{ such that for all }v\in\mH^1(\Om_L),\\[3pt]
\dsp\int_{\Om_L}\hspace{-0.2cm}\nabla u_+\cdot\nabla\overline{v}-k^2 u\overline{v}\,dxdy-\langle\Lambda_+(u_+|_{\Sigma_{L}}),\overline{v} \rangle_{\Sigma_L}-\langle\Lambda_-(u_+|_{\Sigma_{-L}}),\overline{v} \rangle_{\Sigma_{-L}}=-2ik\int_{\Sigma_{-L}}\hspace{-0.2cm}w_+\overline{v}\,dy
\end{array}
\]
with 
\[
\langle\Lambda_\pm(u|_{\Sigma_{\pm L}}),\overline{v}\rangle_{\Sigma_{\pm L}} = \dsp\sum_{n=0}^{+\infty} i\beta_n\, (u,\varphi_n)_{\mL^2(\Sigma_{\pm L})}\,\overline{(v,\varphi_n)_{\mL^2(\Sigma_{\pm L})}}.
\]
\noindent For the numerics, we work with a formulation as in (\ref{FormuNumD}), namely
\[
\begin{array}{|l}
\mbox{Find }u^h_+\in\mV_h \mbox{ such that for all }v^h\in\mV_h,\\[3pt]
\dsp\int_{\Om_L}\hspace{-0.2cm}\nabla u^h_+\cdot\nabla\overline{v^h}-k^2 u^h_+\overline{v^h}\,dxdy-
 \dsp\sum_{n=0}^{M} i\beta_n\, (u^h_+,\varphi_n)_{\mL^2(\Sigma_{ L})}\,\overline{(v^h,\varphi_n)_{\mL^2(\Sigma_{ L})}}\\[10pt]
\hspace{4.95cm}-\dsp\sum_{n=0}^{M} i\beta_n\, (u^h_+,\varphi_n)_{\mL^2(\Sigma_{ -L})}\,\overline{(v^h,\varphi_n)_{\mL^2(\Sigma_{-L})}}
=-2ik\int_{\Sigma_{-L}}\hspace{-0.3cm}w_+\overline{v^h}\,dy,
\end{array}\hspace{-0.8cm}
\]
where this time $\mV_h$ is a space of the form 
\[
\mV_h \coloneqq \left\{v\in \mH^1(\Om_L)\mbox{ such that }v|_{\tau}\in \mathbb{P}_q(\tau)\mbox{ for all }\tau\in\mathcal{T}_h\right\}.
\]
Finally, for approximating the scattering coefficients, we replace $u_+$ by $u^h_+$ in the exact formulas 
\[
R_+ = \int_{\Sigma_{-L}} (u_+-w_+)\,w_+\,dy,\qquad T = \int_{\Sigma_{L}} u_+w_-\,dy.
\]
We emphasize that for all the results presented in this section, the modes are the ones obtained in (\ref{ModesNeumann}).

\section{Invisibility questions}

\noindent In the following, our general goal will be to find situations, by playing with the geometry, the wavenumber $k$, ... where we have some sort of invisibility, that is scattering features as if there were no obstacle in the waveguide.\\
\newline 
The weakest invisibility that one can look for is $R_\pm=0$. In the sequel, we shall say that one has zero reflection or that the defect is non reflecting. In such a situation, an observer generating an incident plane wave, located a bit far from the defect in the geometry and measuring the resulting backscattering field will only measure the evanescent component. Due to noise, one will get something similar to the field in the reference strip and therefore will be unable to detect the presence of the obstacle. Note that due to conservation of energy ((\ref{ConservationNRJ_Dirichlet}) or (\ref{ConservationNRJ_Neumann})), the fact that $R_\pm=0$ implies $|T|=1$ and so $T=e^{i\theta}$ for a certain $\theta\in\R$. As a consequence in general there is a phase shift in the transmitted field which can reveal the presence of the defect if one probes the field in that part of the guide.\\
\newline
A more demanding definition of invisibility is to have $T=1$. In that situation, we shall say that the defect is perfectly invisible or that we have perfect transmission without phase shift.

\begin{figure}[!ht]
\centering
\includegraphics[height=4cm]{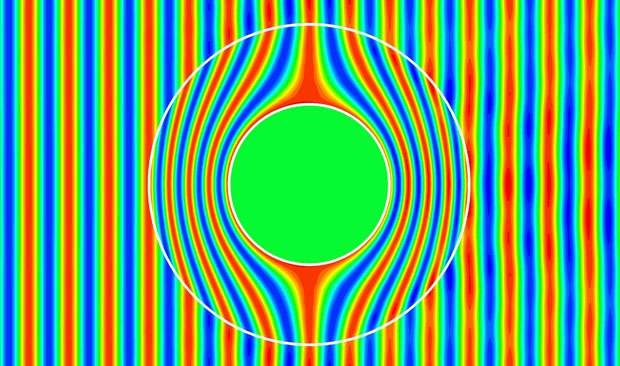}
\caption{Cloaking by transformation optics. \label{TOcloaking}}
\end{figure}

\noindent  The problem of cloaking an object has a large number of applications and has been the subject of intense studies over the last decade in the theory of waves propagation. Different notions of invisibility exist in the literature. Our approach for example is different from the cloaking via transformation optics pictured in Figure \ref{TOcloaking} (see \cite{PeSS06,Leon06}). Let us describe this latter device. Imagine that one wished to hide a given obstacle. One technique consists in surrounding it by a well chosen penetrable material, localized in the annulus around the green region on the picture, so that an incoming wave leaves the whole device as if there were nothing. Said differently, on the picture, an observer probing the field outside of the larger disk marked by the white thin circle obtains the same measurements as in free space and so cannot detect the presence of the obstacle. Mathematically, this idea is quite simple to implement, it boils down to a change of variable. However this change of variable is singular and for this reason the physical parameters of the \textit{ad hoc} material in the cloaking device should take infinite values. For this reason, for the moment designing such structures, even working with metamaterials, is still unreachable, in particular due to the presence of important losses. \\
\newline
What we propose is less ambitious because we only wish to control the scattering coefficients and not to act on the evanescent component of the field. In short, what we aim for is only cloaking at infinity. For this reason, it is more easily doable. Actually, we will not even need to play with penetrable materials. We will see that by working with a homogeneous material and acting only on the geometry is sufficient. On the other hand, though our setting is less ambitious, it is still relevant in numerous applications. Indeed, the evanescent part of the field that we neglect is exponentially decaying at infinity and therefore is really difficult to distinguish from noise a few wavelengths far from the obstacle. Finally, though we simply wish to control a finite number of complex coefficients, this problem is not trivial because the link between the variation of the parameters (geometry, $k$, ...) and the variation of the scattering coefficients is non linear and not explicit. Additionally, let us emphasize that due to the fact that there is no coercivity in the problem, optimization methods fail due to the presence of local minima.\\
\newline
In the next three chapters, we present several ideas to reach invisibility.

\chapter{Invisible perturbations of the reference geometry}\label{chapterContinuation}
\setcounter {section} {0}

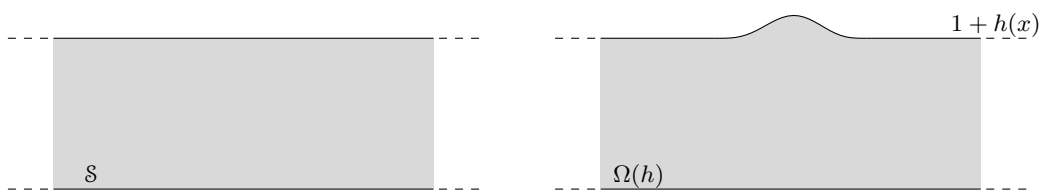
\begin{figure}[!ht]
\centering
\begin{tikzpicture}[scale=1]
\draw[fill=gray!30,draw=none](-2.5,0) rectangle (2.5,2);
\draw (-2.5,0)--(2.5,0);
\draw (-2.5,2)--(2.5,2);
\draw [dashed](-3.1,0)--(-2.5,0);
\draw [dashed](3.1,0)--(2.5,0);
\draw [dashed](-3.1,2)--(-2.5,2);
\draw [dashed](3.1,2)--(2.5,2);
\node at (-2,0.1) [anchor=base] {\footnotesize$\mathcal{S}$};
\begin{scope}[xshift=7.2cm]
\draw[fill=gray!30,draw=none](-2.5,0) rectangle (2.5,2);
\draw (-2.5,0)--(2.5,0);
\draw (-2.5,2)--(-1,2);
\draw (1,2)--(2.5,2);
\draw [dashed](-3.1,0)--(-2.5,0);
\draw [dashed](3.1,0)--(2.5,0);
\draw [dashed](-3.1,2)--(-2.5,2);
\draw [dashed](3.1,2)--(2.5,2);
\draw[gray!30] (-0.8,2)--(0.8,2);
\draw[samples=30,domain=-1:1,draw=black,fill=gray!30] plot(\x,{2+0.1*(\x+1)^4*(\x-1)^4*(\x+3)});
\node at (2.7,2.1) [anchor=base] {\footnotesize$1+ h(x)$};
\node at (-2,0.1) [anchor=base] {\footnotesize$\Om(h)$};
\end{scope}
\end{tikzpicture}
\caption{Reference strip $\mathcal{S}$ (left) and perturbed waveguide $\Om(h)$ (right). \label{PerturbedOm}}
\end{figure}

\noindent In this chapter, we work with techniques of perturbations to construct invisible obstacles. The idea, proposed in the article \cite{BoNa13}, is as follows: in the reference strip $\mathcal{S}=\R\times I$, one has no reflection and perfect transmission, the total field is equal to the incident field, and so $R_\pm=0$, $T=1$. How to slightly modify the geometry while keeping these values for the scattering coefficients? To proceed, we will adapt the proof of the implicit function theorem.

\section{General scheme}\label{ParaGeneScheme}

Let us describe the method for the simplest problem, namely obtaining $R_\pm=0$. At this stage, the approach is the same whether one considers Problem (\ref{WaveguidePbSca}) with Dirichlet BCs or Problem (\ref{WaveguidePbNeumann}) with Neumann BCs (pick $k\in(\pi;2\pi)$  in the first case, $k\in(0;\pi)$ in the second situation). 
Let us focus our attention on $R_+$ and simply write $R$ instead of $R_+$. Note that from conservation of energy (see relations (\ref{ConservationNRJ_Dirichlet})), $R_+=0$ implies $R_-=0$. Consider some real valued profile function $h\in\mathscr{C}^{\infty}_0(\R)$ and let $\Om(h)$ be the waveguide whose upper boundary coincides with the graph of the function $1+h$ (see Figure \ref{PerturbedOm} right). Note that we make some assumption of smoothness here for simplicity but $\mathscr{C}^{2}_0(\R)$ would be enough. Let $R(h)$ be the reflection coefficient of the scattering solution $u_+$ in the geometry $\Om(h)$. Importantly, we have $R(0)=0$ because when $h\equiv0$, $\Om(h)$ is simply the reference strip $\mathcal{S}$. With this notation, the problem we consider writes
\[
\begin{array}{|l}
\mbox{Find }h\not\equiv0 \mbox{ such that }\\[2pt]
R(h)=0.
\end{array}
\]
Let us look for non reflecting geometries which are small perturbations of $\mathcal{S}$. To proceed, let us look for $h=\eps \mu$ with $\eps>0$ small and $\mu\in\mathscr{C}^{\infty}_0(\R)$ to be determined. Since $\eps$ is small, we can write an asymptotic expansion of $R$ with respect to $\eps$. We obtain
\begin{equation}\label{TaylorR}
\begin{array}{rcl}
R(\eps\mu)&=&R(0)+\eps dR(0)(\mu)+\eps^2 \tilde{R}(\eps\mu)\\[3pt]
&=& \eps dR(0)(\mu)+\eps^2 \tilde{R}(\eps\mu)
\end{array}
\end{equation}
where $dR(0)(\mu)$ stands for the differential of $R$ at zero in the direction $\mu$ and $\tilde{R}(\eps\mu)$ is an abstract remainder. Since $dR(0):\mathscr{C}^\infty_0(\R)\to\Cplx$ is a linear map from a space of infinite dimension to a space of dimension two (remember that we work with real valued functions $h$), we have $\dim\,\ker\,dR(0)=+\infty$. Therefore, we can pick $\mu_0\not\equiv0$ such that 
\begin{equation}\label{DefBasis0}
dR(0)(\mu_0)=0.
\end{equation}
By choosing $h=\eps\mu=\eps\mu_0$, we obtain a perturbation of order $\eps$ which produces a reflection in $O(\eps^2)$. This is interesting because this is almost zero reflection but not completely satisfactory yet. To compensate for the remainder, we need to work a bit more. Below we will explain how to compute $dR(0)$ and prove that $dR(0):\mathscr{C}^\infty_0(\R)\to\Cplx$ is onto. This allows us to introduce $\mu_1\in\mathscr{C}^\infty_0(\R)$ and $\mu_2\in\mathscr{C}^\infty_0(\R)$ such that 
\begin{equation}\label{DefBasis}
dR(0)(\mu_1)=1\qquad\mbox{ and }\qquad dR(0)(\mu_2)=i.
\end{equation}
Finally, we look for $\mu$ of the form
\[
\mu=\mu_0+\tau_1\mu_1+\tau_2\mu_2
\]
where $\tau_1$, $\tau_2\in\R$ are parameters to tune. Inserting this $\mu$ in the expansion (\ref{TaylorR}), to get $R(\eps \mu)=0$, we see that we must have
\[
0= \eps dR(0)(\mu_0+\tau_1\mu_1+\tau_2\mu_2)+\eps^2 \tilde{R}(\eps(\mu_0+\tau_1\mu_1+\tau_2\mu_2)).
\]
By exploiting (\ref{DefBasis}), this yields 
\[
0=\tau_1+i\tau_2+\eps \tilde{R}(\eps(\mu_0+\tau_1\mu_1+\tau_2\mu_2)).
\]
In other words, we find that the vector $\vec{\tau}\coloneqq(\tau_1,\tau_2)\in\R^2$ must satisfy the fixed point equation
\begin{equation}\label{FixedPointEqn}
\vec{\tau}=G^\eps(\vec{\tau})\quad\mbox{ with }\quad G^\eps(\vec{\tau})=-\eps(\Re e\,\tilde{R}(\eps(\mu_0+\tau_1\mu_1+\tau_2\mu_2)),\Im m\,\tilde{R}(\eps(\mu_0+\tau_1\mu_1+\tau_2\mu_2))).
\end{equation}
Now by proving uniform (with respect to $\vec{\tau}$) error estimates in the asymptotic expansions as $\eps$ tends to zero, one can show that for any $r>0$, there is $\eps_0>0$ such that the map $\vec{\tau}\mapsto G^\eps(\vec{\tau})$ is a contraction from $\overline{B(O,r)}$ to $\overline{B(O,r)}$ for all $\eps\in(0;\eps_0]$. Here $B(O,r)$ denotes the open ball of $\R^2$ centered at $O$ of radius $r$. Therefore the Banach fixed point theorem guarantees that (\ref{FixedPointEqn}) admits a unique solution $\vec{\tau}^{\,\mrm{sol}}\in\overline{B(O,r)}$. Then for $h^{\,\mrm{sol}}=\eps(\mu_0+\tau_1^{\mrm{sol}}\mu_1+\tau_2^{\mrm{sol}}\mu_2)$, we have $R(h^{\mrm{sol}})=0$. This proves the existence of non reflecting geometries.\\
\newline
Let us comment a bit this method.\\[4pt]
1) First, it is important to prove that the constructed $h^{\,\mrm{sol}}$ is not trivial. To proceed, let us work by contradiction and assume that $h^{\,\mrm{sol}}\equiv0$. Then we get
\[
0=dR(0)(h^{\,\mrm{sol}})=dR(0)(\eps(\mu_0+\tau_1^{\mrm{sol}}\mu_1+\tau_2^{\mrm{sol}}\mu_2))=\eps(\tau_1^{\mrm{sol}}+i\tau_2^{\mrm{sol}}),
\]
and so $\tau_1^{\mrm{sol}}=\tau_2^{\mrm{sol}}=0$. This implies $h^{\,\mrm{sol}}=\eps\mu_0\equiv0$ which is not true due to our choice for $\mu_0$.\\[4pt]
2) Observe that this technique guarantees the existence of an infinite number of non reflecting perturbations. Indeed, first, $h^{\,\mrm{sol}}$ depends on $\eps\in(0;\eps_0]$ (remark that it is not only a scaling because there are non linear terms involved). Additionally, for $\mu_0$   we have some freedom because $\dim\,\ker\,dR(0)=+\infty$.\\[4pt] 
3) We can establish that there is a constant $C>0$ independent of $\eps$ such that we have $|G^\eps(\vec{\tau})| \le C\eps$ in $\overline{B(O,r)}$. As a consequence, as $\eps$ tends to zero, we have $|\vec{\tau}^{\,\mrm{sol}}|=O(\eps)$ and the shape of the perturbation $\eps\mu=\eps(\mu_0+\tau_1\mu_1+\tau_2\mu_2)$ is mainly characterized by $\mu_0$.\\[4pt]
4) Let us clarify the connection with the implicit function theorem. Introduce the functional
\[
\begin{array}{ccc}
\mathscr{F}  :  \R\times\R^2  & \to & \R^2 \\
\qquad (\tau_0,\vec{\tau}) & \mapsto & (\Re e\,R(\tau_0\mu_0+\tau_1\mu_1+\tau_2\mu_2),\Im m\,R(\tau_0\mu_0+\tau_1\mu_1+\tau_2\mu_2))
\end{array}
\]
which is of class $\mathscr{C}^1$ in a neighborhood of $0_{\R^3}$. From what will be shown below, we will be able to deduce that $\partial_{\vec{\tau}}\mathscr{F}(0,0): \R^2 \to \R^2$ is well-defined and bijective. On the other hand, we remark that $\mathscr{F}(0_{\R^3})=0_{\R^2}$. The implicit function theorem applies: there are some neighborhoods $\mathscr{U} \subset \R$, $\mathscr{V} \subset  \R^2 $ of $0$, $0_{\R^2}$ and a unique function $\varphi : \mathscr{U} \to \mathscr{V}$ of class $\mathscr{C}^1$ such that 
\[
[\,(\tau_0,\vec{\tau})\in \mathscr{U} \times \mathscr{V}\ \mbox{ and }\ \mathscr{F}(\tau_0,\vec{\tau})=0\,] \qquad\Leftrightarrow\qquad \vec{\tau} = \varphi (\tau_0).
\]

\section{Zero reflection for the Dirichlet problem}

Let us apply the generic strategy described above to Problem (\ref{WaveguidePbSca}) with Dirichlet BCs. Pick some $k\in(\pi;2\pi)$. First, we need to compute $dR(0)$. 

\begin{proposition}\label{PropoDiffR_Dirichlet}
For $\mu\in\mathscr{C}^\infty_0(\R)$, we have
\[
\begin{array}{rcl}
dR(0)(\mu)&=&\dsp\cfrac{i \pi^2}{\beta_1}\,\int_{\R}\mu(x) e^{2i\beta_1x}\,dx.
\end{array}
\]
As a consequence, $dR(0):\mathscr{C}^\infty_0(\R)\to\Cplx$ is onto.
\end{proposition}
\begin{proof}
The quantity $dR(0)$ corresponds to the derivative of $R$ with respect to the geometry. To identify it, we have to understand how $R$ is changed when the boundary of the waveguide is perturbed around the reference situation ($\Om=\mathcal{S}$). Such results are met classically in shape optimization. To obtain $dR(0)$, we work with techniques of asymptotic analysis.\\
\newline
For $\eps>0$ small and a given $\mu\in\mathscr{C}^\infty_0(\R)$ supported in $(-L;L)$ (change $L$ if necessary), denote by $\Om^\eps$ the domain $\Om(h)$ with $h=\eps\mu$. Let $u^\eps$ stand for the solution $u_+$ introduced in (\ref{WaveguidePbSca}) in the geometry $\Om^\eps$. It satisfies 
\begin{equation}\label{fortEpsDirichlet} 
\begin{array}{|rcll}
\multicolumn{4}{|l}{\mbox{Find }u^\eps\in\mH^1_{0,\loc}(\Om^\eps)\mbox{ such that }u^\eps-w_+\mbox{ is outgoing and } }\\[3pt]
\Delta u^\eps+k^2u^\eps&=&0&\mbox{ in }\Om^\eps \\[2pt]
u^\eps&=&0&\mbox{ on }\partial\Om^\eps.
\end{array}
\end{equation}
For $u^\eps$, as $\eps$ tends to zero, we consider the ansatz 
\begin{equation}\label{ExpansionD}
u^{\eps}=u_0+\eps u_1+\dots
\end{equation}
where $u_0$, $u_1$ are functions to be determined and the dots correspond to higher order terms. On the upper part of $\partial\Om^\eps$, we have, formally, 
\begin{equation}\label{TaylorExpansionDirichlet}
\begin{array}{rcl}
0=u^\eps(x,1+h(x))&=&u^\eps(x,1)+\eps\mu(x)\partial_yu^\eps(x,1)+\dots \\[3pt]
&=& u_0(x,1)+\eps \left(u_1(x,1)+\mu(x)\partial_yu_0(x,1)\right)+\dots .
\end{array}
\end{equation}
Now by inserting (\ref{ExpansionD}) in (\ref{fortEpsDirichlet}) and by exploiting (\ref{TaylorExpansionDirichlet}), collecting the terms of orders $\eps^0$, $\eps^1$, we find that $u_0$, $u_1$ satisfy respectively the problems
\begin{equation}\label{PbAsympto1}
\begin{array}{|rccl}
\multicolumn{4}{|l}{\mbox{Find }u_0\in\mH^1_{0,\loc}(\mathcal{S})\mbox{ such that }u_0-w_+\mbox{ is outgoing and }\hspace{0.4cm}~ }\\[3pt]
\Delta u_0 + k^2u_0 &=& 0&  \mbox{in }\mathcal{S}\\[3pt]
 u_0&=&0 & \mbox{on } \partial\mathcal{S}
\end{array}
\end{equation}
\begin{equation}\label{PbAsympto2}
\begin{array}{|rccl}
\multicolumn{4}{|l}{\mbox{Find }u_1\in\mH^1_{0,\loc}(\mathcal{S})\mbox{ such that }u_1\mbox{ is outgoing and } }\\[3pt]
\Delta u_1 + k^2u_1 &=& 0&  \mbox{in } \mathcal{S}\\[3pt]
 u_1&=&0 & \mbox{on } \R\times\{0\}\\[3pt]
 u_1&=& -\mu\partial_y u_0 & \mbox{on }\R\times\{1\}.
\end{array}\hspace{1.35cm}
\end{equation}
Solving (\ref{PbAsympto1}), we obtain 
\begin{equation}\label{Identifiu0}
u_0(x,y)=w_+(x,y)=e^{i\beta_1x}\varphi_1(y).
\end{equation} 
Denote by $R^\eps$ the reflection coefficient of $u^\eps$. According to Proposition \ref{PropoRepresentationD}, we have 
\[
R^\eps=\cfrac{1}{2i\beta_1}\dsp\int_{\Sigma_{L}\cup \Sigma_{-L}}\cfrac{\partial u^\eps}{\partial \nu}\,\overline{w_-}-u^\eps\cfrac{\partial \overline{w_-}}{\partial \nu}\,dy.
\]
Inserting the expansion (\ref{ExpansionD}) of $u^\eps$ in the above identity and using that $\overline{w_-}=w_+$, $u_0=w_+$, this gives
\[
\begin{array}{rcl}
R^\eps&=&\dsp\cfrac{1}{2i\beta_1}\dsp\int_{\Sigma_{L}\cup \Sigma_{-L}}\cfrac{\partial u_0}{\partial \nu}\,w_+-u_0\cfrac{\partial w_+}{\partial \nu}\,dy+\cfrac{\eps}{2i\beta_1}\dsp\int_{\Sigma_{L}\cup \Sigma_{-L}}\cfrac{\partial u_1}{\partial \nu}\,w_+-u_1\cfrac{\partial w_+}{\partial \nu}\,dy+\dots \\[8pt]
&=& 0+\cfrac{\eps}{2i\beta_1}\dsp\int_{\Sigma_{L}\cup \Sigma_{-L}}\cfrac{\partial u_1}{\partial \nu}\,w_+-u_1\cfrac{\partial w_+}{\partial \nu}\,dy+\dots ,
\end{array}
\]
where again the dots correspond to higher order terms. We deduce that 
\[
dR(0)(\mu)=\cfrac{1}{2i\beta_1}\dsp\int_{\Sigma_{L}\cup \Sigma_{-L}}\cfrac{\partial u_1}{\partial \nu}\,w_+-u_1\cfrac{\partial w_+}{\partial \nu}\,dy.
\]
Integrating by parts in $\mathcal{S}_L\coloneqq(-L;L)\times I$, we deduce
\[
\begin{array}{rcl}
dR(0)(\mu)&=&\dsp\cfrac{1}{2i\beta_1}\int_{\mathcal{S}_L} \Delta u_1 w_+-u_1\Delta w_+\,dxdy-\cfrac{1}{2i\beta_1}\dsp\int_{\partial \mathcal{S}_L\cap\partial \mathcal{S}}\cfrac{\partial u_1}{\partial \nu}\,w_+-u_1\cfrac{\partial w_+}{\partial \nu}\,dx \\[10pt]
& = & -\cfrac{1}{2i\beta_1}\dsp\int_{\partial \mathcal{S}_L\cap\partial \mathcal{S}}\cfrac{\partial u_1}{\partial \nu}\,w_+-u_1\cfrac{\partial w_+}{\partial \nu}\,dx=\cfrac{1}{2i\beta_1}\dsp\int_{\partial \mathcal{S}_L\cap\partial \mathcal{S}}u_1\cfrac{\partial w_+}{\partial \nu}\,dx.
\end{array}
\]
Using the third line of (\ref{PbAsympto2}), this gives
\[
dR(0)(\mu)= \cfrac{i}{2\beta_1}\dsp\int_{\R } \mu(x) \bigg(\cfrac{\partial w_+}{\partial y}\bigg)^2(x,1)\,dx.
\]
Finally, from (\ref{Identifiu0}) (remember that $\varphi_1(y)=\sqrt{2}\sin(\pi y)$), we obtain the formulas 
\[
\begin{array}{rcl}
dR(0)(\mu)&=&\dsp\cfrac{i \pi^2}{\beta_1}\,\int_{\R}\mu(x) e^{2i\beta_1x}\,dx.
\end{array}
\]
Since the real part of $x\mapsto e^{2i\beta_1 x}$ is even while its imaginary part is odd, we see that it is easy to find functions $\mu_1$, $\mu_2$ satisfying the relations (\ref{DefBasis}), which ensures that $dR(0):\mathscr{C}^\infty_0(\R)\to\Cplx$ is onto.\\ 
\newline
The formal calculus above can be rigorously justified by proving error estimates. To proceed, one method consists in rectifying the boundary of $\Om^\eps$ using ``almost identical'' diffeomorphisms to transform the perturbed domain into the reference strip $\mathcal{S}$ (see \textit{e.g.} \cite[Chap.\,7,\,\S6.5]{Kato95}). Then one can prove the estimate, for $\eps$ small enough,
\[
\|u^\eps-(u_0+\eps u_1)\|_{\mH^1(\mathcal{S}_L\setminus\overline{\mathcal{V}})} \le C\eps^2,
\]
where $C$ is a constant independent of $\eps$ and $\mathcal{V}$ is a neighborhood of $\mrm{supp}(\mu)\times\{1\}$.
\end{proof}

\section{Numerical implementation}
The theoretical approach presented above leads very naturally to an algorithm to construct numerically non reflecting perturbations. Let us describe the strategy.\\
\newline
The main idea consists in solving the fixed point equation
\[
\vec{\tau}=G^\eps(\vec{\tau})
\]
(see (\ref{FixedPointEqn})) using an iterative procedure. First, we choose $\mu_0$, $\mu_1$, $\mu_2$ once for all. Then we start with $\vec{\tau}^{\,0}=(0,0)$ and for $p\in\N$, we set $\vec{\tau}^{\,p+1}=G^\eps(\vec{\tau}^{\,p})$. Denote $\mu^p\coloneqq\mu_0+\tau_1^p\mu_1+\tau_2^p\mu_2$. From (\ref{FixedPointEqn}), we have
\[
\begin{array}{rcl}
G^\eps(\vec{\tau}^{\,p})&=&-\eps(\Re e\,\tilde{R}(\eps\mu^p),\Im m\,\tilde{R}(\eps\mu^p)) \\[4pt]
&=& \vec{\tau}^{\,p}-\eps^{-1}(\Re e\,R(\eps\mu^p),\Im m\,R(\eps\mu^p)).
\end{array}
\]
Therefore, for $\vec{\tau}^{\,p}$ we obtain the recursive equation
\begin{equation}\label{IterativeProc}
\vec{\tau}^{\,p+1}=\vec{\tau}^{\,p}-\eps^{-1}(\Re e\,R(\eps\mu^p),\Im m\,R(\eps\mu^p)).
\end{equation}
We stop the loop when we have $|R(\eps\mu^p)| \le \eta$ where $\eta>0$ is a small given criterion. We then define $\vec{\tau}^{\,\mrm{sol}}$ as the last value of $\vec{\tau}^{\,p}$. If the iterative process does not converge, we try again with a smaller value of $\eps>0$. Note that at each step $p\ge0$, we need to solve a scattering problem of the form
\begin{equation}\label{PbNum}
\begin{array}{|rl}
\multicolumn{2}{|l}{\mbox{Find }u\mbox{ such that }u-w_+\mbox{ is outgoing and } }\\[3pt]
\Delta u + k^2 u = 0 & \mbox{ in }\Om^p\coloneqq\Om(\eps\mu^p) \\[3pt]
 u=0 & \mbox{ on }\partial\Om^p.
\end{array}
\end{equation}
To proceed, we approximate the solution of (\ref{PbNum})
by working with Formulation (\ref{FormuNumD}). We use a P2 finite element method in $\Om^p_5\coloneqq\{(x,y)\in\Om^p\,|\,|x|<5\}$. At $x=\pm 5$, a truncated Dirichlet-to-Neumann map with 10 terms serves as a transparent boundary condition. In other words, we take $q=2$, $L=5$, $M=10$ in (\ref{FormuNumD}). Note that at each step, it is necessary to mesh a new domain. For the computations, we use the \texttt{FreeFem++}\footnote{\texttt{FreeFem++}, \url{https://freefem.org/}.} software while we display the results with \texttt{Paraview}\footnote{\texttt{Paraview}, \url{http://www.paraview.org/}.}.
\begin{figure}[!ht]
\centering
\includegraphics[width=0.85\textwidth]{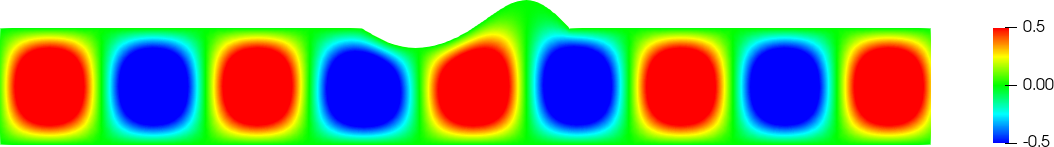}\\[2pt]
\includegraphics[width=0.85\textwidth]{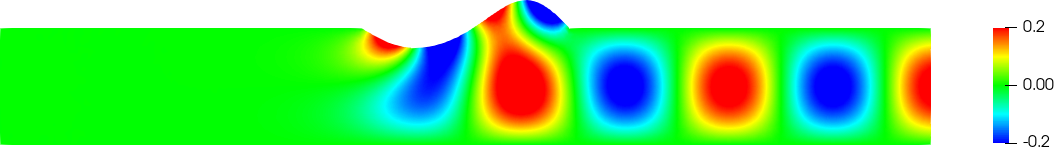}\\[2pt]
\includegraphics[width=0.85\textwidth]{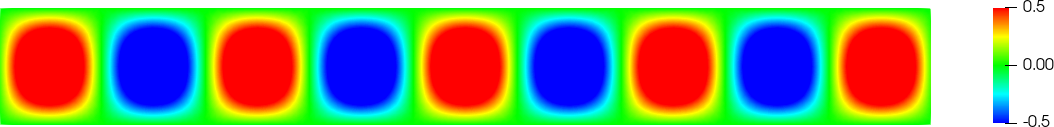}
\caption{Example of non reflecting obstacle for the problem (\ref{WaveguidePbSca}) with Dirichlet BCs. Top: $\Re e\,u$. Middle: $\Re e\,u_s$ with $u_s\coloneqq u-w_+$. Bottom: $\Re e\,w_+$ in the reference strip. Here $k=1.5\pi$.
\label{ZeroRgeom}}
\end{figure}

\noindent Let us give a concrete application. Set $\delta\coloneqq\pi/\beta_1=\pi/\sqrt{k^2-\pi^2}$, 
$I_{\mrm{defect}}\coloneqq(-\delta;\delta)$ and for $\mu_0$, $\mu_1$, $\mu_2$, let us work with the functions such that 
\[
\mu_0(x)=\sin(\beta_1 x),\quad \mu_1(x)=-\cfrac{\beta_1^2}{\pi^3}\,\sin(2\beta_1 x),\quad \mu_2(x)=\cfrac{7\beta_1^2}{12\pi^2}\,\cos(3\beta_1 x/2),\quad \mbox{ for }x\in I_{\mrm{defect}},
\]
$\mu_0=\mu_1=\mu_2=0$ in $\R\setminus\overline{I_{\mrm{defect}}}$. These functions are continuous and compactly supported but do not belong to $\mathscr{C}^{\infty}_0(\R)$. However this is not actually needed for the above theory. On the other hand, one can check that they indeed satisfy relations (\ref{DefBasis0}), (\ref{DefBasis}) (remark in particular that $\mu_0$, $\mu_1$ are odd while $\mu_2$ is even). We set $\eta=10^{-4}$, $k=1.5\pi$. In Figure \ref{ZeroRgeom}, we display the geometry at the end of the iterative procedure for $\eps=0.2$. We have obtained $|R|\approx 8.10^{-5}$ in 24 iterations. As expected, we observe that the scattered field is exponentially decaying as $x\to-\infty$ (the incident wave comes from the left). We note that for the transmitted field, there is a small shift of phase. This is not surprising because $R=0$ only implies $|T|=1$ and not $T=1$.  
Interestingly, the algorithm converges though $\eps$ is not that small. This is because there is only one complex coefficient to cancel and this allows us to get not so small non reflecting perturbations of the reference strip.

\section{Perfect transmission for the Dirichlet problem}\label{ParaT1_Dirichlet}

Can we hope for more and obtain $T=1$ with the above approach? To proceed, a natural idea is to work with the quantity $T-1$. More precisely, in the reference strip we have $T-1=0$. Is it possible to perturb the geometry while keeping $T-1=0$? Let us compute the differential of $T$ with respect to the geometry.  

\begin{proposition}\label{PropoDiffT_Dirichlet}
For $\mu\in\mathscr{C}^\infty_0(\R)$, we have
\[
\begin{array}{rcl}
dT(0)(\mu)&=&\dsp\cfrac{i \pi^2}{\beta_1}\,\int_{\R}\mu(x)\,dx.
\end{array}
\]
As a consequence, $dT(0):\mathscr{C}^\infty_0(\R)\to\Cplx$ is not onto.
\end{proposition}
\begin{proof}
To show this result, one works as for $R$ in Proposition \ref{PropoDiffR_Dirichlet}. Let us keep the same notation. 
According to Proposition \ref{PropoRepresentationD}, we have 
\[
T^\eps=1+\cfrac{1}{2i\beta_1}\dsp\int_{\Sigma_{L}\cup \Sigma_{-L}}\cfrac{\partial u^\eps}{\partial \nu}\,\overline{w_+}-u^\eps\cfrac{\partial \overline{w_+}}{\partial \nu}\,dy.
\]
Inserting the expansion (\ref{ExpansionD}) of $u^\eps$ in the above identity and using that $\overline{w_+}=w_-$, $u_0=w_+$, this gives
\[
\begin{array}{rcl}
T^\eps&=&1+\dsp\cfrac{1}{2i\beta_1}\dsp\int_{\Sigma_{L}\cup \Sigma_{-L}}\cfrac{\partial u_0}{\partial \nu}\,w_--u_0\cfrac{\partial w_-}{\partial \nu}\,dy+\cfrac{\eps}{2i\beta_1}\dsp\int_{\Sigma_{L}\cup \Sigma_{-L}}\cfrac{\partial u_1}{\partial \nu}\,w_--u_1\cfrac{\partial w_-}{\partial \nu}\,dy+\dots \\[8pt]
&=& 1+\cfrac{\eps}{2i\beta_1}\dsp\int_{\Sigma_{L}\cup \Sigma_{-L}}\cfrac{\partial u_1}{\partial \nu}\,w_--u_1\cfrac{\partial w_-}{\partial \nu}\,dy+\dots ,
\end{array}
\]
where the dots correspond to higher order terms. We deduce that 
\[
dT(0)(\mu)=\cfrac{1}{2i\beta_1}\dsp\int_{\Sigma_{L}\cup \Sigma_{-L}}\cfrac{\partial u_1}{\partial \nu}\,w_--u_1\cfrac{\partial w_-}{\partial \nu}\,dy.
\]
Integrating by parts in $\mathcal{S}_L=(-L;L)\times I$, and using the third line of (\ref{PbAsympto2}), this gives 
\[
dT(0)(\mu)=\dsp \cfrac{1}{2i\beta_1}\dsp\int_{\partial\mathcal{S}} u_1\cfrac{\partial w_-}{\partial y}\,dx =\dsp \cfrac{i}{2\beta_1}\dsp\int_{\R} \mu(x)\cfrac{\partial w_+}{\partial y}(x,1)\cfrac{\partial w_-}{\partial y}(x,1)\,dx = \cfrac{i\pi^2}{\beta_1}\dsp\int_{\R} \mu(x)\,dx.
\]
Since the real part of $dT(0)(\mu)$ is null for all $\mu\in\mathscr{C}^\infty_0(\R)$, this shows that $dT(0):\mathscr{C}^\infty_0(\R)\to\Cplx$ is not onto.
\end{proof}
\begin{remark}
The fact that the real part of $dT(0)$ is necessarily null could have been guessed from conservation of energy. Indeed, otherwise by linearity of $dT(0)$, we could find some $\mu\in \mathscr{C}^\infty_0(\R)$ such that $dT(0)(\mu)>0$. Then for $\eps>0$ small enough, $T(\eps\mu)$ would have a real part larger than one. This is impossible due to conservation of energy.
\end{remark}
\noindent Though $dT(0):\mathscr{C}^\infty_0(\R)\to\Cplx$ is not onto, Proposition \ref{PropoDiffT_Dirichlet} proves that we can control the imaginary part of $T$. Let us exploit this property.\\
\newline
Define the map
\begin{equation}\label{DefFT1D}
\begin{array}{rccl}
\mathscr{F}: & \mathscr{C}^\infty_0(\R) &\to &\R^3 \\[4pt] 
 & h & \mapsto & (\Re e\,R(h),\Im m\,R(h),\Im m\,T(h)).
\end{array}
\end{equation}
We have $\mathscr{F}(0)=0$ and we wish to find some $h\not\equiv 0$ such that $\mathscr{F}(h)=0$. With the results of Propositions \ref{PropoDiffR_Dirichlet} and \ref{PropoDiffT_Dirichlet}, it is easy to show that $d\mathscr{F}(0):\mathscr{C}^\infty_0(\R)\to\R^3$ is onto. Therefore there are $\mu_0,\dots,\mu_3\in \mathscr{C}^\infty_0(\R)$ such that
\[
d\mathscr{F}(0)(\mu_0)=0,\qquad [d\mathscr{F}(0)(\mu_1)^\top,d\mathscr{F}(0)(\mu_2)^\top,d\mathscr{F}(0)(\mu_3)^\top]=\mrm{Id}_{3\times3}.
\]
Set $\vec{\tau}\coloneqq(\tau_1,\tau_2,\tau_3)\in\R^3$ and consider the new fixed point equation
\begin{equation}\label{FixedPointEqnNew}
\vec{\tau}=F^\eps(\vec{\tau})
\end{equation}
with \hspace{1cm}
$\dsp\begin{array}{|rcl}
F^\eps(\vec{\tau})&\coloneqq&-\eps(\Re e\,\tilde{R}(\eps\mu(\vec{\tau})),\Im m\,\tilde{R}(\eps\mu(\vec{\tau})),\Im m\,\tilde{T}(\eps\mu(\vec{\tau})) \\[4pt]
\mu(\vec{\tau})&\coloneqq&\mu_0+\tau_1\mu_1+\tau_2\mu_2+\tau_3\mu_3.
\end{array}$\\[4pt]
Here $\tilde{R}$, $\tilde{T}$ are the remainders in the expansions 
\[
\begin{array}{rcl}
R(\eps \mu)&=&  \eps dR(0)(\mu)+\eps^2 \tilde{R}(\eps \mu)\\[2pt]
T(\eps \mu)&=& 1+ \eps dT(0)(\mu)+\eps^2 \tilde{T}(\eps \mu).
\end{array}
\]
Then for any given $r>0$, one can show that (\ref{FixedPointEqnNew}) admits a unique solution $\vec{\tau}^{\,\mrm{sol}}\in\overline{B(O,r)}$ for $\eps$ small enough (here $B(O,r)$ denotes the open ball of $\R^3$ centered at $O$ of radius $r$). Defining 
\[
h^{\,\mrm{sol}}=\eps\bigg(\mu_0+\sum_{p=1}^3\tau_p^{\mrm{sol}}\mu_p\bigg),
\]
we obtain $\mathscr{F}(h^{\,\mrm{sol}})=0$.\\ 
\newline
Why does this imply $T^\eps=1$? Conservation of energy imposes $|R^\eps|^2+|T^\eps|^2=1$. Therefore when $R^\eps=0$ and $\Im m\,T^\eps=0$, the only possibility is to have either $T^\eps=-1$ or $T^\eps=1$. Since we made a small perturbation of the reference strip, error estimates for the asymptotic expansion of $u^\eps$ imply that for $\eps$ small, $T^\eps$ is close to one. Therefore, for $\eps$ small enough, necessarily we must have $T^\eps=1$ exactly.

\begin{figure}[!ht]
\centering
\includegraphics[width=0.85\textwidth]{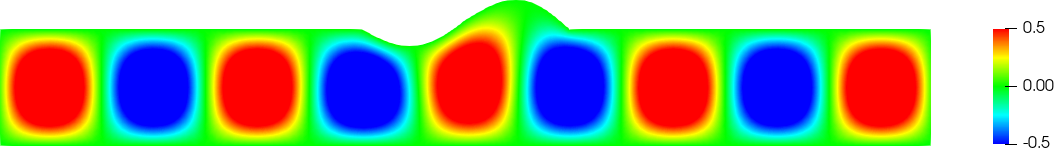}\\[2pt]
\includegraphics[width=0.85\textwidth]{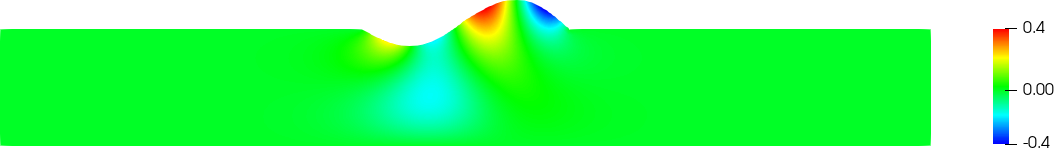}\\[2pt]
\includegraphics[width=0.85\textwidth]{uRef.png}\vspace{-0.1cm}
\caption{Example of perfectly invisible obstacle for the problem (\ref{WaveguidePbSca}) with Dirichlet BCs. Top: $\Re e\,u$. Middle: $\Re e\,u_s$. Bottom: $\Re e\,w_+$ in the reference strip. Here $k=1.5\pi$.
\label{T1geom}}
\end{figure}

\noindent In Figure \ref{T1geom}, we give an example of perfectly invisible perturbation of the reference strip for the problem (\ref{WaveguidePbSca}). This time, compared to Figure \ref{ZeroRgeom} where we imposed only zero reflection, we observe that the scattered field is indeed exponentially decaying both as $x\to-\infty$ and as $x\to+\infty$.

\section{Study of the Neumann problem}

Assume now that we wish to apply the strategy described in \S\ref{ParaGeneScheme} to Problem (\ref{WaveguidePbNeumann}) with Neumann BCs. Pick some $k\in(0;\pi)$ so that only $w_\pm(x,y)=e^{\pm ikx}$ are propagating modes. \\
\newline
Let us compute $dR(0)$ and $dT(0)$. As above, for $\eps>0$ and a given $\mu\in\mathscr{C}^\infty_0(\R)$, denote by $\Om^\eps$ the domain $\Om(\eps\mu)$. Let $u^\eps$ stand for the solution introduced in (\ref{RepresentationNeumannP}) corresponding to the scattering of the incident rightgoing wave $w_+$ in the geometry $\Om^\eps$. It satisfies  
\begin{equation} \label{fortEps} 
\begin{array}{|rccl}
\Delta u^\eps + k^2u^\eps &=& 0&  \mbox{in }\Om^\eps\\[3pt]
\partial_{\nu^\eps} u^\eps&=&0 & \mbox{on } \partial\Om^\eps.
\end{array}	
\end{equation}
Let us emphasize that $\nu^\eps$, the outward unit normal vector to $\partial\Om^\eps$, depends on $\eps$. For $u^\eps$, as $\eps$ tends to zero, we consider the ansatz 
\begin{equation}\label{ExpansionN}
u^{\eps}=u_0+\eps u_1+\dots
\end{equation}
where $u_0$, $u_1$ are functions to be determined and the dots correspond to higher order terms. On $\partial\Om^\eps$, we have the expansions
\begin{equation}\label{ExpansionNormal}
\nu^\eps = \cfrac{1}{\sqrt{1+\eps^2(\mu'(x))^2}}\,\left( \begin{array}{c}
-\eps \mu'(x)\\
1
\end{array}
\right)=\left( \begin{array}{c}
0\\
1
\end{array}
\right)+\eps\left( \begin{array}{c}
-\mu'(x)\\
0
\end{array}
\right)+\dots
\end{equation}
\begin{equation}\label{TaylorExpansion}
\begin{array}{rcl}
\nabla u^\eps(x,1+\eps \mu(x))&=&\nabla u^\eps(x,1)+\eps \mu(x)\left( \begin{array}{c}
\partial^2_{xy}u_\eps(x,1)\\[2pt]
\partial^2_{yy}u_\eps(x,1)
\end{array}
\right)+\dots\,.
\end{array}
\end{equation}
Now we insert (\ref{ExpansionN}) in (\ref{fortEps}) and exploit  (\ref{ExpansionNormal}), (\ref{TaylorExpansion}). Collecting the terms of orders $\eps^0$, $\eps^1$, we find that $u_0$, $u_1$ satisfy respectively the problems
\[
\begin{array}{|rccl}
\Delta u_0 + k^2u_0 &=& 0&  \mbox{in }\mathcal{S}\\[3pt]
\partial_{\nu} u_0&=&0 & \mbox{on } \partial\mathcal{S}
\end{array}\qquad \begin{array}{|rccl}
\Delta u_1 + k^2u_1 &=& 0&  \mbox{in } \mathcal{S}\\[3pt]
\partial_{\nu} u_1&=&0 & \mbox{on } \R\times\{0\}\\[3pt]
\partial_{\nu} u_1&=& \mu'(x)\partial_x u_0-\mu(x)\partial^2_{yy}u_0 & \mbox{on }\R\times\{1\}.
\end{array}
\]
Using additionally that $u^\eps-w_+$ is outgoing, we get first 
\[
u_0(x,y)=w_+(x,y)=e^{ikx}.
\]
Since $w_+$ is independent of $y$, we deduce that $u_1$ satisfies the condition
\begin{equation}\label{ReecritureBC}
\partial_{\nu} u_1= \mu'(x)\partial_x w_+\quad \mbox{on }\R\times\{1\}.
\end{equation}
Denote by $R^\eps$, $T^\eps$ the scattering coefficients of $u^\eps$. Assuming that $\mu$ is supported in $(-L;L)$, from Proposition \ref{PropoRepresentationN}, we know that 
\begin{equation}\label{GeneralFormulaRT}
2ikR^\eps=\int_{\Sigma_{-L}\cup \Sigma_{L}}\cfrac{\partial u^\eps}{\partial\nu}\, \overline{w_-}-u^\eps \cfrac{\partial \overline{w_-}}{\partial\nu}\,dy,\qquad 2ik(T^\eps-1)=\int_{\Sigma_{-L}\cup \Sigma_{L}}\cfrac{\partial u^\eps}{\partial\nu}\,\overline{w_+}-u^\eps \cfrac{\partial \overline{w_+}}{\partial\nu}\,dy.
\end{equation}
Inserting the expansion (\ref{ExpansionN}) of $u^\eps$ in (\ref{GeneralFormulaRT}), integrating by parts and exploiting (\ref{ReecritureBC}) as well as the fact that $\overline{w_\pm}=w_\mp$, we obtain
\[
R^\eps=0-\cfrac{\eps}{2ik}\,\int_{\R}\mu'(x)\partial_xw_+(x) w_+(x)\,dx+\dots,\quad T^\eps=1-\cfrac{\eps}{2ik}\,\int_{\R}\mu'(x)\partial_xw_+(x)w_-(x)\,dx+\dots.
\]
This gives the formulas
\begin{equation}\label{ExpressDiffSNeumann}
\begin{array}{rcl}
dR(0)(\mu)&=&\dsp\cfrac{-1}{2ik}\,\int_{\R}\mu'(x)\partial_xw_+ w_+\,dx=\cfrac{-1}{2}\,\int_{\R}\mu'(x) e^{2ikx}\,dx=ik\,\int_{\R}\mu(x)e^{2ikx}\,dx\\[12pt]
dT(0)(\mu)&=&\dsp\cfrac{-1}{2ik}\,\int_{\R}\mu'(x)\partial_xw_+ w_-\,dx=\cfrac{-1}{2}\,\int_{\R}\mu'(x)\,dx=0.
\end{array}
\end{equation}
As in the Dirichlet case, exploiting that the real part of $x\mapsto e^{2ikx}$ is even while its imaginary part is odd, we see that it is easy to find functions $\mu_1$, $\mu_2$ satisfying the relations (\ref{DefBasis}), which ensures that $dR(0):\mathscr{C}^\infty_0(\R)\to\Cplx$ is onto. Then by applying what has been done in \S\ref{ParaGeneScheme}, one can construct perturbations of the reference strip which are non reflecting at a given $k\in(0;\pi)$.\\
\newline
\noindent Let us give two examples. Set $\delta\coloneqq\pi/k$, 
$I_{\mrm{defect}}\coloneqq(-\delta;\delta)$ and for $\mu_0$, $\mu_1$, $\mu_2$, let us work first with the functions such that 
\begin{equation}\label{DefMuNeumann}
\mu_0(x)=\sin(k x),\quad \mu_1(x)=-\cfrac{1}{\pi}\,\sin(2k x),\quad \mu_2(x)=\cfrac{7}{12}\,\cos(3k x/2),\quad \mbox{ for }x\in I_{\mrm{defect}},
\end{equation}
$\mu_0=\mu_1=\mu_2=0$ in $\R\setminus\overline{I_{\mrm{defect}}}$. One can check that they indeed satisfy relations (\ref{DefBasis0}), (\ref{DefBasis}). We set $\eta=10^{-4}$, $k=0.8\pi$. In Figure \ref{ZeroRgeomNeumann}, we display the geometry at the end of the iterative procedure for $\eps=0.4$. For this choice of functions $\mu_0$, $\mu_1$, $\mu_2$, we are able to obtain a rather large non reflecting perturbation of the reference strip. Here $|R|\approx 4.10^{-5}$ in 15 iterations. 

\begin{figure}[!ht]
\centering
\includegraphics[width=0.85\textwidth]{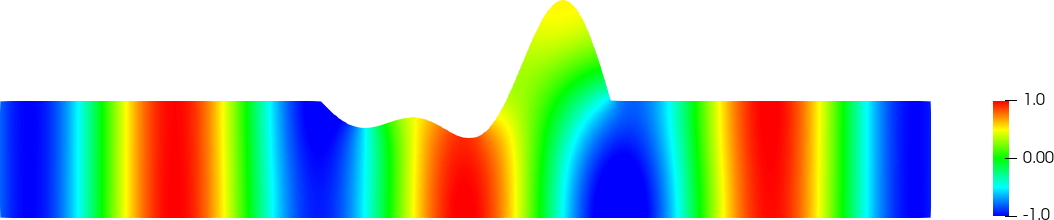}\\[2pt]
\includegraphics[width=0.85\textwidth]{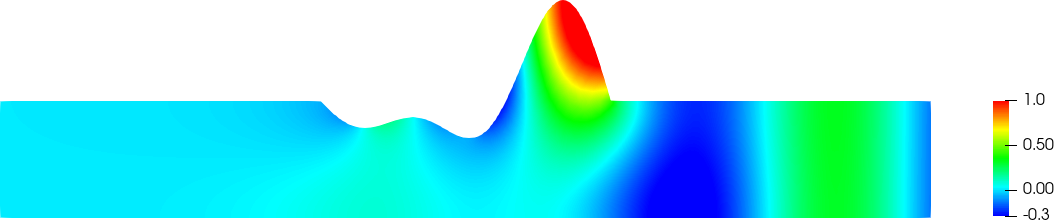}\\[2pt]
\includegraphics[width=0.85\textwidth]{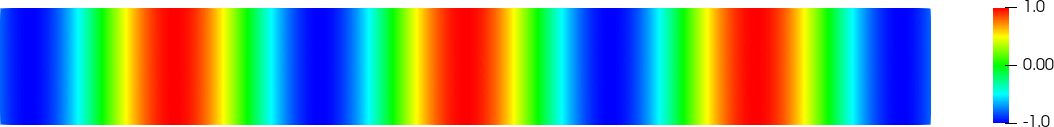}
\caption{Example of non reflecting geometry for Problem (\ref{WaveguidePbNeumann}) with Neumann BCs. Top: $\Re e\,u$. Middle: $\Re e\,u_s$. Bottom: $\Re e\,w_+$ in the reference strip. Here $k=0.8\pi$.
\label{ZeroRgeomNeumann}}
\end{figure}

\noindent In Figure \ref{ZeroRgeomNeumannSym}, we display another example of non reflecting defect. It has been obtained by changing the $\mu_0$ in (\ref{DefMuNeumann}) to 
\[
\mu_0(x)=|x|-\delta,\quad \mbox{ for }x\in I_{\mrm{defect}},
\]
$\mu_0=0$ in $\R\setminus\overline{I_{\mrm{defect}}}$. Here the defect lies entirely in the region $y<1$. Because of this property, we can use symmetry with respect to the line $y=1$ to create a non reflecting obstacle completely embedded in the waveguide (see Figure \ref{ZeroREmbedded}).\\ 
\newline
In \S\ref{ParaT1_Dirichlet} for the Dirichlet problem, by exploiting conservation of energy, we explained how to construct waveguide where $T=1$. Can we adapt this to the Neumann problem? Well, from the computation of $dT(0)$ in (\ref{ExpressDiffSNeumann}) we see that it is impossible because $dT(0)$ is null. It means that a perturbation of order $\eps$ of the reference strip gives a transmission coefficient $T^\eps$ such that $T^\eps-1$ is in $O(\eps^2)$. This looks appealing for perfect transmission. The problem is that since $dT(0)$ is null, one cannot use this term which has a linear dependence with respect to the perturbation of the reference strip to cancel the whole (non linear) expansion of $T^\eps$ via the resolution of the fixed point problem. As a consequence, our technique fails to design perfectly invisible defects in the Neumann case.

\begin{figure}[!ht]
\centering
\includegraphics[width=0.85\textwidth]{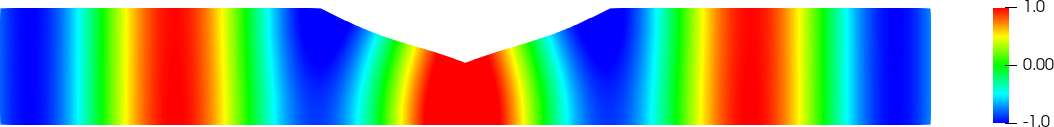}\\[2pt]
\includegraphics[width=0.85\textwidth]{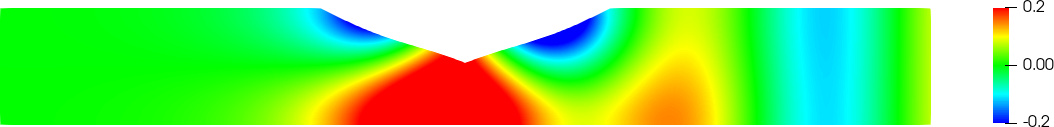}\\[2pt]
\includegraphics[width=0.85\textwidth]{uRef_Neumann.png}
\caption{Example of non reflecting obstacle for Problem (\ref{WaveguidePbNeumann}) with Neumann BCs. Top: $\Re e\,u$. Middle: $\Re e\,u_s$. Bottom: $\Re e\,w_+$ in the reference strip. Here $k=0.8\pi$.
\label{ZeroRgeomNeumannSym}}
\end{figure}

\begin{figure}[!ht]
\centering
\includegraphics[width=0.85\textwidth]{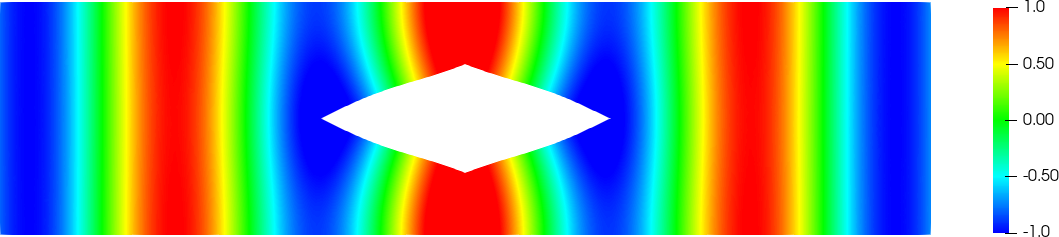}
\caption{Example of non reflecting obstacle for Problem (\ref{WaveguidePbNeumann}) with Neumann BCs. Here $k=0.8\pi$.
\label{ZeroREmbedded}}
\end{figure}

\noindent Above we showed how to construct invisible smooth perturbations of the reference strip $\mathcal{S}$. In the remaining part of this chapter, we wish to explain how to design invisible non smooth perturbations of $\mathcal{S}$. The terminology smooth/not smooth here does not refer to the regularity of the domain but to the form of the asymptotic expansion involved in the asymptotic procedure. In the non smooth case, the field $u^\eps$ exhibits rapid variations in a neighborhood of the perturbation that must be caught with adapted variables. In that situation, the asymptotics of the scattering coefficients is in general more complicated to obtain. The interesting point is that it can bring new useful terms, for example non zero differentials for $T$ for the Neumann problem. Below we consider two types of non smooth perturbations of the reference strip.

\section{Non reflecting clouds of small obstacles}\label{SectionFlies}

\begin{figure}[!ht]
\centering
\begin{tikzpicture}[scale=0.9]
\draw[dashed] (-3,1) ellipse (0.4 and 1);
\draw[fill=gray!5] (-2,1) ellipse (0.4 and 1);
\draw[fill=gray!5,draw=none](-2,0) rectangle (2,2);
\draw (-2,0)--(2,0);
\draw[dashed](-3,0)--(-2,0);
\draw[dashed](3,0)--(2,0);
\draw[dashed](-3,2)--(-2,2);
\draw[dashed](3,2)--(2,2);
\draw(-2,2)--(2,2);
\draw[fill=gray!50,fill opacity=0.4](0:-2 and 0) arc (90:270:0.4 and -1)--(2,2)--(2,0)--(-2,0);
\draw[fill=gray!40] (2,1) ellipse (0.4 and 1);
\node at (-1.5,0.3){\small$\Omega^0$};
\draw[dashed] (3,1) ellipse (0.4 and 1);
\end{tikzpicture}\hspace{0.9cm}\begin{tikzpicture}[scale=0.9]
\node[inner sep=0pt] at (0.7,0.4)  {\includegraphics[width=8mm]{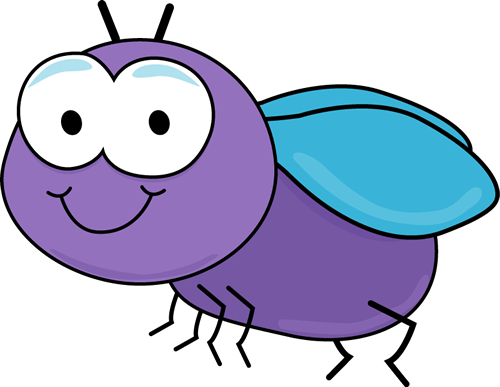}};
\node[inner sep=0pt] at (-1,1.2)  {\includegraphics[width=8mm]{images/mouche1}};
\draw[dashed] (-3,1) ellipse (0.4 and 1);
\draw[fill=gray!50,fill opacity=0.4](0:-2 and 0) arc (90:270:0.4 and -1)--(2,2)--(2,0)--(-2,0);
\node at (-1,1.65){\small$\mathcal{O}^\eps_1$};
\node at (0.7,0.85){\small$\mathcal{O}^\eps_2$};
\draw (-2,0)--(2,0);
\draw (-2,2)--(2,2);
\draw [dashed](-3,0)--(-2,0);
\draw [dashed](3,0)--(2,0);
\draw [dashed](-3,2)--(-2,2);
\draw [dashed](3,2)--(2,2);
\draw[fill=gray!40] (2,1) ellipse (0.4 and 1);
\draw[dashed] (3,1) ellipse (0.4 and 1);
\node at (-1.5,0.3){\small$\Omega^{\eps}$};
\end{tikzpicture}\hspace{0.9cm}\begin{tikzpicture}
\node[inner sep=0pt] at (0,1.6)  {\includegraphics[width=2cm]{images/mouche1}};
\node at (0.3,2.4){\small$\mathcal{O}$};
\node at (0,1){\small \phantom{espa}};
\end{tikzpicture}\vspace{0.4cm}
\caption{Unperturbed waveguide (left), perturbed waveguide with two small obstacles (middle), set $\mathcal{O}$ (right).\label{Domain}} 
\end{figure}

\noindent Let us work first with clouds of small obstacles. To simplify the asymptotic analysis, we  work in 3D with Dirichlet boundary conditions. Note that in 2D, the Green function of the Laplace operator, whose behavior dictates the form of the asymptotic expansions as $\eps$ tends to zero, has a logarithmic singularity which does not appear in 3D. Let 
\[
\Om^0\coloneqq\{z=(x,y)\,|\,x\in\R\mbox{ and }y\in\Theta\}
\]
be a cylinder of $\R^3$ whose transverse section $\Theta\subset\R^2$ is a bounded domain with Lipschitz boundary (see Figure \ref{Domain} left). Consider $\mathcal{O}\subset\R^3$ a bounded domain with Lipschitz boundary and for $M_1$ a point located in $\Om^0$, set, for $\eps$ small,
\[
\mathcal{O}^\eps_1\coloneqq\{z\in\R^3\,|\,\eps^{-1}(z-M_1)\in\mathcal{O}\}.
\]
Finally define the perturbed waveguide 
\[
\Om^\eps\coloneqq\Om^0\setminus\overline{\mathcal{O}^\eps_1}.
\] 
To begin with, for the moment we assume that there is only one obstacle and not a cloud. The problem we consider writes
\begin{equation} \label{fortEpsFly} 
\begin{array}{|rccl}
\Delta u^\eps + k^2u^\eps &=& 0&  \mbox{in }\Om^\eps\\[3pt]
 u^\eps&=&0 & \mbox{on } \partial\Om^\eps.
\end{array}	
\end{equation}
We fix the wavenumber $k>0$ such that only two modes $w_\pm$ can propagate in $\Om^\eps$. These $w_\pm$ are not exactly the same as the ones in (\ref{Defwpm}), they involve the eigenfunctions associated with the first eigenvalue of the Dirichlet Laplacian in $\Theta$, but the situation is similar. As in \S\ref{ParaScaPb}, Problem (\ref{fortEpsFly}) admits a solution with the expansion
\[
u^\eps=\begin{array}{|ll}
w_++R^\eps w_-+\tilde{u}^\eps & \mbox{ for }x\le - d \\[3pt]
\phantom{w_+mi}T^\eps w_++\tilde{u}^\eps & \mbox{ for }x\ge d,
\end{array}
\]
where $R^\eps$, $T^\eps\in\Cplx$ and $\tilde{u}^\eps$ decays exponentially at infinity (we denote $R^\eps$ instead of $R^\eps_+$ to simplify). Again $d>0$ is such that $\Om^\eps=\Om^0$ for $|x|>d$. The first step in the approach is to compute an asymptotic expansion of $u^\eps$ as $\eps$ tends to zero. This is a rather long work that we will not present here (one may consult the  reference \cite[\S2.2]{MaNP00} for more details). Let us simply stress that it appears that $u^\eps$ has a rapid variation in a neighborhood of the obstacle to satisfy the constraint of being null on $\partial\mathcal{O}^\eps_1$. This rapid variation must be caught with adapted variables. Then it is necessary to work both with some inner field and outer field expansions of $u^\eps$ that we match in some intermediate region. This is the method of matched asymptotic expansions which is well documented in the literature (see
the monographs \cite{Ilin,MaNP00}). At the end of the procedure, when $\eps$ tends to zero, we obtain
\[
R^\eps=0+\eps\,4i\pi\mrm{cap}(\mathcal{O})w_+(M_1)^2+O(\eps^2),\qquad\quad T^\eps=1+\eps\,4i\pi\mrm{cap}(\mathcal{O})|w_+(M_1)|^2+O(\eps^2).
\]
Here $\mrm{cap}(\mathcal{O})$ stands for the capacity of the domain $\mathcal{O}$, a constant which appears classically in asymptotic analysis. It is given by 
\[
\mrm{cap}(\mathcal{O})=\cfrac{1}{4\pi}\,\int_{\R^3\setminus\overline{\mathcal{O}}}|\nabla P|^2\,dz
\]
where $P$ is the potential satisfying
\[
\begin{array}{|rcll}
\Delta P&=&0&\mbox{ in }\R^3\setminus\overline{\mathcal{O}} \\[3pt]
P&=&1 & \mbox{ on }\partial\mathcal{O}
\end{array}
\]
which decays at infinity. Note that one can prove that $P$ admits the decomposition
\[
P(z)=\cfrac{\mrm{cap}(\mathcal{O})}{|z|}+O(|z|^{-2})\quad\mbox{ as }|z|\to+\infty.
\]
An important point for our study is that one has always $\mrm{cap}(\mathcal{O})>0$. Additionally, one finds that $w_+$ does not vanish in $\Om^0$. From these two properties and the asymptotics above of $R^\eps$, we infer that one single small obstacle cannot even be non reflecting: whatever the choice for $M_1$ or for the shape $\mathcal{O}$, $\mathcal{O}^\eps_1$ will always generate a reflection whose amplitude is of order $\eps$.\\
\newline
Let us add a second small obstacle 
\[
\mathcal{O}^\eps_2\coloneqq\{z\in\R^3\,|\,\eps^{-1}(z-M_2)\in\mathcal{O}\}
\]
in the waveguide, centred at the point $M_2\in\Om^0$ with $M_2\ne M_1$. Still denoting by $R^\eps$, $T^\eps$ the scattering coefficients in this new geometry, we obtain the expansions, as $\eps\to0^+$,
\begin{equation}\label{AsyExpansionsFlies}
R^\eps=0+\eps\,4i\pi\mrm{cap}(\mathcal{O})\sum_{j=1}^2w_+(M_j)^2+O(\eps^2),\quad T^\eps=1+\eps\,4i\pi\mrm{cap}(\mathcal{O})\sum_{j=1}^2|w_+(M_j)|^2+O(\eps^2).
\end{equation}
Observe that the interactions between the small objects do not appear at order $\eps$, only at higher orders. The interesting point is that now, using the known expression of $w_+$, we can find positions $M_1$, $M_2$ of the obstacles such that
\begin{equation}\label{Constraint1}
\sum_{j=1}^2w_+(M_j)^2=0.
\end{equation}
In that case, we have a perturbation of $\Om^0$ of order $\eps$ which produces a reflection in $O(\eps^2)$. As mentioned above, this is almost no reflection. But by working a bit harder, we can achieve more. Pick $M_1$, $M_2$ such that relation (\ref{Constraint1}) is satisfied. Then by slightly perturbing the position of one obstacle, we can get $R^\eps=0$ exactly. More precisely, for $\tau\in\R^3$, define $M^{\eps\tau}_1=M_1+\eps\tau$ and 
\[
\mathcal{O}^{\eps\tau}_1\coloneqq\{z\in\R^3\,|\,\eps^{-1}(z-M^{\eps\tau}_1)\in\mathcal{O}\}.
\]
One can show that there is $\tau$, which is defined as the solution of a fixed point problem similar to (\ref{FixedPointEqn}), such that the reflection coefficient in the corresponding waveguide is zero (for more details, see \cite{ChNa16}). Note that since $\tau\in\R^3$, we have enough degrees of freedom to cancel one single complex number.\\
\newline
Can we get perfect invisibility, \textit{i.e.} $T^\eps=1$, with this approach? From (\ref{AsyExpansionsFlies}) we see that the answer is no. Indeed, whatever the position of the small obstacles or their number, we always have a phase shift for the transmitted wave.\\
\newline
One can also study what happens at higher wavenumbers $k$. In that situation, as can be seen in the particular case  (\ref{DefModes}), more modes can propagate. The reflection and transmission coefficients then become respectively reflection and transmission matrices $\mathcal{R}^\eps$, $\mathcal{T}^\eps$. One can prove that by working with a sufficiently large number of small obstacles, we can cancel exactly the whole reflection matrix. The idea is the same as above. First we compute an asymptotic expansion of $\mathcal{R}^\eps$ as $\eps\to0$. Then we find positions $M_1$, $M_2$,... , $M_K$ of the obstacles to cancel the term of order $\eps$ in the expansion of $\mathcal{R}^\eps$. Then by slightly perturbing the position of one group of obstacles by solving a fixed point problem in $\R^N$, for a certain $N$ depending on the number of propagating modes, we get $\mathcal{R}^\eps=0$. The higher the number of propagating modes, the more obstacles we need.\\ 
\newline
To conclude this section, let us mention that in this work, the main difficulty mathematically consists in proving error estimates in the asymptotic expansions which are uniform with respect to the parameter $\tau$ in a closed ball to justify that the map appearing in the fixed point problem is indeed a contraction for $\eps$ small enough. We will not elaborate more on that topic here and refer the interested reader to \cite{ChNa16}.

\section{Perfect transmission for the Neumann problem}\label{SectionTOne}

The different strategies presented above fail to provide examples of waveguides where $T = 1$ (perfect transmission without phase shift) for the problem with Neumann BCs. Is there some fundamental obstruction to get $T = 1$ in that case? In this section, we answer negatively to that question. To proceed, we work with another singular perturbation of the reference strip. \\

\begin{figure}[!ht]
\centering
\begin{tikzpicture}[scale=2]
\draw[fill=gray!30,draw=none](-2,1) rectangle (2,2);
\draw (-2,1)--(2,1); 
\draw[dashed] (-2.5,1)--(-2,1); 
\draw[dashed] (2.5,1)--(2,1); 
\draw (-2,2)--(2,2);  
\draw[dashed] (2.5,2)--(2,2);
\draw[dashed] (-2.5,2)--(-2,2);

\draw[dotted,>-<] (-0.2,3.9)--(0.2,3.9);
\draw[dotted,>-<] (-0.8,3.4)--(-1.2,3.4);
\draw[dotted,>-<] (0.8,3.6)--(1.2,3.6);

\draw[dotted,>-<] (0.3,1.95)--(0.3,3.85);
\draw[dotted,>-<] (-0.7,1.95)--(-0.7,3.35);
\draw[dotted,>-<] (1.3,1.95)--(1.3,3.55);

\node at (0,4){\small $\eps$};
\node at (-1,3.5){\small $\eps$};
\node at (1,3.7){\small $\eps$};

\node at (-0.55,2.65){\small $h_1$};
\node at (0.45,2.9){\small $h_2$};
\node at (1.45,2.75){\small $h_3$};

\draw[fill=gray!30](-0.1,2) rectangle (0.1,3.8);
\draw[fill=gray!30](-0.9,2) rectangle (-1.1,3.3);
\draw[fill=gray!30](0.9,2) rectangle (1.1,3.5);

\draw[thick,draw=gray!30](-0.095,2)--(0.095,2);
\draw[thick,draw=gray!30](-0.905,2)--(-1.095,2);
\draw[thick,draw=gray!30](0.905,2)--(1.095,2);

\node[mark size=1pt,color=black] at (-1,2) {\pgfuseplotmark{*}};
\node[color=black] at (-1,1.8) {\small $M_1$};
\node[mark size=1pt,color=black] at (0,2) {\pgfuseplotmark{*}};
\node[color=black] at (0,1.8) {\small $M_2$};
\node[mark size=1pt,color=black] at (1,2) {\pgfuseplotmark{*}};
\node[color=black] at (1,1.8) {\small $M_3$};
\node at (-1.8,1.1){\small$\Omega^{\eps}$};
\end{tikzpicture}
\caption{Geometry of $\Om^{\eps}$.\label{DomainOriginal2D}} 
\end{figure}

\noindent Consider the waveguide $\Om^\eps$ pictured in Figure \ref{DomainOriginal2D}. It is made of the reference strip $\mathcal{S}$ to which we have glued at the points $M_1\coloneqq(x_1,1)$, $M_2\coloneqq(x_2,1)$, $M_3\coloneqq(x_3,1)$, thin chimneys of width $\eps>0$ small and heights respectively equal to $h_1$, $h_2$, $h_3$. We fix $k\in(0;\pi)$ so that only the modes $w_\pm(x,y)=e^{ikx}$ can propagate. The first step is to compute an asymptotic expansion of the scattering coefficients as $\eps$, the width of the thin rectangles, tends to zero. Again, this is a rather long work that we will not present here, the reason being that the fields vary rapidly in a region around the $M_n$, $n=1,2,3$. We refer the reader to \cite{BoCN18} for more details. In this article, by using the method of matched asymptotic expansions, it is shown that when the $h_n$ are such that $kh_n\not\in\{\pi/2+\pi\N\}$, when $\eps$ tends to zero, we have
\begin{equation}\label{ExpansionTrumpet}
\begin{array}{rcl}
R^\eps &=& \dsp0+\eps\Big(ik\sum_{n=1}^3(w_+(M_n))^2\tan(kh_n)\Big)+O(\eps^2) \\[10pt]
T^\eps &=& \dsp1+\eps\Big(ik\sum_{n=1}^3\tan(kh_n)\Big)+O(\eps^2).
\end{array}
\end{equation}

\begin{remark}
When $kh_n\in\{\pi/2+\pi\N\}$, resonant phenomena appear in the chimneys and the asymptotics (\ref{ExpansionTrumpet}) is no more valid. This regime will be exploited in Chapter III (see Section \ref{ParagCloaking}).
\end{remark}

\noindent The crucial point to observe in (\ref{ExpansionTrumpet}) compared to the other perturbations of the reference strip above is the appearance of the term of order $\eps$ in the expansion of $T^\eps$. Its imaginary part is non zero and can change sign according to the value of the $h_n$. Thus by considering another type of perturbation of $\mathcal{S}$, we have been able to obtain a non zero $dT(0)$.\\
\newline
By using that $(w_+(M_n))^2=e^{2ikx_n}$, we can find positions and heights of the chimneys such that $R^\eps=O(\eps^2)$ and $T^\eps-1=O(\eps^2)$. Then perturbing slightly the $h_n$ around these particular values, by solving a fixed point problem in $\R^3$ similar to (\ref{FixedPointEqn}), we can achieve 
\[
R^\eps=0\qquad\mbox{ and }\qquad\Im m\,T^\eps=0
\]
in the new geometry. Finally, by exploiting the relation of conservation of energy $|R^\eps|^2+|T^\eps|^2=1$ as in \S\ref{ParaT1_Dirichlet}, one shows that this implies $T^\eps=1$ for $\eps>0$ sufficiently small.\\
\newline
Initially we have to control two complex numbers ($R^\eps$ and $T^\eps$), so a priori we need four real degrees of freedom. But due to the constraint of conservation of energy, three parameters are sufficient. This explains why three chimneys are involved here. We could also have probably worked with only two chimneys, by perturbing their heights and the distance between them.

\begin{figure}[!ht]
\centering
\includegraphics[width=0.82\textwidth]{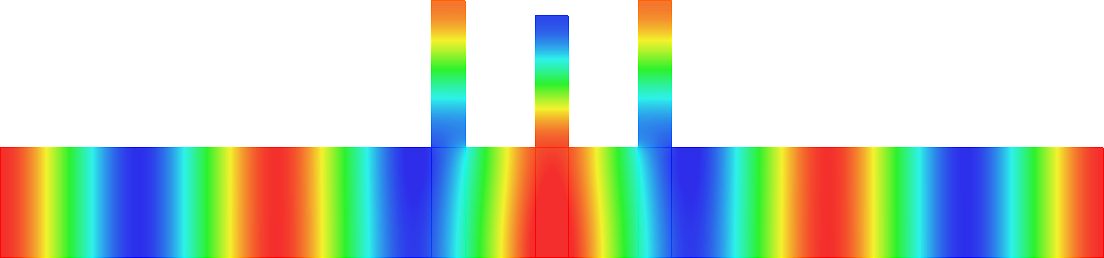}\\[15pt]
\includegraphics[width=0.82\textwidth]{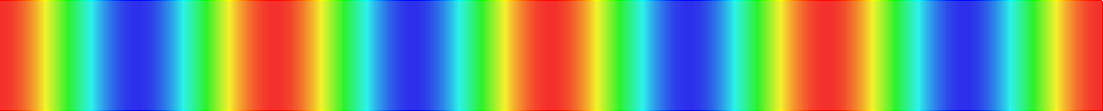}\\[15pt]
\caption{Top: real part of $u^\eps$ in the geometry obtained at the end of the iterative procedure. Bottom: real part of $w_+$ in the reference strip. \label{Trumpet}}
\end{figure}

\noindent This approach can be implemented numerically very naturally. First we set the $M_n$ and the $h_n$ to kill the terms of order $\eps$ in (\ref{ExpansionTrumpet}). Then we tune slightly the length of the chimneys by solving the corresponding fixed point problem iteratively as in (\ref{IterativeProc}). At each step, we solve a scattering problem in a new geometry. One can see this procedure as acting on the pistons on a trumpet to achieve $T^\eps=1$. In Figure \ref{Trumpet} we display a geometry obtained with this method. A bit far from the chimneys so that evanescent terms can be decently neglected, we remark that the field is the same as in the reference strip. Again, numerically one observes that the algorithm converges though $\eps$, the width of the chimneys, is not that small.

\section{Concluding remarks}\label{SectionContinuation}

In this chapter, we constructed smooth and non-smooth perturbations of the reference geometry which are invisible, in a broad sense (non reflecting or perfectly invisible). We worked at a given wavenumber $k$. We could proceed similarly to impose invisibility at given $k_1,\dots,k_N$. However we emphasize that the set of wavenumbers must be discrete and finite. Imposing zero reflection for a continuum of $k$ is probably impossible in general due to the analyticity of the map $k\mapsto R(k)$.
In the numerical results we presented, we were interested in controlling only $R$ and sometimes $\Im m\,T$ at one $k$. As a consequence, we had very few constraints and for this reason, the fixed point algorithm converges with not so small values of $\eps$, which allowed us to obtain rather large invisible defects. When the number of constraints increases, for example when working at higher $k$ so that there are more scattering coefficients to control, in practice we find that $\eps$ must be chosen smaller to have convergence of the method. A natural idea then is to try to reiterate the procedure to obtain larger invisible defects: once an invisible obstacle has been constructed, we can consider it as a new starting configuration and perturb it while keeping the same scattering coefficients. This is usually called a continuation method which allows one to explore the variety (of infinite dimension) of invisible obstacles. We will not implement it here (see \cite{BeBC21} for more details). Instead, we simply illustrate the method in Figure \ref{FigContMeth} in finite dimension. More precisely, instead of working with $\mathscr{F}:\mathscr{C}^{\infty}_0(\R)\to\R^3$ as in the case of perfect invisibility (see (\ref{DefFT1D})), we consider some smooth mapping
\[
\begin{array}{rcll}
\mathscr{F}:&\R^2&\to&\R \\[2pt]
 &  (h_1,h_2) &\mapsto& \mathscr{F}(h_1,h_2)
\end{array}
\]
such that $\mathscr{F}(0,0)=0$ and $d\mathscr{F}(0,0)\ne(0,0)\Leftrightarrow\nabla\mathscr{F}(0,0)\ne(0,0)$. A simple example is 
\[
\mathscr{F}(h_1,h_2)=(h_1-1)^2+(h_2-1)^2-2.
\]

\begin{figure}[!ht]
\centering
\begin{tikzpicture}[scale=1.9]
\node at (-2,0){ \phantom{$h_1$}};
\draw[->] (-1,0)--(3,0);
\draw[->] (0,-1)--(0,2.5);
\draw [draw=blue,line width=0.3mm] plot [smooth cycle, tension=1] coordinates {(0,0)(1,-0.8) (3,0.5) (2,2.2) (0.5,1.8) (-0.1,1)};
\draw[dashed] (0,0)--(-0.7,1);
\node at (-0.2,-0.2){\small $O$};
\node at (3.2,0){ $h_1$};
\node at (0.2,2.35){ $h_2$};
\node at (4,2.2){ \textcolor{blue}{\begin{tabular}{c}
Variety of invisible $h$\\[2pt]
$\{(h_1,h_2)\in\R^2\,|\,\mathscr{F}(h_1,h_2)=0\}$
\end{tabular}}};
\draw[fill=red,draw=none](-0.22,0.55) circle (0.05);
\draw[fill=gray!95,draw=none](-0.33,0.47) circle (0.05);
\node at (-0.65,0.5){ \textcolor{gray!95}{$\eps\,\mu^0_0$}};
\node at (1,0.55){ $h^{\mrm{sol},1}=\eps\,(\mu^0_0+\tau^{\mrm{sol},0}\mu^0_1)$};
\draw[fill=red,draw=none](0,1.4) circle (0.05);
\node at (1.5,1.4){ $h^{\mrm{sol},2}=h^{\mrm{sol},1}+\eps\,(\mu^1_0+\tau^{\mrm{sol},1}\mu^1_1)$};
\draw[fill=red,draw=none](0.8,2) circle (0.05);
\node at (0.85,1.8){ $h^{\mrm{sol},3}$};
\draw[fill=red,draw=none](0,0) circle (0.05);
\node at (0.4,0){ $h^{\mrm{sol},0}$};
\end{tikzpicture}
\caption{Illustration of the continuation method. For $n\in\N$, $\mu^{n}_0,\mu^{n}_1\in\R^2$ are such that $d\mathscr{F}(h^{\mrm{sol},n})(\mu^{n}_0)=0$ and $d\mathscr{F}(h^{\mrm{sol},n})(\mu^{n}_1)=1$ (we can take $\mu^{n}_1=\nabla\mathscr{F}(h^{\mrm{sol},n})/|\nabla\mathscr{F}(h^{\mrm{sol},n})|^2$). \label{FigContMeth}} 
\end{figure}

\chapter{Playing with resonances to reach invisibility}\label{SectionResonance}
\setcounter {section} {0}

\noindent In the previous chapter, we showed how to construct small non reflecting or invisible defects of the reference strip by using variants of the implicit functions theorem. The goal of the present chapter is to create larger invisible obstacles. To proceed, we have to act strongly on scattering coefficients. We will do that by exploiting resonant phenomena. More precisely, in the first section we explain how to work with the Fano resonance phenomenon together with symmetry considerations to construct large non reflecting obstacles in monomode regime\footnote{We call monomode regime the regime of wavenumbers such that only one mode can propagate.}. Then, we modify a bit the point of view and for a given waveguide in acoustics, we show how to perturb its boundary with resonant chimneys to get approximately $T=1$ in the new geometry.

\section{Playing with the Fano resonance}\label{SectionFano}

\begin{figure}[!ht]
\centering
\begin{tikzpicture}[scale=2.4]
\draw[fill=gray!30,draw=none](-1,0) rectangle (1,1);
\begin{scope}[scale=0.5]
\draw [fill=white,draw=black] plot [smooth cycle, tension=1] coordinates {(-0.6,0.9) (0,0.5) (0.7,1) (0.5,1.5) (-0.2,1.4)};
\end{scope}
\draw (-1,1)--(1,1);
\draw (-1,0)--(1,0);
\draw[dashed] (1,1)--(1.2,1);
\draw[dashed] (-1,1)--(-1.2,1);  
\draw[dashed] (1,0)--(1.2,0);
\draw[dashed] (-1,0)--(-1.2,0);  
\node at (0.8,0.15){\small $\Om^0$};
\end{tikzpicture}\qquad\quad\begin{tikzpicture}[scale=2.4]
\draw[fill=gray!30,draw=none](-1,0) rectangle (1,1);
\draw[samples=30,domain=-0.25:0.25,draw=black,fill=gray!30] plot(\x-0.5,{1+0.05*(4*\x+1)^4*(4*\x-1)^4*(4*\x+3)});
\begin{scope}[scale=0.5]
\draw [fill=white,draw=black] plot [smooth cycle, tension=1] coordinates {(-0.6,0.9) (0,0.5) (0.7,1) (0.5,1.5) (-0.2,1.4)};
\end{scope}
\draw (-1,1)--(-0.75,1);
\draw (1,1)--(-0.25,1); 
\draw (-1,0)--(1,0);
\draw[dashed] (1,1)--(1.2,1);
\draw[dashed] (-1,1)--(-1.2,1);  
\draw[dashed] (1,0)--(1.2,0);
\draw[dashed] (-1,0)--(-1.2,0);  
\node at (0.8,0.15){\small $\Om^{\eps}$};
\draw[samples=30,domain=-0.25:0.25,draw=black,fill=gray!30] plot(\x-0.5,{1+0.05*(4*\x+1)^4*(4*\x-1)^4*(4*\x+3)});
\node[above] at (0.1,1.15){\footnotesize  $\partial \Om^{\eps}=(x,1+\eps H(x))$};
\end{tikzpicture}
\caption{Original waveguide $\Om^0$ (left) and perturbed geometry $\Om^{\eps}$ (right). \label{GeometriesFano}} 
\end{figure}

The Fano resonance, named after the physicist Ugo Fano (1912-2001), is a classical phenomenon that arises in many situations in physics. For our particular concern, it appears as follows. Assume that the geometry of our waveguide is characterized by a real parameter $\eps$. Below, $\eps$ will be the amplitude of a local perturbation of the walls (see Figure \ref{GeometriesFano} right). To simplify notation, set $\lambda\coloneqq k^2$ so that the acoustic problem (\ref{WaveguidePbNeumann}) writes
\begin{equation}\label{WaveguidePbFano} 
\begin{array}{|rcll}
\Delta u+\lambda u&=&0&\mbox{ in }\Om^\eps \\[2pt]
\partial_\nu u&=&0&\mbox{ on }\partial\Om^\eps.
\end{array}
\end{equation}
Assume that trapped modes exist for Problem (\ref{WaveguidePbFano}) with $\eps=0$ and $\lambda=\lambda^0\in(0;\pi^2)$. We remind the reader that trapped modes are non zero solutions of the homogeneous problem which belong to $\mH^1$. We will see that for $\eps\ne0$ small, the scattering matrix $\mathbb{S}$, which is of size $2\times2$ in monomode regime according to (\ref{DefScaMatN}), exhibits a rapid change for real $\lambda$ varying in a neighborhood of $\lambda^0$. Then our goal will be to exploit this Fano resonance phenomenon together with symmetry considerations of the geometry to provide examples of waveguides where $R=0$ or $T=0$. Note that the case $T=0$, that we will call zero transmission, is not related to invisibility but can be interesting for other applications. It corresponds to a situation where the energy of an incident wave is completely backscattered, like for a mirror.  \\
\newline
To understand more this Fano resonance phenomenon, let us work on a 1D toy problem.

\subsection{A $\mrm{1D}$ toy problem}\label{SectionToyPb}

\subsubsection{Unperturbed case}

\begin{figure}[!ht]
\centering
\begin{tikzpicture}[scale=1.5]
\begin{scope}[shift={(-1,0.1)}]
\draw[->] (3,0.2)--(3.6,0.2);
\draw[->] (3.1,0.1)--(3.1,0.7);
\node at (3.65,0.3){\small $x$};
\node at (3.25,0.6){\small $y$};
\end{scope}
\draw[dashed] (-4.5,0)--(-4,0);
\draw (-4,0)--(1,0);
\draw (0,0)--(0,1);
\node at (-2.5,-0.3){\small $\Om_1$};
\node at (0.25,0.5){\small $\Om_2$};
\node at (0.5,-0.3){\small $\Om_3$};
\node at (0,-0.25){\small $O$};
\draw[fill=gray!80,draw=none](0,0) circle (0.05);
\end{tikzpicture}
\caption{A $\mrm{1D}$ geometry. \label{Domain1D}} 
\end{figure}

\noindent Consider the geometry 
\[
\Om\coloneqq\Om_1\cup\Om_2\cup\Om_3\qquad\mbox{ with }\Om_1\coloneqq(-\infty;0)\times\{0\},\ \Om_2\coloneqq\{0\}\times(0;1),\,\Om_3\coloneqq(0;1)\times\{0\}
\]
(see Figure \ref{Domain1D}). For a function $\varphi$ defined in $\Om$, set $\varphi_j\coloneqq\varphi|_{\Om_j}$, $j=1,2,3$. Working in suitable coordinates, we can see the $\Om_j$ as $\mrm{1D}$ domains. Similarly to (\ref{WaveguidePbFano}), we study the problem 
\begin{equation}\label{Problem1D1}
-\varphi_j''=k^2\varphi_j\quad\mbox{ in }\Om_j,\qquad j=1,2,3,
\end{equation}
with the conditions 
\begin{equation}\label{Problem1D2}
\begin{array}{|l}
\varphi_1(0)=\varphi_2(0)=\varphi_3(0),\\[2pt]
\varphi'_1(0)=\varphi'_2(0)+\varphi'_3(0),\\[2pt]
\varphi'_2(1)=\varphi'_3(1)=0.
\end{array}
\end{equation}
Thus at the junction point $O$, we impose continuity of the field and conservation of the flux (Kirchhoff law) and we work with Neumann boundary conditions at the ends of the ligaments. It can be shown that this is a good model to describe the properties of the acoustic problem in the 2D domain obtained by thickening the above 1D graph. We are interested in the scattering of the rightgoing incident wave $x\mapsto  e^{ikx}$. We denote by $\varphi$ the corresponding total field. For the simple problem considered here, the radiation condition boils down to assume that $\varphi$ writes as $\varphi(x)=e^{ikx}+R\,e^{-ikx}$ in $\Om_1$, where $R\in\Cplx$ is the reflection coefficient. Using the two boundary conditions of (\ref{Problem1D1})--(\ref{Problem1D2}), we are led to look for a solution $\varphi$ such that
\[
\varphi_1(x)=e^{ikx}+R\,e^{-ikx},\qquad \varphi_2(y)=a\cos(k(y-1)),\qquad \varphi_3(x)=b\cos(k(x-1)),
\]
where $a$, $b\in\Cplx$. Writing the transmission conditions at the junction point $O$, we obtain that $R$, $a$, $b$ must solve the system
\begin{equation}\label{system33}
\mathbb{M}(k)\Phi=F\quad\mbox{ with }\mathbb{M}(k)\coloneqq\left(\begin{array}{ccc}
1 & -\cos k & 0\\
0 & \cos k & -\cos k\\
i & \sin k & \sin k
\end{array}\right),\  \Phi\coloneqq\left(\begin{array}{c}
R \\
a \\
b
\end{array}\right),\  F\coloneqq\left(\begin{array}{c}
-1 \\
0 \\
i
\end{array}\right).
\end{equation}
One finds 
\[
\mrm{det}\,\mathbb{M}(k)=(2\sin k+i\cos k)\cos k.
\]
Therefore this system (and so Problem (\ref{Problem1D1})--(\ref{Problem1D2}) with the above mentioned radiation condition) is uniquely solvable if and only $k\not\in (2\N+1)\pi/2$. When $k\in (2\N+1)\pi/2$, the kernel of Problem (\ref{Problem1D1})--(\ref{Problem1D2}) coincides with $\mrm{span}(\varphi_{\mrm{tr}})$ where $\varphi_{\mrm{tr}}$ is the trapped mode, defined in $\Om$, such that 
\[
\varphi_{\mrm{tr}}(x)=\begin{array}{|cl}
0 & \mbox{ in }\Om_1\\[2pt]
\phantom{-}\sin(ky) & \mbox{ in }\Om_2\\[2pt]
-\sin(kx) & \mbox{ in }\Om_3.
\end{array}
\]
On the other hand, for any $k>0$, one can check that System (\ref{system33}) (and so Problem (\ref{Problem1D1})--(\ref{Problem1D2})) admits a solution because $F\in(\ker\,{}^t\mathbb{M})^{\perp}$ (this can be verified by an explicit calculus). Moreover, as in (\ref{RepresentationupD}) the coefficient $R$ is always uniquely defined (even when $k\in (2\N+1)\pi/2$) and such that
\begin{equation}\label{DefR1D}
R=\cfrac{\cos k+2i\sin k}{\cos k-2i\sin k}\,.
\end{equation}
The map $k\mapsto R(k)$ is $\pi$-periodic and $|R(k)|=1$. The latter relation, which is due to conservation of energy, guarantees that $R(k)=e^{i\theta(k)}$ for some phase $\theta(k)\in\R/(2\pi\Z)$. In Figure \ref{FigPhaseRef}, we represent the map $k\mapsto\theta(k)$ for $k\in(0;\pi)$. We observe that it has a smooth behavior, also around the value $k=\pi/2$ for which trapped modes exist for Problem (\ref{Problem1D1})--(\ref{Problem1D2}). Thus we see that the scattering is insensitive to the existence of trapped modes. 

\begin{figure}[!ht]
\centering
\includegraphics[width=0.47\textwidth]{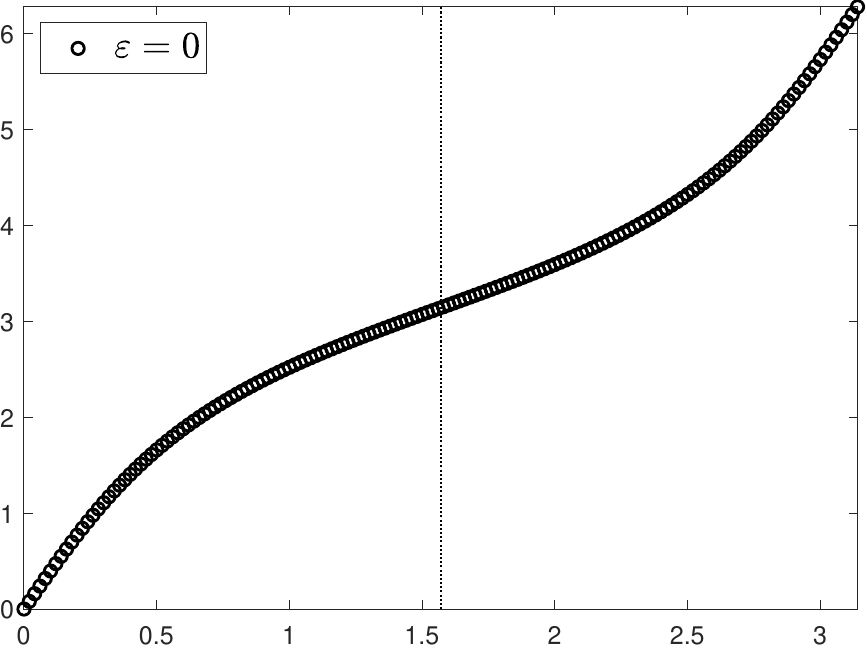}
\caption{Map $k\mapsto\theta(k)$. The vertical black dotted line indicates the value $k=\pi/2$. \label{FigPhaseRef}}
\end{figure}

\noindent It is also interesting to give a spectral description of the properties of the problem. Denote by $A$ the unbounded operator of $\mL^2(\Om)$ such that 
\begin{equation}\label{DefA1}
A\varphi=-\varphi''_j\quad\mbox{ in }\Om_j,\qquad j=1,2,3,
\end{equation}
with domain 
\begin{equation}\label{DefA2}
D(A)=\{\varphi\in\mL^2(\Om)\,|\,\varphi_j\in\mH^2(\Om_j)\mbox{ for }j=1,2,3\mbox{ and }\varphi\mbox{ satisfies }\mrm{(\ref{Problem1D2})}\}.
\end{equation}
Classically, see \textit{e.g.} \cite{BiSo87}, one shows that $A$ is a selfadjoint operator. Therefore its spectrum $\sigma(A)$ is real. One can prove that $\sigma(A)$ coincides with $[0;+\infty)$. More precisely, we have $\sigma_{\mrm{ess}}(A)=[0;+\infty)$ where $\sigma_{\mrm{ess}}(A)$ denotes the essential spectrum of $A$. By definition, $\sigma_{\mrm{ess}}(A)$ corresponds to the set of $\lambda \in\Cplx$ for which there exists a so-called singular sequence, that is a sequence $(u^{(m)})$ of functions of $D(A)$ such that $\|u_m\|_{\mL^2(\Om)}=1$, $(u^{(m)})$ converges weakly to 0 in $\mL^2(\Omega)$ and $((A-\lambda)u^{(m)})$ converges strongly to 0 in $\mL^2(\Omega)$. The fact that $\sigma_{\mrm{ess}}(A)=[0;+\infty)$ is directly related to the existence of the propagating modes $e^{\pm ikx}$ for all $k>0$. On the other hand, the above calculations ensure that $A$ also has point spectrum corresponding to eigenvalues, \textit{i.e.} values of $\lambda$ such that $\ker\,(A-\lambda\mrm{Id})\ne\{0\}$. More precisely, one has $\sigma_{\mrm{p}}(A)= \{(2n+1)^2\pi^2/4,\, n\in\N\}$. Let us emphasize that these eigenvalues are embedded in the essential spectrum (see the illustration of Figure \ref{SpectrumA}).

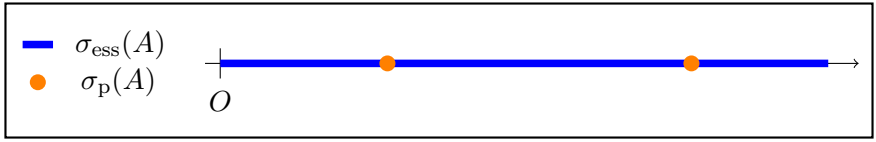
\begin{figure}[!ht]
\centering
\tikz \node[rectangle,inner sep=5pt,fill=none,draw=black,thick]{
\begin{tikzpicture}[scale=2]
\draw[->] (-1.1,0)--(3.2,0);
\draw[-] (-1,-0.1)--(-1,0.1);
\node at (-1,-0.25){ $O$};
\draw[-,line width=1mm,blue] (-1,0)--(3,0);
\fill[orange] (2.1,0) circle (1.5pt) ;
\fill[orange] (0.1,0) circle (1.5pt) ;
\draw[-,line width=1mm,blue] (-2.3,0.125)--(-2.1,0.125)node [anchor=west,black]{\,$\sigma_{\mrm{ess}}(A)$};
\fill[orange] (-2.2,-0.125) circle (1.5pt) node [anchor=west,black]{\quad $\sigma_{\mrm{p}}(A)$};
\end{tikzpicture}\qquad
};
\caption{Spectrum of $A$ in the complex plane. \label{SpectrumA}} 
\end{figure}

\subsubsection{Perturbed case}
Now we consider the same problem in the perturbed geometry $\Om^{\eps}\coloneqq\Om_1\cup\Om_2\cup\Om^{\eps}_3$ with $\Om^{\eps}_3\coloneqq(0;1+\eps)\times\{0\}$ and $\eps\in\R$ small. We denote with a superscript $\eps$ all the above quantities. In $\Om^{\eps}$, the resolution of the previous scattering problem leads to solve the system
\begin{equation}\label{BreakingSYm}
\mathbb{M}^{\eps}(k)\Phi^{\eps}=F\qquad\mbox{ with }\mathbb{M}^{\eps}(k)\coloneqq\left(\begin{array}{ccc}
1 & -\cos k & 0\\
0 & \cos k & -\cos(k(1+\eps))\\
i & \sin k & \sin(k(1+\eps))
\end{array}\right),\  \Phi^{\eps}\coloneqq\left(\begin{array}{c}
R^{\eps} \\
a^{\eps} \\
b^{\eps}
\end{array}\right).
\end{equation}
The vector $F$ is the same as in (\ref{system33}). We find
\[
\mrm{det}\,\mathbb{M}^{\eps}(k)=\sin(k(2+\eps))+i\cos k\cos(k(1+\eps)).
\]
Therefore we find that for $\eps\ne0$ small, the determinant of $\mathbb{M}^{\eps}$ does not vanish for a given range of $k>0$. As a consequence, Problem (\ref{Problem1D1})--(\ref{Problem1D2}) set in $\Om^{\eps}$ has a unique solution. One obtains
\[
R^{\eps}=\cfrac{\cos k\cos(k(1+\eps))+i\sin(k(2+\eps))}{\cos k\cos(k(1+\eps))-i\sin(k(2+\eps))}\,.
\]
Again, we have $|R^{\eps}(k)|=1$ (conservation of energy) so we can write $R^{\eps}(k)=e^{i\theta^{\eps}(k)}$ for some $\theta^{\eps}(k)\in\R/(2\pi\Z)$.  Note that  $\theta^{0}=\theta$ where $\theta$ appears after (\ref{DefR1D}). The map $k\mapsto\theta^{\eps}(k)$ is displayed in Figure \ref{FigPhase} for several values of $\eps$ (see also the alternative representation of Figure \ref{FigFanoSpace1D}). We observe that for $\eps\ne0$, the curve $k\mapsto\theta^{\eps}(k)$ has a fast variation for $k$ close to $\pi/2$. The variation is even faster as $\eps\ne0$ gets small. On the other hand, for $\eps=0$, as represented in Figure \ref{FigPhaseRef}, the curve $k\mapsto\theta^{0}(k)$ has a very smooth behavior. We emphasize that for $(\eps,k)=(0,\pi/2)$, as mentioned above, trapped modes exist for Problem (\ref{Problem1D1})--(\ref{Problem1D2}).		
\begin{figure}[!ht]
\centering
\includegraphics[width=0.47\textwidth]{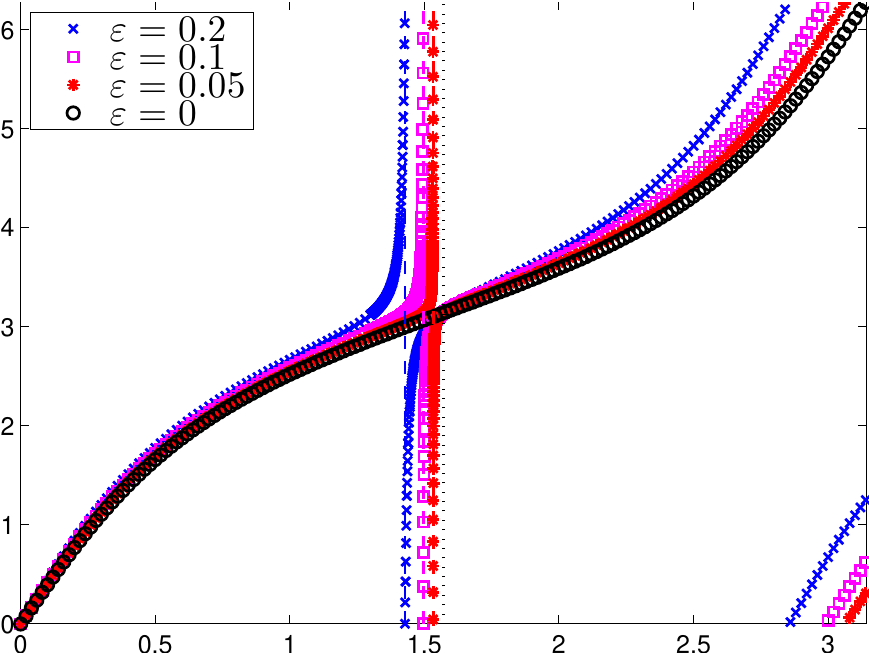}\quad\includegraphics[width=0.47\textwidth]{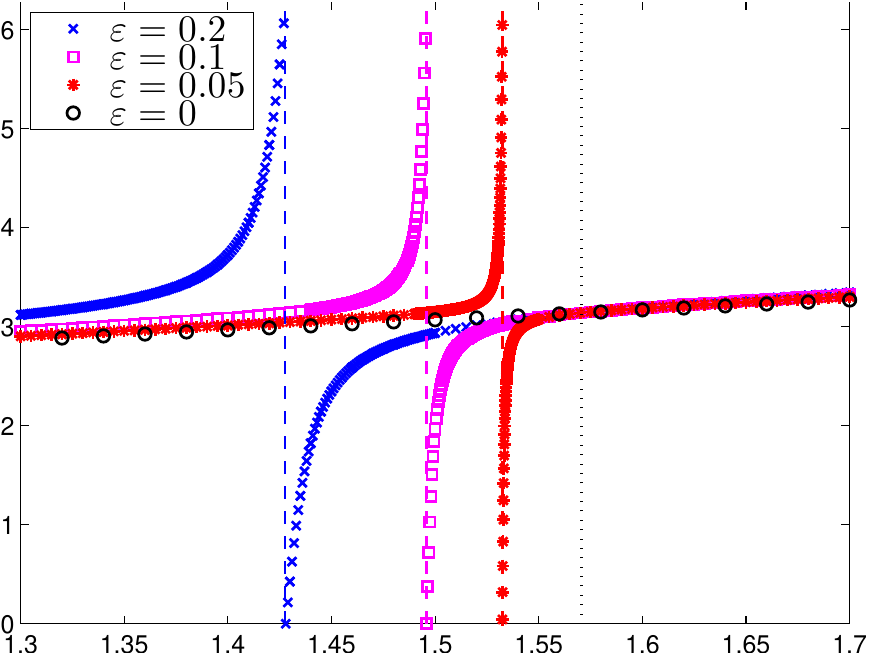}
\caption{Maps $k\mapsto\theta^{\eps}(k)$ for several values of $\eps$. The right picture is a zoom on the left picture around $k=\pi/2$ (marked by the vertical black dotted line). The vertical colored dashed lines indicate the values of $k$ such that $\theta^{\eps}(k)=0$. \label{FigPhase}}
\end{figure}

\begin{figure}[!ht]
\centering
\includegraphics[width=0.5\textwidth]{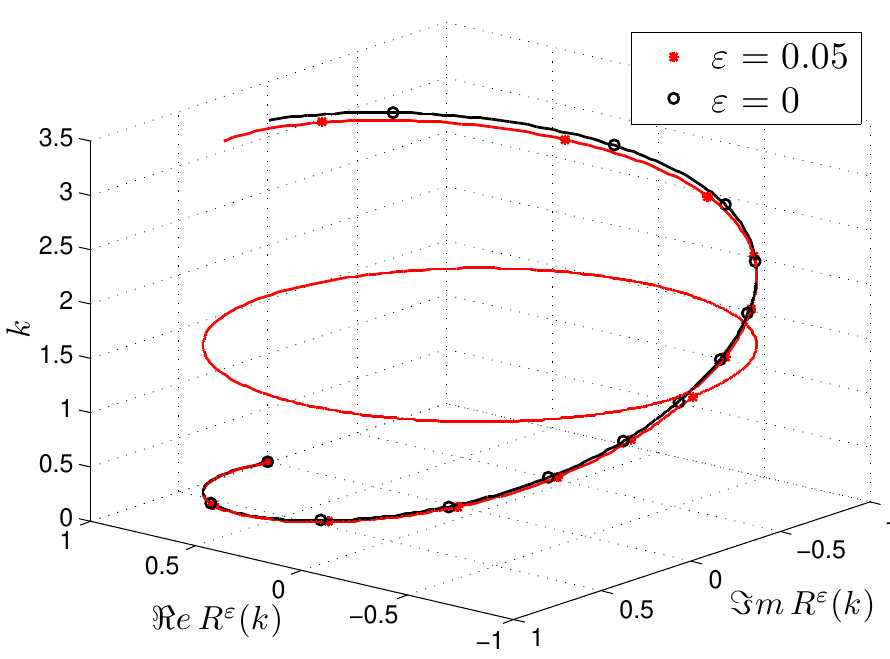}
\caption{Parametric curves $k\mapsto(\Re e\,R^{\eps}(k),\Im m\,R^{\eps}(k))$ for $k\in(0;\pi)$. \label{FigFanoSpace1D}}
\end{figure}
\noindent In order to study the variations of the reflection coefficient with respect to the frequency and the geometry, we define the map $\mathcal{R}:\R^2\to\Cplx$ such that 
\begin{equation}\label{R2variables}
\mathcal{R}(\eps,k)=\cfrac{\cos(k)\cos(k(1+\eps))+i\sin(k(2+\eps))}{\cos(k)\cos(k(1+\eps))-i\sin(k(2+\eps))}\,.
\end{equation}
With such a notation, we have $R^{\eps}(k)=\mathcal{R}(\eps,k)$ and $R(k)=\mathcal{R}(0,k)$. For all $k\in(0;\pi)$, there holds $\lim_{\eps\to0}\mathcal{R}(\eps,k)=\mathcal{R}(0,k)$. Now assume that the frequency and the geometry are related by some prescribed law in a neighborhood of the point $(\eps,k)=(0,\pi/2)$ corresponding to a setting supporting trapped modes. For example, assume that 
\[
k=\pi/2+\eps k'
\]
for a given $k'\in\R$. Then for $k'\ne-\pi/4$, starting from expression (\ref{R2variables}), we find as $\eps\to0$ the expansion
\begin{equation}\label{Fano01}
\mathcal{R}(\eps,\pi/2+\eps k')=-1+\eps\,\Big(\,\cfrac{-2ik'(\pi+2k')}{\pi+4k'}\,\Big)+O(\eps^2).\hspace{4.2cm}
\end{equation}
Note that we have $\mathcal{R}(0,\pi/2)=-1$.\\
\newline
For $k'=-\pi/4$, the asymptotic behavior as $\eps$ tends to zero is more surprising. Indeed, for
\[
k=\pi/2-\eps\pi/4+\eps^2\mu
\]
with $\mu\in\R$ (we take $k'=-\pi/4$ but allow for some freedom at higher order), we obtain
\begin{equation}\label{Resultat1DMobius}
\mathcal{R}(\eps,\pi/2-\eps\pi/4+\eps^2\mu)=g(\mu)+O(\eps)\qquad\mbox{ with }\quad g(\mu)=\cfrac{\pi^2+i(32\mu-4\pi)}{\pi^2-i(32\mu-4\pi)}\,.
\end{equation}
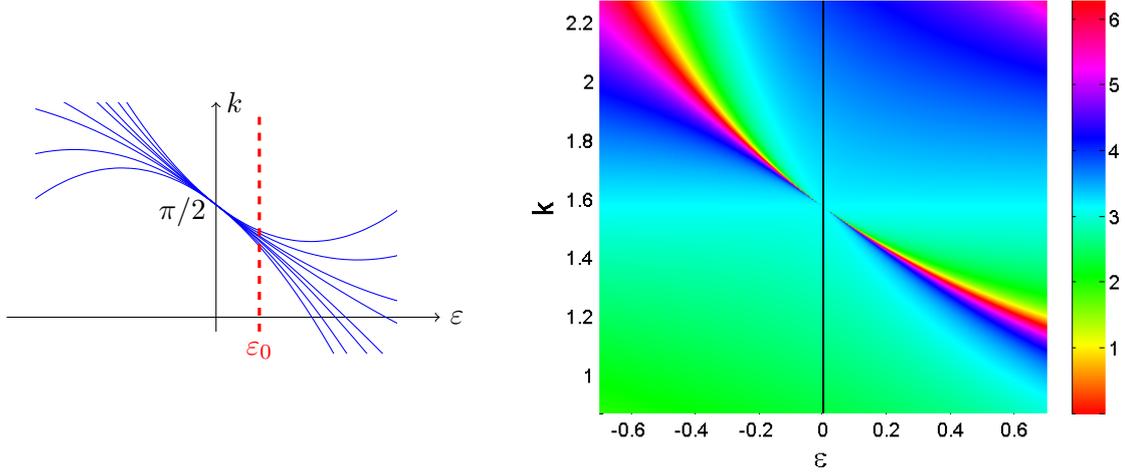
\begin{figure}[!ht]
\centering
\raisebox{1.3cm}{\begin{tikzpicture}[scale=0.95]
\draw[->] (-2.9,0) -- (3.1,0) node[right] {$\eps$};
\draw[->] (0,-0.2) -- (0,3) node[right] {$k$};
\node at (0,3.14/2-0.1) [left] {$\pi/2$};
\begin{scope}
\clip(-2.5,-0.5) rectangle (2.5,3);
\draw[domain=-2.7:2.7,smooth,variable=\x,blue] plot ({\x},{3.14/2-\x*3.14/4+0.2*\x*\x});
\draw[domain=-2.7:2.7,smooth,variable=\x,blue] plot ({\x},{3.14/2-\x*3.14/4+0.3*\x*\x});
\draw[domain=-2.7:2.7,smooth,variable=\x,blue] plot ({\x},{3.14/2-\x*3.14/4+0.1*\x*\x});
\draw[domain=-2.7:2.7,smooth,variable=\x,blue] plot ({\x},{3.14/2-\x*3.14/4+0.05*\x*\x});
\draw[domain=-2.7:2.7,smooth,variable=\x,blue] plot ({\x},{3.14/2-\x*3.14/4-0.05*\x*\x});
\draw[domain=-2.7:2.7,smooth,variable=\x,blue] plot ({\x},{3.14/2-\x*3.14/4-0.2*\x*\x});
\draw[domain=-2.7:2.7,smooth,variable=\x,blue] plot ({\x},{3.14/2-\x*3.14/4-0.3*\x*\x});
\draw[domain=-2.7:2.7,smooth,variable=\x,blue] plot ({\x},{3.14/2-\x*3.14/4-0.1*\x*\x});
\end{scope}
\draw[red,dashed,very thick] (0.6,-0.2) -- (0.6,2.8);
\node at (0.6,-0.2) [below] {\textcolor{red}{$\eps_0$}};
\end{tikzpicture}}\qquad\includegraphics[width=0.47\textwidth]{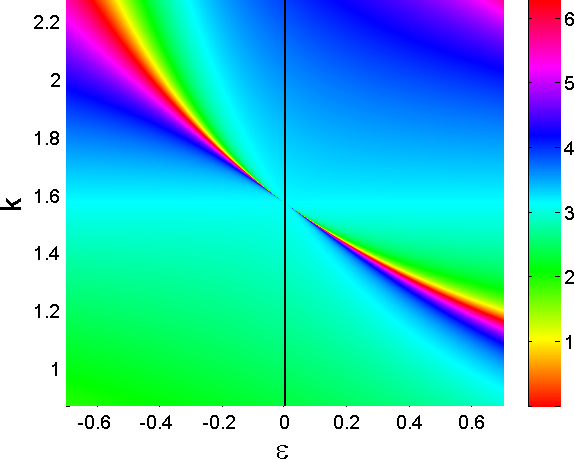}
\caption{Left: parabolic paths $\{(\eps,\pi/2-\eps\pi/4+\eps^2\mu),\,\eps\in(-1;1)\}\subset\R^2$ for several values of $\mu$. According to $\mu$, the limit of the coefficient $\mathcal{R}(\eps,\pi/2-\eps\pi/4+\eps^2\mu)$ defined in (\ref{R2variables}) as $\eps\to0$ is different. Right: the colors indicate the phase of $\mathcal{R}(\eps,k)$. This phase is valued in $[0;2\pi)$.\label{FigFano1DParabola}}
\end{figure}

\noindent Classical results concerning the M\"{o}bius transform (see \textit{e.g.}  \cite[Chap.\,5]{Henr74}) guarantee that $g$ is a bijection between $\R$ and $\mathscr{C}(0,1)\setminus\{-1+0i\}$ ($\mathscr{C}(0,1)$ is the unit circle of the complex plane). Thus for any $z_0\in\mathscr{C}(0,1)$, we can find $\mu\in\R$ such that 
\[
\lim_{\eps\to0}\mathcal{R}(\eps,\pi/2-\eps\pi/4+\eps^2\mu)=z_0
\]
(see Figure \ref{FigFano1DParabola} right). In other words, by tuning cleverly the frequency and the amplitude of the perturbation, we can get any desired value for the reflection coefficient on the unit circle. 
Besides, this proves that the map $\mathcal{R}(\cdot,\cdot):\R^2\to\Cplx$ is not continuous at $(0,\pi/2)$. Finally, this shows also that for $\eps_0\ne0$ small fixed (see the vertical red dashed line in Figure \ref{FigFano1DParabola} left), the curve $k\mapsto \mathcal{R}(\eps_0,k)$ must exhibit a rapid change. Indeed, $\mu$ varying in $[-C\eps_0^{-1};C\eps_0^{-1}]$ for some arbitrary $C>0$ (which is only a small change for $k$) leads to a large change for $\mathcal{R}(\eps_0,\pi/2-\eps_0\pi/4+\eps_0^2\mu)$. This is exactly what we observed in Figure \ref{FigPhase}.\\
\newline
Let us come back to the spectral description of the properties  of the problem. Denote by $A^\eps$ the analog of the operator $A$ introduced in (\ref{DefA1})--(\ref{DefA2}) in the perturbed geometry $\Om^\eps$. Since we made a perturbation in a bounded region, it can be shown that $A^\eps$ has the same essential spectrum as $A$, \textit{i.e.} $\sigma_{\mrm{ess}}(A^\eps)=[0;+\infty)$ (again this is directly due to the existence of the propagating modes $e^{\pm ikx}$ for all $k>0$). On the other hand, the above computations guarantee that for $\eps>0$ small, $A^\eps$ has no eigenvalue, \textit{i.e.} $\sigma_{\mrm{p}}(A^\eps)=\emptyset$. This gives the picture of Figure \ref{SpectrumAeps} for $\sigma(A^\eps)$. A natural question then is: what happened to the eigenvalues of $A$ when the geometry has been perturbed? The answer is that they became so-called complex resonances. These are values of $k^2\in\Cplx$ with $\Im m\, k<0$ such that there is a solution to 
\begin{equation}\label{PbCplxRes}
-\varphi_j''=k^2\varphi_j\quad\mbox{ in }\Om^\eps_j,\qquad j=1,2,3,
\end{equation}
with the conditions (\ref{Problem1D2}), admitting the expansion $\varphi_1(x)=e^{-ikx}$ in $\Om_1$ (generalized outgoing behavior). Note that such generalized eigenfunctions are exponentially growing as $x\to-\infty$ and so do not belong to $\mL^2(\Om^\eps)$. In the literature, they are often refereed to as quasi-normal modes or leaky modes. Observe that they cannot exist for $\Im m\,k>0$ because they would be exponentially decaying as $x\to-\infty$ and so would belong to $\mL^2(\Om^\eps)$. But this is impossible because the operator is selfadjoint and therefore cannot have eigenvalues in $\Cplx\setminus\R$.

\begin{figure}[!ht]
\centering
\tikz \node[rectangle,inner sep=5pt,fill=none,draw=black,thick]{
\begin{tikzpicture}[scale=2]
\draw[->] (-1.1,0)--(3.2,0);
\draw[-] (-1,-0.1)--(-1,0.1);
\node at (-1,-0.25){ $O$};
\draw[-,line width=1mm,blue] (-1,0)--(3,0);
\fill[red] (2.1,-0.2) circle (1.5pt) ;
\fill[red] (0.1,-0.2) circle (1.5pt) ;
\draw[-,line width=1mm,blue] (-3.3,0.125)--+(0.2,0)node [anchor=west,black]{\,$\sigma_{\mrm{ess}}(A^\eps)$};
\fill[red] (-3.2,-0.125) circle (1.5pt) node [anchor=west,black]{\hspace{0.25cm}\mbox{Complex resonances}};
\end{tikzpicture}\qquad
};
\caption{Spectrum of $A^\eps$ in the complex plane. \label{SpectrumAeps}} 
\end{figure}

\noindent Let us stress that this conversion of eigenvalues into  complex resonances when one perturbs the geometry is crucially related to the fact that we are considering eigenvalues which are embedded in the essential spectrum. In particular, it does not occur in bounded domains when essential spectrum do not exist. Let us illustrate this with a concrete example. Denote by $B$ the Neumann Laplacian on the interval $(0;1)$ such that
\[
\begin{array}{|rcl}
D(B)&=&\{\varphi\in\mH^2(0;1)\,|\,\partial_x\varphi(0)=\partial_x\varphi(1)=0 \}\\[2pt]
B\varphi&=&-\varphi''.
\end{array}
\]
The operator $B$ has only discrete spectrum (no essential spectrum), namely 
\[
\sigma(B)= \{n^2\pi^2,\,n\in\N\}
\]
and the corresponding eigenfunctions are the $\varphi_n$ such that $\varphi_n(x)=\cos(n\pi x)$. Now if we denote by $B^\eps$ the same Neumann Laplacian in the perturbed geometry $(0;1+\eps)$, we find 
\[
\sigma(B^\eps)= \{n^2\pi^2/(1+\eps)^2,\,n\in\N\}
\]
and the corresponding eigenfunctions are the $\varphi^\eps_n$ such that $\varphi^\eps_n(x)=\cos(n\pi x)$. In other words, the eigenvalues of $B$ have been only slightly shifted on the real axis (see Figure \ref{SpectrumNeumannLap}), they did not turn into complex resonances.

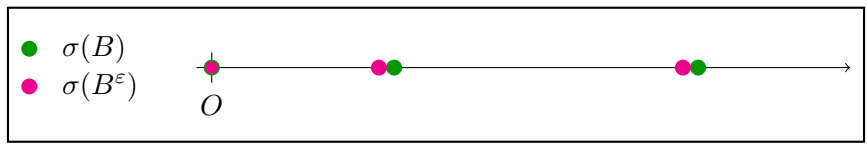
\begin{figure}[!ht]
\centering
\tikz \node[rectangle,inner sep=5pt,fill=none,draw=black,thick]{
\begin{tikzpicture}[scale=2]
\draw[->] (-1.1,0)--(3.2,0);
\draw[-] (-1,-0.1)--(-1,0.1);
\node at (-1,-0.25){ $O$};
\fill[green!60!black] (2.2,0) circle (1.5pt) ;
\fill[green!60!black] (0.2,0) circle (1.5pt) ;
\fill[green!60!black] (-1,0) circle (1.5pt) ;
\fill[magenta] (2.1,0) circle (1.5pt) ;
\fill[magenta] (0.1,0) circle (1.5pt) ;
\fill[magenta] (-1,0) circle (1.2pt) ;
\fill[green!60!black] (-2.2,0.125) circle (1.5pt) node [anchor=west,black]{\hspace{0.25cm}$\sigma(B)$};
\fill[magenta] (-2.2,-0.125) circle (1.5pt) node [anchor=west,black]{\hspace{0.25cm}$\sigma(B^\eps)$};
\end{tikzpicture}\qquad
};
\caption{Spectra of $B$ and $B^\eps$ in the complex plane. \label{SpectrumNeumannLap}} 
\end{figure}

\noindent On the other hand, not all geometric perturbations of Problem (\ref{DefA1})--(\ref{DefA2}) convert embedded eigenvalues into complex resonances. For example, if both $\Om_2$, $\Om_3$ are perturbed into $\Om^\eps_2=\{0\}\times(0;1+\eps)$, $\Om^\eps_3=(0;1+\eps)\times\{0\}$,  one can check that embedded eigenvalues remain embedded eigenvalues for the perturbed problem. Somehow, we need to break the symmetry as in (\ref{BreakingSYm}), to create some coupling between the trapped modes and the scattering properties of the problem. In that case, complex resonances located closed to the real axis have a strong impact on the scattering.

\subsection{Fano resonance in the 2D waveguide}\label{ParaFano2D}

Let us return to the problem (\ref{WaveguidePbFano}) in 2D. We assume that $\Om^0$ is such that the Neumann Laplacian has a simple eigenvalue $\lambda_0\in(0;\pi^2)$ (geometric multiplicity equal to one). We perturb the geometry from some smooth compactly supported profile function $H$ with amplitude $\eps\ge0$ as in Figure \ref{GeometriesFano} right. We denote by $\Om^{\eps}$ the new waveguide and $\mathbb{S}(\eps,\lambda)$, $T(\eps,\lambda)$, $R_\pm(\eps,\lambda)$ the scattering matrix/coefficients in the geometry $\Om^{\eps}$ at frequency $\lambda$. For short, we set $\mathbb{S}^0=\mathbb{S}(0,\lambda^0)$, $T^0=T(0,\lambda^0)$, $R^0_\pm=R_\pm(0,\lambda^0)$. Introduce $u_0$ an eigenfunction associated with $\lambda^0$ such that $\|u_0\|_{\mL^2(\Om)}=1$ as well as $d>0$ such that $\Om^0=\{(x,y)\in\R\times(0;1)\}$ for $\pm x\ge d$ (the geometric perturbation with respect to the straight waveguide is located in the zone $-d<x<d$). Using decomposition in Fourier series in the region $\{(x,y)\in\Om^0\,|\, \pm x \ge d\}$, as in Chapter \ref{ChapWaveguides}, one shows that $u_0$ admits the expansion 
\[
u_0(x,y)=K_{\pm}e^{-\sqrt{\pi^2-\lambda^0}|x|}\cos(\pi y)+\dots
\]
where $K_{\pm}\in\Cplx$ and the $\dots$ correspond to terms which decay as $O(e^{-\sqrt{4\pi^2-\lambda^0}|x|})$ as $|x|\to+\infty$. In \cite{ChNa18}, the following theorem is proved. 
\begin{theorem} \label{TheoremFano}
Assume that $(K_+,K_-)\ne(0,0)$. There is a quantity $\ell(H)\in\R$, which depends linearly on $H$, such that when $\eps\to0$,
\begin{equation}\label{Fano1}
\mathbb{S}(\eps,\lambda^0+\eps\lambda')=\mathbb{S}^0+O(\eps)\qquad\mbox{ for $\lambda'\ne\ell(H)$},\hspace{0.7cm}
\end{equation}
and, for any $\mu\in\R$,
\begin{equation}\label{Fano2}
\mathbb{S}(\eps,\lambda^0+\eps\ell(H)+\eps^2\mu)=\mathbb{S}^0+\cfrac{\tau^{\top}\cdot\tau}{i\tilde{\mu}-|\tau|^2/2}+O(\eps).
\end{equation}
In this expression $\tau=(a,b)\in\Cplx\times\Cplx$ depends only on $\Om$ and $\tilde{\mu}=c\mu+C$ for some unessential real constants $c$, $C$ with $c\ne0$.
\end{theorem}
\noindent The justification of this result is more difficult than in 1D because one cannot perform explicit calculations. It requires adapted tools (see \cite{ShVe05,ShTu12,ShWe13,AbSh16,ChNa18}) that we will not present here. However the interpretation of the result is completely similar to that of the 1D case. Observe that (\ref{Fano1}) and  (\ref{Fano2}) are respectively the analogous of (\ref{Fano01}) and (\ref{Resultat1DMobius}). As explained above, Theorem \ref{TheoremFano} shows that the mapping $\mathbb{S}(\cdot,\cdot)$ is not continuous at $(0,\lambda^0)$ (setting where trapped modes exist). Moreover for $\eps_0$ small fixed, it proves that the scattering matrix $\lambda\mapsto \mathbb{S}(\eps_0,\lambda)$ exhibits a quick change in a neighborhood of $\lambda^0+\eps_0\ell(H)$: this is the Fano resonance phenomenon. When $(K_+,K_-)=(0,0)$ a faster Fano resonance phenomenon occurs.\subsection{Zero reflection and zero transmission}
In the sequel, to simplify we denote by $\mathbb{S}^{\eps}(\mu)$, $T^{\eps}(\mu)$, $R^{\eps}_\pm(\mu)$ the values of $\mathbb{S}$, $T$, $R_\pm$ in $\Om^{\eps}$ at the frequency $\lambda=\lambda^0+\eps\ell(H)+\eps^2\mu$.\\
\newline
Assume now that $\Om^{\eps}$ is symmetric with respect to the vertical axis, \textit{i.e.} such that 
\[
\Om^{\eps}=\{(-x,y)\,|\,(x,y)\in\Om^{\eps}\}.
\] 
Then we can decompose the scattering solutions into even/odd components by working with two half-waveguide problems involving Neumann/Dirichlet boundary conditions at $x=0$.

\begin{figure}[!ht]
\centering 
\begin{tikzpicture}[scale=1.4]
\draw[fill=gray!30,draw=none](-2,0) rectangle (2,1);
\draw (-2,0)--(2,0);
\draw (-2,1)--(-0.25,1);
\draw (2,1)--(0.25,1);
\draw[dashed] (-2.5,1)--(-2,1); 
\draw[dashed] (-2.5,0)--(-2,0); 
\draw[dashed] (2.5,1)--(2,1); 
\draw[dashed] (2.5,0)--(2,0); 
\node at (-1.5,0.2){\small $\Om^\eps$};
\draw[fill=white](0,0.5) circle (0.4cm);
\draw[samples=30,domain=-0.25:0.25,draw=black,fill=gray!30] plot(\x,{1+0.05*(4*\x+1)^4*(4*\x-1)^4*(4*\x+3)});
\end{tikzpicture}\qquad\quad
\begin{tikzpicture}[scale=1.4]
\draw[fill=gray!30,draw=none](-2,0) rectangle (0,1);
\draw (-2,1)--(-0.25,1);
\draw[dashed] (-2.5,0)--(0,0); 
\draw[dashed] (-2.5,1)--(0,1); 
\node at (-1.5,0.2){\small $\om^\eps$};
\draw (-2,0)--(0,0)--(0,0.1);
\draw[fill=white,shift={(0.005,0.9)}]  (0:0) arc (90:270:0.4cm);
\draw[samples=30,domain=-0.25:0,draw=black,fill=gray!30] plot(\x,{1+0.05*(4*\x+1)^4*(4*\x-1)^4*(4*\x+3)})--(0,0.9);
\end{tikzpicture}
\caption{Domain $\Om^\eps$ (left) and $\om^\eps$ (right).\label{LimitDomain}} 
\end{figure}
\noindent More precisely, define the half-waveguide 
\[
\om^\eps\coloneqq\{(x,y)\in\Om^\eps\,|\,x<0\}
\]
(see Figure \ref{LimitDomain} right). Introduce the problem with Neumann BCs
\begin{equation}\label{PbChampTotalSym}
\begin{array}{|rcll} 
\multicolumn{4}{|l}{\mbox{Find }v^\eps\in\mH^1_{\loc}(\om^\eps)\mbox{ such that }v^\eps-w_+\mbox{ is outgoing and } }\\[3pt]
\Delta v^\eps +\lambda v^\eps & = & 0 & \mbox{ in }\om^\eps\\[3pt]
 \partial_\nu v^\eps  & = & 0  & \mbox{ on }\partial\om^\eps
\end{array}
\end{equation}
as well as the problem with mixed BCs
\begin{equation}\label{PbChampTotalAntiSym}
\begin{array}{|rcll}
\multicolumn{4}{|l}{\mbox{Find }V^\eps\in\mH^1_{\loc}(\om^\eps)\mbox{ such that }V^\eps-w_+\mbox{ is outgoing and } }\\[3pt]
\Delta V^\eps + \lambda V^\eps & = & 0 & \mbox{ in }\om^\eps\\[3pt]
 \partial_\nu V^\eps  & = & 0  & \mbox{ on }\partial\om^\eps\cap\partial\Om^\eps \\[3pt]
V^\eps  & = & 0  & \mbox{ on }\Sigma^\eps\coloneqq\partial\om^\eps\setminus\partial\Om^\eps.
\end{array}
\end{equation}
For $\lambda\in(0;\pi^2)$, Problems (\ref{PbChampTotalSym}) and (\ref{PbChampTotalAntiSym}) admit respectively the solutions 
\[
v^\eps = w_++R^\eps_N\,w_- + \tilde{v}^\eps,
\]
\[
V^\eps= w_++R^\eps_D\,w_- + \tilde{V}^\eps,
\]
where $R^\eps_N$, $R^\eps_D\in\Cplx$ and $\tilde{v}^\eps$, $\tilde{V}^\eps\in\mH^1(\om^\eps)$. Due to conservation of energy, one proves as in (\ref{ConservationNRJ_Dirichlet}) the identities
\[
|R^\eps_N|=|R^\eps_D|=1
\]
(since there is only one output in $\om^\eps$, all the energy propagated by the incident wave is backscattered).   
Now, direct inspection shows that if $u^\eps$ is a solution of Problem (\ref{WaveguidePbFano}) associated to an incident wave coming from left or right, then we have
\[
u^\eps(x,y)=\cfrac{v^\eps(x,y)+V^\eps(x,y)}{2}\quad\mbox{ in }\om^\eps,\qquad u^\eps(x,y)=\cfrac{v^\eps(-x,y)-V^\eps(-x,y)}{2}\quad\mbox{ in }\Om^\eps\setminus\overline{\om^\eps}
\]
(up possibly to a term which is exponentially decaying at $\pm\infty$ if there are trapped modes at the given $\lambda$). We deduce that the scattering coefficients $R_\pm^\eps$, $T^\eps$ for Problem (\ref{WaveguidePbFano}) are such that
\[
R_+^\eps=R_-^\eps=\frac{R^\eps_N+R^\eps_D}{2}\qquad\mbox{ and }\qquad T^\eps=\frac{R^\eps_N-R^\eps_D}{2}.
\]
If we indicate the dependence with respect to $\mu$ as in \S\ref{ParaFano2D}, this writes 
\begin{equation}\label{FormulaParity}
R^{\eps}_+(\mu)=R^{\eps}_-(\mu)=\frac{R_N^\eps(\mu)+R_D^\eps(\mu)}{2}\qquad\mbox{ and }\qquad T^\eps(\mu)=\frac{R_N^\eps(\mu)-R_D^\eps(\mu)}{2}\,.
\end{equation}
To set ideas assume that the trapped modes associated with $\lambda^0$ are even. In that case, $(\eps,\lambda)\mapsto R_D(\eps,\lambda)$ is smooth at $(0,\lambda_0)$. As a consequence, for $\eps$ small, $\mu\mapsto R_D^\eps(\mu)$ does not vary much on the unit circle $\mathscr{C}(0,1)$ for $\mu\in(-\eps^{-1/2};\eps^{-1/2})$. On the other hand, by adapting the result of Theorem \ref{TheoremFano}, one establishes that $\mu\mapsto R_N^\eps(\mu)$ runs once on $\mathscr{C}(0,1)$ for $\mu\in(-\eps^{-1/2};\eps^{-1/2})$. From formulas (\ref{FormulaParity}), this ensures that the curves $\mu\mapsto T^{\eps}(\mu)$, $\mu\mapsto R^{\eps}_+(\mu)$ for $\mu\in(-\eps^{-1/2};\eps^{-1/2})$, pass exactly through zero for $\eps$ small enough. This provides examples of geometries where we have either zero reflection or zero transmission. 

\begin{figure}[!ht]
\centering
\begin{tikzpicture}[scale=1.7]
\draw[fill=gray!30,draw=none](0,0) rectangle (3,1);
\draw (3,0)--(0,0)--(0,1)--(3,1); 
\draw[dashed] (3,1)--(3.5,1); 
\draw[dashed] (3,0)--(3.5,0);
\draw[fill=white] (1,0.6) circle (0.25cm);
\draw[dotted,>-<] (1,-0.05)--(1,0.65);
\draw[dotted,>-<] (-0.05,0.6)--(1.05,0.6);
\node at (1.5,0.3){\small $0.5+\eps$};
\node at (0.5,0.7){\small $1$};
\end{tikzpicture}\qquad\raisebox{0.1cm}{\includegraphics[width=0.58\textwidth]{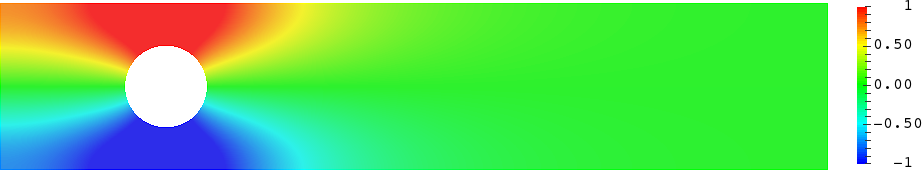}}
\caption{Left: geometry of the half waveguide. Right: real part of a trapped mode for $\eps=0$ and $k^0\coloneqq\sqrt{\lambda^0}\approx2.7403$. The computation has been realized by using Perfectly Matched Layers (see Chapter IV Section \ref{SectionClassical}).\label{GeomDisk}} 
\end{figure}
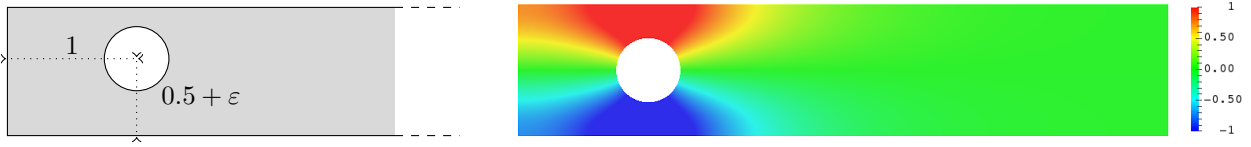

\noindent Let us illustrate this numerically. In Figure \ref{GeomDisk} right, we give an example of geometry supporting trapped modes for the Neumann problem (\ref{PbChampTotalSym}) at a particular $\lambda=\lambda_0\in(0;\pi^2)$. Note that the waveguide is symmetric with respect to the line of equation $y=1/2$ which can be used to give some proofs of existence of such trapped modes in certain circumstances (see \cite{EvLV94}). Then we perturb  the domain by slightly shifting vertically the disk\footnote{Observe that a horizontal shift of the position of the disk would maintain the decoupling between symmetric and skew-symmetric modes. As a consequence, the eigenvalue embedded in the continuous spectrum would remain an eigenvalue embedded in the continuous spectrum and no Fano resonance phenomenon would be observed.}. Then the symmetry is broken and from the analysis above, we know that when we sweep in $\lambda$ in a small neighborhood of $\lambda_0$, there is one $\lambda$ for which one has zero reflection (see Figure \ref {CercleRNull}) and one $\lambda$ for which one gets zero transmission (see Figure \ref{CercleTNull}). Let us stress that the smaller $\eps$, the more delicate the adjustment of $\lambda$.

\begin{figure}[!ht]
\centering
\hspace{-0.15cm}\scalebox{-1}[1]{\includegraphics[trim={0cm 1cm 2.8cm 1cm},clip,width=0.8\textwidth]{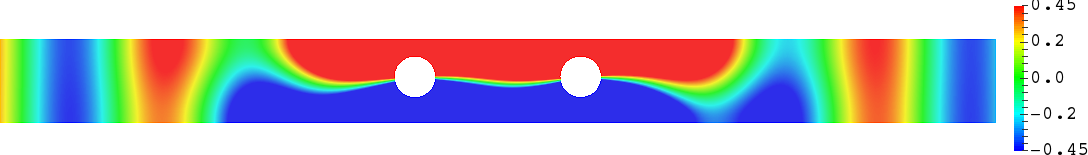}}\\[5pt]
\noindent\scalebox{-1}[1]{\includegraphics[trim={0cm 1cm 2.8cm 1cm},clip,width=0.8\textwidth]{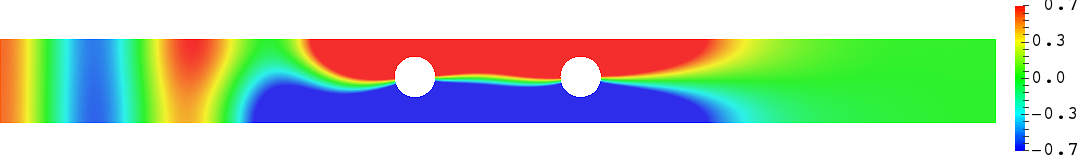}}
\caption{$\Re e\,u^{\eps}$ (top) and $\Re e\,(u^{\eps}-w_+)$ in a setting where $R^\eps_\pm=0$ ($\eps=0.05$ and $k=\sqrt{\lambda}=2.751$).\label{CercleRNull}}
\end{figure}
\begin{figure}[!ht]
\centering
\scalebox{-1}[1]{\includegraphics[trim={0cm 1cm 2.8cm 1cm},clip,width=0.8\textwidth]{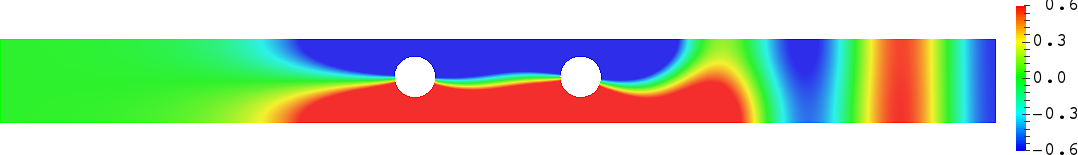}}
\caption{$\Re e\,u^{\eps}$ in a setting where $T^\eps=0$ ($\eps=0.05$ and $k=\sqrt{\lambda}=2.75495$).\label{CercleTNull}}
\end{figure}

\noindent When the domain $\Om^\eps$ is not symmetric with respect to the vertical axis, we cannot use the decomposition with the two half-waveguide problems. In that case, for a fixed small $\eps>0$, in general the curves $\mu\mapsto R^\eps_\pm(\mu)$ do not pass through zero in the complex plane and we do not observe zero reflection. However we can show, quite surprisingly, that $\mu\mapsto T^\eps(\mu)$ always vanishes for some particular $\mu$. Let us give the main ingredients of the proof.
\begin{theorem}\label{MainThmPart1}
Assume that $T^0=T(0,\lambda^0)\ne0$. Then there is $\eps_0>0$ such that for all $\eps\in(0;\eps_0]$, there is $\mu\in\R$, depending on $\eps$, such that $T^\eps(\mu)=0$.
\end{theorem}
\begin{proof}
Theorem \ref{TheoremFano} provides the estimate
\begin{equation}\label{MainAsymptoT}
|T^{\eps}(\mu)-T^{\mrm{asy}}(\mu)|\le C\,\eps
\end{equation}
\[
\hspace{-0.7cm}\mbox{with }\quad T^{\mrm{asy}}(\mu)=T^0+\cfrac{ab}{i\tilde{\mu}-(|a|^2+|b|^2)/2}.
\]
For any compact set $I\subset\R$, the constant $C>0$ in (\ref{MainAsymptoT}) can be chosen independent of $\mu\in I$. We will prove first that the asymptotic term  $T^{\mrm{asy}}(\cdot)$ always passes through zero. Then, from that and Estimate (\ref{MainAsymptoT}), we will deduce, due to the rigid structure of the scattering matrix, that $T^{\eps}(\cdot)$ also passes through zero.\\[2pt]
$\star$ First, we study the set $\{T^{\mrm{asy}}(\mu),\,\mu\in\R\}$. Classical results concerning the M\"{o}bius transform guarantee that $\{T^{\mrm{asy}}(\mu),\,\mu\in\R\}$ coincides with $\mathscr{C}^{\mrm{asy}}\setminus\{T^0\}$ where $\mathscr{C}^{\mrm{asy}}$ is a circle passing through $T^0$. Let us show that $\mathscr{C}^{\mrm{asy}}$ also passes through zero. One finds that $T^{\mrm{asy}}(\mu)=0$ for some $\mu\in\R$ if and only if there holds 
\begin{equation}\label{MainIdentity}
\cfrac{|a|^2+|b|^2}{2}=\Re e\left(\cfrac{ab}{T^0}\right).
\end{equation}
An intermediate calculus of \cite{ChNa18} implies $R^0_+\,\overline{a}+T^0\,\overline{b}=a$ and $T^0\,\overline{a}+R^0_-\,\overline{b}=b$. From this and the unitarity of $\mathbb{S}^{0}$ which imposes $R^0_-=-\overline{R^{0}_+} T^{0}/\overline{T^{0}}$, we can obtain (\ref{MainIdentity}).\\ 
\newline
Denote by $\mu_{\star}$ the value of $\mu$ such that $T^{\mrm{asy}}(\mu_{\star})=0$ and for $\eps>0$, define the interval $I^{\eps}=(\mu_{\star}-\sqrt{\eps};\mu_{\star}+\sqrt{\eps})$.
From (\ref{MainAsymptoT}), for $\eps>0$ small, we know that the curve $\{T^{\eps}(\mu),\,\mu\in I^{\eps}\}$ passes close to zero. Now, using the unitary structure of $\mathbb{S}^{\eps}(\mu)$, we show that this curves passes exactly through zero for $\eps$ small. \\[2pt]
$\star$ 
Assume by contradiction that for all $\eps>0$, $\mu\mapsto T^{\eps}(\mu)$ does not pass through zero in $I_\eps$. Since $\mathbb{S}^{\eps}(\mu)$ is unitary, there holds $R^{\eps}_+(\mu)\,\overline{T^{\eps}(\mu)}+T^{\eps}(\mu)\,\overline{R^{\eps}_-(\mu)}=0$ and so
\[
-R^{\eps}_+(\mu)/\overline{R^{\eps}_-(\mu)}=T^{\eps}(\mu)/\overline{T^{\eps}(\mu)}\qquad\mbox{$\forall\mu\in I^{\eps}$}.
\]
But if $\mu\mapsto T^{\eps}(\mu)$ does not pass through zero on $I^{\eps}$, one can verify that the point $T^{\eps}(\mu)/\overline{T^{\eps}(\mu)}=e^{2i\mrm{arg}(T^{\eps}(\mu))}$ must run rapidly on the unit circle for $\mu\in I_\eps$ as $\eps\to0$. On the other hand, $R^{\eps}_+(\mu)/\overline{R^{\eps}_-(\mu)}$ tends to a constant on $I_\eps$ as $\eps\to0$. This way we obtain a contradiction.
\end{proof}
\begin{remark}
The fact that $\mathscr{C}^{\mrm{asy}}$ passes through zero is quite mysterious. It is related to the rigid structure of the scattering matrix. 
\end{remark} 
\noindent We illustrate this result in Figure \ref{Figure3}. First we find that trapped modes exist for $\eps=0$ and $\sqrt{\lambda^0}\approx1.2395\in(0;\pi)$. Then we compute  $T(\eps,\lambda)$ (\textcolor{blue}{$\times$}) and $R_+(\eps,\lambda)$ (\raisebox{0.5mm}{\protect\tikz \fill[red] (0,3) circle (0.6mm);}) for $\sqrt{\lambda}\in(1.2;1.3)$ and $\eps=0.05$. As predicted, we observe that $\lambda\mapsto T(\eps,\lambda)$  passes through zero around $\lambda^0$. Finally, we display the real part of $u_+$ in $\Om^\eps$ for $\eps=0.05$ and $\sqrt{\lambda}=1.2449$, a configuration where $T(\eps,\lambda)\approx0$.

\begin{figure}[!ht]
\centering
\hspace{-0.0cm}\begin{tikzpicture}[scale=1.8]
\node at (1.2,-0.2){\includegraphics[scale=0.5]{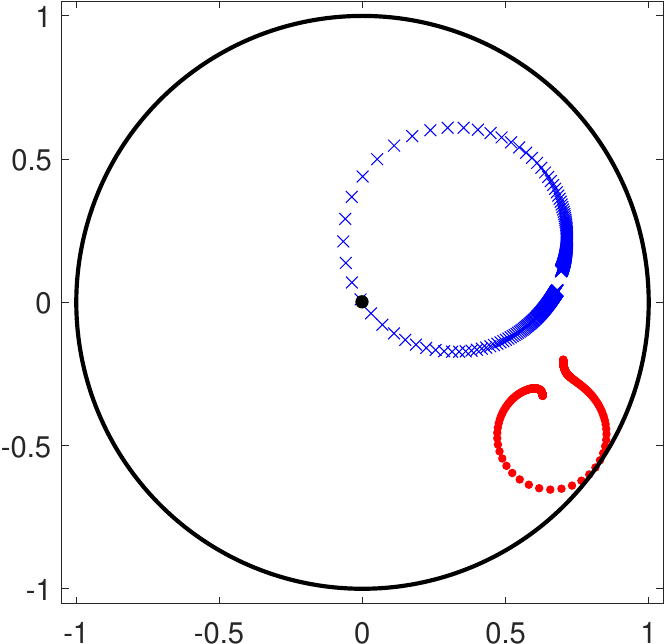}};
\begin{scope}[xshift=-3cm].
\draw[fill=gray!30,draw=none](-2,0) rectangle (2,1);
\draw[fill=white,draw=none](-1,0.35) rectangle (1,0.7);
\draw[fill=white,draw=none](0,0.15) rectangle (1,0.85);
\draw (-2,0)--(2,0); 
\draw (-2,1)--(2,1); 
\draw (-1,0.35)--(0,0.35)--(0,0.15)--(1,0.15)--(1,0.85)--(0,0.85)--(0,0.7)--(-1,0.7)--(-1,0.35); 
\draw[dashed] (2,1)--(2.3,1); 
\draw[dashed] (-2,1)--(-2.3,1); 
\draw[dashed] (2,0)--(2.3,0); 
\draw[dashed] (-2,0)--(-2.3,0); 
\draw[dotted,>-<] (0.4,-0.05)--(0.4,0.6);
\node at (1,0.3){\footnotesize  $0.5+\eps$};
\node at (1,-1){\includegraphics[scale=0.34,trim={0 0 11.4cm 0},clip]{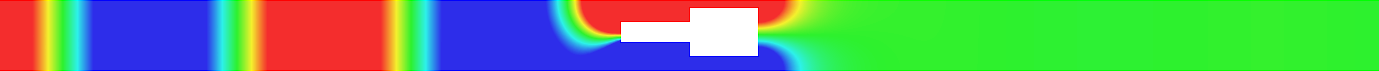}};
\end{scope}
\end{tikzpicture}
\caption{Zero transmission in a waveguide via the Fano resonance mechanism.\label{Figure3}}
\end{figure}
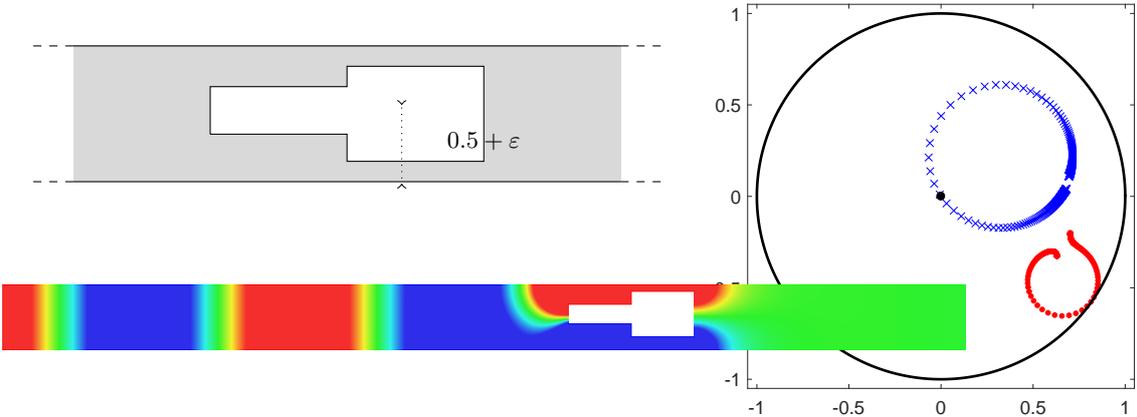

\noindent In this study concerning the exploitation of the Fano resonance mechanism to obtain zero reflection/zero transmission, we considered the case of Neumann BCs. Let us mention that Dirichlet BCs can be studied completely similarly. 

\newpage
\section{Cloaking of a given obstacle by using thin resonant chimneys}\label{ParagCloaking}

\begin{figure}[!ht]
\centering
\vspace{-0.4cm}
\begin{tikzpicture}[scale=2]
\draw[fill=gray!30] (0.3,1) circle (0.4) ;
\draw[fill=gray!30,draw=none](-1.4,0) rectangle (1.4,1);
\draw (-1.4,0)--(1.4,0);
\draw (-1.4,1)--(-0.1,1);
\draw (0.7,1)--(1.4,1);
\draw[fill=white] (-0.5,0.2)--(0.1,0.8)--(0.4,0.3)--cycle;
\draw[dashed] (-1.4,0)--(-1.7,0);
\draw[dashed] (-1.4,1)--(-1.7,1);
\draw[dashed] (1.4,0)--(1.7,0);
\draw[dashed] (1.4,1)--(1.7,1);
\node at (-1.2,0.1){\small $\Om^0$};
\end{tikzpicture}\qquad
\begin{tikzpicture}[scale=2]
\draw[fill=gray!30] (0.3,1) circle (0.4) ;
\draw[fill=gray!30,draw=none](-1.4,0) rectangle (1.4,1);
\draw (-1.4,0)--(1.4,0);
\draw (-1.4,1)--(-0.1,1);
\draw (0.7,1)--(1.4,1);
\draw[dotted,<->] (-0.74,1.02)--(-0.74,1.6);
\draw[dotted,>-<] (-0.7,1.7)--(-0.5,1.7);
\node at (-0.6,1.8){\small $\eps$};
\node at (-0.9,1.3){\small $\ell^\eps$};
\node at (-0.6,0.8){\small $A$};
\draw[fill=gray!30,draw=none](-0.66,0.9) rectangle (-0.54,1.6);
\draw (-0.66,1)--(-0.66,1.6)--(-0.54,1.6)-- (-0.54,1);
\draw[fill=white] (-0.5,0.2)--(0.1,0.8)--(0.4,0.3)--cycle;
\draw[dashed] (-1.4,0)--(-1.7,0);
\draw[dashed] (-1.4,1)--(-1.7,1);
\draw[dashed] (1.4,0)--(1.7,0);
\draw[dashed] (1.4,1)--(1.7,1);
\draw (-0.63,0.97)--(-0.57,1.03);
\draw (-0.63,1.03)--(-0.57,0.97);
\node at (-1.2,0.1){\small $\Om^{\eps}$};
\end{tikzpicture}
\caption{Left: initial setting. Right: geometry with one thin outer resonator. \label{DomainWithOneLigament}} 
\end{figure}
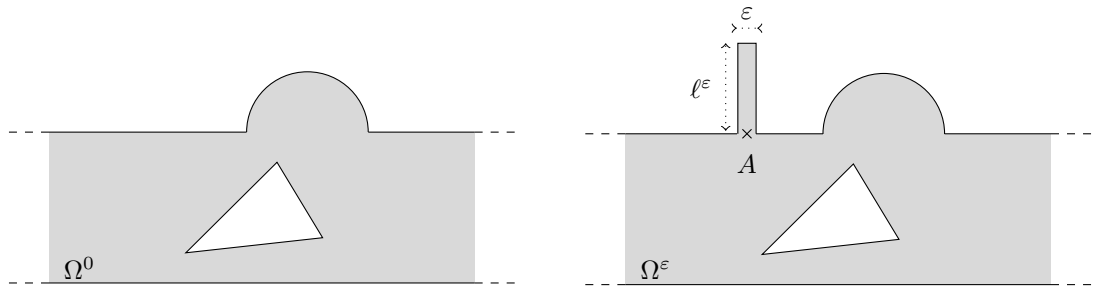

\noindent In this section, we change the point of view. Instead of constructing invisible objects, we assume that some obstacle is given and explain how to hide it (or to cloak it if we use the terminology of physicists). More precisely, starting from a setting where $T\ne1$, we show how to perturb the initial geometry by working with chimneys as pictured in Figure \ref{DomainWithOneLigament} to obtain a new waveguide where $T\approx 1$. Let us mention that we do not work with \textit{ad hoc} penetrable materials as in transformation optics \cite{PeSS06,Leon06,ChCh10}. Additionally, we do not add active sources in the system as people do in active cloaking \cite{Mill06,CDGG21}. What we realize is passive cloaking at infinity by perturbing the shape of the waveguide. \\
\newline
As in Chapter \ref{ChapWaveguides}, the main difficulty of the problem lies in the fact that the dependence of the scattering coefficients with respect to the geometry is not explicit and not linear. In order to address it, techniques of optimization have been considered. We refer the reader in particular to \cite{AlLS17,LDOHG19,LGHDO19}. However, due to the features of the Helmholtz equation, the functionals involved in the analysis are non convex and unsatisfying local minima exist. Moreover, these methods do not allow the user to control the main features of the shape compared with the approach we present and which has been developed in \cite{ChHN22}.\\
\newline
To cloak obstacles, we work with thin outer resonators, the chimneys, of width $\eps>0$ small compared to the wavelength (see again Figure \ref{DomainWithOneLigament} right). These chimneys are interesting because they are almost 1D objects, which allows us to make explicit their influence on the fields and so on the scattering coefficients. However in general, \textit{i.e.} for most lengths, they produce only perturbations of order $\eps$ which is not sufficient to compensate for the scattering due to the initial object. But by working around the resonance lengths (see (\ref{DefResLength})) of the resonators, we can obtain effects of order one. This is a key aspect in our approach which makes it in particular different from the technique presented in \S\ref{SectionTOne}. Note that thin chimneys around resonance lengths have been studied for example in \cite{Krieg,BoTr10,LiZh17,LiSZ19} in a context close to ours, namely in the analysis of the scattering of an incident wave by a periodic array of subwavelength slits. The core of our approach is based on an asymptotic expansion of the scattering solutions with respect to $\eps$ as $\eps$ tends to zero. This allow us to derive formulas for the scattering coefficients with a relatively explicit dependence on the geometrical features. To obtain the expansions, we apply again techniques of matched asymptotic expansions. For related methods, we refer the reader to \cite{Beal73,KoMM94}.\\
\newline
Let us describe the general strategy. We consider the acoustic problem (\ref{WaveguidePbNeumann}) with Neumann BCs and assume that $k\in(0;\pi)$ (the height of the guide is still one outside of some compact region). Denote by $u^\eps_+$ the solution of (\ref{WaveguidePbNeumann}) corresponding to the scattering of the rightgoing plane wave $w_+$ in $\Om^\eps$. The first step is to compute an asymptotic expansion of $u^\eps_+$ as $\eps$ tends to zero. As usual in asymptotic analysis, we work with different ansatz for $u^\eps_+$, depending on the region. More precisely, we consider the outer expansions
\begin{equation}\label{ExpanCloaking}
\begin{array}{|ll}
u^\eps_+(x,y)=u^0(x,y) +\dots & \mbox{ in }\Om^0\\[4pt]
u^\eps_+(x,y)=\eps^{-1}v^{-1}(y)+v^0(y)+\dots & \mbox{ in the resonator}
\end{array}
\end{equation}
where $\Om^0$ denotes the initial waveguide without the chimney and $u^0$, $v^{-1}$, $v^0$ denote unknown functions which are independent of $\eps$. To begin with, suppose that the chimney has a length $\ell>0$ independent of $\eps$. Considering the restriction of Problem (\ref{WaveguidePb}) in $\Om^\eps$ to the thin resonator, when $\eps$ tends to zero, we find that $v^{-1}$ must solve the homogeneous 1D problem
\[
(\mathscr{P}_{\mrm{1D}})\ \begin{array}{|l}
\partial^2_yv+k^2v=0\qquad\mbox{ in }(1;1+\ell)\\[3pt]
v(1)=\partial_yv(1+\ell)=0.\\
\end{array}
\]
An important message is that the features of $(\mathscr{P}_{\mrm{1D}})$ play a key role in the physical phenomena  and so in the asymptotic analysis. We denote by $\ell_{\mrm{res}}$ (resonance lengths), the values of $\ell$, given by
\begin{equation}\label{DefResLength}
\ell_{\mrm{res}}\coloneqq\pi (m+1/2)/k,\qquad m\in\N,
\end{equation}
such that $(\mathscr{P}_{\mrm{1D}})$ admits a non zero solution. Note that the non zero functions solving $(\mathscr{P}_{\mrm{1D}})$ coincide, up to a multiplicative constant, with $\sin(k (y-1))$.\\
\newline
Assume first that $\ell\ne\ell_{\mrm{res}}$ so that the only solution of $(\mathscr{P}_{\mrm{1D}})$ is zero. Then we set $v^{-1}\equiv0$ in (\ref{ExpanCloaking}) and when $\eps\to0$, we can show that
\begin{equation}\label{ResultAsymptoOut}
\begin{array}{|ll}
u^\eps_{\pm}(x,y)=u_{\pm}+o(1) &\quad\mbox{ in }\Om^0\\[4pt]
u^{\eps}_{\pm}(x,y)=u_{\pm}(A)\,v_0(y)+o(1) &\quad\mbox{ in the resonator}
\end{array}
\end{equation}
where $u_{\pm}$ stands for the scattering solution corresponding to an incident plane wave coming from $\mp\infty$ and $A=(x_A,1)$ is the attachment point of the chimney. Moreover in (\ref{ResultAsymptoOut}),
\[
v_0(y)=\cos(k(y-1)) + \tan(k(y-\ell))\sin(k(y-1))
\]
(observe that we have $\partial^2_yv+k^2v=0$ in $(1;1+\ell)$ as well as $v_0(1)=1$ and $\partial_xv_0(1+\ell)=0$). From this, we deduce that 
\[
R^\eps_{\pm}=R_{\pm}+o(1),\qquad\qquad T^\eps= T+o(1),
\]
where $R_{\pm}$, $T$ are the scattering coefficients in the geometry without the chimney. In that situation, we see that the resonator has no influence at order $\eps^0$, which is not interesting for our purpose.\\
\newline 
Assume now that $\ell=\ell_{\mrm{res}}$. In that case, the asymptotic analysis is more involved. At the end of the (long) procedure, see \cite{ChHN22} for the details, we obtain, when $\eps\to0$,
\begin{equation}\label{ResultAsymptoOut1}
\begin{array}{|ll}
u^{\eps}_+(x,y)=u_+(x,y)+ak\gamma(x,y)+o(1) &\mbox{ in }\Om^0\\[4pt]
u^{\eps}_+(x,y)=\eps^{-1}a\sin(k (y-1))+O(1) &\mbox{ in the resonator}
\end{array}
\end{equation}
where $\gamma$ denotes the outgoing Green function such that 
\[
\begin{array}{|rcll}
\Delta \gamma+ k^2\gamma  &=&  0&\mbox{ in }\Om \\ [2pt]
\partial_{\nu} \gamma  &=&  \delta_A&\mbox{ on }\partial\Om.
\end{array} 
\]
Here $\delta_A$ is the Dirac delta distribution supported at $A$. Moreover in (\ref{ResultAsymptoOut1}), we find that $a$ is given by
\[
ak=-\cfrac{u_+(A)}{\Gamma+\pi^{-1}\ln|\eps|+C_{\Xi}}
\]
where $\Gamma$ and $C_{\Xi}$ are some constants which depend only on the initial geometry $\Om^0$. Observe that in (\ref{ResultAsymptoOut1}), the field blows up as $O(\eps^{-1})$ in the resonator, which is directly related to the fact that there is a complex resonance close to the considered real $k$. From (\ref{ResultAsymptoOut1}), we obtain 
\[
R_{+}^{\eps}=R_{+}+ia u_+(A)/2+o(1),\qquad T^{\eps}=T+iau_-(A)/2+o(1).
\]
This time the thin resonator has an influence at order $\eps^0$.  Let us make a small variant by assuming that the length of the chimney is equal to $\ell^\eps=\ell_{\mrm{res}}+\eps \eta$, where $\eta\in\R$ is a parameter that we set as we wish. In that situation the analysis is very similar to the previous case and when $\eps\to0$, we find
\begin{equation}\label{ResultAsymptoOut2}
\begin{array}{|ll}
u^{\eps}_+(x,y)=u_+(x,y)+a(\eta)k\gamma(x,y)+o(1) &\mbox{ in }\Om^0\\[4pt]
u^{\eps}_+(x,y)=\eps^{-1}a(\eta)\sin(k (y-1))+O(1) &\mbox{ in the resonator}
\end{array}
\end{equation}
with
\[
a(\eta) k=-\cfrac{u_+(A)}{\Gamma+\pi^{-1}\ln|\eps|+C_{\Xi}+\eta}\,.
\]
This gives
\begin{equation}\label{ExpansionsMainC}
R_{+}^{\eps}=R_{+}^{\mrm{asy}}(\eta)+o(1),\qquad T^{\eps}=T^{\mrm{asy}}(\eta)+o(1)
\end{equation}
with
\[
R_{+}^{\mrm{asy}}(\eta)\coloneqq R_{+}+ia(\eta)u_+(A)/2,\qquad T^{\mrm{asy}}(\eta)\coloneqq T+ia(\eta)u_-(A)/2+o(1).
\]
Thus, not only the resonator has an influence at order $\eps^0$, but additionally the latter depends on the choice made for $\eta$.  We get something similar to what has been shown in Section \ref{SectionFano} (see in particular Figure \ref{FigFano1DParabola}): for all $\eta\in\R$, the resonator tends to the 1D segment $(1;1+\ell_\mrm{res})$, but depending on the choice of $\eta$, the limit of the corresponding scattering coefficients is not the same. As a consequence, the scattering coefficients, considered as functions of the two variables $(\eps,\ell)$, are not continuous at the point $(0,\ell_\mrm{res})$. Moreover, for $\eps_0$ fixed small, varying slightly $\ell$ around $\ell_\mrm{res}$, which corresponds for example to sweep $\eta\in[-\eps_0^{-1/2};\eps_0^{-1/2}]$, we obtain a large variation for $R_{+}^{\eps_0}$, $T^{\eps_0}$ (again see the illustration of Figure \ref{FigFano1DParabola}). More precisely, working with the M\"{o}bius transform, one shows that the sets 
\[
\{R_{+}^{\mrm{asy}}(\eta),\,\eta\in\overline{\R}\},\qquad \{T^{\mrm{asy}}(\eta),\,\eta\in\overline{\R}\}
\]
where $\overline{\R}\coloneqq\R\cup\{+\infty\}\cup\{-\infty\}$, coincide with circles. We deduce that asymptotically, when $\eps\to0$, when perturbing the length of the chimney around $\ell_\mrm{res}$, the coefficients $R_{+}^{\eps}$, $T^{\eps}$ run on circles. Interestingly for our purpose, the features of these circles depend on $A$, the attachment point of the chimney.\\
\newline
In view of achieving zero reflection, by using the expansions of $u_\pm(A)$ far from the obstacle, the following statement is established in \cite{ChHN22}:
\begin{proposition}
Assume that $R_{+}\ne0$ and $T\ne0$. There are some $A=(x_A,1)$ such that there exists $\eta$ such that $R_{+}^{\mrm{asy}}(\eta)=0$. In that case, when $\eps\to0$, we have $R_{+}^{\eps}=0+o(1)$.
\end{proposition}
\begin{remark}
Note that we exclude the case $R_{+}=0$ because in this situation we already have zero reflection in the initial geometry and there is no need for adding a resonator. In the case $T=0$, \textit{i.e.} $|R_{+}|=1$ due to conservation of energy, the most challenging situation, our approach does not work. However, one possibility to get zero reflection is to add first one or several resonators to obtain a transmission coefficient quite different from zero. And then to add another well-tuned resonator to kill the reflection. Let us mention that this strategy is also interesting when $T$ is small but non zero because in this case achieving almost zero reflection with only one resonator is quite unstable. 
\end{remark}
\begin{remark}
It is important to emphasize that compared to what we presented in the previous sections, we do not reach exactly $R_{+}^{\eps}=0$ but simply get $R_{+}^{\mrm{asy}}(\eta)=0$. In other words, there is a residue due to the error in the expansion. This analysis encourages us to take $\eps$ as small as possible to reduce the residue. However when $\eps$ becomes small, the amplitude of the field in the resonator gets very high and the tuning procedure of the features of the chimney is very sensitive. Thus one must find a compromise between small reflection and robustness with respect to perturbations of the geometry. 
\end{remark}

\noindent In Figures \ref{FigFish01}--\ref{FigFish001}, we consider the scattering of the rightgoing plane wave by a sound hard obstacle (the fish). We have added a well-tuned chimney to obtain almost zero reflection. In Figure \ref{FigFish001}, we observe that by working with a smaller $\eps$, as expected we can reduce the reflection.

\begin{figure}[!ht]
\centering
\includegraphics[width=0.45\textwidth]{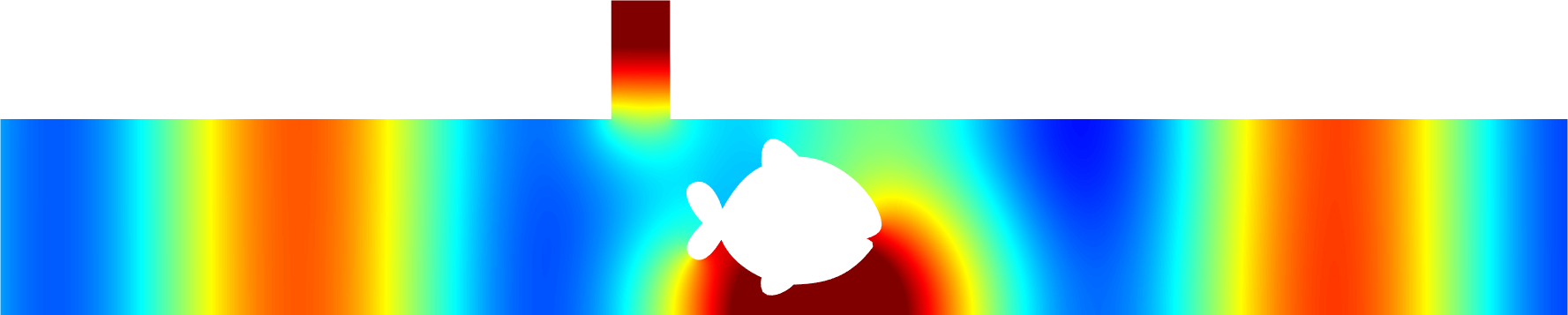}\quad\ \includegraphics[width=0.45\textwidth]{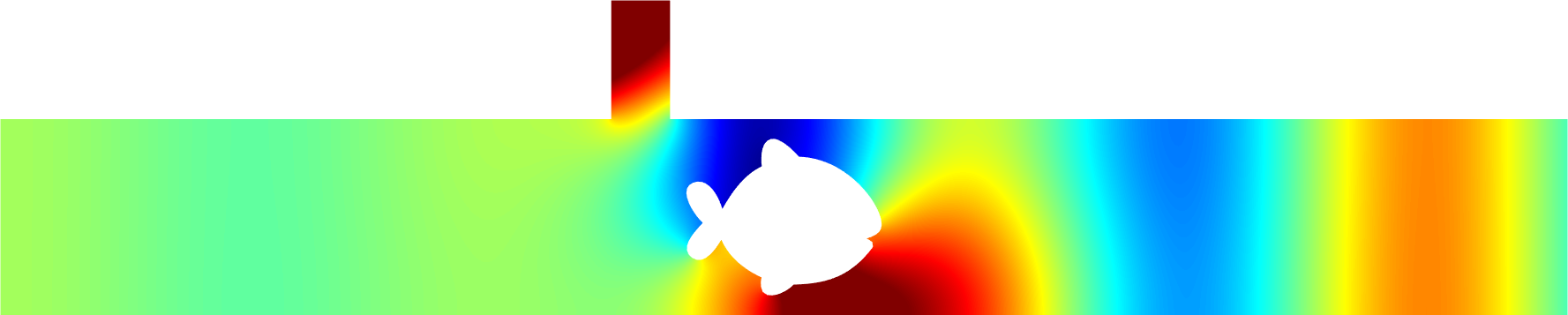}
\caption{Real parts of $u^\eps_+$ (left) and of $u^\eps_+-w_+$ (right). The length of the resonator is tuned to get almost zero reflection. Here $\eps=0.3$.\label{FigFish01}}
\end{figure}

\begin{figure}[!ht]
\centering
\includegraphics[width=0.45\textwidth]{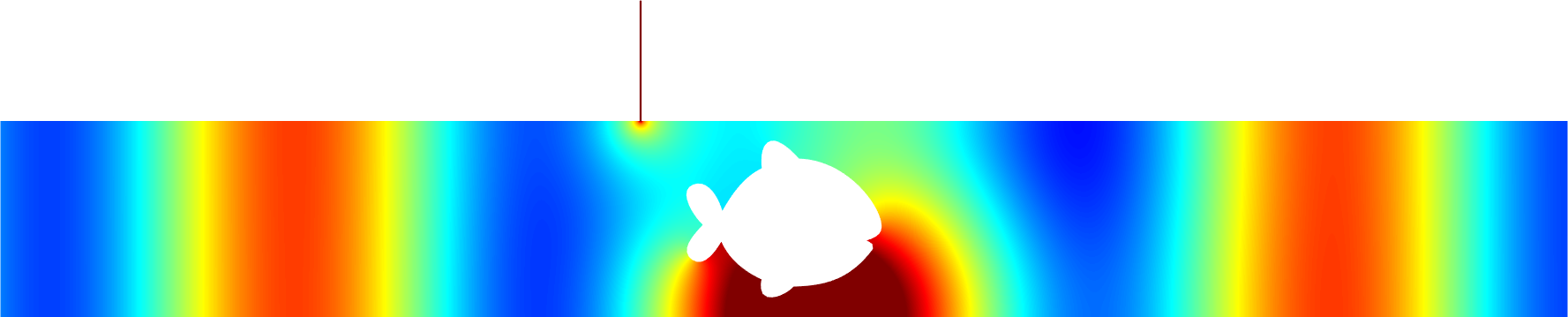}\quad\  \includegraphics[width=0.45\textwidth]{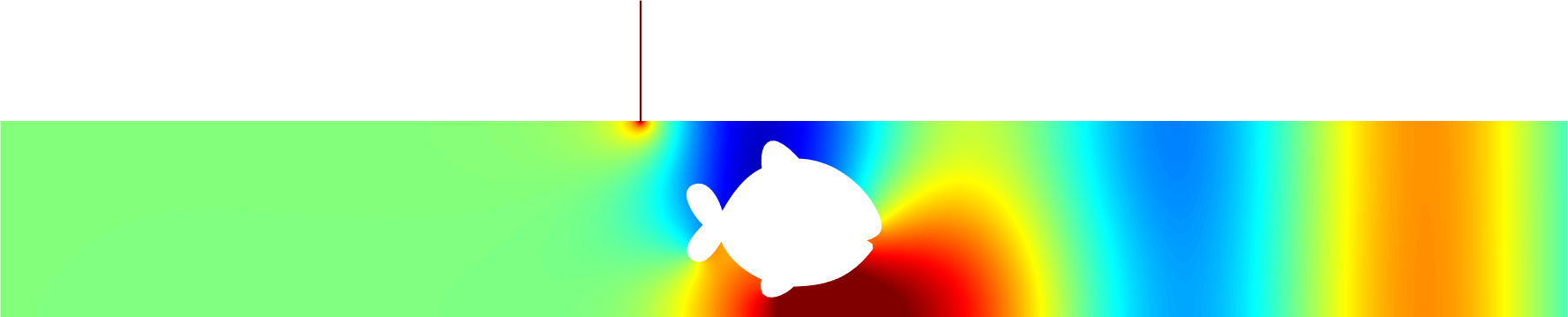}
\caption{Same quantities as in Figure \ref{FigFish01} but with a thinner resonator (here $\eps=0.01$).\label{FigFish001}}
\end{figure}

\noindent Once almost zero reflection has been obtained, it remains to compensate for the phase shift (recall that $R_{+}^{\mrm{asy}}(\eta)=0$ only implies $|T_{+}^{\mrm{asy}}(\eta)|=1$). To proceed, one method consists in coupling the previous waveguide with what we call a phase-shifter. This is a device where one has zero reflection and any prescribed phase. We have shown that such phase shifters can be designed by working with two well-tuned resonant chimneys added to the reference strip $\R\times(0;1)$. At the end, we obtain $T^\eps\approx1$ with three well-tuned resonant chimneys. \\
\newline
On the other hand, by exploiting again the results of the asymptotic analysis (\ref{ResultAsymptoOut2}), (\ref{ExpansionsMainC}), we have established that we can get $T^\eps\approx1$ with only two well-tuned chimneys. Let us assess the degrees of freedom which are involved. We wish to control two complex coefficients, $R^\eps$, $T^\eps$, and so four real parameters. The relation of conservation of energy imposes one constraint. As a consequence, there are three real degrees of freedom. In our strategy here, we play with the two lengths of the resonators and with the distance between them.\\
\newline
In Figures \ref{FigCloakElephant}, \ref{FigCloakTuyau}, we cloak two different obstacles/defects by working with two resonant chimneys. In each case, on the first line we display the field corresponding to the scattering of a rightgoing plane wave without the resonators. The setting of Figure \ref{FigCloakTuyau} is particularly challenging because the initial transmission coefficient is very small. To restore a good transmission, we observe that we have to excite strongly the resonances. Practically, this is probably a limitation because we can imagine that dissipation will then become important.

\begin{figure}[!ht]
\centering
\includegraphics[width=0.9\textwidth]{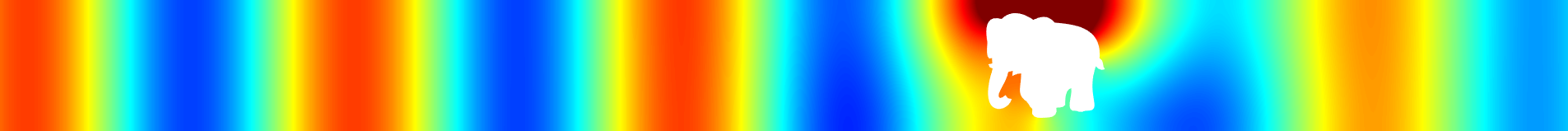}\\[5pt]
\includegraphics[width=0.9\textwidth]{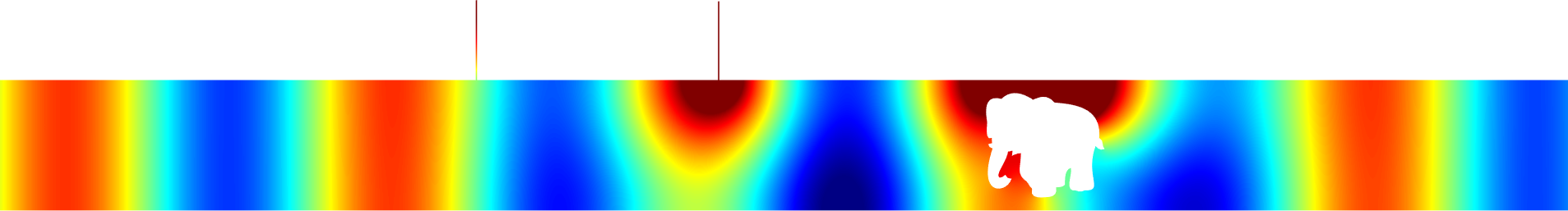}\\[5pt]
\includegraphics[width=0.9\textwidth]{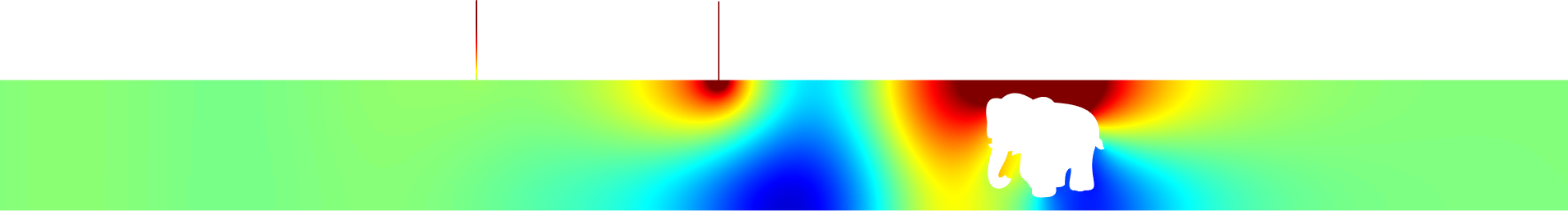}
\caption{Real parts of $u_+$ (top), $u^\eps_+$ (middle) and $u^\eps_+-w_+$ (bottom). The resonators are tuned to get $T^{\eps}\approx1$. Here $\eps=0.01$.\label{FigCloakElephant}}
\end{figure}

\begin{figure}[!ht]
\centering
\includegraphics[trim={5cm 0cm 6cm 0cm},clip,width=0.9\textwidth]{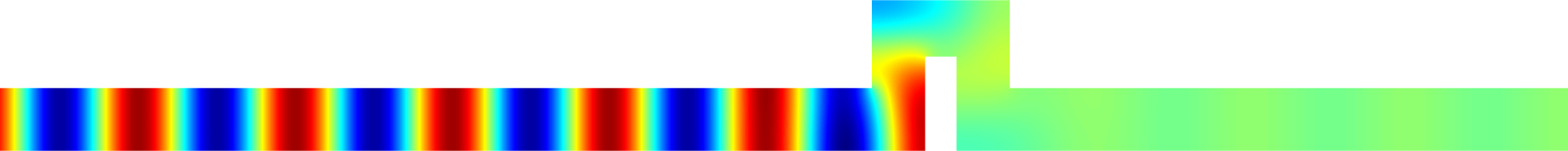}\\[5pt]
\includegraphics[trim={5cm 0cm 6cm 0cm},clip,width=0.9\textwidth]{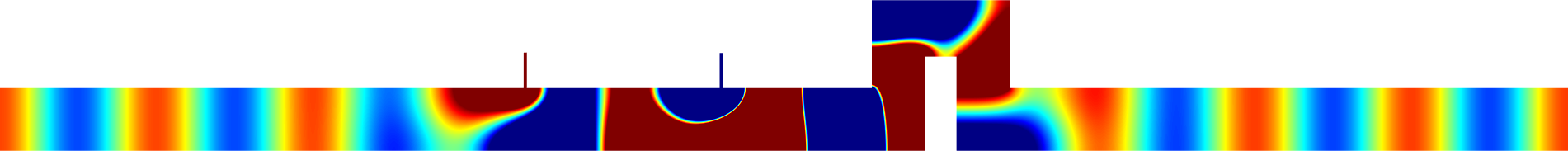}\\[5pt]
\includegraphics[trim={5cm 0cm 6cm 0cm},clip,width=0.9\textwidth]{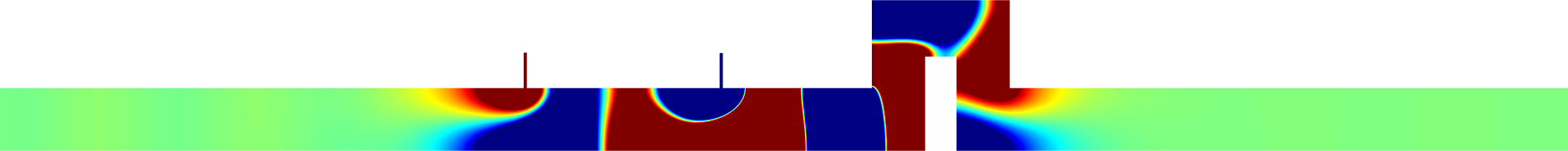}
\caption{Real parts of $u_+$ (top), $u^\eps_+$ (middle) and $u^\eps_+-w_+$ (bottom). The resonators are tuned to get $T^{\eps}\approx1$. Here $\eps=0.05$.\label{FigCloakTuyau}}
\end{figure}

\chapter{A spectral problem characterizing zero reflection}\label{SectionSpectral}
\setcounter {section} {0}

\noindent Let us adopt another point of view concerning questions of invisibility. Instead of considering the wavenumber $k$ fixed in the problem and try to find geometries where one has zero reflection or  perfect invisibility, we assume that the geometry is fixed and we look for $k$ such that there is an incident field whose energy is completely transmitted through the waveguide.

\section{A first idea in 1D} 
\begin{figure}[!ht]
\centering
\begin{tikzpicture}[scale=1.5]
\draw[dashed,thick] (-3.5,0)--(-3,0);
\draw[dashed,thick] (3.5,0)--(3,0);
\draw[draw,thick] (-3,0)--(3,0);
\draw[draw,line width=1mm] (-1,0)--(1,0);
\node at (-1,-0.3){\small $-L$};
\node at (1,-0.3){\small $L$};
\node at (0,0.3){\small $I_L$};
\node at (2,0.3){\small $\gamma=1$};
\node at (-2,0.3){\small $\gamma=1$};
\end{tikzpicture}
\caption{A $\mrm{1D}$ scattering problem. \label{Domain1DNR}} 
\end{figure}
\noindent Consider first a simple scattering problem in 1D, not by a geometric defect, but by a penetrable obstacle located in $I_L\coloneqq(-L;L)$ for $L>0$. Introduce some material index $\gamma\in\mL^{\infty}(\R)$ such that $\gamma>0$ in $\R$ and $\gamma=1$ for $|x|>L$ (see Figure \ref{Domain1DNR}). The problem we are interested in writes
\begin{equation}\label{Pb1DNR}
u''+k^2\gamma u=0\quad\mbox{ in }\R
\end{equation}
with $k>0$. As already seen in \S\ref{SectionToyPb}, for all $k>0$, there is only one propagating mode for this 1D problem, namely $w^\pm(x)=e^{\pm ikx}$, satisfying (\ref{Pb1DNR}) outside of the defect, and no evanescent mode. For that reason, trapped modes do not exist for (\ref{Pb1DNR}). Indeed, if $u\in\mL^2(\R)$ satisfies (\ref{Pb1DNR}), then it must be zero outside of $I_L$ because the functions $w^\pm$ do not belong to $\mL^2(\R)$. This implies $u(\pm L)=\partial_xu(\pm L)=0$ and so $u\equiv0$ in $\R$ according to the Cauchy-Lipschitz theorem.\\ 
\newline
On the other hand, from the theory of Chapter \ref{ChapWaveguides} (see also the computations of \S\ref{SectionToyPb}), we know that for all $k>0$, Problem (\ref{Pb1DNR}) admits a unique solution $u\in\mH^1_{\mrm{loc}}(\R)$ with the expansion 
\[
u=w^++R\,w^-\quad\mbox{ for }x<-L,\qquad\qquad u=T\,w^+\quad\mbox{ for }x>L,
\]
where $R$, $T\in\Cplx$. This $u$ corresponds to the total field associated with the scattering of the rightgoing wave $w^+$. From (\ref{Pb1DNR}), it belongs to $\mH^2_{\mrm{loc}}(\R)$, and therefore to $\mathscr{C}^1(\R)$ according to the classical results on Sobolev embeddings (see \textit{e.g.} \cite[Thm.\,8.8]{Brez11}). Due to conservation of energy, one has $|R|^2+|T|^2=1$.  In the particular situations where $R=0$ (zero reflection), if that happens, $u$ decomposes as 
\[
u=w^+\quad\mbox{ for }x<-L,\qquad\qquad u=T\,w^+\quad\mbox{ for }x>L.
\]
Since $w^+(x)=e^{ikx}$, we deduce that $u$ satisfies the Robin conditions 
\begin{equation}\label{Robin1DNR}
\partial_xu=iku\quad \mbox{ at }x=\pm L.
\end{equation}
Note that (\ref{Robin1DNR}) at $x=L$ is the same outgoing radiation condition as in (\ref{WaveguidePbBounded}) ($\partial_\nu u=\Lambda_+(u|_{\Sigma_L})$). For that 1D problem, the action of the Dirichlet-to-Neumann map is just the multiplication by $ik$. But mind the difference at $x=-L$ where (\ref{Robin1DNR}) selects the incoming behaviour\footnote{The outgoing condition at $x=-L$ is $\partial_\nu u=\Lambda_-(u|_{\Sigma_{-L}})\Leftrightarrow -\partial_xu=iku$.}. Thus if one has zero reflection, then there is $u\in\mH^1(I_L)\setminus\{0\}$ solving
\begin{equation}\label{PbFortNR}
\begin{array}{|rcll}
u''+k^2\gamma u&=&0&\mbox{ in }I_L\\[3pt]
\partial_xu&=&iku &\mbox{ at }x=\pm L.
\end{array}
\end{equation}
The variational formulation associated with (\ref{PbFortNR}) writes
\begin{equation}\label{PbFortNRVaria}
\dsp\int_{I_L}\partial_xu\,\partial_x\overline{v}-k^2\gamma u\overline{v}\,dx-ik\left(u(L)\overline{v(L)}-u(-L)\overline{v(-L)}\right)=0,\qquad \forall v\in\mH^1(I_L).
\end{equation}
Conversely, one can check that if $u\not\equiv0$ solves (\ref{PbFortNRVaria}), then the function $\hat{u}$ such that
\[
\hat{u}(x)=\begin{array}{|ll}
u(-L)\,w^+(x+L)&\mbox{ for }x<-L\\[2pt]
u(x)&\mbox{ for }x\in I_L\\[2pt]
u(L)\,w^+(x-L)&\mbox{ for }x>L
\end{array}
\]
satisfies (\ref{Pb1DNR}), and so that one has zero reflection. Thus Problem (\ref{PbFortNRVaria}) characterizes exactly the $k$ such that one has zero reflection. It is a quadratic spectral problem in $k$ which raises many interesting questions (see \cite{HeKS11} for more details): Can one prove that the set of eigenvalues of (\ref{PbFortNRVaria}) is discrete? Can one show the existence of real eigenvalues? Are there imaginary eigenvalues? Is the numerical approximation of the eigenvalues of (\ref{PbFortNRVaria}) robust?\\
\newline
However we will not go further in that direction because Problem (\ref{PbFortNRVaria}) is quite specific to the 1D case. In 2D or in 3D, as proposed in \cite{Stone-et-al-2020}, we could work similarly and impose the incoming condition 
\[
\partial_\nu u=\overline{\Lambda_-}(u|_{\Sigma_{-L}})\quad\mbox{ on }\Sigma_{-L},
\]
where the operator $\overline{\Lambda_-}:\mH^{1/2}_{00}(\Sigma_{-L}) \to \mH^{-1/2}(\Sigma_{-L})$ is defined as $\Lambda_-$ in (\ref{DefDTN}) with the $i\beta_n$ replaced by $\overline{i\beta_n}$ (remember that the $i\beta_n$ for the propagating modes are purely imaginary while they are real for evanescent modes). But then the problem 
corresponding to (\ref{PbFortNRVaria}) becomes fully non-linear (not only quadratic) with respect to $k$ because $k$ appears in all the $\beta_n$. For that reason, it is rather intractable, in particular its spectrum is hard to compute.

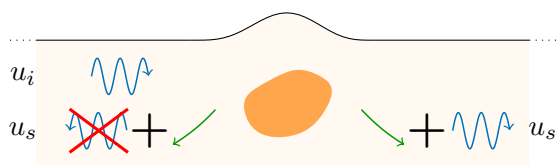
\begin{figure}[!ht]
\centering
\begin{tikzpicture}[scale=1.2]
\draw[fill=orange!5,draw=none](-2.7,0.3) rectangle (2.7,1.7);
\draw[samples=30,domain=-1:1,draw=black,fill=orange!5] plot(\x,{1.7+0.1*(\x+1)^4*(\x-1)^4*(\x+3)});
\begin{scope}[scale=0.7,yshift=0.4cm]
\draw [fill=orange!70,draw=none] plot [smooth cycle, tension=1] coordinates {(-0.6,0.9) (0,0.5) (0.7,1) (0.5,1.5) (-0.2,1.4)};
\end{scope}
\draw (-2.7,0.3)--(2.7,0.3);
\draw (-2.7,1.7)--(-1,1.7);
\draw (1,1.7)--(2.7,1.7);
\draw [dotted](-2.7,0.3)--(-3,0.3);
\draw [dotted](2.7,0.3)--(3,0.3);
\draw [dotted](-2.7,1.7)--(-3,1.7);
\draw [dotted](2.7,1.7)--(3,1.7);
\node at (-2.85,1.3){ $u_i$};
\node at (-2.85,0.7){$u_s$};
\node at (2.85,0.7){$u_s$};
\begin{scope}[xshift=0.9cm,yshift=0.7cm,scale=0.8]
\begin{scope}[xshift=-0.05cm]
\draw[green!60!black,line width=0.2mm,->] plot[very thick,domain=0:pi/4-0.2,samples=100] (\x,{-0.8+0.02*exp(-(\x-4))});
\end{scope}
\node at (0.85,0){\LARGE $\boldsymbol{+}$};
\begin{scope}[xshift=1.2cm]
\draw[RoyalBlue,line width=0.2mm,->] plot[domain=0:pi/4,samples=100] (\x,{0.25*sin(20*\x r)});
\end{scope}
\end{scope}
\begin{scope}[xshift=-2.5cm,yshift=0.7cm,scale=0.8]
\begin{scope}[xshift=0cm]
\draw[RoyalBlue,line width=0.2mm,<-] plot[domain=0:pi/4,samples=100] (0.2+\x,{0.25*sin(20*\x r)});
\node at (1.3,0){\LARGE $\boldsymbol{+}$};
\draw[red,line width=0.4mm] (0.2,-0.3)--(1,0.3);
\draw[red,line width=0.4mm] (1,-0.3)--(0.2,0.3);
\end{scope}
\begin{scope}[xshift=2.5cm,yshift=-0.8cm]
\draw[green!60!black,line width=0.2mm,<-] plot[very thick,domain=8-pi/4+0.2:8,samples=100] (\x-8.3,{0.02*exp((\x-4))});
\end{scope}
\end{scope}
\begin{scope}[xshift=-2.1cm,yshift=1.3cm,scale=0.8]
\draw[RoyalBlue,line width=0.2mm,->] plot[domain=0:pi/4,samples=100] (\x,{0.25*sin(20*\x r)});
\end{scope}
\end{tikzpicture}
\caption{Schematic picture of a reflectionless mode. The propagating wave (blue) is not reflected. The backscattered field is purely evanescent (green). \label{SchematicPicture}}
\end{figure}

\noindent Instead, we will present a strategy proposed in \cite{BoCP18} to show that what we will call reflectionless modes (see below after (\ref{DefReflectionlessMode})) can be characterized as the eigenfunctions of a linear non-selfadjoint spectral problem. The approach is still based on the following basic observation: if for a given incident field, combination of rightgoing propagating modes, one has zero reflection, then the total field is incoming in the input lead and outgoing in the output lead as illustrated in Figure \ref{SchematicPicture}. The new point compared to above is that to select incoming waves on one side of the obstacle and outgoing  waves on the other side, we will use complex scalings \cite{aguilar1971class,balslev1971spectral} (or Perfectly Matched Layers \cite{Bera94}) with imaginary parts of different signs. We will prove that the real eigenvalues of the obtained spectrum correspond either to trapped modes (also called Bound States in the Continuum, BSCs or BICs, in quantum mechanics) or to reflectionless modes. Interestingly, complex eigenvalues also contain useful information on weak reflection cases. When the geometry has certain symmetries, the new spectral problem enters the class of $\mathcal{PT}$-symmetric problems ($\mathcal{PT}$ for $\mathcal{P}$arity-$\mathcal{T}$ime, see Proposition \ref{PropositionPTsym} below). Let us describe this in more details.

\section{2D setting}\label{SectionSetting}

\begin{figure}[!ht]
\centering
\begin{tikzpicture}[scale=1.4]
\draw[fill=orange!5,draw=none](-1.5,0) rectangle (1.5,1);
\draw[fill=orange!70,draw=none](-1,0.25) rectangle (1,0.75);
\draw (-1.5,0)--(1.5,0);
\draw (-1.5,1)--(1.5,1);
\draw [dotted](-1.5,0)--(-1.8,0);
\draw [dotted](1.5,0)--(1.8,0);
\draw [dotted](-1.5,1)--(-1.8,1);
\draw [dotted](1.5,1)--(1.8,1);
\draw (-1,-0.1)--(-1,0.1);
\draw (1,-0.1)--(1,0.1);
\node at (-1,-0.3){\small $-L$};
\node at (1,-0.3){\small $L$};
\node at (0,0.5){\small $\gamma=5$};
\node at (-1.5,0.5){\small $\gamma=1$};
\node at (-2.2,-0.3){\small (a)};
\end{tikzpicture}\qquad\qquad\begin{tikzpicture}[scale=1.4]
\draw[fill=orange!5,draw=none](-1.5,0) rectangle (1.5,1);
\draw[fill=orange!70,draw=none](-1,0.25) rectangle (1,0.75);
\draw[fill=orange!5,draw=none](-1,0.5) rectangle (0,0.75);
\draw (-1.5,0)--(1.5,0);
\draw (-1.5,1)--(1.5,1);
\draw [dotted](-1.5,0)--(-1.8,0);
\draw [dotted](1.5,0)--(1.8,0);
\draw [dotted](-1.5,1)--(-1.8,1);
\draw [dotted](1.5,1)--(1.8,1);
\node at (0.5,0.5){\small $\gamma=5$};
\node at (-1.5,0.5){\small $\gamma=1$};
\begin{scope}[xshift=-1.3cm,yshift=-1cm]
\draw[->] (3,1.2)--(3.5,1.2);
\draw[->] (3.1,1.1)--(3.1,1.6);
\node at (3.65,1.3){\small $x$};
\node at (3.25,1.6){\small $y$};
\end{scope}
\draw (-1,-0.1)--(-1,0.1);
\draw (1,-0.1)--(1,0.1);
\node at (-1,-0.3){\small $-L$};
\node at (1,-0.3){\small $L$};
\node at (-2.2,-0.3){\small (b)};
\end{tikzpicture}\\[-6pt]
\caption{Symmetric (a) and non-symmetric (b) obstacles considered in the numerics below. \label{figsetting}} 
\end{figure}
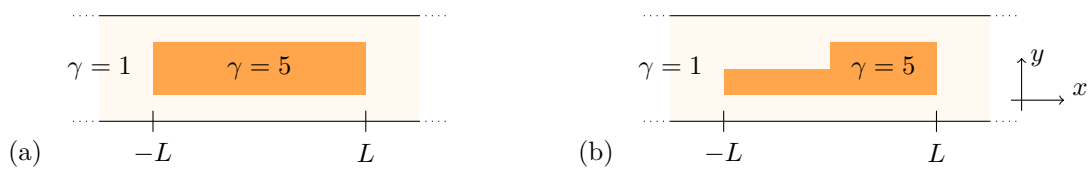

\noindent  To make the presentation as simple as possible, we study the scattering of waves in 2D by a penetrable obstacle. The waveguide coincides with the region $\Om\coloneqq\{(x,y)\in\R^2\,|\,0<y<1\}$ and we consider the problem
\begin{equation}\label{PbInitial}
\begin{array}{|rcll}
\Delta u + k^2\gamma u & = & 0 & \mbox{ in }\Om\\[3pt]
 \partial_y u & = & 0  & \mbox{ on }\partial\Om
\end{array}
\end{equation}
with Neumann BCs. The coefficient $\gamma$ corresponds to the material index of the medium filling $\Om$. We assume that $\gamma\in\mL^\infty(\Om)$ satisfies $\gamma>0$ in $\Om$ and is such that $\gamma=1$ for $|x|\ge L$ where $L>0$ is given. In other words, the obstacle is located in the region $\Om_L\coloneqq\{(x,y)\in\Om\,|\,|x|<L\}$ (see Figure \ref{figsetting}). Pick $k\in(N\pi;(N+1)\pi)$, with $N\in\N\coloneqq\{0,1,\dots\}$. 
Let us change a bit the definition of the modes (\ref{ModesNeumann}) by modifying the normalization and set
\begin{equation}\label{defModes}
w_n^{\pm}(x,y)=\cfrac{e^{\pm i\beta_n x}\varphi_n(y)}{(2|\beta_n|)^{1/2}}\,,\quad\beta_n\coloneqq\sqrt{k^2-n^2\pi^2},\quad \varphi_n(y)=\begin{array}{|ll}
1 & \hspace{-0.2cm}\mbox{if }n=0\\
\sqrt{2}\cos(n\pi y) & \hspace{-0.2cm}\mbox{for }n>0.
\end{array}\hspace{-0.2cm}
\end{equation}
For $n=0,\dots,N$, the wave $w_n^{\pm}$ propagates along the $(Ox)$ axis from $\mp\infty$ to $\pm\infty$. On the other hand, for $n> N$, $w_n^{\pm}$ is exponentially  growing at $\mp\infty$ and exponentially decaying at $\pm\infty$. For $n=0,\dots,N$, we consider the scattering of the wave $w_n^{+}$ by the obstacle located in $\Om$. By adapting Proposition \ref{PropoWPDirichletSca}, one shows that Problem (\ref{PbInitial}) admits a solution $u_n=w_n^{+}+u^s_n$ with the outgoing scattered field $u^s_n$ written as
\begin{equation}\label{RadiationCondition}
u^s_n=\sum_{p=0}^{+\infty} s^{-}_{np}w_p^{-}\quad\mbox{ for }x\le-L, \qquad\quad u^s_n=\sum_{p=0}^{+\infty} s^{+}_{np}w_p^{+}\quad\mbox{ for } x \ge L,
\end{equation}
with $(s^{\pm}_{np})\in\Cplx^{\N}$. The solution $u_n$ is uniquely defined if and only if Trapped Modes (TMs) do not exist at the wavenumber $k$. We remind the reader that trapped modes are non zero  functions $u\in \mL^2(\Omega)$ satisfying (\ref{PbInitial}). We denote by $\mathscr{K}_{\mrm{t}}$ the set of $k^2$ such that TMs exist at the wavenumber $k$. On the other hand, as already mentioned after (\ref{RepresentationupD}), the scattering coefficients $s^{\pm}_{np}$ in (\ref{RadiationCondition}) are always uniquely defined, even when $k^2\in\mathscr{K}_{\mrm{t}}$. In the following, we will be particularly interested in the features of the reflection matrix  (whose size, determined by the number of propagating modes, depends on $k$)
\begin{equation}\label{reflection matrix} 
R(k)\coloneqq(s^{-}_{np})_{0\le n,p\le N}\in\Cplx^{N+1\times N+1}.
\end{equation} 
\begin{definition} The wavenumber $k\in(0;+\infty)\setminus\{\N\pi\}$ is said reflectionless if $\ker\,R(k)\ne\{0\}$.
\end{definition} 
\noindent Let us explain this definition. By linearity, for an incident field (coming from the left)
\begin{equation}\label{DefIncidentField}
u_i=\sum_{n=0}^Na_n w_n^{+},\qquad (a_n)_{n=0}^N\in\Cplx^{N+1},
\end{equation}
Problem (\ref{PbInitial}) admits a solution $u$ such that $u=u_i+u_s$ with
\begin{equation}\label{DefScatteredField}
u_s=\sum_{p=0}^{+\infty}b_p^{\pm} w_p^{\pm}\mbox{ for }\pm x\ge L\qquad\mbox{ and }\qquad b_p^{\pm}=\sum_{n=0}^Na_{n}s^{\pm}_{np}\in\Cplx.
\end{equation}
The above definition says that, if $k$ is reflectionless, then there is a vector $(a_n)_{n=0}^N\in\Cplx^{N+1}\setminus\{0\}$ such that the $b_p^{-}$ in (\ref{DefScatteredField}) satisfy $b_p^{-}=0$, $p=0,\dots,N$. In other words, the scattered field is exponentially decaying for $x\le -L$. Finally notice that the corresponding total field $u=u_i+u_s$ decomposes as
\begin{equation}\label{DefReflectionlessMode}
\begin{array}{ll}
\dsp u=\sum_{n=0}^{N}a_n w_n^{+}+\tilde{u}&\mbox{ for } x\le -L\\[13pt]
\dsp u=\sum_{n=0}^{N}t_n w_n^{+}+\tilde{u}&\mbox{ for } x\ge L\\
\end{array}
\end{equation}
where $t_n=a_n+b_n^+$ and where $\tilde{u}$ decays exponentially for $\pm x\ge L$. In other words, the total field is incoming for $x\le-L$ and outgoing for $x\ge L$.\\
\newline
In the following,  we call Reflectionless Modes (RMs) the functions $u$ satisfying (\ref{PbInitial}) and admitting expansion (\ref{DefReflectionlessMode}). We denote by $\mathscr{K}_{\mrm{r}}$ the set of $k^2$ such that the wavenumber $k$ is reflectionless. Our objective is to explain how to determine directly the set $\mathscr{K}_{\mrm{r}}$ and the corresponding RMs by solving a linear eigenvalue problem, instead of computing the reflection matrix for all values of $k$.  

\section{Classical complex scaling}\label{SectionClassical}
As a first step, we remind briefly how to use a complex scaling to compute trapped modes. Define the unbounded operator $A$ of $\mL^2(\Om)$ such that
\[ 
Au=-\cfrac{1}{\gamma}\,\Delta u
\]
with domain 
\[
D(A)=\{u\in\mH^1(\Om)\,|\,\Delta u\in \mL^2(\Om)\mbox{ and }\partial_{y}u=0\mbox{ on }\partial\Om\}.
\] 
It is known that the Neumann Laplacian $A$ is a selfadjoint operator when $\mL^2(\Om)$ is endowed with the inner product $(\gamma\,\cdot,\cdot)_{\mL^2(\Om)}$. The spectrum of $A$, denoted $\sigma(A)$, coincides with $[0;+\infty)$. More precisely, we have $\sigma_{\mrm{ess}}(A)=[0;+\infty)$ where $\sigma_{\mrm{ess}}(A)$ denotes the essential spectrum of $A$. By definition, $\sigma_{\mrm{ess}}(A)$ corresponds to the set of $\lambda \in\Cplx$ for which there exists a so-called singular sequence $(u^{(m)})$, that is an orthonormal sequence $(u^{(m)})\in\mL^2(\Omega)^{\N}$ such that $(u^{(m)})$ converges to 0 weakly in $\mL^2(\Omega)$ and $((A-\lambda)u^{(m)})$ converges to 0 strongly in $\mL^2(\Omega)$. 
Besides, $\sigma(A)$ may contain eigenvalues (at most a sequence accumulating at $+\infty$, adapt the proof of Proposition \ref{PropoDiscreteTrappedModes} to show this result) corresponding to TMs. In order to reveal these eigenvalues which are embedded in $\sigma_{\mrm{ess}}(A)$, one can use a complex change of variables. For $0<\theta<\pi/2$, set $\eta=e^{i\theta}$ and define the function $\mathcal{I}_{\theta}:\R\to\Cplx $ such that 
\begin{equation}\label{defDilatationx}
\mathcal{I}_{\theta}(x)=\begin{array}{|ll}
-L+(x+L)\,\eta&\mbox{ for }x\le- L\\
x&\mbox{ for }|x|< L\\
+ L+(x-L)\,\eta&\mbox{ for } x\ge L.
\end{array}
\end{equation}
For the sake of simplicity, we  will use abusively the same notation $\mathcal{I}_{\theta}$ for the following map:
$\{\Om \to \Cplx\times(0;1), \;
(x,y) \mapsto (\mathcal{I}_{\theta}(x),y)\}$.
One can easily check that for all $n\geq 0$, 
$w_n^+\circ \mathcal{I}_{\theta}$ is exponentially decaying for $x\ge L$, while $w_n^-\circ \mathcal{I}_{\theta}$ is exponentially decaying for $x\le-L$. Note that here we have extended the definition of the modes $w_n^\pm$ (see (\ref{defModes})) to complex $x$.
As a consequence, defining from expansion (\ref{DefScatteredField})  the function $v_{\theta}=u_s\circ \mathcal{I}_{\theta}$, one has $v_{\theta}=u_s$ for $|x|< L$ and $v_{\theta}\in \mL^2(\Omega)$ (which is in general not true for $u_s$). Moreover $v_{\theta}$ satisfies the following equation in $\Om$:
\begin{equation}\label{defEquationDilatation}
\alpha_{\theta}\frac{\partial}{\partial x}\Big(\alpha_{\theta}\frac{\partial v_{\theta}}{\partial x}\Big)+\frac{\partial^2 v_{\theta}}{\partial y^2}+k^2\gamma v_{\theta}=k^2(1-\gamma)u_i
\end{equation}
with $\alpha_{\theta}(x)=1$ for $|x|<L$ and $\alpha_{\theta}(x)=\eta^{-1}=\overline{\eta}$ for $\pm x\ge L$. In particular, for a TM, $v_{\theta}$ solves (\ref{defEquationDilatation}) with $u_i=0$. This leads us to consider the unbounded operator $A_{\theta}$ of $\mL^2(\Om)$ such that
\begin{equation}\label{defOpClassicalPMLs}
A_{\theta}v_{\theta}=-\cfrac{1}{\gamma}\left(\alpha_{\theta}\frac{\partial}{\partial x}\Big(\alpha_{\theta}\frac{\partial v_{\theta}}{\partial x}\Big)+\frac{\partial^2 v_{\theta}}{\partial y^2}\right)
\end{equation}
with domain
\[
D(A_{\theta})=\{v\in\mH^1(\Om)\,|\,A_{\theta}v\in \mL^2(\Om)\mbox{ and }\partial_{y}v=0\mbox{ on }\partial\Om\}.
\]
Since $\alpha_{\theta}$ is complex valued, the operator $A_{\theta}$ is not selfadjoint. For non-selfadjoint operators, several definitions exist for the essential spectrum which are in general not equivalent (see more details in \cite{EdEv87}). We shall say that $\lambda$ belongs to $\sigma_{\mrm{ess}}(A_{\theta})$ if $A_{\theta}-\lambda\mrm{Id}$ is not a Fredholm operator (see Definition \ref{definition fredholm op}). We recall below the main spectral properties of $A_{\theta}$ (see details in \cite{Simo78}):
\begin{theorem}\label{thmUsualPMLs}
i) There holds
\begin{equation}\label{essClassicalPMLs}
\sigma_{\mrm{ess}}(A_{\theta})=\bigcup_{n\in\N,\,t\ge0}\{n^2\pi^2+te^{-2i\theta}\}.
\end{equation}
ii)  The spectrum of $A_{\theta}$ satisfies $\sigma(A_{\theta})\subset\mathscr{R}_{\mrm{\theta}}^-$ with 
\[
\mathscr{R}_{\mrm{\theta}}^-\coloneqq\{z\in\Cplx\,|\,-2\theta\le\mrm{arg}(z)\le 0\}.
\]
iii)  $\sigma (A_{\theta})\setminus \sigma_{\mrm{ess}}(A_{\theta})$ is discrete and contains only eigenvalues of finite multiplicity.\\[3pt]
iv) Assume that $k^2\in\sigma(A_{\theta})\setminus\sigma_{\mrm{ess}}(A_{\theta})$. Then $k^2$ is real if and only if $k^2\in\mathscr{K}_{\mrm{t}}$. Moreover if $v_{\theta} $ is an eigenfunction associated to $k^2$ such that $\Im m\,k^2<0$, then $v_{\theta}\circ\mathcal{I}_{-\theta}$ is a solution of the original problem (\ref{PbInitial}) whose amplitude is exponentially growing at $+\infty$ or at $-\infty$.
\end{theorem} 
\noindent The interesting point is that now TMs correspond to isolated eigenvalues of $A_{\theta}$, and as such, they can be computed numerically as illustrated below. Note that the elements $k^2$ of $\sigma(A_{\theta})\setminus\sigma_{\mrm{ess}}(A_{\theta})$ such that $\Im m\,k^2<0$, if they exist, correspond to the complex resonances met in (\ref{PbCplxRes}) (whose corresponding generalized eigenfunctions are the quasi normal aka leaky modes). Let us point out that the complex scaling is just a technique to reveal them. Indeed, complex resonances are intrinsic objects defined as the poles of the meromorphic extension from $\{z\in\Cplx\,|\,\Im m\,z>0\}$ to $\{z\in\Cplx\,|\,\Im m\,z\le0\}$ of the operator valued map $z\mapsto (\Delta +z\gamma)^{-1}$ considered in adapted weighted spaces. For more details, we refer the reader to \cite{AsPV00}.

\section{Conjugated complex scaling}\label{SectionUnusual}

Now, we show that by replacing the classical complex scaling by an unusual conjugated complex scaling, and proceeding  as in the previous section, we can define a new complex spectrum which contains the reflectionless values  $k^2\in \mathscr{K}_{\mrm{r}}$ we are interested in. We define the map $\mathcal{J}_{\theta}:\Om\to\Cplx\times(0;1)$ using the following  complex change of variables  
\begin{equation}\label{defDilatationConjugated}
\mathcal{J}_{\theta}(x)=\begin{array}{|ll}
-L+(x+L)\,\overline{\eta}&\mbox{ for }x\le- L\\
x&\mbox{ for }|x|< L\\
+ L+(x-L)\,\eta&\mbox{ for } x\ge L,
\end{array}
\end{equation}
with again $\eta=e^{i\theta}$ ($0<\theta<\pi/2$). Note the important difference in the definitions of $\mathcal{I}_{\theta}$ and $\mathcal{J}_{\theta}$ for $x\le-L$: $\eta$ has been replaced by the conjugated parameter $\overline{\eta}$ to select the incoming modes instead of the outgoing ones in accordance with (\ref{DefReflectionlessMode}). Now, if $u$ is a RM associated to  $k^2\in \mathscr{K}_{\mrm{r}}$, setting  $w_{\theta}=u\circ \mathcal{J}_{\theta}$ (again we extend the definition of the modes $w_n^\pm$ to complex $x$), one has $w_{\theta}=u$ for $|x|< L$ and $w_{\theta}\in \mL^2(\Omega)$ (which is not the case for $u$). The function $w_{\theta}$ satisfies the following equation in $\Om$:
\begin{equation}\label{DefPbFortCPMLs}
\beta_{\theta}\frac{\partial}{\partial x}\Big(\beta_{\theta}\frac{\partial w_{\theta}}{\partial x}\Big)+\frac{\partial^2 w_{\theta}}{\partial y^2}+k^2\gamma w_{\theta}=0
\end{equation}
with $\beta_{\theta}(x)=1$ for $|x|< L$, $\beta_{\theta}(x)=\eta$ for $x\le-L$ and $\beta_{\theta}(x)=\overline{\eta}$ for $x\ge L$. This leads us to define the unbounded operator $B_{\theta}$ of $\mL^2(\Om)$ such that
\begin{equation}\label{defOpConjugatedPMLs}
B_{\theta}w_{\theta}=-\cfrac{1}{\gamma}\left(\beta_{\theta}\frac{\partial}{\partial x}\Big(\beta_{\theta}\frac{\partial w_{\theta}}{\partial x}\Big)+\frac{\partial^2 w_{\theta}}{\partial y^2}\right)
\end{equation}
with domain 
\[
D(B_{\theta})=\{w\in\mH^1(\Om)\,|\,B_{\theta}w\in \mL^2(\Om)\mbox{ and }\partial_{y}w=0\mbox{ on }\partial\Om\}.
\]
As $A_{\theta}$, the operator $B_{\theta}$ is not selfadjoint. Its spectral properties are summarized in the following theorem (see \cite[Thm.\,4.1]{BoCP18}). 
\begin{theorem}\label{thmConjugatedPMLs} 
i) There holds 
\begin{equation}\label{essConjPMLs}
\sigma_{\mrm{ess}}(B_{\theta})=\bigcup_{n\in\N,\,t\ge0}\{n^2\pi^2+te^{-2i\theta},\,n^2\pi^2+te^{+2i\theta}\}.
\end{equation}
ii) The spectrum of $B_{\theta}$ satisfies $\sigma(B_{\theta})\subset\mathscr{R}_{\mrm{\theta}}$ with 
\begin{equation}\label{DefCPML}
\mathscr{R}_{\mrm{\theta}}\coloneqq\{z\in\Cplx\,|\,-2\theta\le\mrm{arg}(z)\le 2\theta\}.
\end{equation}
iii) Assume that $k^2\in\sigma(B_{\theta})\setminus\sigma_{\mrm{ess}}(B_{\theta})$. Then $k^2$ is real if and only if $k^2\in\mathscr{K}_{\mrm{t}}\cup\mathscr{K}_{\mrm{r}}$. Moreover if $w_{\theta}$ is an eigenfunction associated to $k^2$ such that $\pm\Im m\,k^2<0$, then $w_{\theta}\circ\mathcal{J}_{-\theta}$ is a solution of (\ref{PbInitial}) whose amplitude is exponentially growing at $\pm\infty$ and exponentially decaying at $\mp\infty$.
\end{theorem} 
\noindent The important result is that isolated real eigenvalues of $B_{\theta}$ correspond precisely to TMs and RMs. The following proposition (see the proof in \cite[Prop.\,4.2]{BoCP18}), which is directly related to Proposition \ref{PropoTrappedModes}, provides a criterion to determine whether an eigenfunction associated to a real eigenvalue of $B_{\theta}$ is a TM or a RM. 
\begin{proposition}\label{PropositionCarac}
Assume that $(k^2,w_{\theta})\in\R\times \mL^2(\Om)$ is an eigenpair of $B_{\theta}$ such that $k\in(N\pi;(N+1)\pi)$, $N\in\N$. Set
\begin{equation}\label{DefAlephFunction}
\Indicator(w_{\theta})=\sum_{n=0}^N\Big|\int_{0}^1w_{\theta}(-L,y)\varphi_n(y)\,dy\Big|^2
\end{equation}
where $\varphi_n$ is defined in (\ref{defModes}). If $\Indicator(w_{\theta})=0$ then $w_{\theta}\circ\mathcal{J}_{-\theta}$ is a TM ($k^2\in\mathscr{K}_{\mrm{t}}$). If $\Indicator(w_{\theta})>0$ then $w_{\theta}\circ\mathcal{J}_{-\theta}$ is a RM ($k^2\in\mathscr{K}_{\mrm{r}}$). In this case, the incident field $u_i$ defined in (\ref{DefIncidentField}) with
\[
a_n=\int_{0}^1w_{\theta}(-L,y)\varphi_n(y)\,dy,\quad n=0,\dots,N,
\]
yields a scattered field which decays exponentially for $x\le -L$.
\end{proposition}
\noindent The next proposition tells  that  $B_{\theta}$ satisfies the celebrated $\mathcal{PT}$ symmetry property when the obstacle is symmetric with respect to the $(Oy)$ axis. This ensures in particular the stability of simple real eigenvalues, with respect to perturbations of the obstacle satisfying the same symmetry constraint. 
\begin{proposition}\label{PropositionPTsym}
Assume that $\gamma$ satisfies $\gamma(x,y)=\gamma(-x,y)$ for all $(x,y)\in\Om$. Define the operators $\mathcal{P}$, $\mathcal{T}$ such that $\mathcal{P}\varphi(x,y)=\varphi(-x,y)$, $\mathcal{T}\varphi(x,y)=\overline{\varphi(x,y)}$ for $\varphi\in\mL^2(\Om)$. Then one has
\[
\mathcal{PT}B_{\theta}=B_{\theta}\mathcal{PT},
\]
\textit{i.e.} $B_{\theta}$ is $\mathcal{PT}$-symmetric. This implies $\sigma(B_{\theta})=\overline{\sigma(B_{\theta})}$ (if $\lambda$ is an eigenvalue of $B_{\theta}$, $\overline{\lambda}$ as well).
\end{proposition}
\noindent The proof is straightforward observing that the $\beta_{\theta}$ defined after (\ref{DefPbFortCPMLs}) satisfies $\beta_{\theta}(-x,y)=\overline{\beta_{\theta}(x,y)}$.\\
\newline
Finally let us mention a specific difficulty which appears in the spectral analysis of $B_{\theta}$. While Theorem \ref{thmUsualPMLs} guarantees that $\sigma(A_{\theta})\setminus\sigma_{\mrm{ess}}(A_{\theta})$ is discrete, we do not write such a statement for the operator $B_{\theta}$ in Theorem \ref{thmConjugatedPMLs}. A major difference between both operators is that $\Cplx\setminus \sigma_{\mrm{ess}}(A_{\theta})$ is connected whereas $\Cplx\setminus \sigma_{\mrm{ess}}(B_{\theta})$ has a countably infinite number of connected components. As a consequence, to prove that $\sigma (B_{\theta})\setminus \sigma_{\mrm{ess}}(B_{\theta})$ is discrete using the Fredholm analytic theorem, it is necessary to find one $\lambda$ such that $B_{\theta}-\lambda$ is invertible in each of the components of $\Cplx\setminus \sigma_{\mrm{ess}}(B_{\theta})$. In general, in presence of an obstacle, such a $\lambda$ probably exists (proofs for certain classes of $\gamma$ can be obtained working as in \cite{BoCN15}). But for this problem, we can have surprising perturbation results. Thus, if there is no obstacle ($\gamma\equiv1$ in $\Om$), then there holds $\sigma(B_{\theta})=\mathscr{R}_{\mrm{\theta}}$ (see (\ref{DefCPML})): all connected components of $\Cplx\setminus \sigma_{\mrm{ess}}(B_{\theta})$, except the one containing the  complex half-plane $\Re e\,\lambda<0$, are filled with eigenvalues. 
 To show this result, observe that for $k^2\in\sigma(B_{\theta})\setminus\sigma_{\mrm{ess}}(B_{\theta})$, the function $u\circ \mathcal{J}_{\theta}$, with $u(x,y)=e^{ik x}$, is a non-zero element of $\ker\,B_{\theta}$. Notice that this pathological property is also true when $\Om$ contains a family of sound hard cracks (homogeneous Neumann boundary condition) parallel to the $(Ox)$ axis (see the illustration of Figure \ref{HorizontalCracks}).

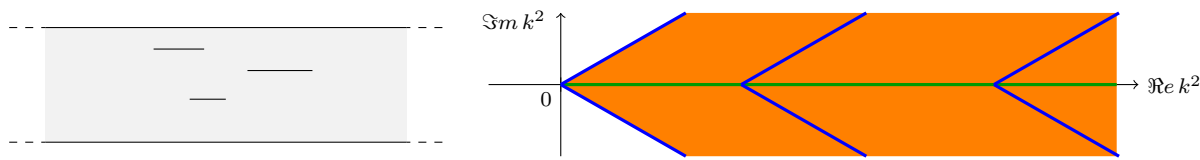
\begin{figure}[!ht]
\centering
\raisebox{0.2cm}{\begin{tikzpicture}[scale=0.95]
\draw[fill=gray!10,draw=none](-2.5,0.4) rectangle (2.5,2);
\draw (-2.5,0.4)--(2.5,0.4);
\draw (-2.5,2)--(2.5,2);
\draw (-0.5,1)--(0,1);
\draw (1.2,1.4)--(0.3,1.4);
\draw (-1,1.7)--(-0.3,1.7);
\draw [dashed](-3,0.4)--(-2.5,0.4);
\draw [dashed](3,0.4)--(2.5,0.4);
\draw [dashed](-3,2)--(-2.5,2);
\draw [dashed](3,2)--(2.5,2);
\end{tikzpicture}}\quad\begin{tikzpicture}[scale=0.95]
\draw (0.5,0) arc (0:-30:0.5) ;
\node at (0.85,-0.2){\tiny $2\theta$};
\draw (0.5,0) arc (0:30:0.5) ;
\node at (0.85,0.2){\tiny $2\theta$};
\draw[->] (-1,0)--(8,0);
\draw[->] (0,-1)--(0,1);
\draw (0,-0.1)--(0,0.1);
\node at (-0.2,-0.2){\scriptsize $0$};
\node at (8.5,0){\scriptsize $\Re e\,k^2$};
\node at (-0.65,0.9){\scriptsize $\Im m\,k^2$};
\draw[fill=green!60!black,draw=none] (1.4,0)  circle[radius=0.08];
\draw[fill=green!60!black,draw=none] (3.8,0)  circle[radius=0.08];
\draw[fill=orange,draw=none] (0,0) -- (1.73,-1) -- (7.7,-1)--(7.7,1)--(1.73,1)--(0,0);
\draw[very thick,green!60!black] (0,0)--(7.7,0);
\draw[very thick,blue,rotate=-30] (0,0)--(2,0);
\draw[very thick,blue,shift={(2.5,0)},rotate=-30] (0,0)--(2,0);
\draw[very thick,blue,shift={(6,0)},rotate=-30] (0,0)--(2,0);
\draw[very thick,blue,rotate=30] (0,0)--(2,0);
\draw[very thick,blue,shift={(2.5,0)},rotate=30] (0,0)--(2,0);
\draw[very thick,blue,shift={(6,0)},rotate=30] (0,0)--(2,0);
\end{tikzpicture}
\caption{Left: waveguide with horizontal cracks. Right: the spectrum of the corresponding $B_{\theta}$ fills the whole sector delimited by $\sigma_{\mrm{ess}}(B_{\theta})$. Here the blue lines represent $\sigma_{\mrm{ess}}(B_{\theta})$.} 
\label{HorizontalCracks}
\end{figure}

\section{Numerical experiments}\label{SectionNumerics}

\subsection{ Classical complex scaling: classical complex resonance modes} 
We first compute the spectrum of the operator $A_{\theta}$ defined in (\ref{defOpClassicalPMLs}) with a classical complex scaling (complex resonance spectrum).  For the numerical experiments, we truncate the computational domain at some distance of the obstacle and use finite elements. This corresponds to the so-called Perfectly Matched Layers (PMLs) method. We refer the reader to \cite{Kalv13} for the numerical analysis of the error due to truncation of the waveguide and discretization. The setting is as follows. We take $\gamma$ such that $\gamma=5$ in $\mathscr{O}=(-1;1)\times(0.25;0.75)$ and $\gamma=1$ in $\Om\setminus\overline{\mathscr{O}}$ (see Figure \ref{figsetting} (a)).  In the definition of the maps $\mathcal{I}_{\theta}$, $\alpha_{\theta}$ (see (\ref{defDilatationx}), (\ref{defEquationDilatation})), we take $\theta=\pi/4$ (so that $\eta=e^{i\pi/4}$) and $L=1$. In practice, we use a $\mrm{P2}$ finite element method in the bounded domain $\Om_{12}\coloneqq\{(x,y)\in\Om\,|\,-12<x<12\}$ with Dirichlet boundary condition at $x=\pm12$ (a Neumann condition would work as well). This leads us to solve the spectral problem
\begin{equation}\label{PbVariaClassicalPmls}
\begin{array}{|l}
\mbox{Find }(\lambda,u^h)\in\Cplx\times\mV_h\setminus\{0\} \mbox{ such that for all }v^h\in\mV_h,\\[3pt]
\dsp\int_{\Om_{12}}\alpha_\theta\partial_xu^h\partial_xv^h+\alpha^{-1}_\theta\partial_yu^h\partial_yv^h\,dxdy=\lambda\int_{\Om_{12}}\gamma\alpha^{-1}_\theta u^hv^h\,dxdy,
\end{array}
\end{equation}
where, similarly to (\ref{DefVh}),  
\begin{equation}\label{DefVhBis}
\mV_h \coloneqq \left\{v\in \mH^1(\Om_{12})\mbox{ such that }v|_{\tau}\in \mathbb{P}_2(\tau)\mbox{ for all }\tau\in\mathcal{T}_h\mbox{ and }v=0\mbox{ at }x=\pm12\right\}.
\end{equation}
Note that to obtain (\ref{PbVariaClassicalPmls}), we exploited that $\alpha_\theta$ depends only on $x$. In Figure \ref{figClassicalPMLs} and below, we display the square root of the spectrum ($k$ instead of $k^2$). The vertical marks on the real axis correspond to the thresholds ($0$, $\pi$, $2\pi$, ...). In accordance with Theorem \ref{thmUsualPMLs}, we observe that $\sqrt{\sigma(A_{\theta})}$ is located in the region $\sqrt{\mathscr{R}_{\theta}^-}=\{z\in\Cplx\,|\,-\theta\le\mrm{arg}(z)\le 0\}$. Moreover, the discretisation of the essential spectrum $\sigma_{\mrm{ess}}(A_{\theta})$ defined in (\ref{essClassicalPMLs}) and forming branches starting at the threshold points appears clearly. Note that a simple calculation shows that $\sqrt{\{n^2\pi^2+te^{-2i\theta},\,t\ge0\}}$ is a half-line for $n=0$ and a piece of hyperbola for $n\ge1$. This is precisely what we get. 
Eigenvalues located on the real axis correspond to trapped modes ($k^2\in\mathscr{K}_{\mrm{t}}$).
 In the chosen setting, which is symmetric with respect to the axis $\R\times\{0.5\}$, one can prove that trapped modes exist \cite{EvLV94}. 
On the other hand, the eigenvalues in the complex plane which are not the discretisation of the essential spectrum correspond to complex resonances. 
\begin{figure}[!ht]
\centering
\includegraphics[width=.49\linewidth]{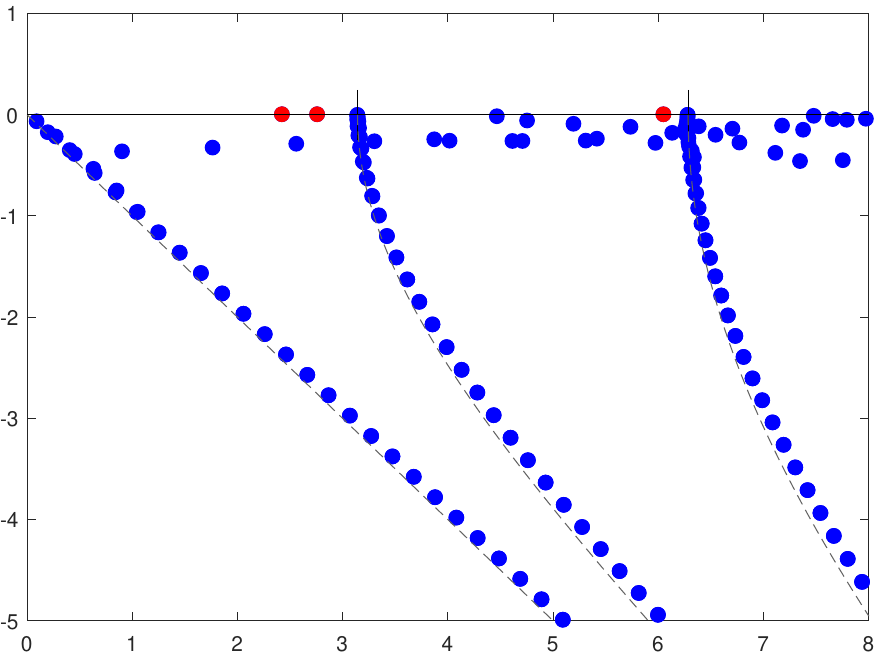}\ \includegraphics[width=.49\linewidth]{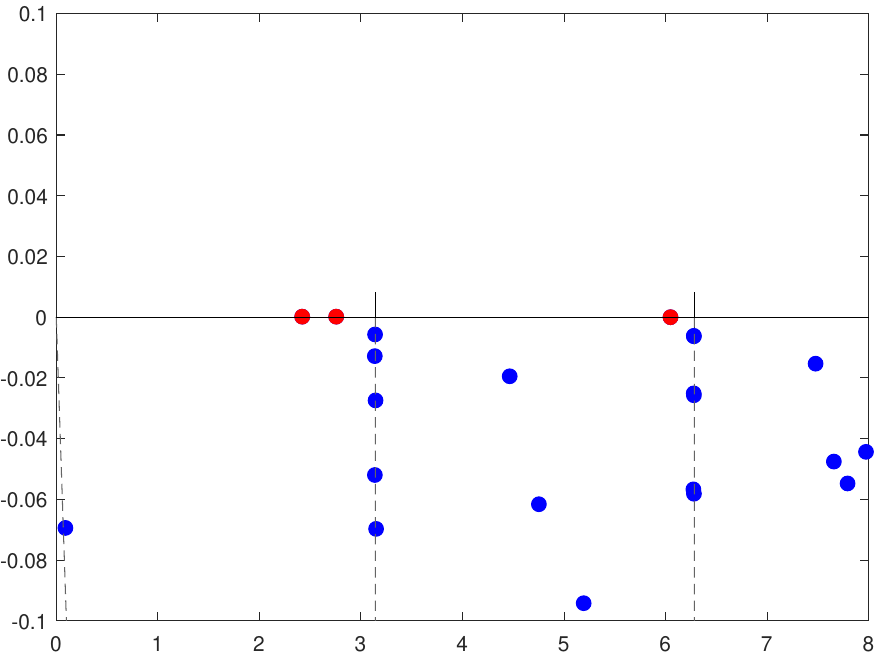}
\caption{Classical complex resonances in the complex $k$ plane corresponding to the spectrum of $A_{\theta}$ for a symmetric obstacle (Figure \ref{figsetting} (a)). The trapped modes are in red, the dashed lines represent the essential spectrum of  $A_{\theta}$ (see (\ref{essClassicalPMLs})). The picture on the right is a zoom-in of that on the left.} 
\label{figClassicalPMLs}
\end{figure}

\subsection{Conjugated complex scaling: reflectionless modes} 
Now we compute the spectrum of the operator $B_{\theta}$ defined in (\ref{defOpConjugatedPMLs}) with a conjugated complex scaling. To proceed, we solve the variational spectral problem
\[
\begin{array}{|l}
\mbox{Find }(\lambda,u^h)\in\Cplx\times\mV_h\setminus\{0\} \mbox{ such that for all }v^h\in\mV_h,\\[3pt]
\dsp\int_{\Om_{12}}\beta_\theta\partial_xu^h\partial_xv^h+\beta^{-1}_\theta\partial_yu^h\partial_yv^h\,dxdy=\lambda\int_{\Om_{12}}\gamma\beta^{-1}_\theta u^hv^h\,dxdy,
\end{array}
\]
with the same $\mV_h$ as in (\ref{DefVhBis}). Note that compared to (\ref{PbVariaClassicalPmls}), $\alpha_\theta$ has been replaced by $\beta_\theta$ (both functions are piecewise constant). First, we use exactly the same symmetric setting (see Figure \ref{figsetting} (a)) as in the previous paragraph. In Figure \ref{figConjugatedPMLs}, we display the square root of the spectrum $\sqrt{\sigma(B_{\theta})}$. Since $\gamma$ satisfies $\gamma(x,y)=\gamma(-x,y)$, according to Proposition \ref{PropositionPTsym} we know that $B_{\theta}$ is $\mathcal{PT}$-symmetric and that therefore its spectrum is stable by conjugation ($\sigma(B_{\theta})=\overline{\sigma(B_{\theta})}$). This is indeed what we obtain. Note that the mesh has been constructed so that $\mathcal{PT}$-symmetry is preserved at the discrete level. $\mathcal{PT}$-symmetry is useful because it guarantees that eigenvalues located close to the real axis which are isolated (no other eigenvalue in a vicinity) are real. Therefore, according to Theorem \ref{thmConjugatedPMLs}, they correspond to trapped modes or to reflectionless modes. 
Remark that, for the same geometry, the spectrum of $B_{\theta}$ (Figure \ref{figConjugatedPMLs}) contains more elements on the real axis than the spectrum of $A_{\theta}$ (Figure \ref{figClassicalPMLs}): the additional elements (the green points in Figure \ref{figConjugatedPMLs}) correspond to reflectionless modes. 

\begin{figure}[!ht]
\centering
\includegraphics[width=.49\linewidth]{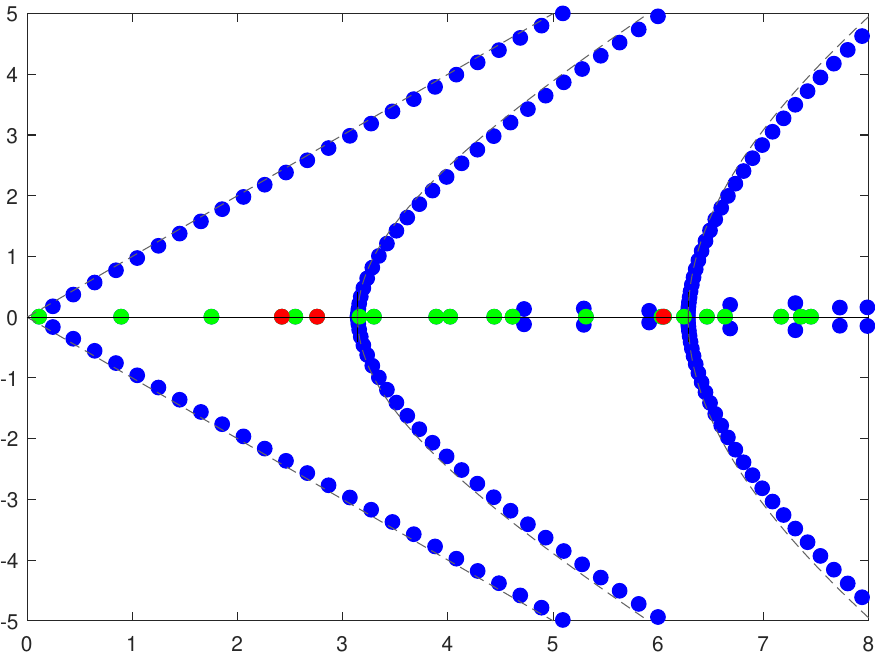}\ \includegraphics[width=.4805\linewidth]{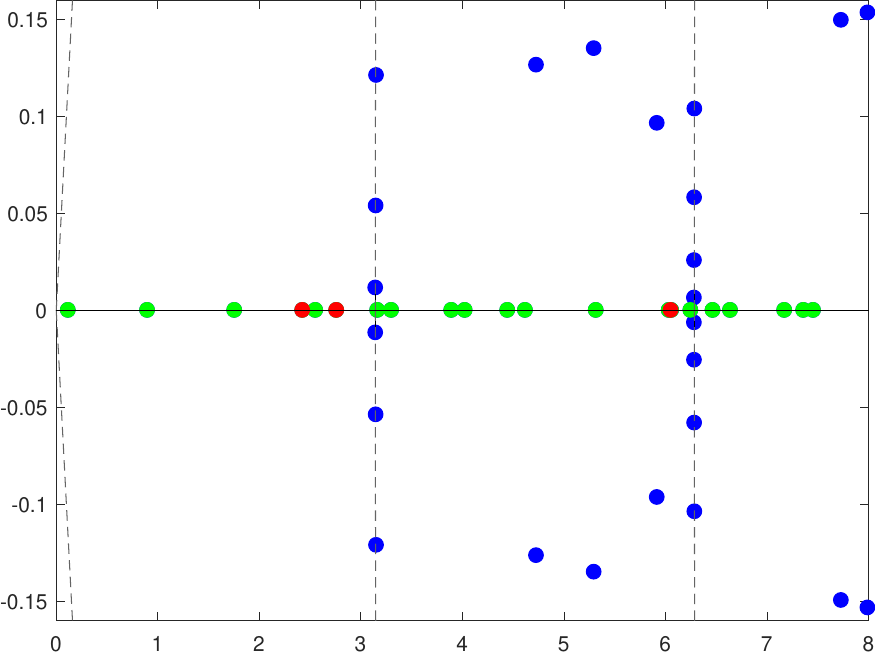}\\[-10pt]
\caption{Reflectionless eigenvalues in the complex $k$ plane corresponding to the spectrum of $B_{\theta}$ for a symmetric obstacle (Figure \ref{figsetting} (a)). The trapped modes in red are the same as in Figure \ref{figClassicalPMLs}. The reflectionless modes are in green and the dashed lines represent the essential spectrum of  $B_{\theta}$ (see (\ref{essConjPMLs})). The picture on the right is a zoom-in of that on the left.}\label{figConjugatedPMLs}
\end{figure}

\noindent In Figure \ref{figEigenfunctions} top, we represent the real part of eigenfunctions associated with seven real eigenvalues of $B_{\theta}$. 
To obtain these pictures, we take $L=4$ in the definition of $\mathcal{J}_{\theta}$ in (\ref{defDilatationConjugated})
and display only the restrictions of the eigenfunctions to $\Om_{L}=\{(x,y)\in\Om\,|\,-L<x<L\}$. 
We recognize two trapped modes (images 3 and 5). 
The other modes are reflectionless modes. In Figure \ref{figEigenfunctions} bottom, we provide the value of the indicator function $\Indicator$ defined in (\ref{DefAlephFunction}) for the seven eigenmodes. The eigenmodes are normalized so that their $\mL^2$ norm  is equal to one. The indicator function $\Indicator$ offers a clear criterion to distinguish between trapped modes and reflectionless modes.

\begin{figure}[!ht]
\centering
\includegraphics[width=.58\linewidth]{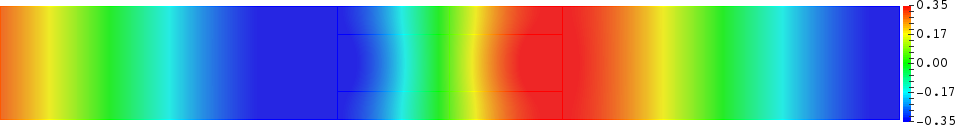}\\[1pt]
\includegraphics[width=.58\linewidth]{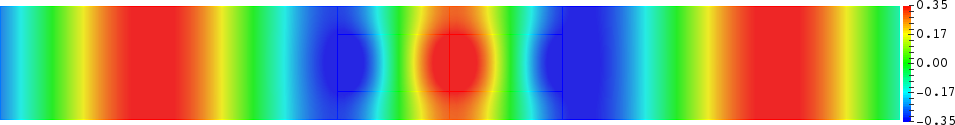}\\[1pt]
\includegraphics[width=.58\linewidth]{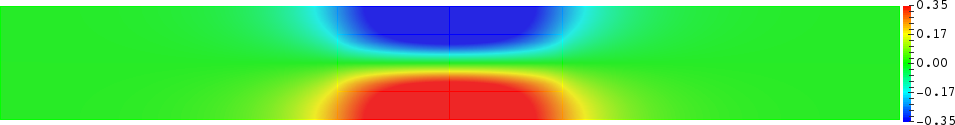}\\[1pt]
\includegraphics[width=.58\linewidth]{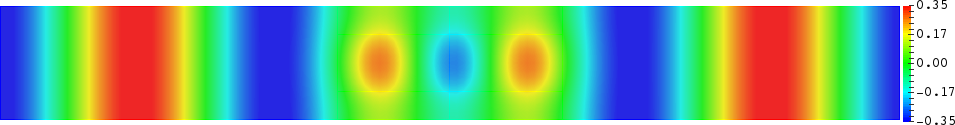}\\[1pt]
\includegraphics[width=.58\linewidth]{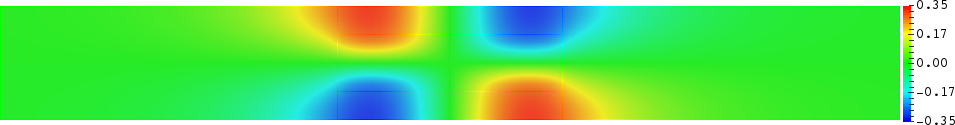}\\[1pt]
\includegraphics[width=.58\linewidth]{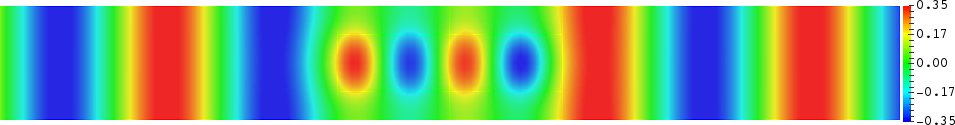}\\[1pt]
\includegraphics[width=.58\linewidth]{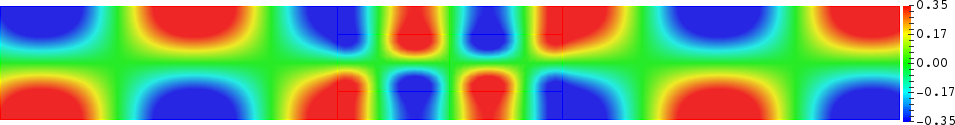}\\[8pt]
\renewcommand{\arraystretch}{1.2}
\begin{tabular}{|>{\centering\arraybackslash}m{1.2cm}|>{\centering\arraybackslash}m{0.7cm}|>{\centering\arraybackslash}m{0.7cm}|>{\centering\arraybackslash}m{1.5cm}|>{\centering\arraybackslash}m{0.7cm}|>{\centering\arraybackslash}m{1.5cm}|>{\centering\arraybackslash}m{0.7cm}|>{\centering\arraybackslash}m{0.7cm}|}
\hline 
\cellcolor{Gray} $k$ &   0.9 &  1.8 &  2.4 &  2.6 &  2.8 &  3.3 &  3.9 \\
\hline
\cellcolor{Gray} $\Indicator(w_{\theta})$  & 0.14 & 0.14 & 8.0\,$10^{-10}$ & 0.14 & 4.3\,$10^{-9}$ & 0.14& 0.14 \\
\hline
\end{tabular}
\caption{Top: real part of eigenmodes associated with real eigenvalues of $B_{\theta}$ from Figure \ref{figConjugatedPMLs}. 
Bottom: value of $k$ and of the indicator function $\Indicator$ for each of these $7$ eigenmodes. The 3rd and the 5th eigenmodes are trapped modes, the five others are reflectionless modes.}
\label{figEigenfunctions}
\end{figure}

\newpage
\noindent In order to inspect the scattering coefficient, we remark that for reflectionless modes associated with wavenumbers $k$ smaller than $\pi$, the incident field $u_i$ in (\ref{DefIncidentField}) decomposes only on the piston mode $w^{+}_0(x,y)=e^{ikx}/\sqrt{2k}$ (monomode regime). In this case the reflection matrix $R(k)$ in (\ref{reflection matrix}) is nothing but the usual reflection coefficient. In Figure \ref{figBalayage}, we thus display the modulus of this coefficient $R(k)$ with respect to $k\in(0.1;3.1)$. As expected, we observe that $R$ vanishes for the values of $k$ obtained in Figure \ref{figEigenfunctions} solving the spectral problem for $B_{\theta}$. 
Of course obtaining the curve $k\mapsto |R(k)|$ is relatively costly and it is precisely 
what we want to avoid by computing the reflectionless $k$ as eigenvalues. Here it is simply a way to check our results. 

\begin{figure}[!ht]
\centering
\includegraphics[width=.49\linewidth]{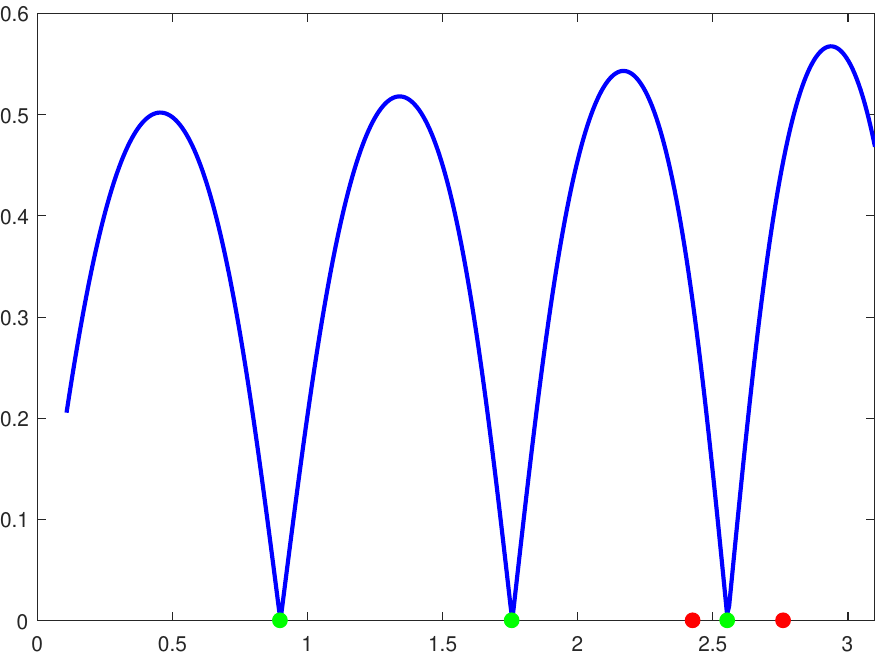}
\caption{Curve $k\mapsto |R(k)|$ (modulus of the reflection coefficient) for $k\in(0.1;3.1)$. The green and red dots represent respectively the reflectionless modes and the trapped modes computed in Figure \ref{figConjugatedPMLs}. We indeed observe that $R(k)$ is null for reflectionless $k$.}
\label{figBalayage}
\end{figure}

\noindent In Figure \ref{figBrokenSym}, we represent the modulus of reflectionless mode eigenfunctions of $B_{\theta}$ associated with one real eigenvalue and two complex conjugated eigenvalues. We observe, and this is true in general, a symmetry with respect to the $(Oy)$ axis for modes corresponding to real eigenvalues which disappears for complex ones. This is the so-called broken symmetry phenomenon which is well-known for $\mathcal{PT}$-symmetric operators (see \textit{e.g.} the review \cite{Bend07}).

\begin{figure}[!ht]
\centering
\includegraphics[width=.65\linewidth]{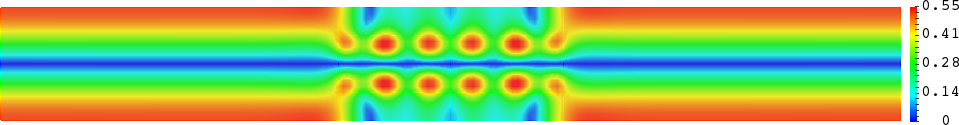}\\[1pt]
\includegraphics[width=.65\linewidth]{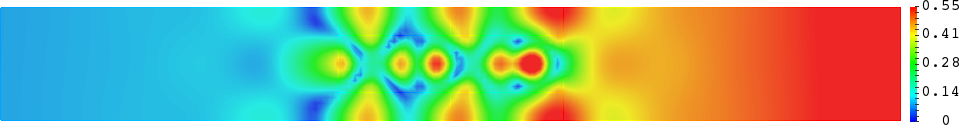}\\[1pt]
\includegraphics[width=.65\linewidth]{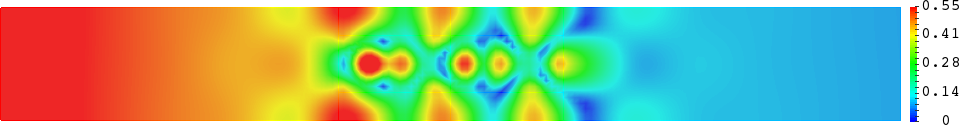}
\caption{Modulus of eigenfunctions of $B_{\theta}$ associated to the eigenvalues $k\approx 5.31$ (top), $k\approx 5.29 - 0.13i$ (middle) and $k\approx 5.29+0.13i$ (bottom) obtained in Figure \ref{figConjugatedPMLs}. The symmetry $x\to-x$ for real $k$ (due to $\mathcal{PT}$-symmetry) disappears for complex $k$. }
\label{figBrokenSym}
\end{figure}~

\noindent Finally in Figure \ref{figSpectrumNonSym}, we display the square root of the spectrum of $B_{\theta}$ for a coefficient $\gamma$ which is not symmetric in $x$ nor in $y$. More precisely, we take $\gamma$ such that $\gamma=5$ in $\mathscr{O}=(-1;0]\times(0.25;0.5)\cup[0;1)\times(0.25;0.75)$ and $\gamma=1$ in $\Om\setminus\overline{\mathscr{O}}$ (see Figure \ref{figsetting} (b)). We observe that the  spectrum is no longer stable by conjugation ($\sigma (B_{\theta})\ne\overline{\sigma (B_{\theta})}$) since
the operator $B_{\theta}$ is not  $\mathcal{PT}$-symmetric, and there is no ``help'' for the eigenvalues to be real.
However, a closer look shows the presence of eigenvalues close to the real axis, in particular for $k\approx1.0 + 0.13i$, $k\approx1.9 + 0.005i$, $k\approx2.5 + 0.02i$, $k\approx2.8 + 0.08i$ and $k\approx3.0 - 0.008i$. In Figure \ref{figNonSymCurve}, we represent $k\mapsto|R(k)|$ for $k\in(0.1;3.1)$ where there is only one propagating mode in the leads.
 It is interesting to note that the above computed complex reflectionless modes (located close to the real axis) affect this curve. More precisely, $k\mapsto|R(k)|$ attains minima for $k\in(0.1;3.1)$ close to the real part of these complex reflectionless modes. Therefore complex reflectionless modes also have  significance for scattering at real frequncies.

\begin{figure}[!ht]
\centering
\includegraphics[width=.49\linewidth]{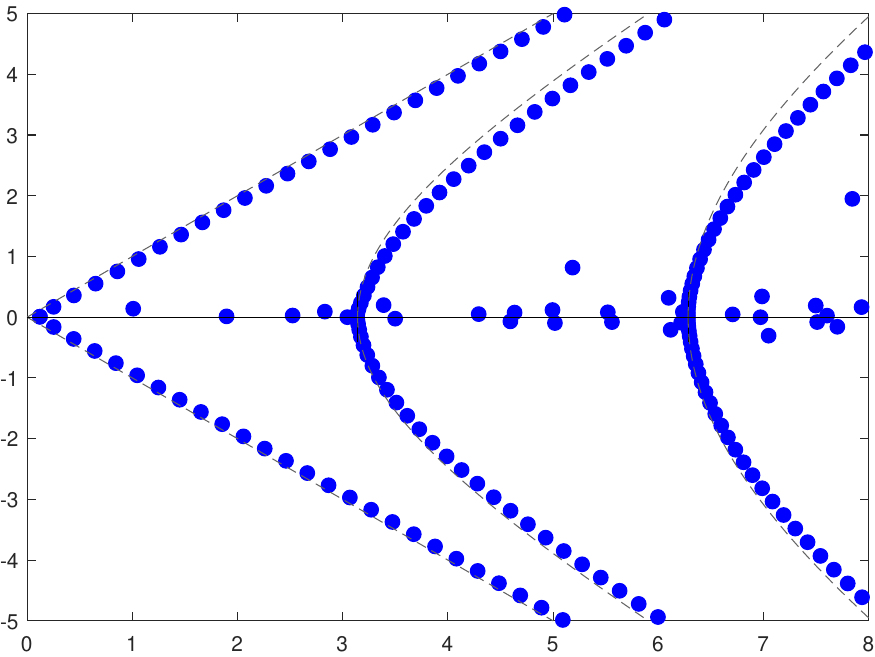}\ \includegraphics[width=.4805\linewidth]{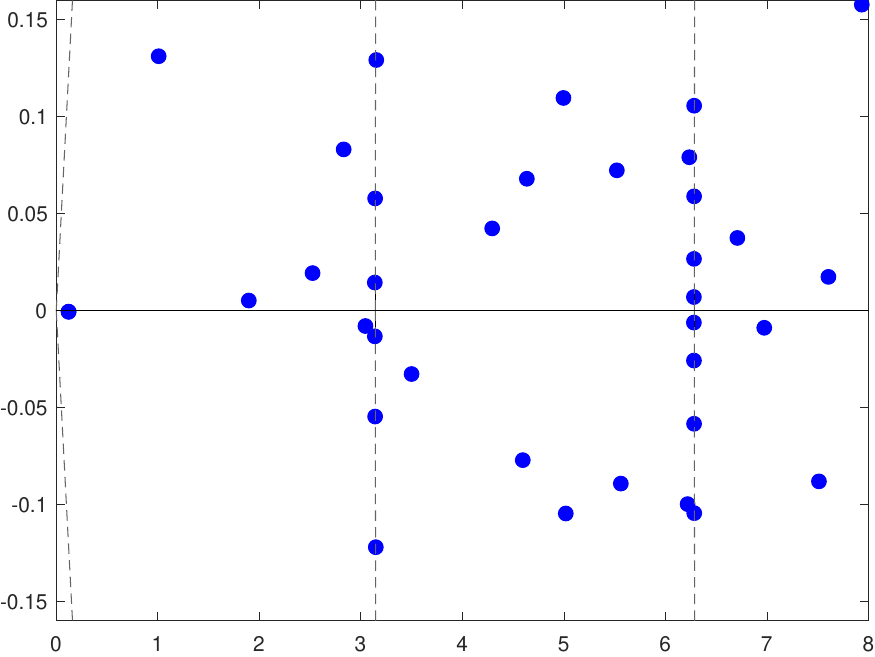}
\caption{Spectrum of $B_{\theta}$ in the complex $k$ plane for a non symmetric obstacle (Figure \ref{figsetting} (b)). The dashed lines represent the essential spectrum of  $B_{\theta}$ (see (\ref{essConjPMLs})). The spectrum is not stable by conjugation. The picture on the right is a zoom-in of that on the left.}
\label{figSpectrumNonSym}
\end{figure}

\begin{figure}[!ht]
\centering 
\includegraphics[width=.49\linewidth]{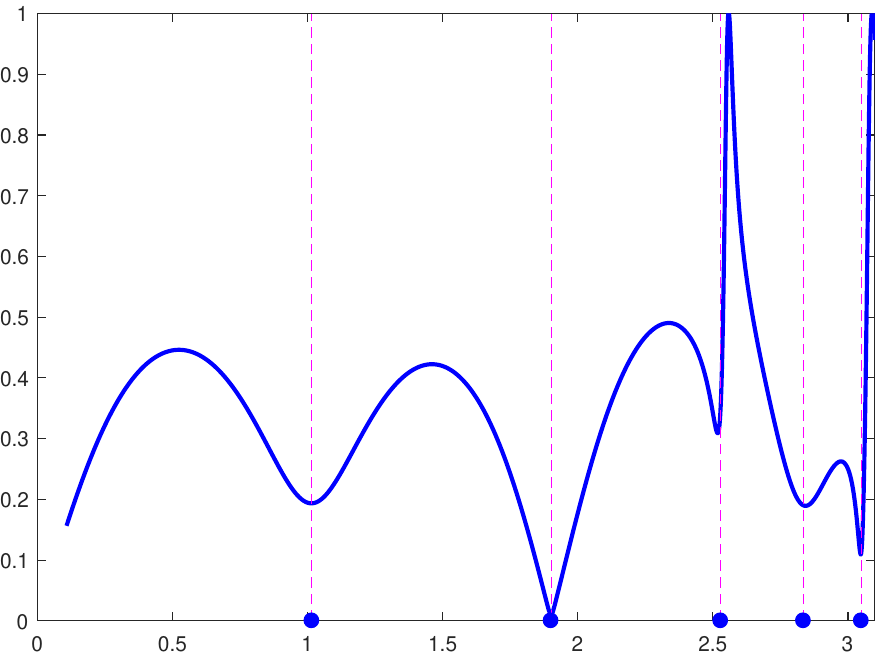}
\caption{Curve $k\mapsto |R(k)|$ (modulus of the reflection coefficient) for $k\in(0.1;3.1)$ and a non symmetric obstacle. The blue dots and the vertical dashed lines correspond to the real parts of the eigenvalues of $B_{\theta}$ located close to the real axis computed in Figure \ref{figSpectrumNonSym}. We observe that $|R(k)|$ is minimal for these particular $k$.}
\label{figNonSymCurve}
\end{figure}

\newpage

\section{Concluding remarks}\label{SectionConclusion}

\noindent Determining the scattering coefficients for a range of frequencies to identify the $k$ for which there are incident fields which produce zero reflection is a tedious work. This chapter shows that reflectionless frequencies can be directly computed as the eigenvalues of a non-selfadjoint operator $B_{\theta}$ (see (\ref{defOpConjugatedPMLs})) with conjugated complex scalings enforcing incoming behavior in the incident lead and outgoing behavior in the other lead. The reflectionless spectrum of this operator $B_{\theta}$ provides a  complementary information to the one contained in the classical complex resonance spectrum associated with quasi normal  modes which decompose only on outgoing waves (see the operator $A_\theta$ in (\ref{defOpClassicalPMLs})). Note that eigenvalues corresponding to trapped modes belong to both the reflectionless spectrum and to the classical complex resonance spectrum because trapped modes do not excite propagating waves. Let us make a few additional comments and highlight future directions as well as open questions. \\[3pt]
$i)$ We have seen that the non-selfadjoint operator $B_{\theta}$ is $\mathcal{P}\mathcal{T}$ symmetric when the structure has mirror symmetry. Interestingly, a direct calculus shows that in the very simple case of a $\mrm{1D}$ transmission problem through a slab of constant index, reflectionless frequencies are all real. This gives an example of a non-selfadjoint $\mathcal{P}\mathcal{T}$ symmetric operator with only real eigenvalues. \\[3pt]
$ii)$ In this work, we investigated scattering problems in waveguides with $N=2$ leads for which two reflectionless spectra exist: one associated with incident waves propagating from the left and another corresponding to incident waves propagating from the right. The more general case with $N$ ($N \ge 2$) leads can be considered as well. Among the total of $2^N$ different spectra with an incoming or an outgoing complex scaling in each lead, two spectra correspond to eigenmodes which decompose on waves which are all outgoing or all incoming. As a consequence, there are $2^N-2$ reflectionless spectra.\\[3pt]
$iii)$ Above we computed reflectionless modes in waveguides containing penetrable obstacles. We can work completely similarly with perturbations of the geometry. In Figure \ref{figRire}, we give two examples of reflectionless modes in such structures. Note that in each case, $B_\theta$ is $\mathcal{P}\mathcal{T}$ symmetric due to the symmetry of the geometry. Other kinds of BCs (Dirichlet, ...) 	and higher dimension ($\mrm{3D}$) can be dealt with similarly.\\

\begin{figure}[!ht]
\centering 
\includegraphics[width=.8\linewidth]{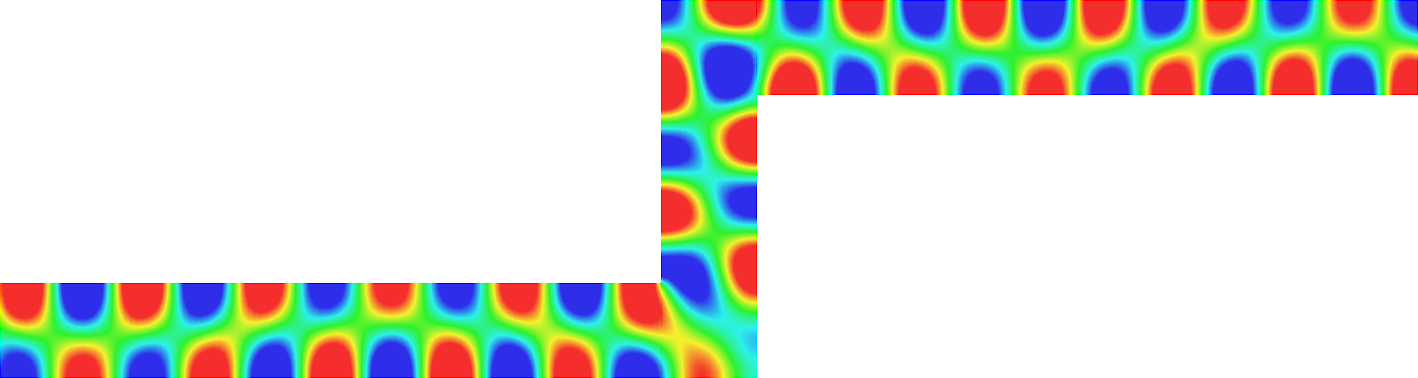}\\[12pt]
\includegraphics[width=.8\linewidth]{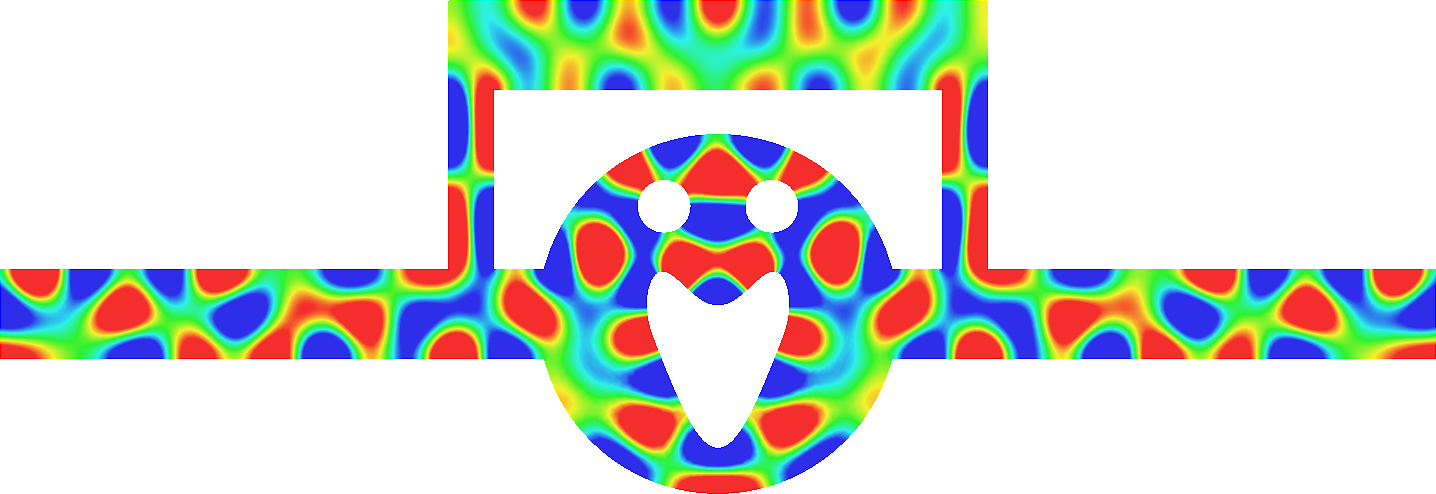}
\caption{Examples of reflectionless modes in waveguides with sound hard walls (Neumann).}
\label{figRire}
\end{figure}

\newpage

\noindent There are many open questions with this work. Here we list just a few of them.\\[3pt]
$iv)$ First, it would be nice to obtain criteria on the index material/geometry ensuring that the spectrum of $B_\theta$ is discrete outside of $\sigma_{\mrm{ess}}(B_\theta)$. In 1D, a rather general condition can be obtained. In higher dimension this is quite open. \\[3pt]
$v)$ On the other hand, could we show that $B_\theta$ has always real eigenvalues corresponding to reflectionless modes, at least for  $\mathcal{P}\mathcal{T}$ symmetric problems? \\[3pt]
$vi)$ The question of the approximation of the spectrum of $B_\theta$ is a field of research in itself because $B_\theta$ is a non-selfadjoint operator. Indeed, the theory of perturbations of non-selfadjoint operators is not well developed and many phenomena can occur. Here it is easy to see in the situation where the waveguide contains only (Neumann) horizontal cracks. In that case we explained before Figure \ref{HorizontalCracks} that $\sigma(B_\theta)$ fills a whole sector. However when we truncate the domain at some distance $L$, we can show that the corresponding spectrum is discrete. This seems particularly pathological. But even when $\sigma(B_\theta)\setminus\sigma_{\mrm{ess}}(B_\theta)$ is discrete, proving that the spectrum of the problem in the truncated waveguide converges to the one of $B_\theta$ when $L$ tends to $+\infty$ is a challenging task. Additionally, it would be interesting to prove that spurious eigenvalues, which would converge to some non physical values, do not exist. In practice, we not only truncate the domain but also approximate the problem by working in finite dimension with finite elements. This is also an approximation which has to be studied. In this context, it seems that the notion of pseudo-spectrum, see \textit{e.g.} \cite{Tref20}, which has been developed to understand the properties of non normal operators, may provide useful information.

\chapter*{Acknowledgments}
The author wishes to express his sincere thanks to J\'er\'emi Dard\'e and Julien Royer for the organization of this convivial  summer school. He also wants to thank the referee for her/his very thorough proofreading and her/his remarks which have helped to improve these lecture notes. Part of this document was conceived at the Isaac Newton Institute for Mathematical Sciences, Cambridge, during the programme \textit{Mathematical theory and applications of multiple wave scattering}, supported by the EPSRC grant no EP/R014604/1. It was a real pleasure to stay at this beautiful institution.

\bibliography{Bibliography}
\bibliographystyle{plain}

\end{document}